\documentclass{amsart}
\usepackage{verbatim,amsfonts,color,mathrsfs,stmaryrd,graphicx, mathdots}
\usepackage{xcolor}
\usepackage{tikz}
\usepackage{tikz-cd}
\usepackage{tkz-euclide}
\usetikzlibrary{patterns}   
\usetikzlibrary{positioning}
\usetikzlibrary{decorations.pathreplacing,calligraphy}

\usepackage{pgfplots}

\pgfplotsset{compat=1.6}
\pgfplotsset{soldot/.style={color=blue,only marks,mark=*}} 
\pgfplotsset{holdot/.style={color=blue,fill=white,only marks,mark=*}}

\theoremstyle{plain}
\newtheorem{Thm}{Theorem}
\newtheorem{Cor}[Thm]{Corollary}
\newtheorem{Conj}[Thm]{Conjecture}

\newtheorem{Lem}[Thm]{Lemma}
\newtheorem{Prop}[Thm]{Proposition}

\theoremstyle{definition}
\newtheorem{Def}[Thm]{Definition}
\newtheorem{Rmk}[Thm]{Remark}
\newtheorem{Eg}[Thm]{Example}

\theoremstyle{remark}

\def\pprec{\mathrel{\scalebox{.9}[1]{$\prec$}\mkern-5mu%
  \scalebox{.4}[1]{$\prec$}\mkern-5.5mu\scalebox{.4}[1]{$\prec$}}} 
 \def\ppreceq{\mathrel{\scalebox{.9}[1]{$\preceq$}\mkern-5mu%
  \scalebox{.4}[1]{$\preceq$}\mkern-5.5mu\scalebox{.4}[1]{$\preceq$}}} 

\usepackage{hyperref}

\begin{document}

 
\title[Continuity of entropy and natural extensions]{Continuity of entropy  for all  $\alpha$-deformations of an infinite class of continued fraction transformations}
\author{Kariane Calta}
\address{Vassar College\\
Poughkeepsie
NY 12604, U.S.A.}  
\email{kacalta@vassar.edu }

\author{Cor Kraaikamp}
\address{Technische Universiteit Delft and Thomas Stieltjes Institute of Mathematics\\ EWI\\ Mekelweg 4\\ 2628 CD Delft, the Netherlands}
\email{c.kraaikamp@tudelft.nl}

\author{Thomas A. Schmidt}
\address{Oregon State University\\Corvallis, OR 97331, USA}
\email{toms@math.orst.edu}
\date{16 March 2023}


\begin{abstract}    We extend the results of our 2020 paper in the Annali della Scuola Normale Superiore di Pisa, Classe di Scienze.
There, we associated to each of an infinite family of triangle Fuchsian groups a one-parameter family of continued fraction maps and  showed that the matching (or, synchronization) intervals are of full measure.   Here, we find planar extensions of each of the maps, and prove the continuity of the entropy function associated to each one-parameter family.   

We also  introduce a notion of ``first pointwise expansive power" of an eventually expansive interval map. We prove that  for every map in one of our  one-parameter families its first pointwise expansive power map  has its natural extension given by the first return of the geodesic flow to a cross section in the unit tangent bundle of the hyperbolic orbifold uniformized by the corresponding group.   We conjecture that this holds for all of our maps. We give  numerical evidence for the conjecture.
\end{abstract}

\maketitle

\tableofcontents

\section{Introduction}  Shortly after the introduction at the end of the 1950s  of the Kolmogorov--Sinai  entropy, hereafter simply \emph{entropy}, Rohlin \cite{Rohlin61} introduced the notion of the natural extension of a dynamical system.
In briefest terms,    the natural extension is the minimal invertible dynamical system of which the original system is a factor under a surjective map. 
Rohlin showed that the original system and its natural extension share entropy values.

At least since the 1977 \cite{NakadaItoTanaka},  natural extensions have been exploited to  determine various properties of continued fractions and related interval maps. In particular,  the continuity of the entropy function on the parameter interval for  Nakada's  $\alpha$-continued fractions, of the 1981  \cite{N}, was proven using natural extensions.  Initiated in \cite{N}, this approach was  pursued in \cite{KraaikampNewClass}, with a further significant step  provided in 2008  by \cite{LuzziMarmi08}, who conjectured that continuity holds on the full parameter interval. The conjecture   was proven  in 2012 independently by  \cite{CT}   and \cite{KraaikampSchmidtSteiner}.    

A common technique in those 2012 proofs is the use of `matching' (also known in the literature as: synchronization, cycle property, etc.).  On a matching parameter interval, for each map the orbits of endpoints of the interval of definition meet one another in finitely many steps and do so in  a common fashion;  planar models of the  natural extensions vary continuously over the interval.   This causes the entropy to change continuously along these parameter intervals.

Katok-Ugarcovici  \cite{KatokUgarcoviciStructure} introduced the family of $(a,b)$-continued fraction maps and determined the full subset of its two-dimensional parameter set for which matching occurs; see also    \cite{CIT,   KatokUgarcovici12}.   These continued fractions, as also the Nakada family, are  associated to the modular group $\text{PSL}_2(\mathbb Z)$.    To each of the triangle Fuchsian groups known as the Hecke groups, \cite{DKS} associated a one-parameter family of continued fraction maps and began the study of their entropy functions; see also \cite{KraaikampSchmidtSmeets}.     In \cite{CaltaKraaikampSchmidt},   to each of another infinite family of triangle Fuchsian groups we associated a one-parameter family of continued fraction maps and determined the set of matching intervals (there called ``synchronization intervals").   

 Series \cite{Series} showed that the regular continued fraction interval map is a factor of  the first return of the geodesic flow to a cross section in the unit tangent bundle of the hyperbolic orbifold uniformized by the modular group.
Luzzi and Marmi \cite{LuzziMarmi08} conjectured that  this is also true for all of the  Nakada $\alpha$-continued fraction maps.  Using the 2012 results on the entropy function, this was proven in \cite{ArnouxSchmidtCross}.    Related results for 
$(a,b)$-continued fractions maps, using coding of geodesics, can be found in \cite{KatokUgarcovici12}.

See 
\S~\ref{ss:quoVadimus}  for more on our motivation for this project.

\subsection{Main results}

The bulk of this paper is devoted to determining explicit planar natural extensions for the infinite family of continuous deformations of interval maps defined and studied in \cite{CaltaKraaikampSchmidt}.   We give detailed descriptions of candidate domains for the planar natural extensions  for maps parametrized along  matching  intervals; these candidates are bijectivity domains for the planar two-dimensional map associated to the given interval map.   This is fairly straightforward for small values of $\alpha$, see \S ~\ref{s:relationsOnHts},~\ref{s:MainForSmallAlps};  but rather intricate for larger values, see \S~\ref{ss:terseForBigAlps}--~\ref{s:largeAlpsRightSide}.   
With these  domains in hand,  we can prove both that they do indeed form the planar natural extensions and furthermore 
that their appropriate limits give natural extensions for the remaining maps.    The explicit natural extensions $\Omega_{n,\alpha}$ are key to proving the following dynamical properties of the systems of the interval maps.   (The basic terminology and notation of \cite{CaltaKraaikampSchmidt} are recalled in \S\S ~\ref{ss:notation} and~\ref{ss:terseRev}.)

 \begin{Thm}\label{t:naturallyErgodicParTout}  Fix $n \ge 3$.   For  $\alpha \in (0,1)$, let $\mu_{\alpha}$  be the normalization of $\mu$ to a probability measure on $\Omega_{\alpha}= \Omega_{n, \alpha}$, and $\mathscr B'_{\alpha}$ denote the Borel sigma algebra on $\Omega_{\alpha}$.      Then the system $(\mathcal T_{\alpha}, \Omega_{\alpha}, \mathscr B'_{\alpha}, \mu_{\alpha})$ is the natural extension of  $(T_{\alpha},\mathbb I_{\alpha},  \mathscr B_{\alpha}, \nu_{\alpha})$, where $\nu_{\alpha}$ is the marginal measure of $\mu_{\alpha}$ and $\mathscr B_{\alpha}$  the Borel sigma algebra on $\mathbb I_{\alpha}$.   Furthermore, both systems are ergodic. 
\end{Thm}  

For the proof of the theorem, as well as the following step in its proof, we rely upon the theoretical tools established in \cite{CKStoolsOfTheTrade}.
  \begin{Prop}\label{p:expansive}   Every $T_{n,\alpha}$ with  $\alpha \in (0,1)$ is eventually expansive.
\end{Prop}

The detailed knowledge of the domains $\Omega_{\alpha}$ allows us to prove the following.
 \begin{Thm}\label{t:continuityOfMass}   For each $n \ge 3$,   the function  $\alpha \mapsto  \mu(\Omega_{\alpha})$ is continuous on $(0,1)$.
 \end{Thm}

 \begin{Thm}\label{t:prodIsConst}     For each $n \ge 3$,  the function  $\alpha \mapsto  h(T_{\alpha}) \mu(\Omega_{\alpha})$ is constant on $(0,1)$.
 \end{Thm}
 
The previous two results immediately imply our main result.   
 \begin{Thm}\label{t:continuityOfEntropy}     The function  $\alpha \mapsto  h(T_{\alpha})$ is continuous on $(0,1)$.
 \end{Thm}

  We introduce the notion of ``first pointwise expansive power" of an eventually expansive interval map, see Definition~\ref{d:firstExpansiveReturn}.
  Numerical evidence leads us to the following conjecture.  See, say,  \cite{Manning} for related  background. 
\begin{Conj}\label{con:returnReturn}   For all $n\ge 3$ and for all $\alpha \in (0,1)$ we conjecture that 
the  first pointwise expansive power of  $T_{n, \alpha}$ has its natural extension given by
the first return of the geodesic flow to a cross section in the unit tangent bundle of the hyperbolic orbifold uniformized by $G_n$.
\end{Conj}

We prove the following, where  the ``set of M\"obius transformations" of a  piecewise M\"obius interval map $T$ denotes the set of M\"obius transformations required to define $T$. 
 \begin{Thm}\label{t:expansiveUisFactorOfFlow}     Suppose that the set of M\"obius transformations of a piecewise M\"obius interval map $T$ and its first pointwise expansive power $U$  generate the same Fuchsian group $G$. Suppose further  that $h(T)\, \mu(\Omega)$ equals the volume of the unit tangent bundle of $G\backslash \mathbb H$  and both are finite.   Then there is a cross section in the unit tangent bundle of $G\backslash \mathbb H$ to geodesic flow, for  which the first return map under the geodesic flow gives the natural extension of the system of $U$.
 \end{Thm}

From this last, the conjecture is equivalent to the following.
 
\begin{Conj}\label{con:Vol}   For all $n\ge 3$ and for all $\alpha \in (0,1)$  one has that
$h(T_{\alpha})\, \mu(\Omega_{\alpha})$ equals the volume of the unit tangent bundle of the hyperbolic orbifold uniformized by the triangle group $G_n$. 
\end{Conj}

We report on numerical corroboration of this latter conjecture in \S~\ref{ss:conjComputation4to12}, and also prove the case when $n=3$.   
 \begin{Thm}\label{t:conjConfirmedNis3}     Conjecture~\ref{con:Vol},  and therefore also Conjecture~\ref{con:returnReturn}, holds when $n=3$.
 \end{Thm}

\phantom{here}\\  
\noindent
{\bf Conventions}   Throughout, we will allow ourselves the minor abuse of using adjectives such as injective, surjective and bijective to mean in each case {\em up to measure zero}, and thus similarly where we speak of disjointness and the like we again will assume the meaning being taken to include the proviso ``up to measure zero" whenever reasonable.  
 For legibility,  we suppress various subscripts, trusting that this causes no confusion for the reader.
  We interchangeably use our old term ``synchronization" and the now more standard ``matching". Finally, it is convenient to refer to the complement of the closure of the synchronization intervals as the ``non-synchronization" points.

\subsection{Overarching  questions}\label{ss:quoVadimus}
We undertook the project of which this is the third paper motivated by a desire to gather more information to hopefully contribute to eventual answers to the following questions. 

Our core motivation springs from the connection between Diophantine approximation and metric number theory.    The central result is that of  L\'evy \cite{L}, which can be seen as relating the entropy value of the regular continued fractions map to a value determined by the convergents to $x$ of almost every $x$.  See say \cite{ArnouxCodage, NakadaLenstra} for explicit formulations of this.      
Thus, for any continued fraction-like map, we desire to know its entropy.  For any family of such maps, we desire to know properties such as continuity of the function that gives  the entropy values of the maps.

The question presents itself:  {\em Are their properties of a family of interval maps which ensure that the associated entropy map is continuous?}      To date, it seems that only direct argumentation, as we give here, has resulted in proofs of such continuity.  We are most familiar with the proofs of continuity of this function in the setting of the Nakada \cite{N} $\alpha$-continued fractions which are directly related to the modular group, and in the setting of their generalizations to similar maps related to each of the Hecke groups, \cite{DKS, KraaikampSchmidtSmeets}.  The proofs there center on planar natural extensions for the interval maps;   with the following Steps: (1)  the natural extensions vary nicely as geometric objects on `certain' parameter intervals, (2) the natural extensions $\Omega$ in fact vary continuously on the full parameter space, (3) the normalizing constant $\mu(\Omega)$ for the mass of each of these natural extensions  then varies continuously, (4) the product of the entropy $h(T)$ of each map with the normalizing constant for the measure of the natural extension defines a constant function on the parameter space: $h(T)\mu(\Omega)$ is constant.

The `certain' parameter intervals mentioned above have to date been intervals on which `matching' (or, synchronization, cycle property, etc.)  occurs.   On such an interval, for each map the orbits of endpoints of the interval of definition meet one another in finitely many steps.  A second question arises:  {\em Are their properties of a family of interval maps, parametrized by an interval, which ensure that matching intervals exist? That they are of full measure?}    We ourselves tend towards the use of the technique of `quilting'  \cite{KraaikampSchmidtSmeets},  which directly shows Step (1) at least locally.  In general, Step (2) requires refined analysis (in the case of `small $\alpha$' our very detailed analysis is contained in some three sections, beginning with \S~\ref{s:relationsOnHts}).  Upon showing finiteness of the normalizing constants, Step (3) follows.   Step (4) results from successful use of Abramov's formula for entropy,  for example see \S~\ref{ss:EM=C}.

Alternately, if the family of maps is associated to a Fuchsian group uniformizing a hyperbolic surface (or orbifold) of finite volume,  one can hope to show that each of the interval maps has an extension given by the return to a cross section by the geodesic flow on the unit tangent bundle of the  surface.   It then follows from the ergodicity of this flow with respect to the standard measure that the integral which appears in Rohlin's formula is equal to the quotient of the volume of the unit tangent bundle by the normalizing constant.   We first saw this in Arnoux's \cite{ArnouxCodage} (to be precise,  in that setting of the regular continued fractions, due to the presence of M\"obius transformations of negative determinant, an extra factor of two belongs in the quotient that we just mentioned).   A third question presents itself: {\em Which interval maps    have extensions given by the return of the geodesic flow to a cross section in the unit tangent bundle of a hyperbolic surface/orbifold?}

In the usual way of  determining an extension of an interval map given by the geodesic flow,
one has that the interval map  itself must be expansive.  
However,  for each of our infinitely many groups there is a nonempty parameter subinterval along which our maps are {\em non-expansive}. 
For each of these, there is a region in the cross section of the corresponding tangent bundle
 where our image unit tangent vectors point so that their geodesics flow backwards; see p.~\pageref{c:oneFactorOfFlow} for details.    Still, experimentation led us to  conjecture that all of our maps are such that the product of their entropy with the normalizing constant  of Step (3) above is equal to the volume of the unit tangent bundle of the orbifold uniformized by the group generated by their set of M\"obius transformations.   Thus our Conjecture~\ref{con:Vol}, see above.    In particular, we realized that there is a way to pointwise ``accelerate"  any eventually expansive map so as to obtain a candidate for  having an extension given by the  {\em first} return map of the geodesic flow to a cross section.
This we defined as the first pointwise expansive power of the map, and Theorem~\ref{t:expansiveUisFactorOfFlow} presented itself to us. Having proved it, we have Conjecture~\ref{con:returnReturn}.  
Of course, we also ask for a proof of the conjectures.

  In \cite{CaltaKraaikampSchmidt} we also defined a one parameter family of interval maps related to each of the groups denoted there as $G_{m,n}$ with $n\ge m>3$.  We fully expect that for each of these that matching occurs on a full measure subset of the parameter interval,  as for the case of $m=3$ proven in \cite{CaltaKraaikampSchmidt},  and even that the tree of words $\mathcal V$ can play an analogous central role in the description of the matching intervals.   Computer experimentation suggests that  as $m\to\infty$ the direct analog of the set of small $\alpha$ occupies an ever decreasing portion of the parameter interval,  and the dynamics of the range where $C$ can appear to large powers seem quite complicated.     We mention in passing that one can also define families of interval maps for each of the groups where the interval maps also involve M\"obius transformations  which are orientation-reversing. Vaguely, we ask: {\em  Can one obtain the analog of the results here for these other families?}

\subsection{Terse review of notation and terminology}\label{ss:notation}   
 We give a quick review of some of the notation and terminology from \cite{CaltaKraaikampSchmidt}.  In fact, we simplify some of that notation as throughout this paper the index $m$ there is set here to equal 3.

\subsubsection{Groups, maps, digits, cylinders, parameter interval partition, ``small" $\alpha$}\label{gpsMapsEtc}   For integer  $n\ge 3$  we let $\nu = \nu_n =  2 \cos \pi/n$ and 

\[ t  := t_{n} = 1 + 2 \cos \pi/n.\]

We use the projective group  $G_{n}$ generated by 
\begin{equation}\label{e:generators} 
A = \begin{pmatrix} 1& t\\
                                      0&1\end{pmatrix},\,  C = \begin{pmatrix} -1& 1\\
                                      -1&0\end{pmatrix}\,.
\end{equation}
(We will have no direct use of the matrix $B = A^{-1}C$ of \cite{CaltaKraaikampSchmidt}.) 

 \bigskip
 Fix $n\ge 3$. For each  $\alpha \in [0,1]$, let  $\ell_0(\alpha) = (\alpha - 1)t,    r_0(\alpha) = \alpha t$ and $\mathbb I_{\alpha} = \mathbb I_{n, \alpha} =  [\ell_0(\alpha), r_0(\alpha))$. 
Our interval maps are piecewise M\"obius,  of the form 
\begin{equation}\label{e:theMaps} 
T_{\alpha} = T_{n,\alpha}:   [\ell_0(\alpha), r_0(\alpha)] \to  \mathbb I_{\alpha}, \;\;\; x \mapsto A^{k} C^{l}\cdot x
\end{equation}
where $\ell\in\{1,2\}$ is minimal such that $C^{l}\cdot x \notin  \mathbb I_{\alpha}$
 and $k$ is the unique integer such that then $ A^{k} C^{l}\cdot x \in \mathbb I_{\alpha}$.     We call $b_{\alpha}(x)= (k,l)$ the $\alpha$-digit of such an $x$, and say that $x$ lies in the cylinder $\Delta_{\alpha}(k,l)$.   To a each element of $\mathbb I_{\alpha}$ we then associate the sequence of its $\alpha$-digits.   The $T_{\alpha}$-orbits of the endpoints $\ell_0(\alpha), r_0(\alpha)$ are of extreme importance in our discussions,  one writes 
 \[\underline{b}_{[1, \infty)}^{\alpha} =  (k_1, l_1)(k_2, l_2)\dots \] for the $\alpha$-digits  of $\ell_0(\alpha)$ and replaces the lower bar by an upper bar,  $\overline{b}_{[1, \infty)}^{\alpha}$ for $r_0(\alpha)$.   We also label entries in these orbits by $\ell_0, \ell_1, \dots$ and $r_0, r_1, \dots$.   The collection of all finite words in the $(k,l)$ which arise in the expansions of any $x \in \mathbb I_{\alpha}$ form the language $\mathcal L_{\alpha}$ for $T_{\alpha}$; any word in $\mathcal L_{\alpha}$ is called $\alpha$-admissible.   There is an order, $\prec$ on the words in the $(k,l)$ such that given $\alpha$, for any $x_1, x_2 \in \mathbb I_{\alpha}$ one has $x_1 \le x_2$ if and only if their $\alpha$-digits sequences are similarly related under  $\prec$.  From this, a word is in $\mathcal L_{\alpha}$ if and only if it lies between $\underline{b}_{[1, \infty)}^{\alpha}$ and $\overline{b}_{[1, \infty)}^{\alpha}$.     See (\cite{CaltaKraaikampSchmidt}, Subsection~1.6) for more about this order.

 The parameter interval is naturally partitioned, 
 $(0, 1] = (0, \gamma_n)\cup [\gamma_n, \epsilon_n) \cup [\epsilon_n, 1]$,  where $\alpha < \gamma_n$ if  $\forall x \in [\ell_0(\alpha), r_0(\alpha)]$ the $\alpha$-digit $(k,l)$ of $x$ has $l=1$ and $\alpha \ge \epsilon_n$ if and only if both $\alpha > \gamma_n$ and the $\alpha$-digit of $\ell_0(\alpha)$ equals  $(k,1)$ with $k \ge 2$,  see (\cite{CaltaKraaikampSchmidt}, Figure~4.1) for plots indicating dynamical behavior related to this partition.   We informally refer to the set of $\alpha< \gamma_n$ as the {\em small $\alpha$},  and all others as  {\em large $\alpha$}.       For small $\alpha$, we use {\em simplified digits:} since $l =1$ we only report the exponent of $A$;  in this setting we use $\underline{d}_{[1, \infty)}^{\alpha}, \overline{d}_{[1, \infty)}^{\alpha}$  in place of $\underline{b}_{[1, \infty)}^{\alpha}, \overline{b}_{[1, \infty)}^{\alpha}$, respectively.  
 
\subsubsection{Expansions given words, synchronization intervals, tree of words;  for small $\alpha$} 
For small $\alpha$,  we define parameter intervals on which an initial portion of the expansion of the $r_0(\alpha)$ are fixed.   That is, there is a common prefix of the  $\overline{d}_{[1, \infty)}^{\alpha}$.     Given a word $v = c_1 d_1 \cdots d_{s-1}c_s$ of positive integers, let  $\overline{S}(v) = \sum_{i = 1}^s\, c_i + \sum_{j=1}^{s-1} d_j$ and for each  $k \in \mathbb N$  let 
\begin{equation}\label{e:overDkV} 
\overline{d}(k,v) = k^{c_1}, (k+1)^{d_1},\cdots,  (k+1)^{d_{s-1}},k ^{c_s}, \;\;\; \text{when}\;\; v = c_1 d_1 \cdots d_{s-1}c_s.
\end{equation}   The corresponding subinterval of parameters is  $I_{k,v} = \{ \alpha\mid \overline{d}_{[1, \overline{S}(v) ]}^{\alpha} = \overline{d}(k,v)\,\}$; in words, this is the set of $\alpha$ for which  the $\alpha$-expansion (in simplified digits)  of 
$r_0(\alpha)$ has $\overline{d}(k,v)$ as a prefix. 
Thus, setting   
\begin{equation}\label{e:rKv}
R_{k,v}  =    (A^kC)^{ c_s}\; (A^{k+1}C)^{d_{s-1}}(A^kC)^{c_{s-1}}\cdots (A^{k+1}C)^{d_1} (A^kC)^{c_1},
\end{equation}    for $\alpha \in I_{k,v}$ one has  $T_{\alpha}^{\overline{S}(v)}(\, r_0(\alpha) \,)= R_{k,v} \cdot  r_0(\alpha)$. The left endpoint of $I_{k,v}$ is denoted $\zeta_{k,v}$, one finds that $R_{k,v} \cdot  r_0( \zeta_{k,v}) =  \ell_0( \zeta_{k,v})$.   We define $J_{k,v} = [\zeta_{k,v}, \eta_{k,v})$  where  with $L_{k,v} = C^{-1}ACR_{k,v}$ we have $L_{k,v}\cdot r_0(\eta_{k,v}) = r_0(\eta_{k,v})$.   Confer (\cite{CaltaKraaikampSchmidt}, Figure~4.2).

A main result of \cite{CaltaKraaikampSchmidt} is that each $J_{k,v}$ is a {\em synchronization interval}\,:   for all $\alpha$ in the interior of $J_{k,v}$, the $T_{\alpha}$-orbits of $\ell_0(\alpha)$ and of $r_0(\alpha)$ meet and do so in a common fashion.   We furthermore showed that the complement in $(0, \gamma_n)$ of the union of the $J_{k,v}$ is of measure zero.   Key to this was determining a maximal common prefix of the $\underline{d}_{[1, \infty)}^{\alpha}$ for $\alpha \in J_{k,v}$.   For that, we used 
$w = w_{n} = (-1)^{n-2}, -2,  (-1)^{n-3},-2$ and for $k$ fixed, we let  $\mathcal C = \mathcal C_k =(-1)^{n-3}, -2, w^{k-1}$ and $\mathcal D = C_{k+1}$, and defined 
$\underline{d}(k,v) = w^k, \mathcal C^{c_1-1} \mathcal D^{d_1}\cdots \mathcal D^{d_{s-1}}\mathcal C^{c_s}, (-1)^{n-2}$.    This is the common prefix, and also 
\begin{equation}\label{e:lowerDetaExpansion}
 \underline{d}{}^{\eta_{k,v}}_{[1,\infty)}  =  \overline{ w^k, \mathcal C^{c_1-1} \mathcal D^{d_1}\cdots \mathcal D^{d_{s-1}}\mathcal C^{c_s}, (-1)^{n-3},-2}\,,
 \end{equation} 
where the overline indicates a period.   We denote the length as a word in $\{-1, -2\}$ of $\underline{d}(k,v)$  by $\underline{S}(k,v)$.  There is an expression for $L_{k,v}$ related to $\underline{d}(k,v)$ in a manner similar to how \eqref{e:rKv} relates $R_{k,v}$ to  $\overline{d}(k,v)$.   The aforementioned synchronization is of the form
\begin{equation}\label{e:synchronizationExplicit}
T_{\alpha}^{\overline{S}(v)+1}(\, r_0(\alpha) \,)= T_{\alpha}^{\underline{S}(v)+1}(\, \ell_0(\alpha) \,).
 \end{equation} 

The words $v$ that we use form a tree, $\mathcal V$.    For general  $v \in \mathcal V$   and nonnegative integers $q$,  we define a new word $\Theta_q(v) = v (v')^q v''$ where $v'$ is defined --- see \eqref{e:vPrimeDef} below --- so that any $\overline{d}_{[1, \infty)}^{\eta_{k,v}}$ equals  $\overline{d}(k,v(v')^{\infty}\,)$ using the obvious extension of our notation,  and $v''$ is an appropriate suffix of $v$.   The tree is rooted at $v=1$, and for words of particularly short length there are special details of the definition of the $\Theta_q$, see   (\cite{CaltaKraaikampSchmidt},  \S 4.2).   There is also a type of self-similarity of $\mathcal V$ which allows the explicit definition of the {\em derived words} operator $\mathscr D$ such that (again for  general $v$)  $\mathscr D\circ \Theta_q (v) = \Theta_q\circ\mathscr D(v)$.    In general $\mathscr D$ decreases the length of words while preserving various properties, and hence assists in induction proofs.   For example, we applied it to prove  that every $v \in \mathcal V$ is a palindrome.

Very similar structures were used in \cite{CaltaKraaikampSchmidt} for the large values of $\alpha$.  We will remind the reader of any of those used in this work, nearer to where they are employed; see \S~\ref{ss:terseRev}.

\subsubsection{Two-dimensional maps}\label{ss:2dMaps}
 
The standard number theoretic planar map associated to a M\"obius transformation $M$ is 
\[
\mathcal{T}_M(x,y) := \bigg( M\cdot x,  N\cdot y\,\bigg) := \bigg( M\cdot x,  RMR^{-1}\cdot y\,\bigg)\, \quad  \text{for}\;\; x \in \mathbb{I}_M,\ y \in \mathbb R\setminus \{(RMR^{-1})^{-1}\cdot \infty\}\,,
\]
where $R = \begin{pmatrix}0&-1\\1&0\end{pmatrix}$.  Thus,   $\mathcal{T}_M(x,y) = (\,M\cdot x, -1/(M\cdot (-1/y))\,)$.
An elementary Jacobian matrix calculation verifies that the measure  $\mu$ on $\mathbb R^2$ given by
\begin{equation}\label{e:muDefd}
d\mu = \dfrac{dx\, dy}{(1 + xy)^2}
\end{equation}
is (locally) $\mathcal{T}_M$-invariant.  

Fixing $n$, each $T_{\alpha}$ is piecewise M\"obius, so that there is some partition of its domain into subintervals,  $\mathbb I_{\alpha} = \cup_\beta\, K_\beta$,  such that  $T_\alpha(x) = M_{\beta}\cdot x$ for all $x \in K_\beta$.   We thus set  
\[
\mathcal{T}_\alpha(x, y) = \bigg( M_\beta \cdot x,  RM_\beta R^{-1}\cdot y\,\bigg)\, \quad  \text{for}\;\; x \in K_\beta,\ y \in \mathbb R\setminus \{N^{-1}\cdot \infty\}\,,
\]
  We thus consider $\mathcal T_{\alpha}$ as being defined on the infinite cylinder fibering over $I_{\alpha}$ and seek a subset on which this map is bijective.

 \subsection{Outline of remainder of paper}\label{ss:outline}   
 The bulk of this paper is devoted to explicitly describing the domains $\Omega = \Omega_{\alpha} = \Omega_{n,\alpha}$, and then proving that each is a bijectivity domain, up to measure zero, for the corresponding two dimensional map.   We give the theorem announcing this in \S~\ref{s:biDom}.  We collect some main background results in \S~\ref{s:elemIdsAndOrder}.  
 
 In \S~\ref{s:relationsOnHts} and~\ref{s:MainForSmallAlps}, we verify that the proposed domains are bijectivity domains, in the case of    small  $\alpha$ in synchronization intervals.    In \S~\ref{s:ContSmall} we treat the remaining small $\alpha$;   the section ends with the proof of the continuity of the function sending small $\alpha$ to the $\mu$-measure of the associated bijectivity domain.

 \S~\ref{ss:terseForBigAlps} tersely reviews notation and terminology for the setting of large $\alpha$. 
 \S~\ref{s:bigLefties} and ~\ref{leftBlocks} are the analogs of  \S~\ref{s:relationsOnHts} and~\ref{s:MainForSmallAlps}, but now for the more challenging setting of large $\alpha$ in the left portion of synchronization intervals (see \S~\ref{ss:terseRev}  for the notion of the left- and right-portions).   \S~\ref{s:largeAlpsRightSide} treats the case of large $\alpha$ in the right portion of synchronization intervals.   \S~\ref{s:limLarge} and  \S~\ref{s:conLarge} together form the analog of \S~\ref{s:ContSmall}  for large $\alpha$.
 
 In \S~\ref{s:dynamProp}  we rely on results of \cite{CKStoolsOfTheTrade} to prove:   our maps are eventually expansive;  the systems on the bijectivity domains give natural extensions  of the interval maps; and, these are all ergodic.   In \S~\ref{s:contEntropy} we use Abramov's formula for the entropy value of an induced system to prove that the function assigning to $\alpha$ the product of the entropy and $\mu$-measure is constant on $(0,1)$. 
 
 Finally, in \S~\ref{ss:theEnd} we treat the setting of the conjectures,  by way of a connection between Rohlin's integral formula for entropy of an interval map and the volume of the unit tangent bundle of the hyperbolic surface/orbifold associated to the map.

\subsection{Thanks}\label{ss:bedankt}   The last named authors thankfully acknowledge the kind hospitality extended by the Mathematics and Statistics Department of Vassar College during a visit where this work was furthered.
 
\section{Bijectivity domains}\label{s:biDom}  The bulk of this paper is involved in providing explicit bijectivity domains for our maps.    For parameter $\alpha$ in the closure of synchronization intervals the behavior of the associated two-dimensional map leads to a partitioning of the domain into an upper $\Omega^+$, and a lower portion $\Omega^-$.

\subsection{Two dimensional bijectivity domain associated to every map}\label{ss:OmegaForSmallAlps}      The overall shape of these upper and lower portions of these domains depends upon whether $\alpha$ is in the interior, or is an endpoint of its synchronization interval; in the case of large $\alpha$, synchronization intervals naturally divide into two portions, as we recall in \S~\ref{ss:terseRev}.   In each case, we give the domain in detail.  In order to announce this fundamental result, here we simply point forward to where these corresponding portions are defined.

\begin{table}[h] 
\begin{tabular}{c|c|c|c}
&$\alpha$ in Interior of $J_{k,v}$ &$\alpha = \zeta_{k,v}$ & $\alpha = \eta_{k,v}$\\
\hline\hline
$\Omega^+$ &Definition~\ref{d:topYvalues}&Definition~\ref{d:topYvalues}&Proposition~\ref{p:OmegaForEta}\\
\hline
$\Omega^-$&Definition~\ref{d:bottomYvalues}&Proposition~\ref{p:OmegaForZeta} &Definition~\ref{d:bottomYvalues}\\
\end{tabular}\\
\phantom{Need vertical space}
\newline
\caption{Location of definitions of $\Omega^{\pm}$ for small $\alpha$ in the closure of a synchronization interval.}
\label{t:defsSmallOm}
\end{table}


\begin{table}[h] 
\begin{tabular}{c|c|c|c|c|c}
&$\eta_{-k,v}< \alpha<\delta_{-k,v}$&$\delta_{-k,v}<\alpha< \zeta_{-k,v}$ &$\alpha = \eta_{-k,v}$ & $\alpha = \zeta_{-k,v}$&$\alpha = \delta_{-k,v}$\\
\hline\hline
$\Omega^+$ &Definition~\ref{d:topYvaluesLargeAlp}&Definition~\ref{d:topYvaluesLargeAlp}&Definition~\ref{d:topYvaluesLargeAlp}&Proposition~\ref{p:OmegaForZetaKneg}&Definition~\ref{d:topYvaluesLargeAlp}\\
\hline
$\Omega^-$&Definition~\ref{d:bottomYvaluesLargeAlps}&Definition~\ref{d:bottomYvaluesPastDelta}&Definition~\ref{d:bottomYvaluesLargeAlps}&Definition~\ref{d:bottomYvaluesPastDelta}&Definition~\ref{d:bottomYvaluesLargeAlps}\\
\end{tabular}
\phantom{Need vertical space}
\newline
\caption{Location of definitions of $\Omega^{\pm}$ for large $\alpha$ in the closure of a synchronization interval.}
\label{t:defsLargeOm}
\end{table}

\medskip

Compare the following with Figures~\ref{f:omegaSmallAlp}, \ref{f:topBottomVerticesSmallAlpsZeta} and \ref{f:topBottomVerticesSmallAlpsEta} for small $\alpha$ values;  and with \ref{f:omegaLargeAlpLessThanDelta_{-k,v}}, \ref{f:omegaLargeAlpBiggerThanDelta_{-k,v}} for large $\alpha$ values. 
\begin{Thm}\label{t:Omega}   Fix $n\ge 3$ and $\alpha \in (0,1)$.  If $\alpha$ is in the closure of a synchronization interval, let $\Omega_{n, \alpha}  := \Omega^{+} \cup \Omega^{-}$ where $\Omega^\pm$ are given according to the locations pointed to in Tables ~\ref{t:defsSmallOm} or ~\ref{t:defsLargeOm}.   For all other $\alpha$, let $\Omega_{n, \alpha}$ be given as in Propositions~\ref{p:fillUpFromZ} or ~\ref{p:fillUpFromW}  according to whether $\alpha<\gamma_n$ or not. 

Then 
$\mathcal T_{n,\alpha}$ is bijective on $\Omega_{n, \alpha}$, up to 
$\mu$-measure zero.   
\end{Thm}

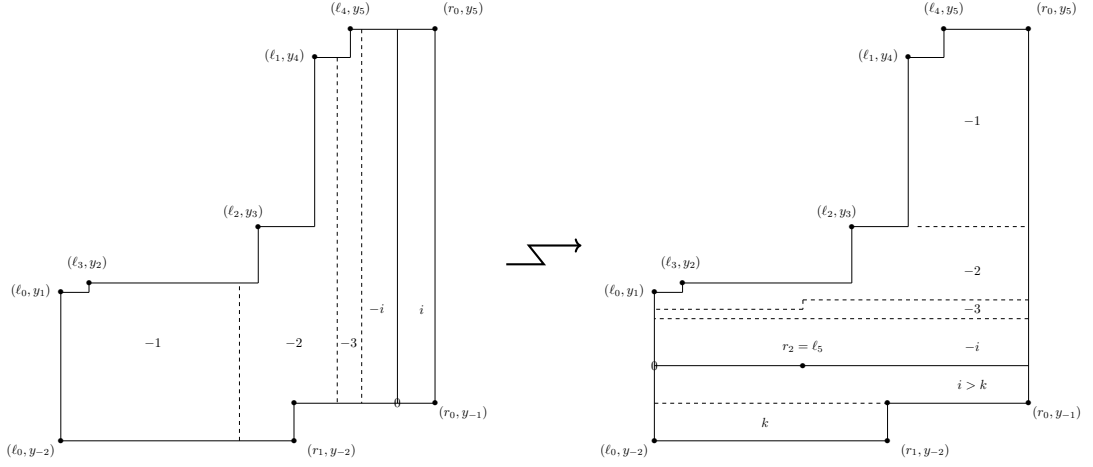
\begin{figure}[h]
\scalebox{.5}{
\noindent
\begin{tabular}{lll}
\begin{tikzpicture}[x=5cm,y=5cm] 
\draw  (-1.79, -0.4)--(-0.55, -0.4); 
\draw  (-0.55, -0.4)--(-0.55, -0.2); 
\draw  (-0.55, -0.2)--(0.2, -0.2); 
\draw  ( 0.2, -0.2)--(0.2,  1.79); 
\draw  (0.2,  1.79)--(-0.25,  1.79); 
\draw  (-0.25,  1.79)--(-0.25,  1.64); 
\draw  (-0.25,  1.64)--(-0.44,  1.64); 
\draw  (-0.44,  1.64)--(-0.44,  0.74); 
\draw  (-0.44,  0.74)--(-0.74,  0.74);
\draw  (-0.74,  0.74)--(-0.74,  0.44);
\draw  (-0.74,  0.44)--(-1.64,  0.44);
\draw  (-1.64,  0.44)--(-1.64,  0.39);   
\draw  (-1.64,  0.39)--(-1.79,  0.39);  
\draw  (-1.79,  0.39)--(-1.79,  -0.4); 
\draw  (0.0, -0.2)--(0.0, 1.79);         
\draw[thin,dashed] ( -0.84, -0.4)--(-0.84, 0.44);        
\draw[thin,dashed] (-0.32, 1.64)--(-0.32, -0.2); 
\draw[thin,dashed] (-0.19, 1.79)--(-0.19, -0.2); 
\node at (-1.3, 0.12) {$-1$};  
\node at (-0.55, 0.12) {$-2$};      
\node at (-0.26, 0.12) {$-3$};      
\node at (-0.11, 0.3) {$-i$};  
\node at (0.13, 0.3) {$i$}; 
 \node at (-1.95, -0.45) {$(\ell_0, y_{-2})$}; 
 \node at (-1.95,  0.40) {$(\ell_0, y_1)$}; 
 \node at ( -1.65, 0.55) {$(\ell_3, y_2)$}; 
 \node at (-0.82, 0.82) {$(\ell_2, y_3)$};
 \node at (-0.6, 1.65)  {$(\ell_1, y_4)$}; 
 \node at (-0.25, 1.9)  {$(\ell_4, y_5)$}; 
 \node at (0.35, -0.25) {$(r_0, y_{-1})$}; 
 \node at (0.35, 1.9) {$(r_0, y_5)$}; 
 \node at (-0.35, -0.45)  {$(r_1, y_{-2})$}; 
 \node at (0, -0.2)  {$0$}; 
 \foreach \x/\y in {-1.79/-0.4,-0.55/-0.4, -0.55 /-0.2,
 0.2/-0.2,  0.2 /1.79, -0.25/1.79, -0.44/1.64, -0.74/0.74,
-1.64/0.44, -1.79/0.39%
} { \node at (\x,\y) {$\bullet$}; }  
\end{tikzpicture}
&
\begin{tikzpicture}[x=2cm,y=5cm] 
\node at (2, 0) {\phantom{here}};
\draw[->, ultra thick] (1.5, 1)--(2, 1) -- (1.8, 1.1) -- (2.5,1.1);
\end{tikzpicture}
&
\begin{tikzpicture}[x=5cm,y=5cm] 
\draw  (-1.79, -0.4)--(-0.55, -0.4); 
\draw  (-0.55, -0.4)--(-0.55, -0.2); 
\draw  (-0.55, -0.2)--(0.2, -0.2); 
\draw  ( 0.2, -0.2)--(0.2,  1.79); 
\draw  (0.2,  1.79)--(-0.25,  1.79); 
\draw  (-0.25,  1.79)--(-0.25,  1.64); 
\draw  (-0.25,  1.64)--(-0.44,  1.64); 
\draw  (-0.44,  1.64)--(-0.44,  0.74); 
\draw  (-0.44,  0.74)--(-0.74,  0.74);
\draw  (-0.74,  0.74)--(-0.74,  0.44);
\draw  (-0.74,  0.44)--(-1.64,  0.44);
\draw  (-1.64,  0.44)--(-1.64,  0.39);   
\draw  (-1.64,  0.39)--(-1.79,  0.39);  
\draw  (-1.79,  0.39)--(-1.79,  -0.4); 
\draw  (-1.79,  0.0)--(0.2, 0.0);  
\draw[thin,dashed] (-1.0, 0.35 )--(0.2, 0.35);         
\draw[thin,dashed] (-1.0, 0.35 )--(-1.0, 0.3 );
\draw[thin,dashed] (-1.0, 0.3 )--(-1.79, 0.3 );
\draw[thin,dashed] (0.2, 0.25)--(-1.79, 0.25 );
\draw[thin,dashed] (-0.39, 0.74)--(0.2, 0.74);
\draw[thin,dashed] (-1.79, -0.2)--(-0.55, -0.2);
\node at (-0.1,  1.3) {$-1$};    
\node at (-0.1,  0.5) {$-2$};      
\node at (-0.1,  0.3) {$-3$};      
\node at (-0.1,  -0.1) {$ i>k$};  
\node at (-0.1,  0.1) {$-i$}; 
\node at (-1, 0.1) {$r_2 = \ell_5$};      
\node at (-1.2,  -0.3) {$k$};
\node at (-1.95, -0.45) {$(\ell_0, y_{-2})$}; 
 \node at (-1.95,  0.40) {$(\ell_0, y_1)$}; 
 \node at ( -1.65, 0.55) {$(\ell_3, y_2)$}; 
 \node at (-0.82, 0.82) {$(\ell_2, y_3)$};
 \node at (-0.6, 1.65)  {$(\ell_1, y_4)$}; 
 \node at (-0.25, 1.9)  {$(\ell_4, y_5)$}; 
 \node at (0.35, -0.25) {$(r_0, y_{-1})$}; 
 \node at (0.35, 1.9) {$(r_0, y_5)$}; 
 \node at (-0.35, -0.45)  {$(r_1, y_{-2})$}; 
\node at (-1.79,  0 ) {$0$}; 
 \foreach \x/\y in {-1.79/-0.4,-0.55/-0.4, -0.55 /-0.2,
 0.2/-0.2,  0.2 /1.79, -0.25/1.79, -0.44/1.64, -0.74/0.74,
-1.64/0.44, -1.79/0.39%
} { \node at (\x,\y) {$\bullet$}; } 
\node at (-1, 0){$\bullet$};
\end{tikzpicture}
%
\end{tabular}
}
\caption{The domain $\Omega_{3, 0.14}$, with blocks $\mathcal B_i$ (see Subsection~\ref{ss:theBlocks}), and their images, both denoted by $i$.   Here $R_{k,v} = AC$ and $L_{k,v} = A^{-1}C A^{-2}CA^{-2}CA^{-1}CA^{-1}$, and $\alpha$ is an interior point of $J_{1,1}$.}%
\label{f:omegaSmallAlp}%
\end{figure}
\bigskip

      For $\alpha$ in a synchronization interval,  the surjectivity of $\mathcal T_\alpha$  requires  that the upper boundary of $\Omega^+$ surjects onto itself, and the lower boundary of $\Omega^-$ does as well;  this  imposes certain relations on the heights of their constituent rectangles.   Using these relations, we solved for the heights, and thus simply define $\Omega^{\pm}$  appropriately.   Surjectivity also requires that   the top and bottom heights satisfy what we call lamination relations, see Subsection~\ref{ss:lamin}.      Section~\ref{s:relationsOnHts} is devoted to verifying that these relations do indeed hold for small $\alpha$ lying in any synchronization interval $J_{k,v}$.   We treat the remaining cases of  $\alpha$ in subsequent sections. 


\subsection{Definition of $\Omega_{\alpha}$  for small $\alpha$ in synchronization interval}\label{ss:theTwoHalves}

We define our domains $\Omega_{\alpha}$ as the union of an upper and lower region.    Compare each of the following two definitions with both Figures~\ref{f:omegaSmallAlp}, ~\ref{f:topBottomVerticesSmallAlps}.  
 
\begin{Def}\label{d:topYvalues}       Fix $n\ge 3, v \in \mathcal V, k\in \mathbb N$ and $\alpha\in (\zeta_{k,v}, \eta_{k,v})$ and let  $\underline{S} = \underline{S}(k,v)$. 
\begin{enumerate}
\item[i.)]     The upper region is the union of rectangles. We let  

\begin{equation}\label{e:omPlusSmallAlp}
 \Omega^{+} = \bigcup_{a=1}^{ \underline{S}+1}\,   K_a \times [0, y_a], 
 \end{equation}
where $K_a$ and $y_a$ are defined below.  

\medskip 
\item[ii.)]  for $1\le a \le \underline{S}$ we let 
\[ K_a =   [\ell_{i_{a}}, \ell_{i_{a+1}}),\]
after having labelled as  $\ell_{i_1}, \ell_{i_2},\dots, \ell_{i_{\underline{S}+1}}$ the first $\underline{S}+1$ elements of the $T_{\alpha}$-orbit of $\ell_0(\alpha)$ in increasing real order.   

\medskip 

\item[iii.)]  Ordering and labelling the elements of the $T_{\zeta_{k,v}}$-orbit of $\ell_0(\zeta_{k,v})$ in exactly the same way,  also let $\tau: \{0, \dots, \underline{S}\}\to \{1, \dots, \underline{S}+1\}$,  be defined by $\tau(j) = a$ exactly when $i_a = j$.    And, finally, for $i \in \{0, \dots,  \underline{S}\}$,   set 
\[y_{\tau(i)} = - \ell_{\underline{S}-i}(\zeta_{k,v}).\] 
\end{enumerate}
\end{Def}

The ordering of the  $\ell_i(\alpha)$ is the same as that of the $\ell_i(\zeta_{k,v})$ because the initial $\underline{S}$ simplified digits of $\ell_0(\alpha)$ are constant on $J_{k,v}$, and by  Lemma~\ref{l:lastEllValueIsLargest} below.

\begin{figure}[h]
\scalebox{.8}{
\noindent
\begin{tabular}{ll}
\begin{tikzpicture}[x=3.5cm,y=5cm] 
\draw  (-.32, 0)--(-.32, 0.2); 
\draw  (-.32, 0.2)--(-.15, 0.2);
\draw  (-.15, 0.2)--(-.15, 0.3);
\draw  (-.15, 0.3)--(-.1, 0.3);
\draw  (-.1, 0.3)--(-.1, 0.4);
\draw  (-.1, 0.4)--(-.05, 0.4);  
\draw  (0.0, 0.45)--(0.0, 0.5);    
\draw  (0.0, 0.5)--(0.15, 0.5); 
\draw  (0.15, 0.5)--(0.15, 0.6); 
\draw  (0.15, 0.6)--(0.2, 0.6);  
\draw  (0.3, 0.65)--(0.3, 0.7);      
\draw  (0.3, 0.7)--(0.35, 0.7);  
\draw  (0.35, 0.7)--(0.35, 0.8);  
\draw  (0.35, 0.8)--(1, 0.8);   
\draw  (1, 0.8)--(1, 0.9); 
\draw  (1, 0.9)--(1.4, 0.9);  
\draw  (1.4, 0.9)--(1.4, 0);  
\draw[thin,dashed] (0.1, 0)--(0.1, 0.5);     
\draw[thin,dashed] (0.55, 0)--(0.55, 0.8); 
\node at (-0.09, 0.12) {$-1$};    
\node at (.35, 0.12) {$-2$};           
\node at (-.55,   0.25) {$(\ell_0, y_1)$};           
\node at (-0.4, 0.55) {$(\ell_{\underline{S}-1}, y_{\tau(\underline{S}-1)})$};   
\node at (0.25,   .9) {$(\ell_{i_{\underline{S}}}, y_{\underline{S}})$}; 
\node at (.9,   1){$(\ell_{\underline{S}}, y_{\underline{S}+1})$}; 
\node at ( 1.45,   1) {$(r_0,  y_{\underline{S}+1})$}; 
 \foreach \x/\y in {-.32/0.2,  0.0/0.5, 0.35/0.8, 1/0.9, 1.4/0.9%
} { \node at (\x,\y) {$\bullet$}; } 
\end{tikzpicture}
&\;\;\;
\begin{tikzpicture}[x=3.5cm,y=5cm] 
\draw  (-.32, 0.6)--(-.32, 0 ); 
\draw  (-.32, 0)--(0, 0 ); 
\draw  (0, 0)--(0, 0.05 ); 
\draw  (0, 0.05)--(0.58, 0.05 ); 
\draw  (0.58, 0.05)--(0.58, 0.1); 
\draw  (0.58, 0.1)--(0.63, 0.1); 
\draw  (0.8, 0.1)--(0.8, 0.15); 
\draw  (0.8, 0.15)--(0.9, 0.15);   
\draw  (0.9, 0.15)--(0.9, 0.2);  
\draw  (0.9, 0.2)--(1.05, 0.2);   
\draw  (1.05, 0.2)--(1.05, 0.25);  
\draw  (1.05, 0.25)--(1.1, 0.25); 
\draw  (1.2, 0.3)--(1.2, 0.35); 
\draw  (1.2, 0.35)--(1.25, 0.35);     
\draw  (1.25, 0.35)--(1.25, 0.4);  
\draw  (1.25, 0.4)--(1.3, 0.4); 
\draw  (1.3, 0.4)--(1.4, 0.4);  
\draw  (1.4, 0.4)--(1.4, 0.6);   
\draw[thin,dashed] (0.55, 0.05)--(0.55, 0.6);  
\draw[thin,dashed] (0.95, 0.2)--(0.95, 0.6); 
\node at ( .75, 0.5) {$k+1$};    
\node at (1.1, 0.5) {$k$};   
\node at (-.38,   -0.15) {$(\ell_0, y_{-\overline{S}-1})$};          
\node at (0.14,   -0.07) {$(r_{\overline{S}}, y_{-\overline{S}-1})$};          
\node at (0.75,  -0.06) {$(r_{\iota}, y_{\beta(\iota)})$};      
\node at (1.28, 0.07) {$(r_{\overline{S}-1}, y_{\beta(\overline{S}-1)})$}; 
\node at (1.5,  .3) {$(r_0,  y_{-1})$}; 
 \foreach \x/\y in {-.32/0,  0.0/0, 0.58/0.05, 1.05/0.2, 1.4/0.4%
} { \node at (\x,\y) {$\bullet$}; } 
 
\end{tikzpicture} 
\end{tabular}
}
\caption{Schematic representations showing the most important vertices of the tops and bottoms of blocks, for general $\alpha \in (\zeta_{k,v}, \eta_{k,v})$.  Labels in part justified by  Lemmas~\ref{l:lastRValueIsLeast} and ~\ref{l:2ndLargestR}.  See Definitions~\ref{d:topYvalues},     ~\ref{d:bottomYvalues} and \ref{d:iota}  for notation at vertices.   Interior label  $i$ denotes partitioning ``block" $\mathcal B_i$, see Subsection~\ref{ss:theBlocks} for definitions.}
\label{f:topBottomVerticesSmallAlps}%
\end{figure}
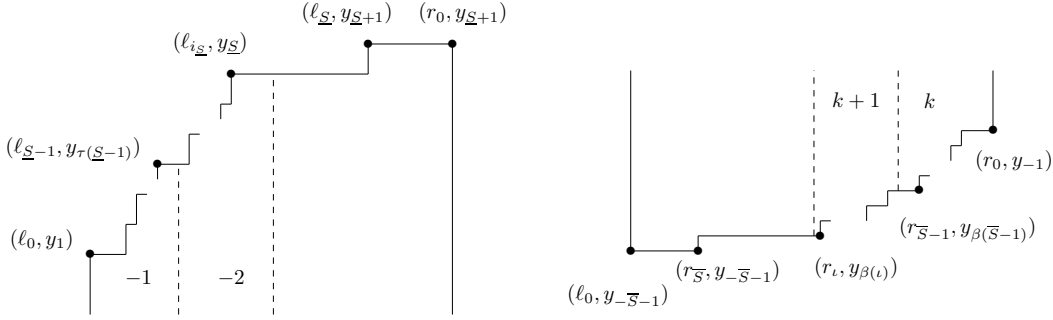

\bigskip 

The lower region is also a union of rectangles.  (In the following we will use  singly subscripted $L$ to denote subintervals, we trust that there will be no confusion with our notation for certain group elements.)

\begin{Def}\label{d:bottomYvalues} 
\phantom{hi} 
\begin{enumerate}
\item[i.)] Let 
\begin{equation}\label{e:omMinusSmallAlp}
 \Omega^{-} = \bigcup_{b=-1}^{ -\overline{S}-1}\,   L_{b} \times [y_b, 0],
\end{equation} 
where the $L_b$ and $y_b$ in the following items. 

\medskip
\item[ii.)]  For $-\overline{S}\le b \le -1$ we let 
  
\[ L_b =   [ r_{j_{b-1}}, r_{j_b}]\]
and $L_{-\overline{S}-1} = [\ell_0, r_{\overline{S}}]$, after having labelled as 
$r_{j_{-1}}, r_{j_{-2}},\dots, r_{j_{-\overline{S}-1}}$ the points      $r_0, r_1, \dots, r_{\overline{S}}$ in {\em decreasing} order as real numbers.

(Lemma~\ref{l:lastRValueIsLeast} below guarantees that there is no nontrivial intersection of the various $L_b$.) 
 
\medskip 
 
\item[iii.)]  Ordering and labelling the elements of the $T_{\eta_{k,v}}$-orbit of $r_0(\eta_{k,v})$ in exactly the same way,  also let  $\beta: \{0, \dots, \overline{S}\}\to \{-1, \dots, -\overline{S}-1\}$,  be defined by $\beta(i) = b$ exactly when $j_b = i$.    And, finally, for $j \in \{0, \dots,  \overline{S}\}$,   set 
\[y_{\beta(j)} = - r_{\overline{S}-j}(\eta_{k,v}).\] 

\end{enumerate}
\end{Def}

\begin{Eg}\label{e:kIs1}  Fix  $n=3$, $k=1$, $v=1$.   As indicated in the caption to Figure~\ref{f:omegaSmallAlp}, we have 
$R_{1,1} = AC$ and $L_{1,1} = A^{-1}C A^{-2}CA^{-2}CA^{-1}CA^{-1}$.    In general, $AR_{k,v}$ fixes $r_0(\zeta_{k,v}) = \zeta_{k,v} t$; here, we find 
$\zeta_{1,1} = (5 - \sqrt{21})/4 \sim 0.104$.  
Similarly,   $L_{k,v}$ fixes $r_0(\eta_{k,v})$ and thus here $\eta_{1,1} = (-1+ \sqrt{21})/20 \sim 0.179$.   Table~\ref{t:orbsInEg} gives the initial portions of the orbits of each of the endpoints of the interval of definition for both $\alpha = 
\zeta_{1,1}$ and $\alpha = \eta_{1,1}$.

\begin{table}  
$\begin{array}{c|c|c||c|c|}
&\zeta_{1,1}&\text{approx. value}&\eta_{1,1}&\text{approx. value}\\
\hline\hline
r_0 &\phantom{\bigg)}(5 - \sqrt{21})/2&\phantom{-}0.209&(-1 + \sqrt{21})/2&\phantom{-}0.358\\
r_1 &\phantom{\big)}(1 - \sqrt{21})/2&-1.791&(5 - \sqrt{21})/2&\phantom{-}0.209\\
\hline
\ell_0 &\phantom{\bigg)}(1 - \sqrt{21})/2&-1.791&(-21+ \sqrt{21})/10&-1.642\\
\ell_1 &\phantom{\big)}(-9+ \sqrt{21})/10&-0.442&(-21+ \sqrt{21})/42&-0.391\\
\ell_2 &\phantom{\big)}(-9+ \sqrt{21})/6&-0.736&(-9+ \sqrt{21})/10&-0.442\\
\ell_3 &\phantom{\big)}(-21+ \sqrt{21})/10&-1.642&(-9+ \sqrt{21})/6&-0.736\\
\ell_4 &\phantom{\big)}(-21+ \sqrt{21})/42&-0.391&\ell_0(\eta_{1,1})\\
\ell_5 &\phantom{\big)}\ell_1(\zeta_{1,1})&\\
\end{array}
$
\phantom{Need vertical space}
\newline
\\
\caption{Endpoint orbits for $\alpha = 
\zeta_{1,1}$ and $\alpha = \eta_{1,1}$ when $n=3$.}
\label{t:orbsInEg}
\end{table}

Since $\underline{S}(1,1) = 4$ and $\overline{S}(1,1) = 1$, the various $y_a, a>0$ and $y_b, b<0$ for the $\Omega_{\alpha}$ with $\alpha \in (\zeta_{1,1}, \eta_{1,1})$ can now be determined, see Table~\ref{t:yValsInEg}. In particular, this second table gives the values of the various ``heights" of rectangles shown in Figure~\ref{f:omegaSmallAlp}.

\begin{table}
$\begin{array}{c|c||c|c|}
&\text{approx. value}&&\text{approx. value}\\
\hline\hline
y_5 &1.791&y_{-1}&-0.209\\
y_4 &1.642&y_{-2}&-0.358\\
y_3&0.736\\
y_2&0.442\\
y_1&0.391\\
\end{array}
$
\phantom{Need vertical space}
\newline
\phantom{Need vertical space}
\newline
\caption{The horizontal boundary values, the ``heights",   of $\Omega_{3,\alpha}$ for $\alpha \in  (\zeta_{1,1}, \eta_{1,1})$.}
\label{t:yValsInEg}
\end{table} 
\end{Eg}

\section{Elementary identities and another order on words}\label{s:elemIdsAndOrder}   We collect a few basic tools for use in proving that the images of regions under the various two-dimensional maps $\mathcal T_M$ are as we claim.    We imagine that the reader will initially skip this section and return to it as its results are applied. 

\subsection{Relating  $y$-coordinate action of $\mathcal T_M$ to $x$-coordinate action} 
In what follows, a left arrow over a word indicates the word taken in reverse order.   Recall from Subsection~\ref{ss:2dMaps} that our two-dimensional maps are defined using the matrix $R$.
 
 \begin{Lem}\label{l:conjByRofAtoPc}       If $X$ is a word in $\{ A^pC, \,p \in \mathbb Z,\}$ and $x_1, x_2 \in \mathbb R$,  then 
 \[ R X R^{-1} \cdot (-x_1) =  - x_2\;\; \text{if and only if}\;\; x_1 = \overleftarrow{X}\cdot x_2\,.\]

In particular,   for 
real $x$, 
\[(R A^pC R^{-1})^{-1}\cdot (-x) = - \; (A^pC\cdot x).\]
 
 \end{Lem}
\begin{proof}    We prove the statements in opposite order.  First, observe that  $A^p C$ is of the form $\begin{pmatrix} a&b\\-b& 0\end{pmatrix}$.   And therefore, using $M^{t}$  to denote the transpose of a matrix $M$,  
$(A^p C)^{t}\cdot(-x) =-\; (A^pC\cdot x)$  for any real $x$.   Since for any $M \in \text{SL}_2(\mathbb R)$,  $RM R^{-1} =   (M^{t})^{-1}$ (projectively), we have that 
$(R  A^p C R^{-1})^{-1} \cdot (-x) =  -\; (A^pC\cdot x)$  for any real $x$.   

Now,  $ R X R^{-1} \cdot (-x_1) =  - x_2$ if and only if   $-x_1 =  (R X R^{-1})^{-1}  \cdot (-x_2)$.   Writing $X = X_1 \cdots X_{|X|}$, we have $(R X R^{-1})^{-1} 
= (R X_{|X|} R^{-1})^{-1} \cdots (R X_1 R^{-1})^{-1}$, and induction gives $ R X R^{-1} \cdot (-x_1) =  - x_2$ if and only if  $x_1 = \overleftarrow{X}\cdot x_2$.    
\end{proof}

\subsection{Relations for lamination} \label{ss:lamin}  
The simplest manner for  regions to have images under $\mathcal T_{n,\alpha}$ that abut one another, thus to {\em laminate}, occurs when regions fibering above two distinct cylinders   are mapped to lie directly one above the other.   We give the simple formulas that imply such behavior in the special cases that arise in this work.

 \begin{Lem}\label{l:aAndCrCommute}      For any $n\ge 3$  the  matrices  $A$ and $CR$ commute.   Furthermore,  $RAR^{-1} = C^{-1} A C$.   
 \end{Lem}
\begin{proof}  First, $CR = \begin{pmatrix} 1 &1\\0&1\end{pmatrix}$  and thus clearly commutes with $A$.      The stated equality then easily follows.   
\end{proof}

 \begin{Lem}\label{l:lamEqsSameCexpon}  Fix $n \ge 3$ and $\alpha \in [0,1]$.  Suppose that $a<b, c<d$ are real numbers.   The rectangle 
$\Delta_{\alpha}(k, l)\times [a,b]$  is mapped by $\mathcal T_{n,\alpha}$ below the image of  $\Delta_{\alpha}(k+1, l)\times [c, d]$ 
so as to share a common horizontal line segment  
 if and only if 
 \[ b = R^{-1}C^{-l}AC^lR\cdot c\,.\]
 In particular,  if $l=1$ then this holds if and only if $b = c + t$.  
 \end{Lem}
\begin{proof} The second coordinate of $\mathcal T_{n,\alpha}(x,y)$ for $x \in \Delta(k, l)$ is given by 
$RA^kC^lR^{-1}\cdot y$, and similarly when $x \in \Delta(k+1, l)$.  Since all of the matrices here are of positive determinant, 
the images of the first rectangle highest horizontal edge of $y$-value $RA^{k}C^lR^{-1}\cdot b$, while the second will have 
lowest horizontal edge $RA^{k+1}C^lR^{-1}\cdot c$, respectively.   Setting these equal and solving gives the first result. 

The second result follows since $CR$ and $A$ commute, and $R^2$ acts as the identity.  
\end{proof}
 
 \begin{Lem}\label{l:lamEqsSameAexpon}   The rectangle $\Delta(k, l)\times [a,b]$ is mapped by $\mathcal T_{n,\alpha}$ above
 the image of $\Delta(k, l+1)\times [e,f]$   so as to share exactly a common horizontal line 
 if and only if 
 \[ a =  RCR^{-1}\cdot f\,.\]
 \end{Lem}


\begin{proof} Here we simply solve  for $a$ in the equation $R A^k C^{l}R^{-1} \cdot a = R A^{k} C^{l+1}R^{-1} \cdot f$.
\end{proof}

  The next lemma is also immediate. 
 
 \begin{Lem}\label{l:lamEqsDifferingAandCexpons}   The rectangle $\Delta(k, l)\times [a, b]$ is mapped by $\mathcal T_{n,\alpha}$ below the image of 
 $\Delta(k+1, l+1)\times [g,h]$   so as to share exactly a common horizontal line 
 if and only if 
 \[ b =  RC^{-l}A C^{l+1} R^{-1}\cdot g\,.\]  
 \end{Lem}

\subsection{An order on words} 
The order on words mentioned in Subsection~\ref{gpsMapsEtc} agrees with the standard ordering of real numbers, and since each $T_{\alpha}$ takes on only values less than $r_0(\alpha)$,  this leads to a corresponding order on words of $\mathcal V$.    This is what might be called an alternating dictionary order.  In fact, we made use this order in    (\cite{CaltaKraaikampSchmidt}, Lemmas~4.16 and  ~6.7)   
 without giving it a distinct notation,  including stating, without proof, a variant of  our Lemma~\ref{l:bothOrders} below.   
\begin{Def}\label{d:orderOnV} 
Let $\pprec$ be the dictionary order  on words $a = a_1 a_2 \cdots a_m$ in natural numbers induced by a dictionary ordering extending (for all words $a,b$ and any $i,j \in \mathbb N$):
$a_{2i-1} \pprec b_{2i-1}$ if $a_{2i-1}  < b_{2i-1}$;    $a_{2j} \pprec b_{2j}$ if $a_{2j}  > b_{2j}$; $b_{2j-1} \pprec a_{2i}$.
\end{Def} 
\bigskip
 
The following result justifies our introduction of the second word order; recall that $\overline{d}(k,v)$ is given in \eqref{e:overDkV}.    
\begin{Lem}\label{l:bothOrders}  Fix $n, k$.     Suppose $a, b \in \mathcal V$, then $\overline{d}(k,a) \preceq \overline{d}(k,b)$ if and only if $a \ppreceq b$.
 \end{Lem}
\begin{proof}     The word order on simplified digits is such that $k+1 \prec k$ and hence $\overline{d}(k,a) \prec \overline{d}(k,b)$ if and only if there is some $i$ such that $\overline{d}(k,a)_{[1,i-1]} =  \overline{d}(k,b)_{[1,i-1]}$ and $\overline{d}(k,b)_{[i,i]} =k$ while $\overline{d}(k,a)_{[i,i]}  = k+1$.    In particular,   either (1) $\overline{d}(k,a)_{[i-1,i-1]} =  \overline{d}(k,b)_{[i-1,i-1]} = k$ and thus there is some $j$ such that $b_{2j-1} \ge a_{2j-1}+1$;  or, (2)   $\overline{d}(k,a)_{[i-1,i-1]} =  \overline{d}(k,b)_{[i-1,i-1]} = k+1$ and thus there is some $j$ such that $a_{2j} \ge b_{2j} +1$.   In either case, we find that $a \pprec b$.   The case of equality is easily verified.
\end{proof}

 A proof of the following can be given by use of the derived words map $\mathscr D$.   We let  $\sigma$ be the left-shift on letters, thus  $\sigma(c_1d_1\cdots d_{s-1}c_s) = d_1\cdots d_{s-1}c_s$.
\begin{Lem}\label{l:selfD}  Each $v \in \mathcal V$ is {\em self-dominant} in the sense that 
  $\sigma^j(v) \ppreceq v$  for all $1\le j\le |v|$. 
 \end{Lem}

 Recall that  
 (\cite{CaltaKraaikampSchmidt}, Definition ~4.8) 
  defines
\begin{equation}\label{e:vPrimeDef}  v' =  \begin{cases} 1 (c_1 - 1) 1c_2 \cdots 1 c_s &\text{if}\;\;\; c_1\neq 1\,,\\
                               (d_1+1) 1 \cdots d_{s-1} 1 &\text{otherwise}\,,
                              \end{cases}
 \end{equation}
and that we show in \cite{CaltaKraaikampSchmidt} that when $c_1>1$ then for all $i$,  $c_i \in \{c_1, c_1-1\}$ as well as that when $c_1 = 1$, each  $d_i \in \{d_1,  d_1+1\}$.    In  (\cite{CaltaKraaikampSchmidt}, Definition ~4.10), 
 we give $\Theta_q(v)$ for short words, and then recursively define values of the operators  $\Theta_q$: whenever $v = \Theta_p(u) = u v''$ for some $p\ge 0$ and some suffix $v''$, we let    $\Theta_q(v) = v (v')^q v''$.

 \begin{Lem}\label{l:wordOrder}     Suppose that $v \in \mathcal V$.  Then $v' \pprec \sigma^j(v)$ for any $j<|v|$.
 \end{Lem}
\begin{proof}  Note that if $v$ is such that no $c_i = c_1-1$ or no $d_i = d_1 + 1$ then it is immediate that every $v' \pprec \sigma^j(v)$.  Furthermore, for $v$ of length one that the result holds is easily verified.    We therefore proceed by induction, assuming our result for a given $v$ and proving that it holds for $\Theta_q(v)$.

Now suppose that $v' \pprec \sigma^j(v)$ holds for all $j< |v|$.   Since, $\Theta_q(v) = v (v')^q v''$  we have  $(\,\Theta_q(v)\,)' = (v')^{q+1} v''$.   We aim to show that $(\,(\Theta_q(v)\,)'\,)^{\infty} \pprec \sigma^j(\, \Theta_q(v)\,)$ for all $j \le |\Theta_q(v)|$.   From our hypothesis,  if 
 $j<|v|$ then   $(\,\Theta_q(v)\,)' \pprec\sigma^j(\,\Theta_q(v)\,)$. Also by this hypothesis we have $v' \pprec v''$,  and hence    $(\,\Theta_q(v)\,)' \pprec \sigma^j(\,\Theta_q(v)\,)$ when $j=|v| + p |v'|$, with $0\le p \le q$.  The inequality for the values $j= 1+|v| + p |v'|$ holds due to the form of $v'$. For   $j = r +|v| + p |v'|$, with $2\le r < |v|$,   
$\sigma^j(\,\Theta_q(v)\,)$  begins with a suffix of $v$, and thus our induction hypothesis shows that the inequality holds for all of these values.  
\end{proof}

Combined with the property of the self-dominance of $v$,  the following result aids in determining explicit orbit elements with extremal value.   See for instance the proof of Lemma~\ref{l:lastEllValueIsLargest}. 
 \begin{Lem}\label{l:maxSuffixPrefix}     Suppose that $v \in \mathcal V$ is of length greater than one.  The parent of $v$ is the longest proper suffix of $v$ that is also a prefix.    
 \end{Lem}
\begin{proof}  We have that $v = u (u')^q u''$ with $q\ge 0$ and $u$ the parent of $v$.    Since every word in $\mathcal V$ is a palindrome, certainly $u$ is a suffix of $v$.   Any longer suffix/prefix must of course admit $u$ as both a prefix and a suffix.  In particular, it suffices to show that there are no internal appearances of $u$ within $v$.    Since this sort of containment is a property preserved by the derived words operator $\mathscr D$, it suffices to check this in the setting of the shortest possible $v$.   But for these short $v$, the result is clear. 
\end{proof}  

\bigskip

\section{Relations on heights of rectangles, for small $\alpha$}\label{s:relationsOnHts}
Throughout this section, unless otherwise stated, we assume that $\alpha$ is a fixed value in the interior of some   synchronization interval $J_{k,v}$.    Recall that a review of definitions and notation is given in Subsection~\ref{ss:notation}.

 We state and prove the main results on the heights of the rectangles of Definitions~\ref{d:topYvalues}  and ~\ref {d:bottomYvalues} whose union gives $\Omega_{\alpha}$.   
For larger values of $\alpha$, see Subsection~\ref{s:RelHtsLargeAlps} and Section~\ref{s:largeAlpsRightSide}.

\subsection{Heights of rectangles, top half}
     
\begin{Def}\label{d:indivStepsTop} 
  Considering $\underline{d}(k,v)$ as a word in $\{-1, -2\}$, for $i \in \{0, \dots, \underline{S}\}$ let $p_i$ be the first letter of $\sigma^i(\,\underline{d}(k,v)\,)$, where $\sigma$ is the left-shift on letters.  That is,  $p_i$ is the first simplified digit of $\ell_i = \ell_i(\alpha)$; equivalently,  $\ell_i \in \Delta_{\alpha}(p_i, 1)$. 
\end{Def} 

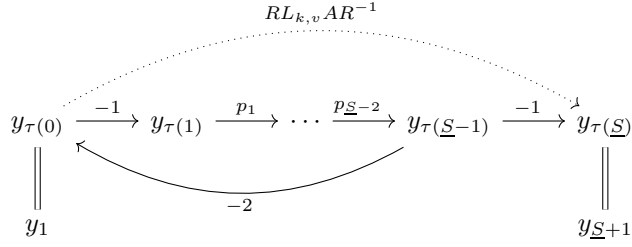
\begin{figure}[t]
\begin{tikzcd}[column sep=2pc,row sep=2pc]
y_{\tau(0)}\ar[pos=0.5]{r}{-1} \ar[bend left, dotted]{rrrr}{R L_{k,v}AR^{-1}}\ar[equal]{d}&y_{\tau(1)}  \ar{r}{p_1}& \cdots  \ar[pos=0.4]{r}{p_{\underline{S}-2}} &y_{\tau(\underline{S}-1)} \ar[pos=0.4]{r}{-1}\ar[bend left]{lll}{-2} &y_{\tau(\underline{S})}\ar[equal]{d}\\
y_1&&&&y_{\underline{S}+1}
\end{tikzcd} 

\caption{Relations on the top heights of rectangles comprising $\Omega^{+}$ for $\alpha \in [\zeta_{k,v}, \eta_{k,v})$.  Here $-1$ denotes $RA^{-1}CR^{-1}$, and $-2$ denotes $RA^{-2}CR^{-1}$, similarly each $p_i$ denotes $RA^{p_i}CR^{-1}$ with notation as in Definition~\ref{d:indivStepsTop}.  The leftmost vertical equality sign is due to Definition ~\ref{d:topYvalues}, the rightmost is due to Lemma~\ref{l:lastEllValueIsLargest}.  Note that up to relabeling vertices and arrows, the top row is the reverse diagram of that for the orbit of $\ell_0(\zeta_{k,v})$; see Figure~\ref{f:zetaLzeroPeriodic}.}
\label{f:topIsPeriodic} 
\end{figure}

\begin{figure}[b]
\begin{tikzcd}[column sep=2pc,row sep=2pc]
\ell_0\ar[pos=0.64]{r}{-1} \ar[bend left, dotted]{rrrr}{L_{k,v}A}& \ell_1 \ar{r} & \cdots  \ar{r}  &\ell_{\underline{S}-1} \ar[pos=0.2]{r}{-1} &\ell_{\underline{S}}\ar[bend left]{lll}{-2}
\end{tikzcd} 
\caption{The orbit of $\ell_0=\ell_0(\zeta_{k,v})$ under $T_{n,\zeta_{k,v}}$.  Here $-1$ denotes $A^{-1}C$, and $-2$ denotes $A^{-2}C$. Compare with Table~\ref{t:orbsInEg}.}
\label{f:zetaLzeroPeriodic} 
\end{figure}
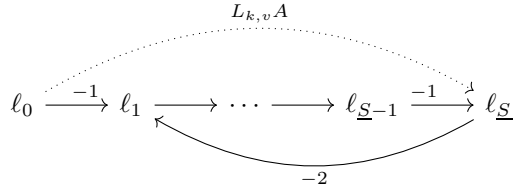

 \medskip  
 Compare the following  with Figures~\ref{f:topBottomVerticesSmallAlps}, \ref{f:topIsPeriodic} and ~\ref{f:zetaLzeroPeriodic}.
 \begin{Lem}\label{l:topValuesFirstRelations}     The following hold:
 \begin{enumerate}
\item[i.)]   For all  $i \in \{0, \dots, \underline{S}-1\}$,  $y_{\tau(i+1)} = R A^{p_i}C R^{-1}\cdot  y_{\tau(i)}$;
 \item[ii.)]  $y_{\tau(\underline{S})} = R L_{k,v} A R^{-1}\cdot  (-\ell_{\underline{S}}(\zeta_{k,v})\,)$;
 \item[iii.)]   $R A^{-2}C R^{-1}\cdot y_{\tau(\underline{S}-1)}  = y_{\tau(0)}\,$;
 \item[iv.)]   $y_1  = R A^{-1}R^{-1} \cdot (-\ell_0(\zeta_{k,v})\,).$
\end{enumerate} 
\end{Lem}
\begin{proof} 
 Definition~\ref{d:topYvalues} gives   $y_{\tau(i)} = - \ell_{\underline{S}-i}(\zeta_{k,v})$ and $y_{\tau(i+1)} = - \ell_{\underline{S}-i-1}(\zeta_{k,v})$.  Of course,  $A^{p_{\underline{S}-i-1}}C\cdot  \ell_{\underline{S}-i-1}(\zeta_{k,v}) = \ell_{\underline{S}-i}(\zeta_{k,v})$.  Since $v \in \mathcal V$  is a palindrome,  it follows that $\underline{d}(k,v)$ is a palindrome as a word in $\{-1, -2\}$; in particular, $p_i = p_{\underline{S}-i-1}$.  Hence, $-y_{\tau(i)} = - A^{p_i}C\cdot (- y_{\tau(i+1)} )$. Lemma~\ref{l:conjByRofAtoPc} now implies $y_{\tau(i+1)} = R A^{p_i}C R^{-1}\cdot  y_{\tau(i)}$.
 
That $y_{\tau(\underline{S})} = R L_{k,v} A R^{-1}\cdot  (-\ell_{\underline{S}}(\zeta_{k,v})\,)$ follows by elementary induction. 

By (\cite{CaltaKraaikampSchmidt}, Lemma~4.5),  
  see also (\cite{CaltaKraaikampSchmidt}, Lemma~4.33),
 $A^{-2}C \cdot \ell_{\underline{S}}(\zeta_{k,v}) = \ell_1(\zeta_{k,v})$.   Lemma~\ref{l:conjByRofAtoPc}   yields  $y_{\tau(0)} = R A^{-2}C R^{-1}\cdot y_{\tau(\underline{S}-1)}\,$. 
 
 Finally, since $\underline{d}(k,v)$ ends with a $-1$,   we find $y_1= y_{\tau(0)} = R A^{-2}C (A^{-1}C)^{-1} R^{-1} \cdot y_{\tau(\underline{S})}$, and thus  $ y_1  = R A^{-1}R^{-1} \cdot (-\ell_0(\zeta_{k,v})\,).$
\end{proof} 

 \begin{Lem}\label{l:finiteMassUpperSmallAlps}  We have  $\mu(\Omega^+)< \infty$.  
 \end{Lem}
\begin{proof} 
To show that $\mu(\Omega^+)< \infty$, we show that for each negative $x \in \mathbb I_\alpha$, the fiber in $\Omega^+$ above it, say $\{x\}\times[0,y_a]$,  lies below the curve $y = -1/x$.     Set $\zeta = \zeta_{k,v}$; by Lemma~\ref{l:topValuesFirstRelations},  $y_1 = RA^{-1}R\cdot(-\ell_0(\zeta)\,) = 1/(t - 1/\ell_0(\zeta))$ and thus $y_1 < -1/\ell_0(\zeta)$ holds if $\ell_0(\zeta)$ lies between the roots of $x^2 + t x - 1 =0$.  Since this interval has left endpoint less than $-t$ and has a positive right endpoint, this condition holds and hence $y_1 < -1/\ell_0(\zeta)$.    Thus, all of the line segment $\{\ell_{i_1}(\zeta)\}\times [0, y_1]$ lies below  $y = -1/x$.       

Recall that $\ell_{i_1} = \ell_0$.  Lemma~\ref{l:topValuesFirstRelations} shows that for each $a$,  there is an $M \in G_{n}$ sending $(\ell_{i_1}(\zeta), y_1)$  to 
$(\ell_{i_a}(\zeta), y_a)$; compare Figures~\ref{f:topIsPeriodic} and \ref{f:zetaLzeroPeriodic}.   A simple check confirms that  any $M$ is such that $\mathcal T_M$ sends the  locus $y=-1/x$ to itself (and since $M^{-1}$ exists this is the only preimage of the locus).    Since every 
$\mathcal T_M$ is both orientation preserving and $\mu$-mass preserving, each   $\{\ell_{i_a}(\zeta)\}\times [0, y_a]$ lies below  $y = -1/x$.  For each $\alpha \in J_{k,v}$, $\ell_{i_a}(\zeta)< \ell_{i_a}(\alpha)$ and it follows that each $\{\ell_{i_a}(\alpha)\}\times [0,y_a]$  lies below  $y = -1/x$ (whenever $\ell_{i_a}(\alpha)<0$), and thus all of $K_a \times [0,y_a]$ does also.  Thus, the result holds.
\end{proof}

\subsection{Heights of rectangles, bottom half}
\begin{Def}\label{d:indivStepsBot} 
Considering $\overline{d}(k,v)$ as a word in $\{k, k+1\}$, for each $j \in \{-1, \dots, -\overline{S}\}$ let $q_j$ be the first letter of $\sigma^j(\overline{d}(k,v)\,)$.  That is,  $q_j$ is the first simplified digit of $r_j$; equivalently,  $r_j \in \Delta(q_j, 1)$. 
\end{Def} 
\bigskip

To simplify typography, we define the following. 
\begin{Def}\label{d:iota} 
If $u\in \mathcal V$ is the parent of $v$, 
let 
\[ \iota = |\overline{d}(k,u)|,\]
where the length of $\overline{d}(k,u)$ is as a word in $\{k, k+1\}$.   When $v=1$, let $\iota = 0$.  
\end{Def} 
\bigskip
  
\begin{figure}
\begin{tikzcd}[column sep=2pc,row sep=2pc]
y_{\beta(0)}\ar[pos=0.64]{r}{k} \ar[bend left, swap]{rrr}{R R_{k,u}R^{-1}}\ar[bend left]{rrrrrr}{R R_{k,v}R^{-1}}&y_{\beta(1)} \ar{r} & \cdots  \ar{r} {k} &y_{\beta(\iota)}\ar[pos=0.2]{r}{k+1} &\cdots  \ar{r}&y_{\beta(\overline{S}-1)}  \ar{r}{k}\ar[bend left]{lllll}{k+1} &y_{\beta(\overline{S}) }
\end{tikzcd}
%
\caption{Relations on the bottom heights of rectangles comprising $\Omega^{-}$ for $\alpha \in (\zeta_{k,v}, \eta_{k,v})$, when $v$ has parent $u$.   Arrow label $k$ denotes $RA^{k}CR^{-1}$, and $k+1$ denotes $RA^{k+1}CR^{-1}$.}
\label{f:bottomOrbit}
\end{figure}
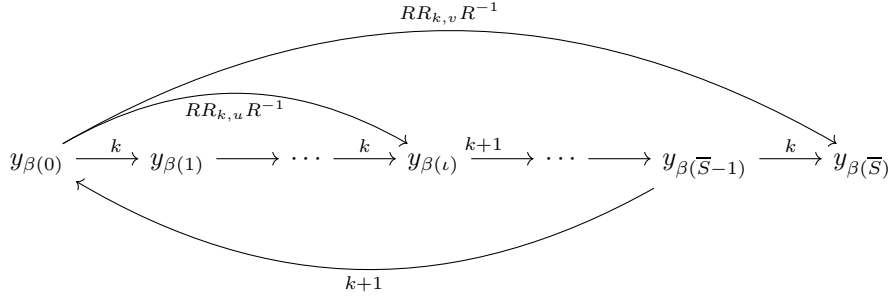

 Compare the following  with Figures~\ref{f:topBottomVerticesSmallAlps}, \ref{f:bottomOrbit} and \ref{f:rZeroEtaOrbit}.  Recall that the details of the construction of $\Omega^-$ are given in Definition~\ref{d:bottomYvalues}.
 \begin{Lem}\label{l:bottomValuesFirstRelations}     The following hold: 
 \begin{enumerate}
 \item[i.)]    For all  $j \in \{-1, \dots, -\overline{S}-2\}$,  $y_{\beta(j+1)} = R A^{q_j}C R^{-1}\cdot  y_{\beta(j)}$;
\item[ii.)]  $y_{\beta(\overline{S})} = R R_{k,v} R^{-1}\cdot y_{\beta(0)}$;
 \item[iii.)]  $y_{\beta(0)} = R A^{k+1}C R^{-1}\cdot y_{\beta(\overline{S}-1)}\,$;
 \item[iv.)]  if $v$ has parent $u$, then   $y_{\beta(\iota)} = R R_{k,u} R^{-1}\cdot y_{\beta(0)}$; 
 \item[v.)]  $y_{\beta(0)} = R A R^{-1}\cdot (- r_{0}(\eta_{k,v})\,)$.
 \end{enumerate} 
\end{Lem}
\begin{proof}  As in the proof of Lemma~\ref{l:topValuesFirstRelations}, the first two statements follow from Lemma~\ref{l:conjByRofAtoPc}.  
By (\cite{CaltaKraaikampSchmidt}, Lemma~4.5), 
  $A^{k+1}C\cdot r_{\overline{S}}(\eta_{k,v}) = r_1(\eta_{k,v})$, and thus Lemma~\ref{l:conjByRofAtoPc} yields the third statement.   
Since $v$ is palindromic and has prefix $u$, it also has suffix $u$ (itself a palindrome).  From this,  Lemma~\ref{l:conjByRofAtoPc} yields the fourth statement.     Finally,  due to the palindromic nature of $v$, we know that the final digit of $\overline{d}(k,v)$ is $k$ and hence 
from (i),  $y_{\beta(\overline{S})} = R A^{k}C R^{-1}\cdot  y_{\beta(\overline{S}-1)}$.   Thus, (iii) gives that $y_{\beta(0)} = R A R^{-1}\cdot y_{\beta(\overline{S}-1)}$.  The result holds, since by definition  $y_{\beta(\overline{S})} =  - r_{0}(\eta_{k,v})$.
\end{proof} 

 \begin{Lem}\label{l:finiteMassLowerSmallAlps}  We have  $\mu(\Omega^-)< \infty$.  
 \end{Lem}
\begin{proof} 
To show that $\mu(\Omega^-)< \infty$, we show that for each positive $x \in \mathbb I_\alpha$, its fiber $\{x\}\times[y_b,0]$ lies above the curve $y = -1/x$.     Set $r_0 = r_0(\eta_{k,v})$; simplifying in  Lemma~\ref{l:bottomValuesFirstRelations} (v) gives   $y_{\beta(0)} = - r_0/( t \,r_0 + 1)$.   Since the interval between the roots of $x^2 - t x - 1 =0$ contains all positive numbers up to at least $t$, it easily follows that   $y_{\beta(0)} > -1/r_0$.   The proof of Lemma~\ref{l:finiteMassUpperSmallAlps}  can be further adjusted {\it mutatis mutandis}.
\end{proof}

\subsection{Ordering  the $\ell_i$; top heights increase}

To further simplify typography, we define the following. 
\begin{Def}\label{d:simplerNotation}
Let  $d_{\alpha}(x)$ denote the sequence of {\em simplified} $\alpha$-digits of $x$.   Similarly,  set $\underline{S} = \underline{S}(k,v)$ and for all $i$ let $\ell_i$ denote $\ell_i(\alpha)$.
\end{Def} 

\medskip 
  
 \begin{Lem}\label{l:lastEllValueIsLargest}     Fix $n \ge 3$ and $k \in \mathbb N$ and $\alpha \in J_{k,v}$.    We have that 
 \begin{enumerate}
 \item [i.)] 
 \[ \ell_{\underline{S}} = \max_{0\le i \le \underline{S}}\, \ell_i\,.\]

\medskip 
 \noindent
Equivalently,   $\tau(\underline{S}) = \underline{S}+1$ and $i_{\underline{S}+1} = \underline{S}$.

\medskip
\item [ii.)] Moreover, 
\[ d_{\alpha}(\ell_{i_{\underline{S}}}) = \begin{cases}  -2, (-1)^{n-3}, -2, (-1)^{n-2},\, d_{\alpha}(\ell_{\underline{S}}) & \text{if}\; v =1 ;\\
                                -2, (-1)^{n-3}, -2, \, \sigma^{n-1}(\,\underline{d}(k,u)\,),\, d_{\alpha}(\ell_{\underline{S}}) & \text{if}\; u \,\text{is the parent of}\; v.
                               \end{cases}
 \]                                 
\end{enumerate} 
\end{Lem}
\begin{proof}  
Since   
 $\ell_{\underline{S}}(\alpha) \ge \ell_{\underline{S}}(\zeta_{k,v})$ for any $\alpha \in J_{k,v}$, and our order of words agrees with the usual order on real numbers, it suffices to show that $ \sigma^j(\,\underline{d}(k,v)\,)\,d_{\alpha}(\, \ell_{\underline{S}}(\alpha)\,) \prec   d_{\alpha}(\, \ell_{\underline{S}}(\zeta_{k,v})\,)$ for all $0\le j< \underline{S}$.    
 
 Letting $\zeta = \zeta_{k,v}$, (\cite{CaltaKraaikampSchmidt}, Lemma~4.33)
 gives
\begin{equation}\label{e:lowerDzetaExpan}
\begin{aligned} 
\underline{d}{}^{\zeta}_{[1,\infty)}  &= -1, \overline{(-1)^{n-3}, -2,  (-1)^{n-3},-2, w^{k-1}, \mathcal C^{c_1-1} \mathcal D^{d_1}\cdots \mathcal D^{d_{s-1}}\mathcal C^{c_s},  (-1)^{n-2}, -2} \\
                                                 &= \underline{d}(k,v), \,\overline{ -2, (-1)^{n-3},-2, \mathcal C^{c_1} \mathcal D^{d_1}\cdots \mathcal D^{d_{s-1}}\mathcal C^{c_s},  (-1)^{n-2}}. 
\end{aligned}
\end{equation}
Hence, for the first result,  it is sufficient to prove that 
 \begin{equation}\label{e:InequalityToProve}
 \sigma^j(\,\underline{d}(k,v)\,)  \prec -2, (-1)^{n-3},-2, \, \mathcal C^{c_1}, \mathcal D^{d_1}\cdots \mathcal C^{c_s}, (-1)^{n-2}
\end{equation}
for all $0\le j< \underline{S}$.  

\bigskip 

Temporarily viewing $\underline{d}(k,v)$ as a word in $a:= (-1)^{n-2}, b:= (-1)^{n-3}, c:=  -2$, and letting $\star$ denote the expansion of $\ell_{\underline{S}}$,     
\[\underline{d}{}^{\alpha}_{[1,\infty)}  = a c [bc (acbc)^{k-1}]^{c_1}\, [bc (acbc)^{k}]^{d_1}\,\cdots\, [bc (acbc)^{k}]^{d_{s-1}}\, [bc (acbc)^{k-1}]^{c_s}\,a \star.\] 
We may use $a\prec b \prec c \preceq \star$  and easily determine that the greatest subword of $\underline{d}(k,v)$ in two of these letters is $cb$, in three is $cbc$ and so forth (of course, if $n=3$ special considerations are necessary).     Thus, when $v=1$ we find that the greatest suffix of $\underline{d}{}^{\alpha}_{[1,\infty)}$  is $cbca\star$.   For the remainder of the proof, we no longer explicitly refer to the letters $a, b, c, \star$, but the reader may still find them of use.
 
Due to the maximality of $c_1$ and the  self-domination of $v$,  we find that when $c_1$ appears only as the initial and final letter of $v$,
$ \sigma^j(\,\underline{d}(k,v)\,)$ has maximal value $-2,  (-1)^{n-3}, -2,\, \mathcal C^{c_1}, (-1)^{n-2}$.   Such appearances of $c_1$ implies that $v =\Theta_q(c_1)$ for $q\ge 1$.  These cases are also easily confirmed. 

For all other $v$, due to Lemma~\ref{l:selfD},     the maximal value of  $ \sigma^j(\,\underline{d}(k,v)\,)$ occurs 
when  $\sigma^{j+n-1}(\,\underline{d}(k,v)\,)$   equals some  prefix of $\mathcal C^{c_1}, \mathcal D^{d_1}\cdots \mathcal C^{c_s}, (-1)^{n-2}$.   Of course, by its very definition,  $\sigma^{j+n-1}(\,\underline{d}(k,v)\,)$ is a suffix of $\underline{d}(k,v)$.  From this simultaneous prefix/suffix property, 
there is an ancestor $u \in \mathcal V$ of $v$ such that $\sigma^{j+n-1}(\,\underline{d}(k,v)\,) = \sigma^{n-1}(\,\underline{d}(k,u)\,)$.  Maximality combined with Lemma~\ref{l:maxSuffixPrefix} implies  that $u$ is in fact the direct parent of $v$.    We have thus proved the second statement. 

 Of course, 
$\underline{d}(k,u)$ ends with $w,\mathcal C^{c_s}, (-1)^{n-2}$.   But, $v$ is strictly longer than $u$, and hence the prefix corresponding to $u$, in $\sigma^{n-1}$ applied to the period of  $\ell_{\underline{S}(\zeta_{k,v})}$,  is followed by a ${-1}^{n-3}, -2$.   Therefore,  inequality ~\eqref{e:InequalityToProve} holds, and the result is proven. 
\end{proof}

 \begin{Lem}\label{l:ellOneIsBigEnough}   The $T_\alpha$-image of the set of $\ell_i$ in $\Delta_{\alpha}(-1,1)$ with $i<\underline{S}$ lies to the right of the $T_\alpha$-image of the set of $\ell_i, i <\underline{S}$ in $\Delta_{\alpha}(-2,1)$.  In particular, 
$\tau(1) = 1+\tau(1+i_{\underline{S}})$.   
\end{Lem}
\begin{proof}   Certainly $\ell_0$ is the smallest $\ell_i$ in $\Delta_{\alpha}(-1,1)$ with $i<\underline{S}$.  By definition,   $\ell_{i_{\underline{S}}}$ is the largest  $\ell_i$ with $i<\underline{S}$;  it certainly lies in $\Delta_{\alpha}(-2,1)$.   Since each $A^{p}C$ is an order preserving function,  it suffices to show that $\ell_1 > \ell_{1+i_{\underline{S}}}$.    We thus refer to Lemma~\ref{l:lastEllValueIsLargest}; our  inequality  when $v = c_1$ can be directly verified.  For general $v$, since the expansion of $\ell_{i_{\underline{S}}}$ begins with a $-2$,  inequality \eqref{e:InequalityToProve} implies that  
$\ell_{i_{\underline{S}}+1}$ has expansion less than the expansion of $\ell_1$.  
The result thus holds. 
\end{proof}

 \begin{Lem}\label{l:oppositeOrderEllsSmallAlp}   For all $0\le i < \underline{S}\,$, 
 \[ \tau(i) + \tau(\underline{S}-i) = \underline{S} + 2.\]
  \end{Lem}
\begin{proof}  Since certainly $\tau(0) = 1$,  Lemma~\ref{l:lastEllValueIsLargest} yields the result when $i=0$.   
 
Lemma~\ref{l:ellOneIsBigEnough} implies that $\tau(i+1)-\tau(i)$ is determined by the simplified digit of $\ell_i$.  Indeed, if $\ell_i \in \Delta_{\alpha}(-1,1)$, then  $\tau(i+1) = \tau(i) + (\tau(1) - \tau(0))$. Similarly,  if $\ell_i \in \Delta_{\alpha}(-2,1)$, then  $\tau(i+1) = \tau(i) + (\tau(1+i_{\underline{S}}) -\tau(i_{\underline{S}}))$.

Due to the palindromic nature of $\underline{d}(k,v)$,  $\ell_{\underline{S}-i-1}$ and $\ell_i$ share the same digit.   Therefore,  $\tau(\underline{S}-i)-\tau(\underline{S}-i-1) = \tau(i+1)-\tau(i)$.  The result thus follows from induction. 
\end{proof} 

 \begin{Lem}\label{l:topHeightsIncrease}   For $1\le a \le \underline{S}$,  
 \[y_{a+1} > y_a\,.\]  
\end{Lem}
\begin{proof}  By definition,  there is an $i$ such that $y_a = y_{\tau(i)} = - \ell_{\underline{S}-i}(\zeta_{k,v})$.   Due to the negative sign, this is the 
$\underline{S}+2 - \tau(\underline{S}-i) ^{\text{th}}$ of the values in $\{-\ell_{i'}(\zeta_{k,v})\,|\,i'<\underline{S}\}$ ordered  as increasing real numbers.  
 Lemma~\ref{l:oppositeOrderEllsSmallAlp} yields that this is the  $\tau(i)^{\text{th}}$ value.   But, $a = \tau(i)$, and thus the position in this ordered set increases with $a$. 
\end{proof}

 \begin{Cor}\label{c:obviousYsubUnderlineS}    We have 
 \[ y_{\underline{S}} = - \ell_{i_2}(\zeta_{k,v}).\]    
 \end{Cor}
\begin{proof} The lemma yields that $y_{\underline{S}}$ is the second largest height.   Since these heights are the various $-\ell_i(\zeta_{k,v})$, this second largest height is indeed negative one times the second smallest $\ell_i(\zeta_{k,v})$.
\end{proof}

\subsection{Ordering the  $r_j$; bottom heights increase} 

 \begin{Lem}\label{l:lastRValueIsLeast}     We have that 
 \[ r_{\overline{S}} = \min_{0\le j \le \overline{S}}\, r_j\,.\]
 Equivalently,   $\beta(\overline{S}) = -\overline{S}-1$.
\end{Lem}
\begin{proof}  To show that $r_{\overline{S}}$ is the minimal value in the initial portion of the orbit of $r_0$ for each $\alpha \in J_{k,v} = [\zeta_{k,v}, \eta_{k,v})$, it suffices to show that $\sigma^j(\, \overline{d}(k,v)\,) \succ d_{[1,\infty)}(\,r_{\overline{S}}\,)$ for all $j< \overline{S}$, where $d_{[1,\infty)}(x)$ denotes the sequence of simplified digits of $x$.    Since $r_{\overline{S}}(\eta_{k,v})$ is larger than $r_{\overline{S}}(\alpha)$ for all $\alpha \in J_{k,v}$, and $\overline{d}{}^{\eta_{k,v} }_{[1,\infty)} = \overline{d}(k, v (v')^\infty)$,    it in fact suffices to show  the $\sigma^j(\, \overline{d}(k,v)\,) \succ \overline{d}(k, (v')^\infty)$.   
But, Lemma~\ref{l:bothOrders} shows that this is equivalent to the result of Lemma~\ref{l:wordOrder}.
\end{proof}

 We can also identify the second smallest $r_j$.  Recall that $\iota = \overline{S}(u)$, when $v = \Theta_q(v)$ and $\iota= 0$ when $v=1$.
 \begin{Lem}\label{l:2ndLargestR}   We have 
    \[  r_\iota = \min_{0\le j < \overline{S}}\, r_j\,.\]
    Equivalently,   $\beta(\iota) =  -\overline{S}$.
\end{Lem}
\begin{proof} The statement for $v=1$ is trivial.  

Due to Lemma~\ref{l:lastRValueIsLeast},  we can repose the statement for $v \neq 1$ as $\sigma^{\iota} (\, \Theta_q(u) \epsilon\,) \prec \sigma^j (\, \Theta_q(u) \epsilon\,)$ for all $j<|v|$, where $\epsilon$ is any word such that $\epsilon \prec \sigma^j (\, \Theta_q(u) \epsilon\,)$ for all  $j<|v|$.   (That is, $\epsilon$ is simply any sufficiently small word, allowing comparisons to accord with the minimality of $r_{\overline{S}(v)}$.) 

Now, this reformulated statement is easily verified for $v$ and for any $\Theta_{0}^{h}\circ \Theta_p(v)$ when $v$ is of length one.    We can thus apply the morphism $\mathscr D$   and deduce that it is true for all $v\in \mathcal V$.  
\end{proof} 

 \begin{Lem}\label{l:rOneIsBigEnough}   We have that $\beta(1+\iota) = 1 + \beta(1)$.
\end{Lem}
\begin{proof}  The argument here is the exact analog of the proof of Lemma~\ref{l:ellOneIsBigEnough}.
\end{proof}
 
 \begin{Lem}\label{l:bottomHeightsIncrease}   For $-1\le b \le -\overline{S}+1$,  
 \[y_{b-1} < y_b\,.\]  
\end{Lem}
\begin{proof}   The proof here is the exact analog of that for Lemma~\ref{l:topHeightsIncrease}.  
\end{proof}

\bigskip
\subsection{Meeting of $T_{\zeta_{k,v}}$-orbits with $T_{\eta_{k,v}}$-orbits}\label{ss:OrbitsMeet}    We now explore a phenomenon that already  can  be noticed in  the setting of Example~\ref{e:kIs1},  see in particular Table~\ref{t:orbsInEg}.

In the following, note that although in general for different values of $\alpha$, a fixed sequence of (simplified) digits determines distinct real values, (after having fixed $n$) purely periodic digit sequences correspond to fixed points of well defined elements of the group $G_{n}$, independent of choice of $\alpha$.    

Recall from Definition~\ref{d:iota} that  if $u$ is the parent of $v$ in $\mathcal V$ and $k\in \mathbb N$ are fixed, then $\iota$ is short hand notation for $|\overline{d}(k,u)|$.   Compare the following with Figure~\ref{f:rZeroEtaOrbit}.

\begin{figure}
\begin{tikzpicture}[baseline= (a).base]
\node[scale=.75] (a) at (0,0){
\begin{tikzcd}[ column sep=2pc,row sep=2pc]
r_0(\eta)\ar[pos=0.64]{r}{k} \ar[bend left=40, swap]{rrrrrrr}{R_{k,v}}&r_1(\eta) \ar[equal]{d}\ar{r} & \cdots  \ar{r} {k} &r_{\overline{S}-\iota+1} (\eta)\ar{r}{k+1} \ar[equal]{d}&r_{\overline{S}-\iota} (\eta)\ar{r}\ar[bend left,  swap]{rrr}{ R_{k,u}}\ar[equal]{d} &\cdots  \ar{r}&  r_{\overline{S}-1}(\eta) \ar{r}{k} \ar[equal]{d} &r_{\overline{S}}(\eta)\ar[equal]{d}\\
&r_{\iota+1}(\zeta)&\cdots&r_{\overline{S}-1}(\zeta)&r_0(\zeta)&\cdots&r_{\iota-1}(\zeta)&r_{\iota}(\zeta)\ar[bend left=18]{llllll}{k+1}
\end{tikzcd} 
};
\end{tikzpicture}
%
\caption{The $T_{\eta_{k,v}}$-orbit of $r_0(\eta_{k,v})$ contains the initial $\overline{S}$ elements of the $T_{\zeta_{k,v}}$-orbit of $r_0(\zeta_{k,v})$, see Proposition~\ref{p:rOrbitsOfSynIntEndptsMeet}.    Here $\eta,\zeta$ denote  $\eta_{k,v}, \zeta_{k,v}$, respectively.     Arrow labels $k, k+1$ denote  $A^{k}C, A^{k+1}C$, respectively.}
\label{f:rZeroEtaOrbit}
\end{figure}
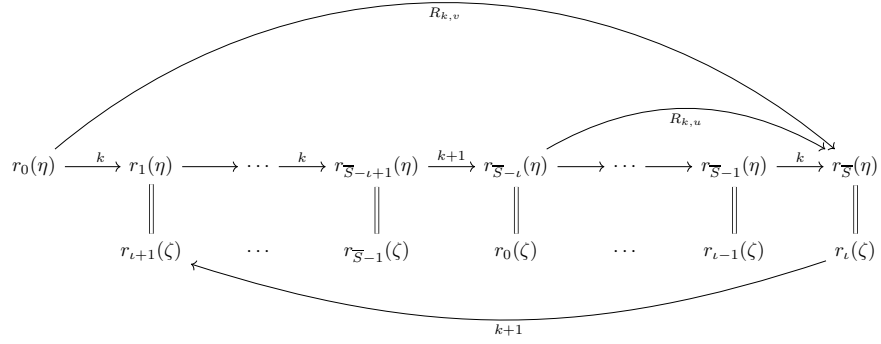
 
\begin{Prop}\label{p:rOrbitsOfSynIntEndptsMeet}   Fix $n\ge 3,  k \in \mathbb N, v \in \mathcal V$.  
Then $T_{\eta_{k,v}}$-orbit of $r_0(\eta_{k,v})$ contains the initial $\overline{S}$ elements of the $T_{\zeta_{k,v}}$-orbit of $r_0(\zeta_{k,v})$.  

In particular, 
\begin{enumerate}
\item [i.)] For any length one $v$, we have $r_1(\eta_{k,v}) = r_0(\zeta_{k,v})$;

\item[ii.)]  For $|v| >1$, 
\[  r_{\overline{S}}(\eta_{k,v}) = r_{\iota}(\zeta_{k,v}).\]
\end{enumerate}

\end{Prop}
\begin{proof} We first prove the enumerated statements.  For all $v$, (\cite{CaltaKraaikampSchmidt}, Lemma~4.5)
 implies that $A^{k+1}CR_{k,v} (A^kC)^{-1}$ fixes $r_1(\eta_{k,v})$,  while $AR_{k,v}$ fixes $r_0(\zeta_{k,v})$.    When $v = c_1$, we have $R_{k,v} = (A^kC)^{c_1}$ and thus $r_1(\eta_{k,v}), r_0(\zeta_{k,v})$ are fixed by the same element of $G_n$.    Both $r_1(\eta_{k,v}), r_0(\zeta_{k,v})$ and their respective orbits are certainly positive and  less than $r_0(\gamma_n)$ which in turn is less than one; in this interval each $A^jC$ defines an expansive map.  Therefore, both $r_1(\eta_{k,v}), r_0(\zeta_{k,v})$ must be the repulsive fixed point of the hyperbolic $AR_{k,v}$ ,  and hence equality does hold. 

For longer length $v$, let $u$ be the parent of $v$.  Then $r_{\overline{S}}(\eta_{k,v})$ is fixed by $R_{k,v} (A^kC)^{-1}A^{k+1}C$ and $r_{\iota}(\zeta_{k,v})$ by $R_{k,u}A R_{k,v}(R_{k,u})^{-1}$.  With mild abuse of notation, the first of these is $R_{k, v'}$, the second is $R_{k, u^{-1} \overleftarrow{v'}  u}$.   (Recall that $v'$ is given in \eqref{e:vPrimeDef}.)  Now,    (\cite{CaltaKraaikampSchmidt}, Remark~4.14)
 gives $v = uaz$, with $v, u, a, z$ and $z'a$   all being palindromes. Thus, $\overleftarrow{v'} = \overleftarrow{(zau)'}  = \overleftarrow{z'au} = u \overleftarrow{(z'a)} = uz'a$.  Hence, $u^{-1} \overleftarrow{v'}  u = z'a u = v'$,  and therefore $r_{\overline{S}}(\eta_{k,v}) = r_{\iota}(\zeta_{k,v})$  does hold. 
  
 Again by the palindromic nature of $v$, one easily sees that   $r_{\overline{S}-\iota + i}(\eta_{k,v}) = r_i(\zeta_{k,v})$ for $0 \le i \le \iota$ and $r_{j}(\eta_{k,v}) = r_{\iota+j}(\zeta_{k,v})$ for $1 \le i \le \overline{S}-\iota-1$.
\end{proof}

\begin{Cor}\label{c:ellZeroEtaAsEllJzetaSmallAlps}   Fix $n\ge 3,  k \in \mathbb N, v \in \mathcal V$.  
Then $\ell_0(\eta_{k,v})$ is in the  $T_{\zeta_{k,v}}$-orbit  of $\ell_0(\zeta_{k,v})$.   

More precisely, 
\begin{enumerate}
\item[i.)]  $\ell_{|w|}(\zeta_{k,1}) = \ell_0(\eta_{k,1})$; 

\item[ii.)]  For $v \neq1$,  
\[  \ell_{ |\underline{d}(k,u)|}(\zeta_{k,v}) = \ell_0(\eta_{k,v}).\]
\end{enumerate}

\end{Cor}

\begin{proof}  From \eqref{e:lowerDetaExpansion} and \eqref{e:lowerDzetaExpan} we have  $\underline{d}{}^{\eta_{k,1}}_{[1,\infty)} = \overline{w^k,(-1)^{n-3}, -2}$ and $\underline{d}{}^{\zeta_{k,1}}_{[1,\infty)} = w^{k+1}, \overline{(-1)^{n-3}, -2, w^k}$.  The case of $v=1$ thus visibly holds.
 
 We now suppose that $u$ is the parent of $v$.
 In order to apply the results of the previous proposition, we use the relation $L_{k,v} = C^{-1}AC R_{k,v}$.  Note that in general $\ell_{|\underline{d}(k,u)|}(\alpha) = A^{-1}L_{k,u}A \cdot \ell_0(\alpha)$, instead of what might be naively expected.  Indeed, the definition of the $L_{k,v}$ in general, see (\cite{CaltaKraaikampSchmidt}, Definition~ 4.32),
 is such that the prefix of $(A^{-1}C)^{n-2}$ of $L_{k,u}$ must be replaced by $A^{-2}C(A^{-1}C)^{n-3}$ in order to achieve an admissible suffix of $L_{k,v}$. 
  
First suppose that $v = c_1$ with $c_1>1$, thus of parent  $u = (c_1-1)\,$.   Then $\ell_{|\underline{d}(k,u)|}(\zeta_{k,c_1}) = A^{-1}L_{k,c_1-1}A \cdot \ell_0(\zeta_{k,c_1}) = A^{-1}C^{-1}AC R_{k,c_1-1}\cdot r_0(\zeta_{k,c_1})= A^{-1}C^{-1}AC R_{k,c_1-1}\cdot r_1(\eta_{k,c_1})= A^{-1}C^{-1}AC R_{k,c_1}\cdot r_0(\eta_{k,c_1})= A^{-1}L_{k,c_1}A \cdot \ell_0(\eta_{k,c_1}) = A^{-1}\cdot r_0(\eta_{k,c_1}) = \ell_0(\eta_{k,c_1})$.
 
  From Proposition~\ref{p:rOrbitsOfSynIntEndptsMeet}, 
$L_{k,u}A \cdot \ell_0(\zeta_{k,v}) = L_{k,u} \cdot r_0(\zeta_{k,v}) = C^{-1}AC R_{k,u}\cdot r_0(\zeta_{k,v}) = C^{-1}AC \cdot r_{\iota}(\zeta_{k,v}) = C^{-1}AC \cdot r_{\overline{S}}(\eta_{k,v}) = C^{-1}AC R_{k,v}  \cdot r_0(\eta_{k,v}) = L_{k,v}  \cdot r_0(\eta_{k,v})= r_0(\eta_{k,v})$. 
Thus, $\ell_{|\underline{d}(k,u)|}(\zeta_{k,v}) = A^{-1}\cdot  r_0(\eta_{k,v}) = \ell_0(\eta_{k,v})$.
 \end{proof}

\begin{Prop}\label{p:ellZeroEtaIsEllSubIsubTwoZetaSmallAlps}   Fix $n\ge 3,  k \in \mathbb N, v \in \mathcal V$.  
Then $\ell_0(\eta_{k,v})$ is   the second smallest element of the  $T_{\zeta_{k,v}}$-orbit  of $\ell_0(\zeta_{k,v})$.   
That is,  
\[ \ell_0(\eta_{k,v}) = \ell_{i_2}(\zeta_{k,v}).\]
\end{Prop}
 \begin{proof}    Suppose that $|v| >1$.   The previous result implies that the $T_{\zeta_{k,v}}$-orbit  of $\ell_0(\zeta_{k,v})$ not only reaches $\ell_0(\eta_{k,v})$ but also thereafter agrees with  the $T_{\eta_{k,v}}$-orbit  of $\ell_0(\eta_{k,v})$.  Suppose that $i$ is such that   $\ell_i(\zeta_{k,v}) < \ell_0(\eta_{k,v})$.  Since every element of the $T_{\eta_{k,v}}$-orbit  of $\ell_0(\eta_{k,v})$ is greater than or equal to $\ell_0(\eta_{k,v})$, we must have   $i < |\underline{d}(k,u)|$.   Now, the  first $|\underline{d}(k,v)|-1$ simplified digits of $\ell_0(\zeta_{k,v})$ agree with those of $\ell_0(\eta_{k,v})$, and of course  $\ell_0(\zeta_{k,v})\le \ell_i(\zeta_{k,v})$.  Thus, the simplified digits of  $\ell_i(\zeta_{k,v})$  begin  $w^{k},  
\mathcal C^{c_1-1} \mathcal D^{d_1}\cdots \mathcal D^{d_{s-1}}\mathcal C^{c_s}, (-1)^{n-3}
 \cdots$. 
 Recall that  $\mathcal D = \mathcal C w = (-1)^{n-3}, -2, w^k$.  
 
 If $i >0$, then in light of \eqref{e:lowerDzetaExpan},  the  digit sequence of $\ell_i(\zeta_{k,v})$  begins with the final $w^k$ obtained by factoring some $\mathcal D^{d_j}$ in the prefix $\underline{d}(k,u)$ of $\underline{d}(k,v)$ as $\mathcal D^{-1+d_j},  (-1)^{n-3}, -2, w^k$.  But then 
$\ell_{i-(n-2)}(\zeta_{k,v})$ corresponds to factoring instead  as 
$\mathcal D^{d_j} =  \mathcal D^{-1+d_j}, \mathbf {\mathcal D}$.   Indeed, in light of the digits of $\ell_i(\zeta_{k,v})$, $\ell_{i-(n-2)}(\zeta_{k,v})$ is associated with
a copy of $v'$, refer to \eqref{e:vPrimeDef}.  But, this copy of $v'$ begins within the prefix $u$ of the word $v = u (u')^q u''$.    This in turn implies that there is  a copy of $u'$ contained in $v$, beginning within the prefix $u$.     However,  Lemma~\ref{l:wordOrder} shows that this is impossible.  Therefore, the  case of $i>0$ is void.

   Finally, the case of $|v|=1$ is argued similarly, but is more straightforward.   
 \end{proof}
 
 \begin{Cor}\label{c:secondHighest}     We have 
 \[y_{\underline{S}} = - \ell_0(\eta_{k,v}).\]
\end{Cor}
\begin{proof}  By definition, $y_{\underline{S}}$ is the second largest of the values $-\ell_i(\zeta_{k,v})$.   Thus,  $y_{\underline{S}} = \ell_{i_2}(\zeta_{k,v})$ and by the proposition this equals $- \ell_0(\eta_{k,v})$.
\end{proof}

\subsection{Differing by $t$: two paired top/bottom heights}

 \begin{Lem}\label{l:easyHeightPairDifference}    
 We have $y_{\underline{S}+1} = t + y_{-\overline{S}}\,$.  
 \end{Lem}
\begin{proof}   By Lemma~\ref{l:2ndLargestR},  $-\overline{S}= \beta(\iota)$.   
By definition,  $y_{\beta(\iota)} = -r_{\overline{S}-\iota}(\eta_{k,v})$.   Proposition~\ref{p:rOrbitsOfSynIntEndptsMeet} now gives 
$y_{\beta(\iota)} = -r_0(\zeta_{k,v})$.

From Lemma~\ref{l:lastEllValueIsLargest},   $\underline{S}+1 = \tau(\underline{S})$.
By definition, $y_{\tau(\underline{S})} = - \ell_0(\zeta_{k,v})$, and thus the result holds.
\end{proof} 

 \begin{Lem}\label{l:harderHeightPairDifference}    
 We have $y_{\underline{S}} = t + y_{-\overline{S}-1}\,$.  
 \end{Lem}
\begin{proof}  By Corollary~\ref{c:secondHighest},  $y_{\underline{S}} = - \ell_0(\eta_{k,v})$.  On the other hand,  from Lemma~\ref{l:lastRValueIsLeast},   
$ y_{-\overline{S}-1} =  y_{\beta(\overline{S})}$, which by definition has the value $-r_0(\eta_{k,v})$.  The result thus holds. 
\end{proof} 

\smallskip

\section{Bijectivity of $\mathcal T_{\alpha}$ on  $\Omega_{\alpha}$ for synchronizing small $\alpha$}\label{s:MainForSmallAlps} 
In this section, we prove Theorem~\ref{t:Omega}  for $\alpha \in (0, \gamma_n)$.   This bijectivity of $\mathcal T_{\alpha}$ on $\Omega_{\alpha}$  
follows from: ($i$) the upper boundary of $\Omega^+$ surjects onto itself; ($ii$)  the lower boundary of $\Omega^-$ surjects onto itself; ($iii$)  the images of the ``blocks", defined directly below, laminate.  

\subsection{Partitioning $\Omega_{\alpha}$ by blocks $\mathcal B_i$}\label{ss:theBlocks}  Recall that the upper and lower  parts of $\Omega_{\alpha}$ are given in \eqref{e:omPlusSmallAlp}  and \eqref{e:omMinusSmallAlp}. 
  
For each $i \in \{-1, -2, \cdots \} \cup \{k, k+1, \dots\}$, let the {\em block} $\mathcal B_i$ be the closure of the set $\{(x,y)\in \Omega\,|\, x \in \Delta_{\alpha}(i, 1)\,\}$.   Thus the  blocks partition $\Omega$ up to $\mu$-measure zero, confer Figures~\ref{f:omegaSmallAlp} and \ref{f:topBottomVerticesSmallAlps}.  

Since $\mathcal T$ is invertible, it is clear that for the values of $\alpha$ under consideration, Theorem ~\ref{t:Omega}  follows from the following two results. 

 \begin{Prop}\label{p:leftGivesTop}     The union of the $\mathcal T(\mathcal B_i)$ taken over all negative $i$ equals $\Omega^{+}$ up to $\mu$-measure zero.  
 \end{Prop}
This holds because:  the  various blocks appropriately laminate (in Subsection~\ref{ss:laminate});   their image includes the upper boundary of  $\Omega^{+}$(in Subsection~\ref{ss:upperBoundaryInImage});  and, it is directly verified that the limit as $i \to - \infty$ of $\mathcal T(\mathcal B_i)$ is $\mathbb I_{\alpha} \times \{0\}$.

 \begin{Prop}\label{p:rightGivesBottom}     The union of the $\mathcal T(\mathcal B_i)$ taken over all positive $i\ge k$ equals $\Omega^{-}$, up to  $\mu$-measure zero. 
 \end{Prop}
  This follows similarly from the results of Subsections~\ref{ss:laminate} and ~\ref{ss:lowerBoundaryInImage}, along with the easily verified fact that  the limit as $i \to   \infty$ of $\mathcal T(\mathcal B_i)$ is $\mathbb I_{\alpha} \times \{0\}$.

\subsection{Blocks laminate one above the other}\label{ss:laminate} 
As an initial step to proving these two results, we have the following results.

\begin{Lem}\label{l:bottomsOfBlocks}    Let $s$ denote the first simplified digit of $r_{\overline{S}}$.  Then 
for all $i \notin \{k+1, k\}$,   
 the lower boundary of the block $\mathcal B_{i}$ has height given by 
\[ \begin{cases}    
    y = y_{-\overline{S}-1}& \text{if} \;\;s \succ i\,; \\
     (y = y_{-\overline{S}-1}) \cup (y = y_{-\overline{S}})& \text{if} \;\;s = i\,; \\
     y = y_{-\overline{S}}& \text{if} \;\;s \prec i\,.
     \end{cases}
\]     
\end{Lem}
\begin{proof}  Recall that $L_{-\overline{S}-1} = [\ell_0, r_{\overline{S}}]$.   Hence, 
the leftmost piece in the partition defining $\Omega^{-}$  is $L_{-\overline{S}-1} \times [y_{-\overline{S}-1},0]$.   Therefore, 
we certainly find that whenever $s \succ i$, the lower boundary of the block $\mathcal B_{i}$ is given by   $ y = y_{-\overline{S}-1}$.
Since the next partition piece to the right is $J_{-\overline{S}} \times [y_{-\overline{S}},0]$,   and for all $j<\overline{S}$ we have 
$r_j \in \Delta_{\alpha}(k+1,1) \cup \Delta_{\alpha}(k,1)$,  the rest of the statement follows. 
\end{proof}

 \begin{Lem}\label{l:topsOfBlocks}    Let $u$ denote the first simplified digit of $\ell_{\underline{S}}$.  Then 
for all $i \notin \{-1, -2\}$,   
 the top boundary of the block $\mathcal B_{i}$ has height given by 
\[ \begin{cases}    
    y = y_{\underline{S}+1}& \text{if} \;\;u \succ i\,; \\
    (y = y_{\underline{S}+1}) \cup (y = y_{\underline{S}})& \text{if} \;\;u = i\,; \\
     y = y_{\underline{S}}& \text{if} \;\;u \prec i\,.
     \end{cases}
\]      
\end{Lem}

\begin{proof}  Due to Lemma~\ref{l:lastEllValueIsLargest},   $K_{\underline{S}+1} = [\ell_{\underline{S}}, r_0]$.   Hence, 
the rightmost piece in the partition defining $\Omega^{+}$  is $K_{\overline{S}+1} \times [0, y_{\underline{S}+1}]$.   
With this, one argues as for the bottoms of the blocks.
\end{proof}

 \begin{Lem}\label{l:laminationInEasyCase}  Suppose that $i \notin\{-1, k+1, k\}$.  
Then   $\mathcal T_{\alpha}(\mathcal B_i)$ laminates above $\mathcal T_{\alpha}(\mathcal B_{i-1})$.
 \end{Lem}

\begin{proof}   Since  $L_{k,v} = C^{-1}AC\, R_{k,v}$, (\cite{CaltaKraaikampSchmidt}, Lemma~4.2)
  and its proof show that $r_{\overline{S}} \in \Delta_{\alpha}(u+1, 1)$ if and only if $\ell_{\underline{S}} \in \Delta_{\alpha}(u, 1)$.
 
 Case:   $s \succ i$.   Here, $\mathcal B_i$ has bottom boundary height $y_{-\underline{S}-1}$.  Also,  $u = s-1 \succ i-1$, and hence 
 $\mathcal B_{i-1}$ has top boundary height $y_{\underline{S}}$.  Lemma~\ref{l:harderHeightPairDifference}    
states that $y_{\underline{S}} = t + y_{-\overline{S}-1}$, and thus Lemma~\ref{l:lamEqsSameCexpon} yields the result in this case. 

 Case:   $s \prec i$.   Here,  $\mathcal B_i$ has bottom boundary height $y_{-\underline{S}}$.  And,  $u = s-1 \prec i-1$, gives that 
 $\mathcal B_{i-1}$ has top boundary height $y_{\underline{S}+1}$.    Lemma~\ref{l:easyHeightPairDifference}  and Lemma~\ref{l:lamEqsSameCexpon} yield the result in this case. 
  
 Case: $s = i$.   Since $T_{\alpha}(r_{\overline{S}}) = T_{\alpha}(\ell_{\underline{S}})$ and the corresponding two pieces of each of the bottom boundary of $\mathcal B_i$ and top boundary of $\mathcal B_{i-1}$ are sent to the same height (by the arguments for the previous two cases),  the result holds in this final case as well.  
\end{proof}

 \begin{Lem}\label{l:laminationFromLeftBlocks}    
The  $\mathcal T_{\alpha}$-image of the block $\mathcal B_{-1}$ laminates above a portion of $\mathcal T_{\alpha}(\mathcal B_{-2})$.
\end{Lem}
\begin{proof}  We have that $\mathcal B_{-1}$ has bottom height $y_{-\overline{S}-1}$ if $r_{\overline{S}} \notin \Delta_{\alpha}(-1,1)$, and both heights $y_{-\overline{S}-1}, y_{-\overline{S}}$ otherwise.   The block $\mathcal B_{-2}$ has multiple heights along its upper boundary, 
with its rightmost height being either $y_{\underline{S}}$ if $\ell_{\underline{S}} \notin \Delta_{\alpha}(-2,1)$ and otherwise its two rightmost heights being $y_{\underline{S}}, y_{\underline{S}+1}$. In either case, the corresponding portion of the upper boundary fibers above an interval whose left endpoint is $\ell_{i_{\underline{S}}}$. 

Arguments as in the proof of the previous lemma, combined with Lemma~\ref{l:ellOneIsBigEnough},  show that the lower boundary of $\mathcal B_{-1}$ has image  equalling the image of a rightmost portion of the upper boundary of $\mathcal B_{-2}$.    
\end{proof}

 \begin{Lem}\label{l:laminationFromRightBlocks}    
 A portion of  $\mathcal T_{\alpha}(\mathcal B_{k+1})$ laminates above $\mathcal T_{\alpha}(\mathcal B_k)$.
 \end{Lem}
\begin{proof} Here Lemma~\ref{l:rOneIsBigEnough} combines with the above arguments to imply the result. 
\end{proof}
 
\subsection{Upper boundary  is in image} \label{ss:upperBoundaryInImage} 
 
 \begin{Lem}\label{l:intervalMeetsTwoCylinders}    
 There is a partition 
 \[ K_ {\tau(\underline{S}-1)} = K'_{\tau(\underline{S}-1)} \cup K''_{\tau(\underline{S}-1)}\]
 such that 
 \[ A^{-1}C \cdot K'_{\tau(\underline{S}-1)} = K_{\underline{S}+1},  \;\;\;A^{-2}C \cdot K''_{\tau(\underline{S}-1)} = K_1.\]
 
\bigskip  
Furthermore,   $\mathcal T$ sends $K_ {\tau(\underline{S}-1)} \times \{y_{\tau(\underline{S}-1)}\}$ to $(K_1\times \{y_1\}) \cup (K_{\underline{S}+1}\times \{y_{\underline{S}+1}\})$.
 \end{Lem}

\begin{proof} 
Since $\alpha \in J_{k,v}$ certainly  $\ell_{\underline{S}}< r_0$, 
and since $\underline{d}(k, v)$ ends with $-1$,   it follows that $\ell_{\underline{S}-1}$ lies in the interior of $\Delta_{\alpha}(-1,1)$.   
Furthermore,  Lemma~\ref{l:lastEllValueIsLargest}  yields that $\ell_{\underline{S}-1}$ is the largest $\ell_i$ lying in $\Delta_{\alpha}(-1,1)$.     By definition,  $K_{\tau(\underline{S}-1)} = [\ell_{\tau(\underline{S}-1)}, \ell_{1+\tau(\underline{S}-1)}]$, and it follows that $\ell_{1+\tau(\underline{S}-1)}$ is the smallest $\ell_i \in \Delta_{\alpha}(-2,1)$.   Thus, $K_{\tau(\underline{S}-1)}$ 
is partitioned into two pieces,  which we define as $K'_{\tau(\underline{S}-1)} \cup K''_{\tau(\underline{S}-1)}$, where the first of these   is mapped by $A^{-1}C$ to  $[\ell_{\underline{S}}, r_0 ] = K_{\underline{S}+1}$ and the second is mapped by  $A^{-2}C$ to $[\ell_0(\alpha), \ell_{i}]$ for some $i$.   Since there is certainly at least one sequence $-2, (-1)^{n-2}$ in $\underline{d}(k,v)$,  surely the smallest $\ell_i \in \Delta_{\alpha}(-2,1)$ has image under $A^{-2}C$ that is larger than any $A^{-1}C\cdot \ell_{i'}$ for $i' \in \Delta_{\alpha}(-1,1)$, as these latter images have digits starting at least (in the sense of our ordering) with $(-1)^{n-3}$.     Therefore,   $A^{-2}C$ sends $K''_{\tau(\underline{S}-1)}$ to $[\ell_0, \ell_{i_2}] = K_1$.
 
 Finally,   Lemma~\ref{l:topValuesFirstRelations}  implies the result on images of the top height.   
\end{proof} 

\begin{Lem}\label{l:rightmostIntervalMeetsTwoCylinders}    
 Let $K'_{\underline{S}}$ be the closure of  $K_{\underline{S}} \cap \Delta_{\alpha}(-2,1)$.   Then

 \[ A^{-2}C \cdot K'_{\underline{S}} = \begin{cases}  K_{\tau(1+ i_{\underline{S}})} &\text{if}\;\; \ell_{\underline{S}} \in \Delta_{\alpha}(-2,1);\\
K_{\tau(1+ i_{\underline{S}})} \cup [\ell_{1+ i_{\underline{S}}}, r_0]&\text{otherwise}.
 \end{cases}
 \]
 
 Furthermore,   $\mathcal T$ correspondingly sends $ K'_{\underline{S}} \times \{y_{\underline{S}}\}$ to $K_{\tau(1+ i_{\underline{S}})} \times \{y_{\tau(1+ i_{\underline{S}})}\}$, respectively  $(K_{\tau(1+ i_{\underline{S}})} \times \{y_{\tau(1+ i_{\underline{S}})}\}) \cup 
  (\, [\ell_{1+ i_{\underline{S}}}, r_0] \times  \{y_{\tau(1+ i_{\underline{S}})}\})$.  
\end{Lem}

\begin{proof}  The first simplified digit of  $\ell_{\underline{S}}$ is 
 at least $-2$ (in our usual ordering). 
On the other hand, all $\ell_i$ with $i <  \underline{S}$  are contained in  $ \Delta_{\alpha}(-1,1) \cup  \Delta_{\alpha}(-2,1)$, and there do exist $i$ with $\ell_i \in \Delta_{\alpha}(-2,1)$.   Thus certainly, $K_{\underline{S}} \cap \Delta_{\alpha}(-2,1) \neq \emptyset$.   This intersection is all of  $K_{\underline{S}}$ if  $\ell_{\underline{S}} \in \Delta_{\alpha}(-2,1)$;  otherwise, this intersection includes the right endpoint of  $\Delta_{\alpha}(-2,1)$.   The first statement thus holds.    The second statement is an immediate consequence of the first. 
\end{proof}

Recall that the various $p_i$ are defined in Definition~\ref{d:topYvalues}.
 \begin{Lem}\label{l:coveringUpperBoundary}    For each $1\le a < \underline{S}$ with $a \neq \tau(\underline{S}-1)$, 
 \[ \mathcal T(\, K_a \times \{y_a\}\,) =  K_{\tau( i_a + 1)}\times \{y_{\tau( i_a + 1)}\}.\]
 
 In particular,  
 \[ \bigcup_{a=1}^{\underline{S}}\; \mathcal T(\, K_a \times \{y_a\}) \supset   \bigcup_{a=1}^{\underline{S}+1}\;   K_a \times \{y_a\}.\]
 \end{Lem}
\begin{proof}    
If $a<\underline{S}+1$ and $a \neq \tau(\underline{S}+1)$, then $K_a \subset \Delta_{\alpha}(p_{i_a},1)$.   Thus,    $\mathcal T(\, K_a \times \{y_a\})  = K_{\tau( i_a + 1)}\times \{y_{\tau( i_a + 1)}\}$.

Lemmas~\ref{l:intervalMeetsTwoCylinders}  and ~\ref{l:rightmostIntervalMeetsTwoCylinders} show that $a \in \{\tau(\underline{S}+1),   \underline{S}\}$  account for images indexed by $a \in \{1, \tau(1+ i_{\underline{S}}), \underline{S}+1\}$.  An elementary counting argument shows that the remaining values of $a$ account for images indexed by the remaining $\underline{S}-2$ values. 
\end{proof}

\subsection{Lower boundary  is in image}\label{ss:lowerBoundaryInImage}

\begin{Lem}\label{l:lowerIntervalMeetsTwoCylinders}
Let 
\[L'_{\beta(\overline{S}-1)}  = L_{\beta(\overline{S}-1)} \cap\Delta_{\alpha}(k+1,1)\]
and 
\[L''_{\beta(\overline{S}-1)}  = L_{\beta(\overline{S}-1)} \cap \Delta_{\alpha}(k,1).\]
Then 
 \[ A^{k+1}C \cdot L'_{\beta(\overline{S}-1)} \supset L_{-1},  \;\;\;A^kC \cdot L''_{\beta(\overline{S}-1)}= L_{-\overline{S}-1}.\]
 
\bigskip 
Furthermore,   $\mathcal T$ applied to  $L_{\beta(\overline{S}-1)} \times  \{y_{\beta(\overline{S}-1)}\}$  contains $(L_{-1}\times \{y_{-1}\}) \cup (L_{-\overline{S}-1)}\times \{y_{-\overline{S}-1)}\})$.
\end{Lem}
 
\begin{proof} Since $\overline{d}(k,v)$ ends with a $k$,  Lemma~\ref{l:lastRValueIsLeast} implies that $r_{\overline{S}-1}$ is the least $r_j$ in $\Delta_{\alpha}(k,1)$.  If $v \neq c_1$, there exist $r_j \in \Delta_{\alpha}(k+1,1)$.  In this case,  $L_{\beta(\overline{S}-1)}$ is partitioned as $L'_{\beta(\overline{S}-1)} \cup L''_{\beta(\overline{S}-1)}$.  Since  $L'_{\beta(\overline{S}-1)}$ contains the right endpoint of $\Delta_{\alpha}(k+1,1)$,  it easily follows that $A^{k+1}C \cdot L'_{\beta(\overline{S}-1)} = L_{-1}$.  Similar considerations show  $A^kC \cdot L''_{\beta(\overline{S}-1)}= L_{-\overline{S}-1}$.     The rest of the proof follows easily in this case.   

If $v = c_1$, then $L_{\beta(\overline{S}-1)}= [r_{\overline{S}}, r_{\overline{S}-1}]$ and here $A^{k+1}C \cdot L'_{\beta(\overline{S}-1)} \supset L_{-1}$ (with equality in general not holding).   Once again, the rest of the proof follows easily.
\end{proof}

The following now easily follows. 
\begin{Lem}\label{l:coveringLowerBoundary}     
We have 
 \[ \bigcup_{b=-1}^{-\overline{S}}\; \mathcal T(\, L_b \times \{y_b\}) \supset   \bigcup_{b=-1}^{-\overline{S}-1}\;   L_b \times \{y_b\}.\]
\end{Lem}

\subsection{Left endpoint values:   $\alpha = \zeta_{k,v}$}     Recall that for any $v$ and with $k \in \mathbb N$, when $\alpha = \zeta_{k,v}$ one has $r_{\overline{S}} = \ell_0$ and $A^{-2}C\cdot \ell_{\underline{S}} = \ell_1$.   The first of these equations shows that the interval $L_{-\overline{S}-1}$ of Definition~\ref{d:bottomYvalues} degenerates to a point in this case;  correspondingly,  $\Omega_{-}$ will now have one fewer lower height.      As well, Lemma~\ref{l:easyHeightPairDifference} gives  $y_{\underline{S}+1} = -\ell_0(\zeta_{k,v})$  and $y_{\beta(\iota)} = -r_0(\zeta_{k,v})$.     Note that Lemma~\ref{l:rOneIsBigEnough} implies that $A^{k+1}C$ sends $r_\iota$ to the right of $r_1$.   Thus the various preceding sections now easily imply the following, see Figure~\ref{f:topBottomVerticesSmallAlpsZeta}. 

\bigskip 

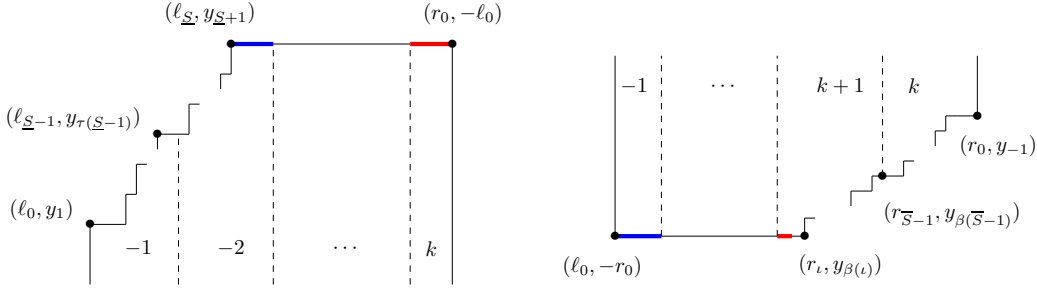
\begin{figure}[h]
\scalebox{.8}{
\noindent
\begin{tabular}{ll}
 
\begin{tikzpicture}[x=3.5cm,y=5cm] 
\draw  (-.32, 0)--(-.32, 0.2); 
\draw  (-.32, 0.2)--(-.15, 0.2);
\draw  (-.15, 0.2)--(-.15, 0.3);
\draw  (-.15, 0.3)--(-.1, 0.3);
\draw  (-.1, 0.3)--(-.1, 0.4);
\draw  (-.1, 0.4)--(-.05, 0.4);  
\draw  (0.0, 0.45)--(0.0, 0.5);    
\draw  (0.0, 0.5)--(0.15, 0.5); 
\draw  (0.15, 0.5)--(0.15, 0.6); 
\draw  (0.15, 0.6)--(0.2, 0.6);  
\draw  (0.3, 0.65)--(0.3, 0.7);      
\draw  (0.3, 0.7)--(0.35, 0.7);  
\draw  (0.35, 0.7)--(0.35, 0.8);  
\draw[blue, line width=2]  (0.35, 0.8)--(0.55, 0.8);   
\draw  (0.55, 0.8)--(1.2, 0.8);
\draw[red, line width = 2]  (1.2, 0.8)--(1.4, 0.8);
\draw  (1.4, 0.8)--(1.4, 0);  
\draw[thin,dashed] (0.1, 0)--(0.1, 0.5);     
\draw[thin,dashed] (0.55, 0)--(0.55, 0.8); 
\draw[thin,dashed] (1.2, 0)--(1.2, 0.8); 
\node at (-0.09, 0.12) {$-1$};    
\node at (.35, 0.12) {$-2$};   
\node at (.9, 0.12) {$\cdots$};         
\node at (1.3, 0.12) {$k$};   
\node at (-.55,   0.25) {$(\ell_0, y_1)$};           
\node at (-0.4, 0.55) {$(\ell_{\underline{S}-1}, y_{\tau(\underline{S}-1)})$};   
\node at (0.25,   .9) {$(\ell_{\underline{S}}, y_{\underline{S}+1})$}; 
\node at ( 1.45,   0.9) {$(r_0,  -\ell_0)$}; 
 \foreach \x/\y in {-.32/0.2,  0.0/0.5, 0.35/0.8,1.4/0.8%
} { \node at (\x,\y) {$\bullet$}; } 
\end{tikzpicture}
&\;\;\;
\begin{tikzpicture}[x=3.5cm,y=5cm] 
\draw  (-.32, 0.6)--(-.32, 0 ); 
\draw[blue, line width = 2]  (-.32, 0)--(-0.1, 0 ); 
\draw  (-.1, 0)--(0.45, 0 ); 
\draw[red, line width = 2] (0.45, 0)--(0.52, 0); 
\draw  (0.52, 0)--(0.58, 0); 
\draw  (0.58, 0)--(0.58, 0.06); 
\draw  (0.58, 0.06)--(0.63, 0.06); 
\draw  (0.8, 0.1)--(0.8, 0.15); 
\draw  (0.8, 0.15)--(0.9, 0.15);   
\draw  (0.9, 0.15)--(0.9, 0.2);  
\draw  (0.9, 0.2)--(1.05, 0.2);   
\draw  (1.05, 0.2)--(1.05, 0.25);  
\draw  (1.05, 0.25)--(1.1, 0.25); 
\draw  (1.2, 0.3)--(1.2, 0.35); 
\draw  (1.2, 0.35)--(1.25, 0.35);     
\draw  (1.25, 0.35)--(1.25, 0.4);  
\draw  (1.25, 0.4)--(1.3, 0.4); 
\draw  (1.3, 0.4)--(1.4, 0.4);  
\draw  (1.4, 0.4)--(1.4, 0.6);   
\draw[thin,dashed] (0.45, 0)--(0.45, 0.6);  
\draw[thin,dashed] (0.95, 0.2)--(0.95, 0.6); 
\draw[thin,dashed] (-0.1, 0)--(-0.1, 0.6);  
\node at ( -.23, 0.5) {$-1$};  
\node at ( .2, 0.5) {$\cdots$};    
\node at ( .75, 0.5) {$k+1$};    
\node at (1.1, 0.5) {$k$};   
\node at (-.38,   -0.1) {$(\ell_0, -r_0)$};          
\node at (0.75,  -0.1) {$(r_{\iota}, y_{\beta(\iota)})$};        
\node at (1.28, 0.07) {$(r_{\overline{S}-1}, y_{\beta(\overline{S}-1)})$}; 
\node at (1.5,  .3) {$(r_0,  y_{-1})$}; 
 \foreach \x/\y in {-.32/0, 0.58/0, 0.95/0.2, 1.4/0.4%
} { \node at (\x,\y) {$\bullet$}; } 
 
\end{tikzpicture} 
\end{tabular}
}
\caption{Schematic representations (not to scale) showing the most important vertices of the tops and bottoms of blocks, for  $\alpha = \zeta_{k,v}$ (and $|v|>1$).    
See Proposition~\ref{p:OmegaForZeta}.   Interior label  $i$ denotes partitioning ``block" $\mathcal B_i$, see Subsection~\ref{ss:theBlocks} for definitions.  Thick line segments of same color have equal images under $\mathcal T_{n,\zeta_{k,v}}$, where the bottommost thick red segment has right endpoint at $(A^{k+1}C)^{-1}\cdot r_1$.}
\label{f:topBottomVerticesSmallAlpsZeta}%
\end{figure}

\begin{Prop}\label{p:OmegaForZeta}   Fix $n\ge 3,  k \in \mathbb N, v \in \mathcal V$   and $\alpha = \zeta_{k,v}$.      Let $\Omega^{+} $ be as in Definition~\ref{d:topYvalues}. 
Let $\Omega^{-}$ be as in Definition~\ref{d:bottomYvalues},   except that we redefine $L_{-\overline{S}}$ to be  $[\ell_0, r_\iota]$, and set \newline 
$\Omega^{-} = \bigcup_{b=-1}^{ -\overline{S}}\,   L_b \times [y_b,0]$.     

Then 
$\mathcal T_{n,\alpha}$ is bijective on $\Omega_{n, \alpha}  := \Omega^{+} \cup \Omega^{-}$, up to 
$\mu$-measure zero.   Furthermore,   $\mu(\Omega)_{n, \alpha}) < \infty$. 
\end{Prop}

\subsection{Right endpoint values: $\alpha =\eta_{k,v}$} Fix $n\ge 3$, and $v \in \mathcal V$.   Although  when   $\alpha = \eta_{k,v}$, the $T_{\alpha}$-obits of $\ell_0(\alpha), r_0(\alpha)$ do not synchronize,  we still can precisely describe the domain on which the associated two-dimensional maps is bijective.

Recall that $\ell_{\underline{S}} = \ell_0(\eta_{k,v})$ and $r_{\overline{S}+1}(\eta_{k,v}) = A^{k+1}C\cdot r_{\overline{S}}(\eta_{k,v}) = r_1(\eta_{k,v})$.  Furthermore,  from Lemmas~\ref{l:harderHeightPairDifference}  and ~\ref{l:lastRValueIsLeast},   $y_{\underline{S}} =  -\ell_0(\eta_{k,v})$  and  $y_{-\overline{S}-1} =  y_{\beta(\overline{S})} = -r_0(\eta_{k,v})$.    In particular,  here $\Omega^{+}$ has one less height than the general case.    

    Thus the various preceding sections now easily imply the following results, see Figure~\ref{f:topBottomVerticesSmallAlpsEta}.  
\begin{Prop}\label{p:OmegaForEta}   Fix $n\ge 3,  k \in \mathbb N, v \in \mathcal V$   and $\alpha = \eta_{k,v}$.     Let $\Omega^{-}$ be as in Definition~\ref{d:bottomYvalues}.  Let $\Omega^{+} $ be as in Definition~\ref{d:topYvalues} except that we redefine $K_{\underline{S}}$ to be  $[\ell_{i_{\underline{S}}}, r_0]$, and $\Omega^{+} = \bigcup_{a=1}^{ \underline{S}}\,   K_a \times [0, y_a]$.     

Then 
$\mathcal T_{n,\alpha}$ is bijective on $\Omega_{n, \alpha}  := \Omega^{+} \cup \Omega^{-}$, up to 
$\mu$-measure zero.     Furthermore,  $\mu(\Omega_{n,  \alpha}) < \infty$. 
\end{Prop}

\begin{figure}[h]
\scalebox{.8}{
\noindent
\begin{tabular}{ll}
\begin{tikzpicture}[x=3.5cm,y=5cm] 
\draw  (-.32, 0)--(-.32, 0.2); 
\draw  (-.32, 0.2)--(-.15, 0.2);
\draw  (-.15, 0.2)--(-.15, 0.3);
\draw  (-.15, 0.3)--(-.1, 0.3);
\draw  (-.1, 0.3)--(-.1, 0.4);
\draw  (-.1, 0.4)--(-.05, 0.4);  
\draw  (0.0, 0.45)--(0.0, 0.5);    
\draw  (0.0, 0.5)--(0.15, 0.5); 
\draw  (0.15, 0.5)--(0.15, 0.6); 
\draw  (0.15, 0.6)--(0.2, 0.6);  
\draw  (0.3, 0.65)--(0.3, 0.7);      
\draw  (0.3, 0.7)--(0.35, 0.7);  
\draw  (0.35, 0.7)--(0.35, 0.8);    
\draw  (0.35, 0.8)--(0.42, 0.8);   
\draw[blue, line width = 2] (0.42, 0.8)--(0.55, 0.8); 
\draw  (0.55, 0.8)--(1.2, 0.8);  
\draw[red, line width = 2] (1.2, 0.8)--(1.4, 0.8); 
\draw  (1.4, 0.8)--(1.4, 0);  
\draw[thin,dashed] (0, 0)--(0,0.5);     
\draw[thin,dashed] (0.55, 0)--(0.55, 0.8); 
\draw[thin,dashed] (1.2, 0)--(1.2, 0.8); 
\node at (-0.15, 0.12) {$-1$};    
\node at (.35, 0.12) {$-2$};           
\node at (1.3, 0.12) {$k$};   
\node at (-.55,   0.25) {$(\ell_0, y_1)$};           
\node at (-0.4, 0.55) {$(\ell_{\underline{S}-1}, y_{\tau(\underline{S}-1)})$};   
\node at (0.25,   .9) {$(\ell_{i_{\underline{S}}}, y_{\underline{S}})$}; 
\node at ( 1.45,   .9) {$(r_0,  -\ell_0)$}; 
 \foreach \x/\y in {-.32/0.2,  0.0/0.5, 0.35/0.8, 1.4/0.8%
} { \node at (\x,\y) {$\bullet$}; } 
\end{tikzpicture}
&\;\;\;
\begin{tikzpicture}[x=3.5cm,y=5cm] 
\draw  (-.32, 0.6)--(-.32, 0); 
\draw[blue, line width = 2] (-0.32, 0)--(-0.1, 0.0); 
\draw  (-.32, 0)--(0.55, 0); 
\draw[red, line width = 2] (0.55, 0)--(0.75, 0); 
\draw  (0.75, 0)--(0.75, 0.05); 
\draw  (0.75, 0.05)--(0.78, 0.05); 
\draw  (0.8, 0.1)--(0.8, 0.15); 
\draw  (0.8, 0.15)--(0.9, 0.15);   
\draw  (0.9, 0.15)--(0.9, 0.2);  
\draw  (0.9, 0.2)--(1.05, 0.2);   
\draw  (1.05, 0.2)--(1.05, 0.25);  
\draw  (1.05, 0.25)--(1.1, 0.25); 
\draw  (1.2, 0.3)--(1.2, 0.35); 
\draw  (1.2, 0.35)--(1.25, 0.35);     
\draw  (1.25, 0.35)--(1.25, 0.4);  
\draw  (1.25, 0.4)--(1.3, 0.4); 
\draw  (1.3, 0.4)--(1.4, 0.4);  
\draw  (1.4, 0.4)--(1.4, 0.6);   
\draw[thin,dashed] (0.55, 0)--(0.55, 0.6);  
\draw[thin,dashed] (0.95, 0.2)--(0.95, 0.6); 
\draw[thin,dashed] (-0.1, 0)--(-0.1, 0.6);  
\node at ( .75, 0.5) {$k+1$};    
\node at (1.1, 0.5) {$k$};   
\node at ( -.23, 0.5) {$-1$};  
\node at (-.38,   -0.15) {$(\ell_0, -r_0)$};          
\node at (0.9,   -0.15) {$(r_{\overline{S}}, y_{-\overline{S}-1})$};          
\node at (1.28, 0.1) {$(r_{\overline{S}-1}, y_{\beta(\overline{S}-1)})$}; 
\node at (1.5,  .3) {$(r_0,  y_{-1})$}; 
 \foreach \x/\y in {-.32/0,   0.75/0, 1.05/0.2, 1.4/0.4%
} { \node at (\x,\y) {$\bullet$}; } 
 
\end{tikzpicture} 
\end{tabular}
}
\caption{Schematic representations (not to scale) showing the most important vertices of the tops and bottoms of blocks, for  $\alpha = \eta_{k,v}$ (and $|v|>1$).    
 See Proposition~\ref{p:OmegaForEta}.   Interior label  $i$ denotes partitioning ``block" $\mathcal B_i$, see Subsection~\ref{ss:theBlocks} for definitions.  Thick line segments of same color have equal images under $\mathcal T_{n,\eta_{k,v}}$; the topmost thick blue segment in the left figure  has left endpoint at $(\,(A^{-2}C)^{-1}\cdot \ell_1, -\ell_0)$. }
\label{f:topBottomVerticesSmallAlpsEta}%
\end{figure}
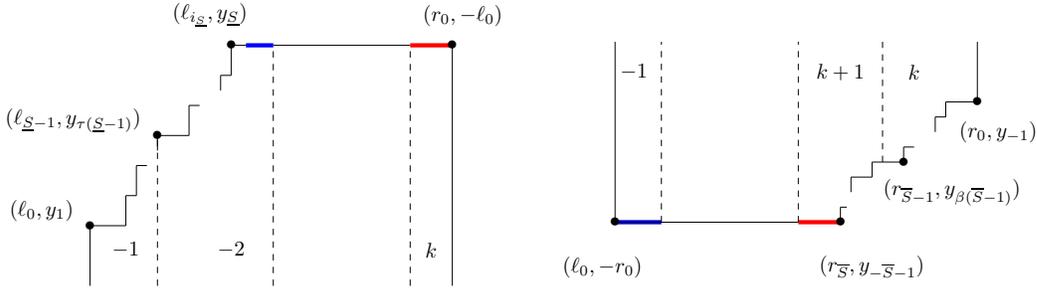

\medskip 
\begin{Cor}\label{c:contAtRightEndPtSyncIntervalSmallAlps}   Fix $n\ge 3,  k \in \mathbb N, v \in \mathcal V$.  Then the union of 
$\Omega_{\eta_{k,v}}$ with the line segment $r_0(\eta_{k,v}) \times [-\ell_0(\eta_{k,v}), -\ell_0(\zeta_{k,v})]$ is the limit, with respect to the Hausdorff metric topology on compact subsets of $\mathbb R^2$, of $\Omega_{\alpha}$ as $\alpha$ tends to $\eta_{k,v}$ from the left.
\end{Cor}

\begin{proof}   When evaluating the limit from the left as $\alpha \to \eta_{k,v}$, we can assume that $\alpha \in J_{k,v}$.     Recall that we  define both $\Omega_{\alpha}, \Omega_{\eta_{k,v}}$ as the union of a top piece and a bottom piece: in simple notation $\Omega^+\cup \Omega^-$.      The bottom piece of each is defined as in Definition~\ref{d:bottomYvalues}.  In particular,  each is the union of the same number of rectangles; the vertices of these rectangles are such that their first coordinates clearly vary continuously in $\alpha$, while their second coordinates are fixed throughout the closure of $J_{k,v}$.     The top pieces are similarly defined,  however for $\alpha \in J_{k,v}$ the number of rectangles is greater by one, with the ``extra" rectangle being   $[\ell_{i_{\underline{S}+1}}(\alpha), r_0(\alpha)] \times [0, -\ell_0(\zeta_{k,v})]$.   Since $\ell_{i_{\underline{S}+1}}(\alpha)$ and  $r_0(\alpha)$ both converge to $r_0(\eta_{k,v})$,  the limit of this rectangle is simply a vertical line segment.  The remaining rectangles comport themselves continuously in the manner discussed above for  the rectangles comprising the bottom portion.   The result thus holds.  
 \end{proof}

\section{Continuity of mass, for small $\alpha$}\label{s:ContSmall}

We aim to show that the  bijectivity domains  vary continuously at points of non-synchronization.     Since these are limit points of sequences of the types of $\alpha$ which we have already considered,  we bootstrap on the above work. 
  
  We prove Proposition~\ref{p:NaturallyErgodicEndPts},  giving in particular the ergodicity for the systems indexed by endpoints of synchronization intervals, the $\zeta_{k,v}$ and $\eta_{k,v}$,   by applying a main result of \cite{CKStoolsOfTheTrade}.   After establishing prerequisite results,  we use this  to show Lemma~\ref{l:admissibility}: Reversing words is a bijection for these systems' languages of admissible words, in the two cases of the letters being restricted to  $\{-1, -2\}$,  or to $\{k, k+1\}$.  A limiting argument then shows that this is also true for the systems of  the non-synchronization $\alpha$-values.  This bijection is key in the proof of Proposition~\ref{p:fillUpFromZ},   which states that the domain we define there for each non-synchronization $\alpha$  is a domain of bijectivity  for its associated 2-dimensional map.   The reversal of words is naturally required in that argument as it is related to the piecewise action on $y$-values, see Lemma~\ref{l:imagesAndReversal}.  
  
  This section closes with its main result, Theorem~\ref{t:continuity}, the continuity of the function taking $\alpha$ to the mass of its associated planar domain.

\subsection{Ergodicity at endpoints of synchronization intervals} 

  Here and later, we rely on (\cite{CKStoolsOfTheTrade}, Theorem~2.3), one of whose hypotheses is that the interval map in question satisfy the ``bounded non-full cylinders" condition.
 
\begin{Rmk}\label{rmk:bddNonfullCylinders}   The bounded non-full cylinders condition of \cite{CKStoolsOfTheTrade} is that the orbits of the endpoints of all non-full cylinders avoid  the interior of some full cylinder.    All of our maps have infinitely many full cylinders.  Hence, the condition is certainly satisfied by any of our maps  whose non-full cylinders have endpoints with expansions involving only finitely many digits. Recall, see  (\cite{CaltaKraaikampSchmidt}, \S~1.6.2), that for any of our maps, the only possible non-full cylinders are those containing:  $r_0(\alpha)$; $\ell_0(\alpha)$ and hence,  for large $\alpha$, $\mathfrak b_{\alpha}$.    Furthermore, the remaining endpoint of each of these possible non-full cylinders always meet the orbit of one of the endpoints. 
 \end{Rmk}

Recall that our underlying two-dimensional measure is $\mu$,  see ~\eqref{e:muDefd}.   
 
\begin{Prop}\label{p:NaturallyErgodicEndPts}   Fix $n\ge 3,  k \in \mathbb N, v \in \mathcal V$ and $\alpha \in \{\zeta_{k,v},\eta_{k,v}\}$.   
Let $\mu_{\alpha}$  be the normalization of $\mu$ to a probability measure on $\Omega_{n, \alpha}$, and $\mathscr B'_{\alpha}$ denote the Borel sigma algebra on $\Omega_{\alpha}$.      Then the system $(\mathcal T_{\alpha}, \Omega_{\alpha}, \mathscr B'_{\alpha}, \mu_{\alpha})$ is ergodic.   Furthermore,    this two dimensional system is the natural extension of  $(T_{\alpha},\mathbb I_{\alpha},  \mathscr B_{\alpha}, \nu_{\alpha})$, where $\nu_{\alpha}$ is the marginal measure of $\mu_{\alpha}$ and $\mathscr B_{\alpha}$  the Borel sigma algebra on $\mathbb I_{\alpha}$.   In particular, the one dimensional system is ergodic. 
\end{Prop}

\begin{proof} [Sketch]   For the given values of $\alpha$,    the digits of $\overline{d}_{[1, \infty)}^{\alpha}$ and $\underline{d}_{[1, \infty)}^{\alpha}$ are contained in $\{-1, -2, k, k+1\}$.  In particular the non-full cylinders of $T_{\alpha}$ are bounded in range.   Due to this and the finiteness of $\mu(\Omega_{\alpha})$ and of its vertical fibers, the hypotheses of (\cite{CKStoolsOfTheTrade}, Theorem~2.3) are met, thus giving the result here.
\end{proof}

We will use the ergodicity to argue that a subsystem gives rise to another description of the  $\Omega_{\alpha}$ when $\alpha = \zeta_{k,v}$ or $\eta_{k,v}$.    This will then be key to showing continuity of $\alpha \mapsto \Omega_{\alpha}$, as we will define the two-dimensional domain for non-synchronizing $\alpha$ analogously. 

\subsection{Bijectivity domain for non-synchronizing $\alpha$}\label{ss:domNonSync}   In this subsection, we announce the bijectivity domain for non-synchronizing $\alpha$.

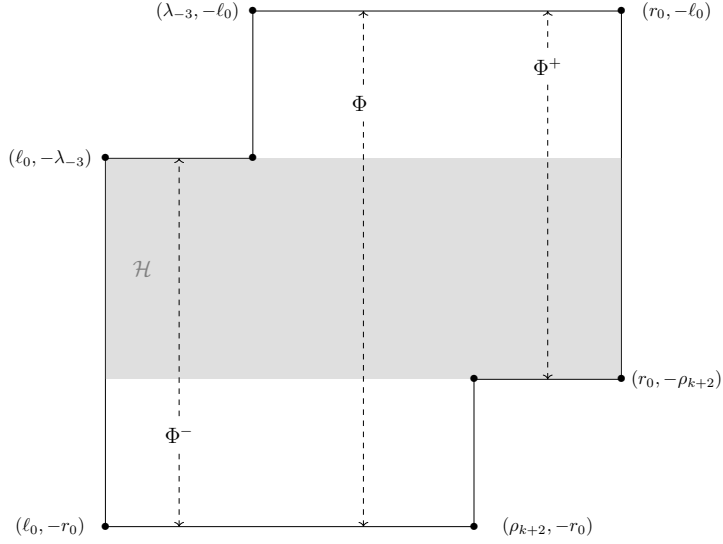
\begin{figure}
\scalebox{.7}{
\begin{tikzpicture}[x=7cm,y=7cm] 
\draw  (0, 0)--(0, 1.0); 
\draw  (0, 1.0)--(0.4, 1);
\draw  (0.4,1)--(0.4, 1.4);
\draw  (0.4, 1.4)--(1.4, 1.4);
\draw (1.4, 1.4)--(1.4, 0.4);
\draw (1.4, 0.4)--(1.0, 0.4);
\draw (1.0, 0.4)--(1.0, 0.0);
\draw (1.0, 0.0)--(0,0);
\fill [opacity= 0.25, gray] (0, 1) -- (1.4,1) --(1.4, 0.4)--(0, 0.4)--cycle;
\node [gray] at (0.1,  .7) {\Large{$\mathcal H$}}; 
\draw[<-, dashed] (0.2, 0)--(0.2, .2);
\node at (0.2,  .25) {\Large{$\Phi^{-}$}}; 
\draw[->, dashed] (0.2, 0.3)--(0.2, 1.0);
\draw[<-, dashed] (0.7, 0)--(0.7, 1.1);
\node at (0.69,  1.15) {\Large{$\Phi$}}; 
\draw[->, dashed] (0.7, 1.2)--(0.7, 1.4);
\draw[<-, dashed] (1.2, 0.4)--(1.2, 1.2);
\node at (1.2,  1.25) {\Large{$\Phi^{+}$}}; 
\draw[->, dashed] (1.2, 1.3)--(1.2, 1.4);
 \foreach \x/\y in {0/0,   0/1, 0.4/1, 0.4/1.4, 1.4/1.4, 1.4/0.4,1.0/0.4,1.0/0.0%
} { \node at (\x,\y) {$\bullet$}; } 
\node at ( -0.15, 0) {$(\ell_0, -r_0)$};  
\node at ( -0.15, 1) {$(\ell_0, -\lambda_{-3})$};  
\node at ( 0.25, 1.4) {$(\lambda_{-3}, -\ell_0)$};  
\node at (1.55, 1.4) {$(r_0, -\ell_0)$};  
\node at (1.55, 0.4) {$(r_0, -\rho_{k+2})$};  
\node at (1.2, 0) {$(\rho_{k+2}, -r_0)$};  
\end{tikzpicture}
}
\caption{Schematic representation of $\mathcal Z = \mathcal Z_{\alpha}$, see Definition~\ref{d:notationForEndpoints}.   For $\alpha < \gamma_{n}$ non-synchronizing or the endpoint of a synchronization interval,  the  central ``vertical" region $[\lambda_{-3}, \rho_{k+2}) \times \Phi$  is sent by $\mathcal T_{\alpha}$ to the (gray) ``horizontal" region $\mathcal H = \mathcal H_\alpha = \mathbb I_{\alpha} \times (\Phi^{-}\cap \Phi^{+})$ while the complementary pieces are sent by powers of  $\mathcal T_{\alpha}$ to fill in a region that strictly includes the complement of $\mathcal H_\alpha$ in $\mathcal Z_{\alpha}$.}
\label{f:theZ}%
\end{figure}

We introduce some more notation and introduce basic sets for our construction.    For legibility,  we suppress various subscripts, trusting that this causes no confusion for the reader. 
\begin{Def}\label{d:notationForEndpoints}   Fix $n\ge 3,  k \in \mathbb N$ and $\alpha < \gamma_{n}$. Let $\mathcal L= \mathcal L_{\alpha}$ denote the language of admissible (simplified) digits.   For finite length $u \in \mathcal L$,  as usual let $\Delta_{\alpha}(u)$ be the corresponding $T_{\alpha}$-cylinder.    Let $\lambda_u,  \rho_u$ denote the left- and right endpoints, respectively of $\Delta_{\alpha}(u)$.      

We then define 
\[ \Phi = [ -r_0, -\ell_0], \; \Phi^{-} = [-r_0, - \lambda_{-3}],  \; \text{and}\; \Phi^{+} = [-\rho_{k+2}, -\ell_0]\,,\]  
\[ \mathcal H= \mathbb I_{\alpha} \times (\Phi^{-}\cap \Phi^{+}) \]
and 
\[\mathcal Z =  [\ell_0, \lambda_{-3})\times \Phi^{-} \;\cup \;[\lambda_{-3}, \rho_{k+2}) \times \Phi  \;\cup \; [\rho_{k+2}, r_0) \times \Phi^{+},\]
see Figure~\ref{f:theZ}.   Note that $\mathcal Z \supset \mathcal H$.
\end{Def}

For each non-synchronizing $\alpha$, we determine a bijectivity domain for its two dimensional map.    Recall that we allow ourselves the freedom to write set equalities (strict inclusion,  disjoint unions, etc) when in fact these only hold up to $\mu$-measure zero (and usually, the measure zero exceptional set is easily found).    As usual, for any finite set $S$ of digits, we let $S^*$ denote the set of finite words in $S$. 
 \begin{Prop}\label{p:fillUpFromZ}     Fix a non-synchronizing $\alpha$ with $\alpha <\gamma_{n}$.  
 With notation as above, let $\mathcal T = \mathcal T_{\alpha}$ and define 
 \[\Omega = \overline{\cup_{j=1}^{\infty}\, \mathcal T^j(\, \mathcal Z)\,}\,.\]
 Then    $\Omega \supset \mathcal Z$ and $\mathcal T$ is bijective  up to $\mu$-measure zero on $\Omega$.   Furthermore, 
 \begin{equation} \label{e:jtUnion} \Omega = \mathcal H \, \bigsqcup \,  \overline{\sqcup_{u\in \{-1,-2\}^{^*}\cap \mathcal L_{\alpha}}\; \mathcal T^{|u|}(\, \Delta_{\alpha}(u) \times \Phi^{-} \,)} \,\bigsqcup \, \overline{\sqcup_{u\in \{k+1,k\}^{^*}\cap \mathcal L_{\alpha}}\; \mathcal T^{|u|}(\, \Delta_{\alpha}(u) \times \Phi^{+} \,)}.
\end{equation}
\end{Prop}
  The surjectivity of $\mathcal T$ on $\Omega$ follows from the definition of $\Omega$.  
 Key to proving the remaining statements of the proposition is to show that each $x\in \mathbb I_{\alpha}$ of $y$-fiber $\Phi$ in $\mathcal Z$ have exactly that same fiber in $\Omega$. 
  Our choice of the basic shapes of $\mathcal Z$ and $\mathcal H$ is motivated by Lemma~\ref{l:imagesAndReversal} below, which also begins the proof of \eqref{e:jtUnion}.   Together, Lemmas~\ref{l:minusOneAndTwoOnZ} and \ref{l:kAndKplusOneOnZ} establish the key step.   From this, both injectivity and  \eqref{e:jtUnion} then follow. %
  
 \medskip

We first introduce notation.  
\begin{Def}\label{d:notationForMatricesOfWords}    Given any simplified digit  $d$,  we let $M_d =  A^{d}C$.  For any finite word $u = u_1\, u_2\, \dots u_{|u|}$ with the $u_i \in \mathbb Z$,  we let 
$M_u = M_{u_{|u|}} \cdots M_{u_1}$.    We also define $\overleftarrow{u}$ to be the reversed word of $u$, thus  $\overleftarrow{u} = u_{|u|} \dots u_2\, u_1$. 
\end{Def}
 
 We have the following immediate implication of Lemma~\ref{l:conjByRofAtoPc}.     (Temporarily, we allow intervals to include $\infty$.) 
 \begin{Lem}\label{l:negAndReverse}   Let $u$ be a finite length word in nonzero integers and $x \in \mathbb R$.  Then $R M_u R^{-1}\cdot (-x) = - (M_{\overleftarrow{u}})^{-1} \cdot x$.  
 \end{Lem}
 
 \begin{Lem}\label{l:imagesAndReversal}    Fix $n\ge 3,  k \in \mathbb N$ and $\alpha < \gamma_{n}$.  Suppose that finite length $u$ is in $\mathcal L_{\alpha}$ and $a, b \in \overline{\mathbb I}_{\alpha}$ with $a<b$.    Then   
 \[ \mathcal T^{|u|}(\, \Delta_{\alpha}(u) \times [-b, -a]\,)  =   [\, M_{u}\cdot\lambda_u,   \, M_{u}\cdot \rho_u \,) \times    [- (M_{\overleftarrow{u}})^{-1}  \cdot b, -(M_{\overleftarrow{u}})^{-1} \cdot a].\]
 \end{Lem}
\begin{proof}   For $x \in \Delta_{\alpha}(u)$, we have that  $T^{|u|}(x) = M_{u}\cdot x$.    Since each $M_{u_j}$ defines an increasing function, the image of $\Delta_{\alpha}(u)$ under $T^{|u|}$ is the interval $[\, M_{u}\cdot\lambda_u,   \, M_{u}\cdot \rho_u \,)$.  Similarly,   the image of the $y$-coordinates is given by $y \mapsto R M_u R^{-1}\cdot y$.   Lemma~\ref{l:negAndReverse}  shows that our result holds.
\end{proof}

\begin{Lem}\label{l:upTheMiddle}   For each full cylinder  $\Delta_{\alpha}(j) $, we have that
 \[ \mathcal T_{\alpha}(\, \Delta_{\alpha}(j) \times \Phi \,)  = \mathbb I_{\alpha} \times -\Delta_{\alpha}(j).\]   Furthermore,  
 \[  \mathcal T_{\alpha}(\, [\lambda_{-3}, \rho_{k+2})\times \Phi  \,) = \mathcal H.\]
\end{Lem}

\begin{proof}   The interval $[\lambda_{-3}, \rho_{k+2})$ is the union of the full cylinders $\Delta_{\alpha}(j), k+2 \succeq j \succeq -3$.    Since $\lambda_j = (M_j)^{-1}\cdot \ell_0$ and $\rho_j = (M_j)^{-1}\cdot r_0$,   we find that $\mathcal T_{\alpha}$ sends $\Delta_{\alpha}(j) \times \Phi$ to   $\mathbb I_{\alpha} \times -\Delta_{\alpha}(j)$.   Taking the union over the various $j$ gives $\mathcal T_{\alpha}(\, [\lambda_{-3}, \rho_{k+2})\times \Phi  \,) =  \mathbb I_{\alpha}\times -[\lambda_{-3}, \rho_{k+2})$.  The result now follows. 
\end{proof}

 \bigskip
 
 When applying Lemma~\ref{l:imagesAndReversal}, the following is a natural aid. 
 \begin{Lem}\label{l:fillingUpI}   Fix $n\ge 3$,  and  either a non-synchronizing $\alpha<\gamma_{n}$ or $\alpha \in \{\zeta_{k,v}, \eta_{k,v}\}$ with $k\in \mathbb N$ and $v \in \mathcal V$.   Then 
 \[ [\ell_0, \lambda_{-3}] = \overline{\sqcup_{u\in \{-1,-2\}^{^*}\cap \mathcal L_{\alpha}}\, [\lambda_{u,-3}, \, \rho_u)\,}\]
 and 
 \[ (\rho_{k+2}, r_0] = \overline{ \sqcup_{u\in \{k+1,k\}^{^*}\cap \mathcal L_{\alpha}}\, [\lambda_u,\,  \rho_{u,k+2})\,}.\] 
 \end{Lem}
\begin{proof}   We consider the first equality.  Note first that $[\ell_0, \lambda_{-3}) = \Delta_{\alpha}(-1) \sqcup \Delta_{\alpha}(-2)$.   Fix any $u\in \{-1,-2\}^{^*}\cap \mathcal L_{\alpha}$.   Then $u,-2$ is also an admissible word, and of course,  $\rho_{u,-2} = \lambda_{u,-3}$.   It follows that 
\[ \Delta_{\alpha}(u) =\Delta_{\alpha}(u,-1) \sqcup \Delta_{\alpha}(u,-2)\sqcup [\lambda_{u,-3}, \, \rho_u),\]
where by definition $\Delta_{\alpha}(u,-1)$ is empty whenever $u, -1 \notin \mathcal L_{\alpha}$.

  Thus, we can iterate the partitioning of $\Delta_{\alpha}(u)$ beginning with each of $u=-1, u=-2$.  By the (eventual) expansiveness of the accelerated $T_0$ map, any infinite admissible word of simplified digits  in $\{-1, -2\}$ corresponds to a unique point.  Therefore,  $\overline{ \Delta_{\alpha}(-1) \sqcup \Delta_{\alpha}(-2)\,} = \overline{\sqcup_{u\in \{-1,-2\}^{^*}\cap \mathcal L_{\alpha}}\, [\lambda_{u,-3}, \, \rho_u)\,}$ and our equality holds. 

  The veracity of the second equation is similarly argued.
\end{proof}

\subsection{Sweeping out $\Omega_{\alpha}$ for $\alpha$ an endpoint of a synchronization interval}
 \begin{Lem}\label{l:fillUpWithZimagesZetaEta}    Fix $n\ge 3,  k \in \mathbb N, v \in \mathcal V$ and $\alpha \in  \{\zeta_{k,v}, \eta_{k,v}\}$.
  Then $\Omega_{\alpha}  \supset \mathcal Z_{\alpha} $ and 
  \[\Omega_{\alpha}  = \overline{\cup_{j=1}^{\infty}\, \mathcal T_{\alpha}^{j}(\, \mathcal Z_{\alpha})\,} \]
  up to $\mu$-measure zero.
 \end{Lem}

\begin{proof}   The containment follows from Propositions~\ref{p:OmegaForZeta} and \ref{p:OmegaForEta}.  Since $\overline{\cup_{j=1}^{\infty}\, \mathcal T_{\alpha}^{j}(\, \mathcal Z)\,}$ is  $\mathcal T_{\alpha}$-invariant, the ergodicity, shown in Proposition~\ref{p:NaturallyErgodicEndPts},  implies  the equality. 
\end{proof}

 \begin{Lem}\label{l:admissibility}  Fix $n\ge 3,  k \in \mathbb N, v \in \mathcal V$ and $\alpha \in  \{\zeta_{k,v}, \eta_{k,v}\}$.   The map $u \mapsto \overleftarrow{u}$ defines a self-bijection on each of  $ \{-1,-2\}^{^*}\cap \mathcal L_{\alpha}$ and $\{k+1, k\}^{^*}\cap \mathcal L_{\alpha}$.
  \end{Lem}
 
\begin{proof}  {\bf  Reversed admissible words are admissible.}   Certainly taking the reversed word has no effect on the set of digits appearing in a word, and also defines a bijective function on the set of words.     Since whenever   $u$ is of length one   we trivially have  $\overleftarrow{u}$ also in $\mathcal L_{\alpha}$,  and also the admissibility of any word implies the admissibility of any of its subwords (that is, of any of its factors),  we prove that the reverses of admissible words are admissible by induction on the lengths of words.  Thus, we suppose that for $u$ given every admissible word of  length less than that of $u$ (satisfying our digit restrictions) has its reversed word also admissible.

We first consider $u \in  \{-1,-2\}^{^*}\cap \mathcal L_{\alpha}$.   For  $u = u_1\cdots u_s$ with $s>1$, we let $u' = u_2\cdots u_s$.

\noindent
{\bf Step 1: Reduction to special case.} Were $\overleftarrow{u}$ inadmissible, then there is some suffix of this word, of say length $j$, which is less than $\underline{d}^{\alpha}_{[1,j]}$ (see \cite{CaltaKraaikampSchmidt}, Lemma~1.5).   Invoking our induction hypothesis, we find that it must be the word itself which has this property; that is,   $\overleftarrow{u} \prec \underline{d}^{\alpha}_{[1,s]}$.  On the other hand,  $\overleftarrow{u'} \succeq \underline{d}^{\alpha}_{[1,s-1]}$, but since we are using a `dictionary order'   $\overleftarrow{u'} = \underline{d}^{\alpha}_{[1,s-1]}$ and (since only digits $-1,-2$ are involved) $u_1 = -1$ while $\underline{d}^{\alpha}_{[s,s]} = -2$.     

Thus, we can and do assume that $\overleftarrow{u'} = \underline{d}^{\alpha}_{[1,s-1]}$ and $u_1 = -1$.

We remind the reader that the expansion for $\eta_{k,v}$ is  given above  in \eqref{e:lowerDetaExpansion} and that of $\zeta = \zeta_{k,v}$ in \eqref{e:lowerDzetaExpan}. 

\medskip 

\noindent
{\bf Step 2:   Special case is void.}  Since appearances of $-2$  in $\underline{d}^{\alpha}_{[1,\infty)}$ are separated by either $(-1)^{n-2}$ or $(-1)^{n-3}$,   but the admissible $u$ cannot begin with $(-1)^{n-1}$, it must be the case that $\underline{d}^{\alpha}_{[1,s]}$ has suffix $-2, (-1)^{n-3}, -2$ and $u$ thus has prefix  $(-1)^{n-2}, -2$.  (It is easily seen that shorter $u$ do have admissible $\overleftarrow{u}$; as we proceed other ``short" possibilities will similarly arise, but these also are easily seen to give admissible $\overleftarrow{u}$ and thus we will not mention them.)

 Now,  the admissible $u$ cannot begin  $(-1)^{n-2}, -2,  (-1)^{n-2}, -2$ and hence $\underline{d}^{\alpha}_{[1,s]}$ has suffix $-2, (-1)^{n-3}, -2, (-1)^{n-3}, -2$ and $u$ has prefix  $w$.   
We now note that for any $j \in \mathbb N$
\begin{equation}\label{e:reverseWordWithWpower}
 \overleftarrow{-2, (-1)^{n-3}, -2, w^j, (-1)^{n-2}} = w^{j+1}.
 \end{equation}
This then gives $\overleftarrow{-2, \mathcal C, (-1)^{n-2}} = w^k$ and  thus $\overleftarrow{-2, \mathcal D, (-1)^{n-2}} = w^{k+1}$.     Since $w^{k+1}$ is not admissible, we see that $\underline{d}^{\alpha}_{[1,s]}$ cannot be of suffix $-2, \mathcal D,  (-1)^{n-3}, -2$.   

Appearances of $-2, (-1)^{n-3}, -2, (-1)^{n-3}, -2$ in $\underline{d}^{\alpha}_{[1,\infty)}$ occur only when the suffix  $(-1)^{n-3}, -2$ of some final $w$ of a $\mathcal C$ or $\mathcal D$ is followed by the prefix $(-1)^{n-3}, -2$ of again some $\mathcal C$ or $\mathcal D$.  (When $k=1$ there is also the possibility of consecutive $\mathcal C$ giving this.)  Hence, we must have that $\underline{d}^{\alpha}_{[1,s]}$ has  suffix   $-2, \mathcal C,   (-1)^{n-3}, -2$, and $u$ has prefix $w^k$.  (The identity  $\mathcal D = [\mathcal C, (-1)^{n-2}], -2, (-1)^{n-3}, -2$ shows that we have yet to find an inadmissible $\overleftarrow{u}$.)    Equation \eqref{e:reverseWordWithWpower} gives in fact that  $\overleftarrow{-2, \mathcal C^i, (-1)^{n-2} }  = w^k   \mathcal C^{i-1}$ for any positive $i$.  Since $c_s = c_1 = \max{c_i}$ and $\mathcal D \prec \mathcal C$, admissibility of $u$ entails that $\underline{d}^{\alpha}_{[1,s]}$ has  suffix $-2, \mathcal C^{c_1},   (-1)^{n-3}, -2 = -2, \mathcal C^{c_s},   (-1)^{n-3}, -2$.   

Recall from \cite{CaltaKraaikampSchmidt} that either all $c_i$ equal one and each $d_i \in \{d_1, d_1+1\}$, or all $d_i$ equal one and each $c_i \in \{c_1, c_1-1\}$.   Since   \eqref{e:reverseWordWithWpower} gives  that the reverse of any  $-2,\mathcal D^{j_1}, \mathcal C^{i_1},\cdots,  \mathcal D^{j_r}, \mathcal C^{i_r}, (-1)^{n-2}$ is  $w^k   \mathcal C^{i_r-1}, \mathcal D^{j_r}, \dots, \mathcal C^{i_1}, \mathcal D^{j_1}$, we next see that $\underline{d}^{\alpha}_{[1,s]}$ has  suffix   $-2, \mathcal D^{d_1}, \mathcal C^{c_1},   (-1)^{n-3}, -2$.    We now do mention a short possibility: If the $c_i$ are all one and $d_2 = d_1 + 1$ then $u$ of the form $w^k, \mathcal C^{c_1-1} \mathcal D^{d_1}, \mathcal C^{c_2}, \mathcal D^{d_2-1}$ is admissible, however as usual $\overleftarrow{u}$ is admissible.   With appropriate generalization of this, we can continue our arguments and find that (under our restrictions on $u$) no finite length $u$  gives an inadmissible $\overleftarrow{u}$.     That is, this case is empty and thus for all $u \in  \{-1,-2\}^{^*}\cap \mathcal L_{\alpha}$ we have also $\overleftarrow{u} \in  \{-1,-2\}^{^*}\cap \mathcal L_{\alpha}$. 
  
\medskip 
\noindent
{\bf Involution shows onto.} Since any admissible $u$ is the reverse of its own reverse, we now have that $u \mapsto \overleftarrow{u}$ defines an involution on $ \{-1,-2\}^{^*}\cap \mathcal L_{\alpha}$.
 
 \medskip 
 
 The proof in the case of $u \in \{k+1, k\}^{^*}\cap \mathcal L_{\alpha}$ is given {\em mutatis mutandis}.  
\end{proof}

\subsection{Finishing the proof of bijectivity, and finite mass for non-synchronizing $\alpha$}  Throughout this subsection, we again fix $n\ge 3,  k \in \mathbb N$ and a non-synchronizing $\alpha < \gamma_{n}$.  Together,  Lemma~\ref{l:upTheMiddle} above, and Lemmas~\ref{l:minusOneAndTwoOnZ} and  ~\ref{l:kAndKplusOneOnZ}  show the statement of Proposition~\ref{p:fillUpFromZ} that $\Omega \supset \mathcal Z$.  The  latter two lemmas rely on the following corollary to Lemma~\ref{l:admissibility}.
 
 \begin{Cor}\label{c:symmetryInLimitAlph}   The map $u \mapsto \overleftarrow{u}$ defines a self-bijection on each of  $ \{-1,-2\}^{^*}\cap \mathcal L_{\alpha}$ and $ \{k+1, k\}^{^*}\cap \mathcal L_{\alpha}$.
 \end{Cor}
\begin{proof}   Lemma~\ref{l:admissibility} holds directly for the various $\eta_{k,v}$.   Any other non-synchronizing $\alpha$ is such that  there is a sequence of $v_i \in \mathcal V$ and $q_i \in \mathbb Z_{\ge 0}$ such that $v_{i+1} = \Theta_i(v_i)$ with  $\underline{d}{}^{\alpha}_{[1,\infty)}  = \lim_{i\to \infty} \underline{d}(k, v_i)$ and  $\overline{d}{}^{\alpha}_{[1,\infty)} = \lim_{i\to \infty} \overline{d}(k, v_i)$ (with the usual metric on sequences).    Furthermore,   $\alpha = \lim_{i\to \infty} \eta_{k, v_i}$ and hence for each $L\in \mathbb N$,   there is some $i$ such that for any word $u$ with $|u| = L$, we have $u \in \mathcal L_{\alpha}$  if and only if $u \in \mathcal L_{\eta_{k,v_i}}$.   Therefore the symmetry for  $ \{-1,-2\}^{^*}\cap \mathcal L_{\alpha}$ and $ \{k+1, k\}^{^*}\cap \mathcal L_{\alpha}$ follows from the symmetry for each of the corresponding sets of words for every $\eta_{k, v_i}$. 
\end{proof}

The following two results provide the heart of the justification of injectivity.  
\medskip

Let $\mathscr L = \sup_j\, \{\ell_j(\alpha)\}$.    Note that $\mathscr L < \lambda_{-3}$.   
 \begin{Lem}\label{l:minusOneAndTwoOnZ}       We have that
 \[ \overline{\sqcup_{u\in \{-1,-2\}^{^*}\cap \mathcal L_{\alpha}}\; \mathcal T_{\alpha}^{|u|}(\, \Delta_{\alpha}(u) \times \Phi^{-} \,)}   \cap \{(x,y)\mid x \in  [\mathscr L, r_0)\,\} =  [\mathscr L, r_0)  \times -[\ell_0, \lambda_{-3}).\]
 \end{Lem}
\begin{proof}  Since  $\underline{d}{}^{\alpha}_{[1,\infty)}  \in \{-1,-2\}^{^*}$ whereas $\overline{d}{}^{\alpha}_{[1,\infty)}  \in \{k+1,-k\}^{^*}$,    for all $u\in \{-1,-2\}^{^*}\cap \mathcal L_{\alpha}$  $T_{\alpha}^{|u|}(\Delta_{\alpha})$ has  some some $\ell_i(\alpha), i\ge 0$ as its left endpoint   and has right endpoint equal to $r_0$.   In particular,    $T_{\alpha}^{|u|}(\Delta_{\alpha}) \supset [\mathscr L, r_0)$.     Furthermore,   $RM_uR^{-1}\cdot \Phi^{-} = -[\lambda_{\overleftarrow{u},-3}, \rho_{\overleftarrow{u}})$. 

It remains to show that $(\ell_0, \lambda_{-3})$ is the union over all $u\in \{-1,-2\}^{^*}\cap \mathcal L_{\alpha}$ of $[\lambda_{\overleftarrow{u},-3}, \rho_{\overleftarrow{u}})$.   By Lemma~\ref{l:fillingUpI}, 
$(\ell_0, \lambda_{-3}) = \sqcup_{u\in \{-1,-2\}^{^*}\cap \mathcal L_{\alpha}}\; [\lambda_{u,-3}, \rho_u)$ and hence 
 Corollary~\ref{c:symmetryInLimitAlph} now implies $(\ell_0, \lambda_{-3}) = \sqcup_{u\in \{-1,-2\}^{^*}\cap \mathcal L_{\alpha}}\; [\lambda_{\overleftarrow{u},-3}, \rho_{\overleftarrow{u}})$.  
\end{proof}

The next result is proven exactly analogously.    Let $\mathscr R = \inf_j\, \{r_j(\alpha)\}$ and note that $\mathscr R > \rho_{k+2}$.

 \begin{Lem}\label{l:kAndKplusOneOnZ}       We have that
 \[ \overline{\sqcup_{u\in \{k+1,k\}^{^*}\cap \mathcal L_{\alpha}}\; \mathcal T_{\alpha}^{|u|}(\, \Delta_{\alpha}(u) \times \Phi^{+} \,)}\cap \{(x,y)\mid x \in  [\ell_0, \mathscr R)\,\} =  [\ell_0, \mathscr R) \times  -[\rho_{k+2}, r_0).\]
 \end{Lem}

\bigskip  
Although not necessary for the proof of bijectivity, more can be shown about the vertical fibers.
 \begin{Lem}\label{l:fibersNice}   For each $x \in \mathbb I_{\alpha}$,  let $\mathcal F_x$ be the vertical fiber of $\Omega$ above $x$.  Then for all $x \in \mathbb I_{\alpha}$,  the fiber $\mathcal F_x$ is an interval.  In particular,    $\mathcal F_{\ell_0} = [-r_0, -L]$ where $L = \sup_j\, \{\ell_j(\alpha)\}$ and $\mathcal F_{\ell_0} = [-\mathscr R, -\ell_0]$ where $\mathscr R = \inf_j\, \{r_j(\alpha)\}$.  
 Furthermore,
(i) if $x < x' < \rho_{k+2}$ then $\mathcal F_x \subseteq F_{x'}$,  and (ii)   if $\rho_{k+2}< x < x' $ then $\mathcal F_x \supseteq F_{x'}$.

 \end{Lem}
\begin{proof}   The arguments for  ``left" versus ``right" sides of our setting are completely analogous.  We thus only give the arguments for the latter case.    

    We first show containment of fibers.    For $L<x < \rho_{k+2}$,  we have that $\mathcal F_x = \Phi$.   For $x\le L$,  we have 
$\mathcal F_x  = \Phi^{-}\bigcup\, \overline{\cup_{\substack{u\in \{-1,-2\}^{^*}\cap \mathcal L_{\alpha}\\x\ge T_{\alpha}^{|u|}(\lambda_u)}}\;  RM_uR^{-1}\cdot \Phi^{-}}$,  thus certainly if $x<x' \le L$ then $\mathcal F_x \subseteq F_{x'}$.

Now,   $\mathcal F_{\ell_0} = \Phi^{-}\bigcup\, \overline{\cup_{\substack{u\in \{-1,-2\}^{^*}\cap \mathcal L_{\alpha}\\ \Delta_{\alpha}(u)\, \text{full}}}\;   -[\lambda_{\overleftarrow{u},-3}, \rho_{\overleftarrow{u}})}$.    It is easily confirmed that Lemma~\ref{l:fillUpWithZimagesZetaEta}, with  Propositions~\ref{p:OmegaForZeta} and \ref{p:OmegaForEta},  implies that when $\alpha$ is replaced by any  $\alpha'$ equal to some $\zeta_{k,v}$ or $\eta_{k,v}$ then the leftmost fiber equals  $[- r_0(\alpha'), - L(\alpha')]$ where $L(\alpha')$ is the largest element of the $T_{\alpha'}$-orbit of $\ell_0(\alpha')$.

We thus now aim to argue that  $\underline{d}{}^{\alpha}_{[1,\infty)}  = \lim_{i\to \infty} \underline{d}(k, v_i)$ leads to the result holding here.    It is easily established that $\Phi^{-}$  is the limit of the corresponding interval for each $\eta_{k, v_i}$ and also $L$ is the limit from below of the $L(\eta_{k, v_i})$.     For any $u\in \{-1,-2\}^{^*}\cap \mathcal L_{\alpha}$ there is some $i_0$ (depending only on $|u|$) such that for all $i \ge i_0$  we have $u \in \mathcal L_{\eta_{k, v_i}}$.      The non-full cylinders are  given by those $u$ of suffix $\underline{d}{}^{\alpha}_{[1,j]}$ for some $j \in \mathbb N$.    For each such $j$, there is some $i_1$ such that for all $i \ge i_1$ we have $\underline{d}{}^{\alpha}_{[1,j]} = \underline{d}{}^{\eta_{k, v_i}}_{[1,j]}$.   Thus,  each ``full" $u$ is such that $\Delta_{\eta_{k, v_i}}(u)$ is full for all $i$ sufficiently large.

 On the other hand,  again since $\underline{d}{}^{\alpha}_{[1,\infty)}  = \lim_{i\to \infty} \underline{d}(k, v_i)$,   for each $i$ every $u \in \mathcal L_{\eta_{k, v_i}}$ of length less than $\underline{d}(k, v_i)$ is $\alpha$-admissible and if $\Delta_{\eta_{k, v_i}}(u)$ is full then also $\Delta_{\alpha}(u)$ is full.    Thus, $\mathcal F_{\ell_0}$ is the limit of ever increasing subsets of the $\mathcal F_{\ell_0(\eta_{k, v_i})}= [- r_0(\alpha'), - L(\alpha')]$ from which the result now follows.

 Finally,   since $\mathcal F_{\ell_0}$ is an interval so also is each  $\mathcal F_{\ell_j}$.    Since any left non-full cylinder $\Delta_{\alpha}(u)$ has its  left endpoint equal to some $\ell_j$ by $T_{\alpha}^{|u|}$,  the fibers are constant (and hence intervals) between the various $\ell_j$.   
 
  We repeat that the analogous arguments succeed for the remaining cases.   
\end{proof} 

\subsection{Ergodic natural extension for non-synchronizing $\alpha$} 
The proof of Theorem~\ref{t:naturallyErgodicParTout} for small, non-synchronizing $\alpha$ relies on (\cite{CKStoolsOfTheTrade}, Theorem~2.3).    The following verifies one of the hypotheses of that theorem for such $\alpha$.

 \begin{Prop}\label{p:finMeas}     
 With $\Omega$ as in Proposition~\ref{p:fillUpFromZ},  $\mu(\Omega)< \infty$.
 \end{Prop}

\begin{proof}   Let $\Omega^{+} = \{ (x,y) \in \Omega \mid y \ge 0\}$ and $\Omega^{-} = \{ (x,y) \in \Omega \mid y \le 0\}$.   

Since    $\underline{d}{}^{\alpha}_{[1,\infty)} \in \{-1, -2\}^{\mathbb N}$ the $T_{\alpha}$-orbit of $\ell_0(\alpha)$ agrees with the $T_0$-orbit of this same value.    The vertical fiber in $\Omega^+$ above $\ell_0$ is $-[\mathscr L, 0]$, where  $\mathscr L = \sup_j\, \{\ell_j(\alpha)\}$.      Now,  $- \mathscr L = \lim_{i\to \infty} -\ell_{\underline{S}(\zeta_{k, v_i})} =  \lim_{i\to \infty} RA^{-1}R^{-1} \cdot -\ell_0(\zeta_{k, v_i}) = RA^{- 1}R^{-1} \cdot -\ell_0$.     Since $-t < \ell_0$, we have    $RA^{- 1}R^{-1} \cdot -\ell_0 <RA^{- 1}R^{-1} \cdot t$.  Direct computation shows $RA^{- 1}R^{-1} \cdot t = 1/(t+1/t)$.  This value is certainly less than $1/t$, and hence it follows that the fiber $\mathcal F_{\ell_0}$ lies inside $\Omega_0$, see \cite{CaltaSchmidt}.    That is,  the $\mathcal T_0$-orbit of the vertical fiber lies strictly inside $\Omega_0$.   But,  $\Omega_0$ touches the locus $y = -1/x$, the locus of non-finiteness for $d \mu = (1+xy)^{-2}\, dx\, dy$,  only at its vertices.  Since  the vertical fibers of $\Omega^{+}$ are constant between the various $\ell_j(\alpha)$ we find that all of $\Omega^{+} \cap \{x\le 0\}$  lies within a set of finite $\mu$-measure.   Elementary considerations  show that all of $\Omega^{+}$ has finite $\mu$-mass. 

  With   $\mathscr R = \inf_j\, \{r_j(\alpha)\}$, Lemma~\ref{l:bottomValuesFirstRelations}, Proposition~\ref{p:OmegaForEta}, and an elementary calculation give that  $-\mathscr R=  \lim_{i\to \infty} R A R^{-1}\cdot -r_0(\eta_{k, v_i})= R A R^{-1}\cdot -r_0 = -r_0/(1 + t r_0)$.     Clearly,   $(r_0, - \mathscr R)$, and hence all of $\mathcal F_{r_0}$,  lies above the locus $y=-1/x$.   Since the locus is invariant under $\mathcal T_M$ for any $M \in \text{SL}_2(\mathbb R)$, we find every fiber $\mathcal F_{r_j}$ also lies above this locus.  Since the fibers  are increasing with respect to $x\ge0$ we find that  $\Omega^{-} \cap \{x\ge 0\}$  lies within a set of finite $\mu$-measure.   Elementary considerations  show that all of $\Omega^{-}$ has finite $\mu$-mass. 
\end{proof}
 
 \begin{Prop}\label{p:bijectivityDomIsNatExtErgSmallNonSyn}       With $\alpha$ as in the notation of  Proposition~\ref{p:fillUpFromZ}, 
the system $(\mathcal T_{\alpha}, \Omega_{\alpha}, \mathscr B'_{\alpha}, \mu_{\alpha})$ is the natural extension of  $(T_{\alpha},\mathbb I_{\alpha},  \mathscr B_{\alpha}, \nu_{\alpha})$, where $\nu_{\alpha}$ is the marginal measure of $\mu_{\alpha}$ and $\mathscr B_{\alpha}$  the Borel sigma algebra on $\mathbb I_{\alpha}$.  Finally,   both systems are ergodic. 
 \end{Prop}
 \begin{proof} From Proposition~\ref{p:fillUpFromZ}, $\Omega_\alpha = \Omega$ is a bijectivity domain for $\mathcal T_{\alpha}$.    From Proposition~\ref{p:finMeas} this domain has finite measure.   From its description in Proposition~\ref{p:fillUpFromZ} the domain has fibers of bounded Lebesgue measure.  Since $\alpha$ is a non-synchronizing value, it corresponds to a limit under a sequence $\Theta_{q_i}$ applied to a word in $\mathcal V$;  the digits which enter the expansion of the endpoints of such sequences form a finite set.  In light of Remark~\ref{rmk:bddNonfullCylinders},  $T_{\alpha}$  satisfies the bounded non-full cylinders hypothesis for (\cite{CKStoolsOfTheTrade}, Theorem~2.3) and indeed all other hypotheses are then also easily verified.  That is, the proof of the proposition is then complete.
 \end{proof}

\subsection{Continuity of $\alpha \mapsto \mu(\Omega_{\alpha})$} 

We can now prove the main aim of this section.  
 \begin{Thm}\label{t:continuity}     The function  $\alpha \mapsto  \mu(\Omega_{\alpha})$ is continuous on $(0, \gamma_{n})$.
 \end{Thm}

\begin{proof}  Since the both top and bottom heights are constant along any synchronization interval $J_{k,v}$, with fibers constant between the various $r_i(\alpha), \ell_j(\alpha)$, continuity and indeed smoothness,  of $\alpha \mapsto  \mu(\Omega_{\alpha})$ in the interior of the interval is clear.  Propositions~\ref{p:OmegaForZeta} and ~\ref{p:OmegaForEta} easily imply that the continuity holds on the closed interval. 

     Now fix $k$.   To  show continuity throughout  $I_k := \cup_{v\in \mathcal V}\, I_{k,v}$, it only remains to show that $\mu(\Omega_{\alpha'}) \to \mu(\Omega_{\alpha})$  when $\alpha$ is non-synchronizing or (for the remaining one sided limits to) some $\eta_{k,v}$ or $\zeta_{k,v}$ with all values in $I_k$.   Fix one such value $\alpha\in I_k$, it suffices to show that  $\mu(\Omega_{\alpha'}) \to \mu(\Omega_{\alpha})$  when the $\alpha'$ are non-synchronizing or various $\eta_{k,v}$ or $\zeta_{k,v}$ also in $I_k$.    Such $\alpha'$ have  
\[\Omega_{\alpha'} = \mathcal H_{\alpha'}\bigsqcup  \overline{\sqcup_{u\in \{-1,-2\}^{^*}\cap \mathcal L_{\alpha'}}\; \mathcal T_{\alpha'}^{|u|}(\, \Delta_{\alpha'}(u) \times \Phi_{\alpha'}^{-} \,)} \bigsqcup  \overline{\sqcup_{u\in \{k+1,k\}^{^*}\cap \mathcal L_{\alpha'}}\; \mathcal T_{\alpha'}^{|u|}(\, \Delta_{\alpha'}(u) \times \Phi_{\alpha'}^{+} \,)}.\]  
Now, as $\alpha'\to \alpha$ we have $\mu( \mathcal Z_{\alpha'} \setminus  \mathcal Z_{\alpha}) \to 0, \mu( \mathcal H_{\alpha'} \setminus \mathcal H_{\alpha}) \to 0$.  Furthermore both  $\underline{d}{}^{\alpha'}_{[1,\infty)}\to  \underline{d}{}^{\alpha}_{[1,\infty)}$ and $\overline{d}{}^{\alpha'}_{[1,\infty)}\to  \overline{d}{}^{\alpha}_{[1,\infty)}$.   In particular,  there are $b_{\alpha'}, c_{\alpha'} \in \mathbb N$  each going to infinity as $\alpha'\to \alpha$ such that  for all $u\in \{-1,-2\}^{^*}$ with $|u| <  b_{\alpha'}$, one has  $u\in  \mathcal L_{\alpha'}$ if and only if    $u \in \mathcal L_{\alpha}$,   and similarly  for all $u\in \{k+1,k\}^{^*}$ with $|u| <  c_{\alpha'}$ we have  $u \in \mathcal L_{\alpha'}$  if and only if  $u \in \mathcal L_{\alpha}$.  For these $u$, we have also $\Delta_{\alpha'}(u) \to \Delta_{\alpha}(u)$. Since  $\Omega_{\alpha}$ is of finite $\mu$-measure,   as $|u| \to \infty$ we have $\mu(\Delta_{\alpha}(u) \times \Phi_{\alpha}^{-})\to 0$ and similarly $\mu(\Delta_{\alpha}(u) \times \Phi_{\alpha}^{+})\to 0$.     Taken all together, we find that $\mu(\Omega_{\alpha'}) \to \mu(\Omega_{\alpha})$ as $\alpha'\to \alpha$.

It only remains to show continuity at the boundary of each of the various $I_k$.     Fix $k$, we must show continuity from the left at $\alpha = \zeta_{k,1}$, we do so by considering the limit given by $\zeta_{k+1,h} \to \zeta_{k,1}$ as $h \to \infty$.     As in \eqref{e:lowerDzetaExpan},  with  $\mathcal D = (-1)^{n-3}, -2, w^{k+1}$,
\[  
\underline{d}{}^{\zeta_{k,1}}_{[1,\infty)}  =  (-1)^{n-2},-2,\overline{ \mathcal D}\,\;\;\; \text{and}\;\;\;  \underline{d}{}^{\zeta_{k+1,h}}_{[1,\infty)}  =  (-1)^{n-2},-2,\overline{ \mathcal D^h, w}\]
Therefore,  in terms of the argument of the previous paragraph, here also there is agreement among ever greater length words of $\{-1,-2\}^{^*}\cap \mathcal L_{\alpha'}$ with those of $\{-1,-2\}^{^*}\cap \mathcal L_{\alpha}$.  That is, the arguments above apply and give $\mu(\Omega_{\zeta_{k+1,h}}^{+}) \to \mu(\Omega_{\zeta_{k,1}}^{+})$.       

Note that that the only words   in  $\mathcal L_{\zeta_{k,1}}$ which contain the letter $k$ are the prefixes of $\overline{d}{}^{\zeta_{k,1}}_{[1,\infty)}\,$; in each,   $k$ occurs only as the initial letter.   Complementing this, words consisting of $k+1$ repeated any number of times are admissible.   On the other hand, in each $\mathcal L_{\zeta_{k+1,h}}$ there is no word with the letter $k$,  and the words consisting of $k+1$ repeated up to $h$ of times are admissible.     Thus, the arguments from above apply to show that $\mu(\Omega_{\zeta_{k+1,h}}^{-}) \to \mu(\Omega_{\zeta_{k,1}}^{-})$.     Therefore, continuity of $\alpha \mapsto  \mu(\Omega_{\alpha})$ holds on all of $(0, \gamma_{n})$.
\end{proof}

\section{Intermezzo:  Preparation for the treatment for larger values of $\alpha$}\label{ss:terseForBigAlps}
\subsection{Terse review of further notation and terminology}\label{ss:terseRev}  We complement Subsection~\ref{ss:notation} with reminders to the reader of further terminology and notation from \cite{CaltaKraaikampSchmidt}.   For the following, confer (\cite{CaltaKraaikampSchmidt}, Figures~4.1 and ~1.3).

\medskip 

Fix $n \ge 3$.   For each $\alpha$,  let 
\begin{equation}\label{e:frakBis}
\mathfrak b_{\alpha} = C^{-1}\cdot \ell_0(\alpha).
\end{equation} 
   Then 
the right endpoint $\gamma_n$  of the set of small $\alpha$ is   the value of $\alpha$ such that $\mathfrak b_{\alpha}  = r_0(\alpha)$.     The left endpoint of the set of large $\alpha$ is $\epsilon_n$, the value of $\alpha$  such that $A^{-1}C\cdot \ell_0(\alpha) = r_0(\alpha)$.   Thus,  the large $\alpha$ are  exactly the values such that the first $\alpha$-digit of $\ell_0(\alpha)$ is  $(-k,1)$ for some integer $k \ge 2$.

 For large $\alpha$,  we use exactly the same tree of words $\mathcal V$ as for small $\alpha$.    For the following, confer (\cite{CaltaKraaikampSchmidt}, Figure~6.1).
  For each $k \in \mathbb N, k \ge 2$ 
we let  
\begin{equation}\label{e:underDminKv}  \underline{d}(-k,v) = (-k)^{c_1}, (-k-1)^{d_1},\cdots,  (-k-1)^{d_{s-1}},(-k) ^{c_s},
\end{equation} 
and 
so that the corresponding subinterval of parameters is $I_{-k,v} = \{ \alpha \,|\, \underline{d}{}^{^{\alpha}}_{[1, |v|\,]} = \underline{d}(-k,v)\}$.    Thus,  (using redundant notation) set $\underline{S}(-k,v)$ to be the sum of the $c_i$ and $d_j$ of $v$, and
\begin{equation}\label{e:lKminV}
L_{-k,v}  =    (A^{-k}C)^{ c_s}\; (A^{-k-1}C)^{d_{s-1}}(A^{-k}C)^{c_{s-1}}\cdots (A^{-k-1}C)^{d_1} (A^{-k}C)^{c_1}A^{-1},
\end{equation}
then for $\alpha \in I_{-k,v}$ one has  $\ell_{\underline{S}(-k,v)}(\alpha) = 
L_{-k,v}A \cdot  \ell_0(\alpha)$.
The right endpoint of $I_{k,v}$ is denoted $\zeta_{-k,v}$;   one finds $L_{-k,v}A\cdot \ell_0(\zeta_{-k,v}) = r_0(\zeta_{-k,v})$.   
Set $R_{-k,v} = CA^{-1}CL_{-k,v}$ and define $J_{-k,v} = [\eta_{-k,v}, \zeta_{-k,v})$  where    $R_{-k,v}\cdot r_0(\eta_{-k,v})= \mathfrak b_{\eta_{-k,v}}$.

Each  $J_{-k,v}$ is partitioned into two left-closed subintervals, whose common endpoint, $\delta_{-k,v}$,  is characterized by $\ell_{\underline{S}}(\alpha) = \mathfrak b_{\alpha}$ when $\alpha = \delta_{-k,v}$.     We will refer to these as the {\em left portion} and {\em right portion} of the interval.  Recall that synchronization occurs after one step further in the $T_{\alpha}$-orbit of $r_0(\alpha)$ for $\alpha$ in the right portion as opposed to in the left portion.   To be precise,  the synchronication is given by 
\begin{equation}\label{e:synchroExplicitLargeAlps}
\begin{cases}   
\, \ell_{1+\underline{S}}(\alpha) = r_{1+\overline{S}}(\alpha) &\text{if}\;  \alpha <\delta_{-k,v}\,,\\
\\
\,\ell_{1+\underline{S}}(\alpha) = r_{2+\overline{S}}(\alpha)&\text{if}\;  \alpha >\delta_{-k,v}\,.
\end{cases}
 \end{equation} 
 
 Key to proving the synchronization results for large $\alpha$ was the determination of a maximal common prefix of the $\overline{b}_{[1, \infty)}^{\alpha}$ for $\alpha \in J_{-k,v}$.   For that, we set $u = u_{n} = (1,2)^{n-2}, (1,1)$ and   for $k\ge 2$, 
$\mathcal E = \mathcal E_k =(1,1) u^{k-2} (1,2)^{n-3}$ and $ \mathcal F =\mathcal E_{k+1}$.  The common prefix is 
\begin{equation}\label{e:digitsNotSmall} 
\overline{b}(-k,v) =  (1,2)^{n-2} \mathcal E^{c_1}  \mathcal F^{d_1} \,\mathcal E^{c_2 } \,\mathcal F^{d_2} \cdots \mathcal E^{c_{s-1}} \,\mathcal F^{d_{s-1}} \,\mathcal E^{c_s},
\end{equation} 
 denote  its length as a word in the $\alpha$-digits by  $\overline{S}(-k,v)$.   There is an expression for $R_{-k,v}$ related to $\overline{b}(-k,v)$ similar to how \eqref{e:lKminV} relates $L_{-k,v}$ to $\underline{b}(-k,v)$;  see \eqref{e:expressRsubminKv}. 

 Finally, in  the case of intermediate sized parameters, that is for $\alpha \in [\gamma_n, \epsilon_n)$,   one has as for small $\alpha$ that the first $\alpha$-digit of $\ell_0(\alpha)$ is  $(-1,1)$  and thus $k=1$; however, the dynamics are mainly of the type of the large $\alpha$.   To see this, first note  that $\gamma_n = \eta_{-1, n-2}$ and  furthermore, there exists  $\alpha \in J_{-1, n-2}$ such that $(A^{-1}C)^{n-2}\cdot \ell_0(\alpha) = C^{-1}\cdot \ell_0(\alpha)$.   For this value of $\alpha$,  $AC^2 \cdot r_0(\alpha) = \ell_0(\alpha)$.  Arguing as in (\cite{CaltaKraaikampSchmidt}, Section ~6), it follows that  $(1,2)$ is the initial digit of $r_0$ for all parameter values at least as large as this $\alpha$.  Thus for $\alpha$ to the right of $J_{-1, n-2}$, the dynamics are indeed of the type of the large $\alpha$.   On the other hand,   see Subsection~\ref{ss:ExamplesKis1FirstHalfOfSyncInt},  synchronization holds on $J_{-1, n-2}$ with $\overline{S}(-1, n-2) = 0$ and hence the fact that there are $\alpha \in  J_{-1, n-2}$ whose digit of $r_0$ is less than $(1,2)$ is insignificant for our discussions. 

  Therefore to describe the dynamics for the intermediate sized parameters,   we  restrict to words in the subtree $\check{\mathcal V}\subset \mathcal V$  defined by ($i$)  $v = c_1 d_1 \cdots c_s \in \check{\mathcal V}_n$, is such that $c_i \le n-2$, and ($ ii$) the only word in $\check{\mathcal V}_n$ with prefix $n-2$ is $v=n-2$ itself.         As discussed in (\cite{CaltaKraaikampSchmidt}, Section~7)  
we can still use  \eqref{e:digitsNotSmall} since the implied appearances of $u^{k-2}= u^{-1}$ turn out to always have neighboring appearances of  $u$ to positive powers.    Thus, we treat the two-dimensional maps for the cases of $\alpha \ge \gamma_n$   all at the same time.    We give various examples with  intermediate sized $\alpha$,  see Subsection~\ref{ss:ExamplesKis1FirstHalfOfSyncInt} and Examples~\ref{e:kIsMin1LeftEndpt}, ~\ref{e:kIs1SecondHalfOfSyncInt}, ~\ref{e:kIsMin1RightEndpt}, ~\ref{e:kIsMin1Delta}. 

\bigskip 
\begin{figure}[h]
\scalebox{.55}{
\noindent
\begin{tabular}{ll}
\begin{tikzpicture}[x=6cm,y=6cm] 
\draw  (-.32, -0.72)--(0.5, -0.72); 
\draw  (0.5, -0.72)--(0.5,  -0.44); 
\draw  (0.5, -0.44)--(1.1,  -0.44);   
\draw  (1.1, -0.44)--(1.1,  -0.39);   
\draw  (1.1, -0.39)--(1.68,  -0.39);   
\draw  (1.68, -0.39)--(1.68,  0.35); 
\draw  (1.68, 0.35)--(0.25,  0.35); 
\draw  (0.25, 0.35)--(0.25,  0.2);  
\draw  (0.25, 0.2)--(-.32,  0.2);  
\draw  (-.32, 0.2)--(-.32,  -0.72);
 \draw  (0.0, -0.72)--(0.0, .2); 
 \draw  (1, -0.44)--(1, .35); 
\draw[thin,dashed] (-0.17, -0.72)--(-0.17, 0.2);    
\draw[thin,dashed] (0.19, -0.72)--(0.19, 0.2); 
\draw[thin,dashed] (0.33, -0.72)--(0.33, 0.35);      
\draw[thin,dashed] ( 0.78, -0.44)--( 0.78, 0.35);  
\draw[thin,dashed] (0.92, -0.44)--(0.92, 0.35);  
\draw[thin,dashed] (1.2, -0.39)--(1.2, 0.35);  
\draw[thin,dashed] (1.4, -0.39)--(1.4, 0.35);  
\draw[thin,dashed] (1.08, -0.44)--(1.08, 0.35);  
\node at (-.25, 0.12) {\tiny{$(-2,1)$}};         
\node at (0.55, -0.2) {\tiny{$(1,1)$}};  
\node at (0.25, 0) {\tiny{$(2,1)$}}; 
\node at (0.85, 0.12) {\tiny{$(-2,2)$}};      
\node at (0.10, 0) {\tiny{$\cdots$}}; 
\node at (-0.1, 0) {\tiny{$\cdots$}};   
\node at (0.96, -0.01) {\tiny{$\cdots$}}; 
\node at (1.05, -0.01) {\tiny{$\cdots$}}; 
\node at (1.14, 0) {\tiny{$(3,2)$}};  
\node at (1.3, 0) {\tiny{$(2,2)$}};   
\node at (1.55,  -0.2) {\tiny{$(1,2)$}};   
\node at (-.38, -0.8) {$(\ell_0, y_{-3})$}; 
\node at (-.38, 0.3) {$(\ell_0, y_1)$}; 
\node at (0.15,  0.35)  {$(\ell_1, y_2)$};  
\node at (0.5, -0.8)  {$(r_1, y_{-3})$}; 
\node at (1.2, -0.5)  {$(r_2, y_{-2})$}; 
\node at (.78, -0.5)  {$\mathfrak b$};
\node at (1.7, -0.5)  {$(r_0, y_{-1})$};  
\node at (1.7,0.4)  {$(r_0, y_{2})$};   
\node at (0,0)  {$0$};  
\node at (1, -0.5)  {$1$};  
 \foreach \x/\y in {-.32/-0.72, 0.5/-0.72, 1.1/-0.44, 
 1.68/-0.39,  1.68/0.35, 0.25/0.35, -.32/0.2%
} { \node at (\x,\y) {$\bullet$}; } 
\draw[->, ultra thick] (1.5,  -0.65)--(1.6,  -0.65) -- (1.55, -0.7) -- (1.8,-0.7);
\end{tikzpicture}  
&
\begin{tikzpicture}[x=6cm,y=6cm] 
\draw  (-.32, -0.72)--(0.5, -0.72); 
\draw  (0.5, -0.72)--(0.5,  -0.44); 
\draw  (0.5, -0.44)--(1.1,  -0.44);   
\draw  (1.1, -0.44)--(1.1,  -0.39);   
\draw  (1.1, -0.39)--(1.68,  -0.39);   
\draw  (1.68, -0.39)--(1.68,  0.35); 
\draw  (1.68, 0.35)--(0.25,  0.35); 
\draw  (0.25, 0.35)--(0.25,  0.2);  
\draw  (0.25, 0.2)--(-.32,  0.2);  
\draw  (-.32, 0.2)--(-.32,  -0.72); 
\draw  (-.32, 0)--(1.68, 0);   
\draw  (-.5,  .75)--(-0.1, .75);      
\draw  (-.1,  .75)--(-0.1, .45);   
\draw  (-.1,  .45)--(-0.5, .45);   
\draw  (-.5,  .45)--(-0.5, .75);  
\draw  (1, -0.55)--( 1.4, -0.55);      
\draw  (1.4, -0.55)--(1.4, -0.85);   
\draw  (1.4, -0.85)--(1,  -0.85);   
\draw  (1, -0.85)--(1,  -0.55);        
\draw[thin,dashed] (-.32, -0.44)--(0.5, -0.44);      
\draw[thin,dashed] (-.32, -0.35)--(1.68, -0.35);      
\draw[thin,dashed] (-.32, -0.28)--(1.68, -0.28);  
\draw[thin,dashed] (0.25, 0.27)--(1.68, 0.27);      
\draw[thin,dashed] (0.25, 0.2)--(1.68, 0.2);     
\draw[thin,dashed] (-.5,  .6)--(-.1,  .6);             
\draw[thin,dashed] (1,  -.7)--(1.4 , -0.7);           
\draw[thin,dashed] (-0.32, -0.2)--(1, -0.2);                
\draw[thin,dashed] (1, -0.2)--(1, -0.17);  
\draw[thin,dashed] (1, -0.17)--(1.68, -0.17);  
\draw[thin,dashed] (-0.32,  -0.12)--(1.68, -0.12); 
\node at (-.35, 0) {$0$};               
\node at (0, -.6) {$(1,2)$};  
\node at (0, -.4) {$(1,1)$};  
\node at (0.7, -.32) {$(2,2)$};   
\node at (1.1, -.23) {$(2,1)$};
\node at (.85, -.16) {$(3,2)$}; 
\node at (1,  .3) {$(-2,1)$};  
\node at (1,  .23) {$(-2,2)$};  
\node at ( 0.7,  .15) {$(-3,1)$}; 
\node at ( -0.3, .7) {$(k,l)$};            
\node at ( -0.3, .5) {$(k,l+1)$};  
\node at (1.2,  -.8) {$(k,l)$};            
\node at (1.2, -.6) {$(k+1,l+1)$};                 
\node at (-.32, 0) {$\bullet$};         
\node at (0.7, .085) {$\vdots$};  
\node at (0.7, -.05) {$\vdots$};  
\end{tikzpicture}  
\end{tabular}
}
\caption{The domain $\Omega_{3, 0.86}$. Blocks $\mathcal B_{i,j}$   and their images, both denoted by $(i,j)$. Here $L_{-k,v} = A^{-2}CA^{-1}$ and $R_{-k,v} = AC \,AC^{2}$, and $\alpha$ is an interior point of $J_{-2,1}$ lying to the left of $\delta_{-2,1}$.  Also, hints as to the lamination ordering given in small boxes.}
\label{f:omegaLargeAlpLessThanDelta_{-k,v}}%
\end{figure}

\bigskip

\section{The case of  large $\alpha$ in left portion of synchronization interval}\label{s:bigLefties}   Devoted to $\alpha \ge \gamma_n$, this section and the next parallel Sections~\ref{s:relationsOnHts} and ~\ref{s:MainForSmallAlps}.    The relationships between the various rectangle heights in this current are more intricate,  see Figure~\ref{f:orbitRelations} for an indication of this.

\subsection{Definition of $\Omega_{\alpha}$  for large $\alpha$ in left portion of synchronization interval}\label{ss:theTwoHalvesLargeAlps}

Fix $n\ge 3, k \in \mathbb N$.    If  $k\ge 2$, fix  $v \in \mathcal V$, and otherwise fix $v\in \check{\mathcal V}$.      Finally, fix $\alpha$ in the interior of $J_{-k,v}$.   

Compare the following definition with  Figures~\ref{f:omegaLargeAlpLessThanDelta_{-k,v}} and \ref{f:topBottomVertices}.

\begin{Def}\label{d:topYvaluesLargeAlp}   Let  $\underline{S} = \underline{S}(-k,v)$ and  
define   $\Omega^{+}$ as in Definition~\ref{d:topYvalues}  except that 
 \[ \text{For all}\;  i \in \{0, \dots,  \underline{S}\},\, \text{we   set}\; y_{\tau(i)} = - \ell_{\underline{S}-i}(\eta_{-k,v}).\] 
\end{Def}

\bigskip 
\begin{figure}[h]
\scalebox{.42}{
\noindent
\begin{tabular}{lcr}
\begin{tikzpicture}[x=8cm,y=8cm] 
\draw  (-.32, 0)--(-.32, 0.2); 
\draw  (-.32, 0.2)--(-.25, 0.2); 
\draw  (-.25, 0.2)--(-.15, 0.2); 
\draw  (-.15, 0.2)--(-.15, 0.3);  
\draw  (-.15, 0.3)--(-.1, 0.3); 
\draw  (-.1, 0.3)--(-.1, 0.4);
\draw  (-.1, 0.4)--(-.05, 0.4); 
\draw  (0,0.45)--(0, 0.5);    
\draw  (0,0.5)--(0.15, 0.5);    
\draw  (0.15,0.5)--(0.15, 0.6);  
\draw  (0.15,0.6)--(0.2, 0.6); 
\draw  (0.3, 0.65)--(0.3, 0.7);
\draw  (0.3, 0.7)--(0.35, 0.7);   
\draw  (0.35, 0.7)--(0.35, 0.8); 
\draw  (0.35, 0.8)--(1, 0.8); 
\draw  (1, 0.8)--(1, 0.9);
\draw  (1, 0.9)--(1.4, 0.9);  
\draw  (1.4, 0.9)--(1.4, 0); 
\draw[thin,dashed] (0.1, 0)--(0.1, 0.5);  
\draw[thin,dashed] (0.55, 0)--(0.55, 0.8); 
\draw[thin,dashed] (1.1, 0)--(1.1, 0.9); 
\node at (-0.09,0.12) {$(-k,1)$};   
\node at (.35, 0.18) {$(-k-1,1)$};  
 \node at (-.42, 0.25) {$(\ell_0, y_1)$}; 
 \node at (-0.2, 0.55) {$(\ell_{\underline{S}-1}, y_{\tau(\underline{S}-1)})$}; 
 \node at (0.34, .9) {$(\ell_{\iota}, y_{\underline{S}})$};  
 \node at (.9, 1) {$(\ell_{\underline{S}}, y_{\underline{S}+1})$};    
 \node at (1.45, 1) {$(r_0,  y_{\underline{S}+1})$};          
 \node at (1.1, -0.05){$\mathfrak b$} ;        
 \foreach \x/\y in {-.32/0.2, 0.0/0.5, 0.35/0.8, 
1/0.9,  1.4/0.9%
} { \node at (\x,\y) {$\bullet$}; }           
\end{tikzpicture}
&\;\;\;\;\;
&
\begin{tikzpicture}[x=8cm,y=8cm] 
\draw  (-.32, .6)--(-.32, 0); 
\draw  (-.32, 0)--(0.25, 0);  
\draw  (0.25, 0)--(0.25, 0.05); 
\draw  (0.25, 0.05)--(0.3, 0.05);  
\draw  (0.3, 0.05)--(0.3, 0.1);   
\draw  (0.3, 0.1)--(0.35, 0.1);  
\draw  (0.4, 0.12)--(0.4, 0.15);  
\draw  (0.4, 0.15)--(0.45, 0.15);
\draw  (0.45, 0.15)--(0.9, 0.15);
\draw  (0.9, 0.15)--(0.9, 0.2);    
\draw  (0.9, 0.2)--(1.05, 0.2);  
\draw  (1.05, 0.2)--(1.05, 0.25); 
\draw  (1.05, 0.25)--(1.1, 0.25);      
\draw  (1.2, 0.3)--(1.2, 0.35);   
\draw  (1.2, 0.35)--(1.25, 0.35);
\draw  (1.25, 0.35)--(1.25, 0.4);    
\draw  (1.25, 0.4)--(1.3, 0.4);   
\draw  (1.3, 0.4)--(1.4, 0.4);    
\draw  (1.4, 0.4)--(1.4, 0.6);  
\draw[thin,dashed] (0.2, 0)--(0.2, 0.6);  
\draw[thin,dashed] (0.55, 0.15)--(0.55, 0.6);
\draw[thin,dashed] (0.95, 0.215)--(0.95, 0.6); 
 \node at (.35, 0.27) {$(1,1)$}; 
  \node at (1.1, 0.4) {$(1,2)$}; 
  \node at (-.46, -0.05) {$(\ell_0, y_{-\overline{S}-1})$};        
  \node at (0.25, -0.1) {$(r_{j_{-\overline{S}-1}}, y_{-\overline{S}-1})$};       
  \node at (.8, 0.05) {$(r_{\overline{S}}, y_{\beta(\overline{S})})$};    
  \node at (1.23, 0.12) {$(r_{-1 + j_{-\overline{S}-1}}, y_{1 + \beta(\overline{S})})$};     
\node at (1.4,  .3) {$(r_0,  y_{-1})$};    
 \foreach \x/\y in {-.32/0, 0.25/0, .9/0.15, 
1.05/0.2,  1.4/0.4%
} { \node at (\x,\y) {$\bullet$}; }  
\end{tikzpicture}
\end{tabular}
}
\caption{Schematic representations showing the most important vertices of the tops and bottoms of blocks (not to scale), for $\alpha \in (\eta_{-k,v}, \delta_{-k,v})$.   Blocks $\mathcal B_{(i,j)}$   denoted by $(i,j)$.} 
\label{f:topBottomVertices}%
\end{figure}
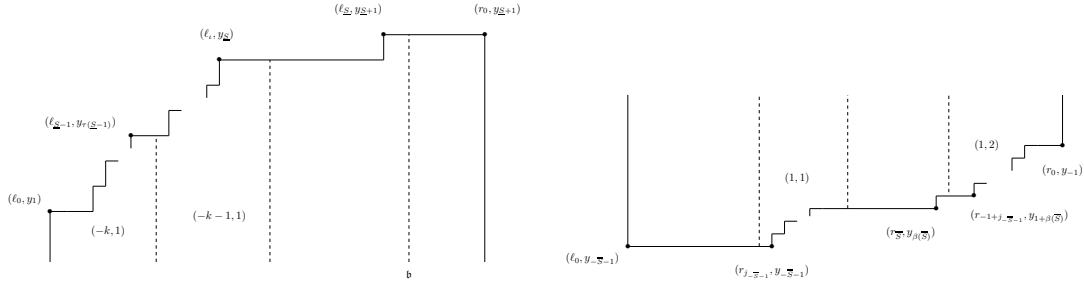

\begin{Def}\label{d:bottomYvaluesLargeAlps}  Let $\overline{S} = \overline{S}(-k,v)$ and 
define   $\Omega^{-}$ as in Definition~\ref{d:bottomYvalues}   except that we let 
\[y_b = \begin{cases}   - \hat r_{j_{-(b+  e +1)}}\;\;\text{if}\; -1\ge b \ge  - e;\\
                                      \\
                                       - \hat r_{j_{-(b+  e  +  \overline{S} +2)}}\;\;\text{otherwise}, 
            \end{cases} 
\]
where for $0 \le i \le \overline{S}$, 
\[\hat r_i = \begin{cases}  C\cdot  r_i(\eta_{-k,v}),\;\;\text{if}\; r_i(\eta_{-k,v}) \in \Delta_{\eta_{-k,v}}(1,2);\\
                                        r_i(\eta_{-k,v}), \;\; \;\; \;\; \;\text{otherwise},
                \end{cases}
\]
and $e = e(-k,v)$ is the number of occurrences of $(1,2)$ in $\overline{b}(-k,v)$.   See Figure~\ref{f:meaningOfrHat} for an indication of these values. 
\end{Def}

\bigskip 
Note that since $\overline{b}(-k,v)$ only has digits $\{(1,2), (1,1)\}$, each $\hat r_i \le \mathfrak b_{\eta_{-k,v}}$ with equality exactly when $i = \overline{S}$.   Furthermore, 
$e(-k,v) = \# \{ i \mid \hat r_i  = C\cdot \hat r_i \}$.\\  

\bigskip

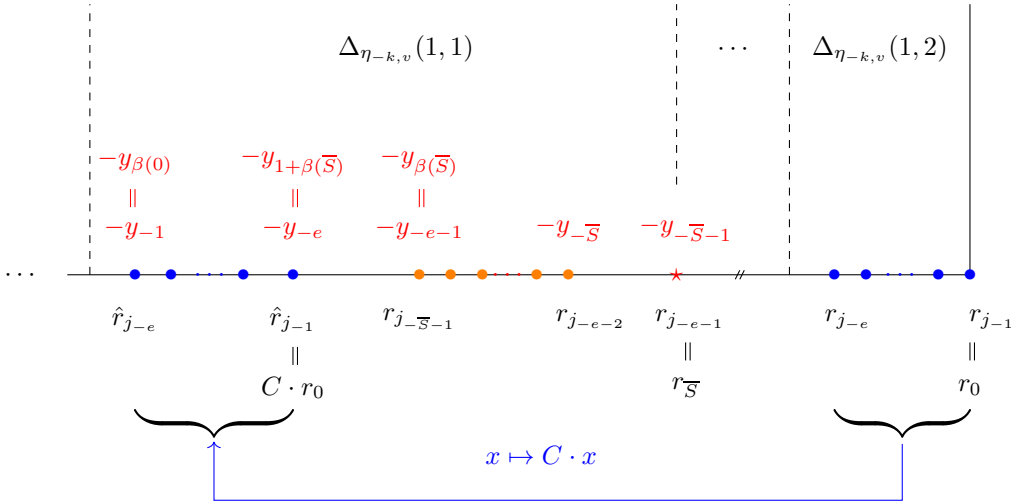
\begin{figure}[h]
\scalebox{1}{
\begin{tikzpicture}[x=6cm,y=6cm] 
\draw  (-1, 0)--(1, 0); 
\node at (-1.1, 0) {$\cdots$}; 
\draw[thin,dashed] (-0.95, 0)--(-0.95, 0.6);  
\draw[thin,dashed] (0.35, 0.2)--(0.35, 0.6);  
\node at (-.85, -0.1) {$\hat{r}_{j_{-e}}$}; 
\node at (-.85, 0.25) [red] {$-y_{\beta(0)}$}; 
\draw[red]  (-0.85, 0.15)--(-0.85, 0.19); 
\draw[red]   (-0.86, 0.15)--(-0.86, 0.19); 
\node at (-.85, 0.1) [red] {$-y_{-1}$}; 
\node at (-0.68, 0)[blue] {$\cdots$}; 
\node at (-0.5, -0.1) {$\hat{r}_{j_{-1}}$}; 
\node at (-.5, 0.25) [red] {$-y_{1 +\beta(\overline{S})}$}; 
\draw[red]  (-0.5, 0.15)--(-0.5, 0.19); 
\draw[red]   (-0.49, 0.15)--(-0.49, 0.19); 
%
\draw  (-0.5, -0.16)--(-0.5, -0.2); 
\draw  (-0.49, -0.16)--(-0.49, -0.2); 
\node at (-0.5, -0.25) {$C \cdot r_0$};
\node at (-.5, 0.1) [red] {$-y_{-e}$}; 
%
\node at (-0.22, 0.25) [red] {$-y_{\beta(\overline{S})}$}; 
\draw[red]  (-0.22, 0.15)--(-0.22, 0.19); 
\draw[red]   (-0.21, 0.15)--(-0.21, 0.19); 
\node at (-0.22, 0.1) [red] {$-y_{-e-1}$}; 
\node at (-0.22, -0.1) {$r_{j_{-\overline{S}-1}}$};
\node at (-0.02, 0) [red] {$\cdots$};
\node at (0.11, 0.1) [red] {$-y_{-\overline{S}}$}; 
\node at (0.16, -0.1) {$r_{j_{-e-2}}$};
\node at (0.37, 0.1) [red] {$-y_{-\overline{S}-1}$}; 
\node at (0.38, -0.1) {$r_{j_{-e-1}}$};
\draw  (0.37, -0.15)--(0.37, -0.19); 
\draw (0.38, -0.15)--(0.38, -0.19); 
\node at (0.37, -0.25) {$r_{\overline{S}}$};
%
\draw  (0.48, -0.01)--(0.49, 0.01); 
\draw   (0.49,  -0.01)--(0.5, 0.01);  
\draw[thin,dashed] (0.6, 0)--(0.6, 0.6);  
\draw[thin] (1, 0)--(1, 0.6);  
\node at (0.73, -0.1) {$r_{j_{-e}}$};
\node at (0.85, 0) [blue] {$\cdots$}; 
\node at (1.05, -0.1) {$r_{j_{-1}}$};
\draw  (1, -0.15)--(1, -0.19); 
\draw (1.01, -0.15)--(1.01, -0.19); 
\node at (1, -0.25) {$r_0$};
\draw [decorate, ultra thick,
    decoration = {calligraphic brace,mirror,amplitude=10pt}] (.7,-0.3) --  (1,-0.3);
\draw [decorate, ultra thick,
    decoration = {calligraphic brace,mirror,amplitude=10pt}] (-0.85,-0.3) --  (-0.5,-0.3);
\draw[->, blue] (0.85,-0.375) -- (0.85, -0.5) -- (-0.675,-0.5) -- (-0.675,-0.37);
\node at (0.05, -0.4) [blue] {$x \mapsto C\cdot x$};
\node at (-0.25, 0.50) {$\Delta_{\eta_{-k,v}}(1,1)$};
\node at (0.48, 0.5){$\cdots$}; 
\node at (0.8, 0.50) {$\Delta_{\eta_{-k,v}}(1,2)$};
 \foreach \x/\y in {-.85/0, -.77/0, -0.61/0,-0.5/0%
} { \node at (\x,\y)[blue] {$\bullet$}; } 
 \foreach \x/\y in {-0.22/0,   -0.15/0, -0.08/0, 0.04/0, 0.11/0}%
{ \node at (\x,\y)[orange] {$\bullet$}; } 
\node at (0.35, 0) [red] {$\star$};
 \foreach \x/\y in {0.7/0,  0.77/0, 0.93/0, 1/0%
} { \node at (\x,\y)[blue] {$\bullet$}; } 
\end{tikzpicture}  
}
\caption{Bottom heights are of the form $-\hat{r}_j$, see Definition~\ref{d:bottomYvaluesLargeAlps}.  For $r_j(\eta_{k,v}) \in \Delta_{\eta_{k,v}}(1,2)$, we have $\hat r_j = C\cdot r_j(\eta_{k,v})$. That $\hat{r}_{j_{-1}} < r_{j_{-\overline{S}-1}}(\eta_{k,v})$ is shown in Lemma~\ref{l:htsIncrease}.}
\label{f:meaningOfrHat}%
\end{figure}

 The following result gives a hint of the use of the $\hat r_j$. 
 \begin{Lem}\label{l:whyHatr}  Considering $\overline{b}(-k,v)$ as a word in $\{(1,2), (1,1)\}$, let $q_j$ be the first letter of $\sigma^j(\overline{b}(k,v)\,)$, and set $M_j = A C^{l_j}$ where $q_j = (1, l_j)$   for $0 \le j < \overline{S}$.
Then 
 \[ R M_j R^{-1} \cdot - \hat r_j = \begin{cases}  - \hat r_{j-1}\;\;\;\text{if } \;\;1 \le j <\overline{S};\\
                                                                           -\hat r_{_{\overline{S}}} \;\;\;\text{if }\;\; j=0 
                                                  \end{cases}
\]
and $R AC R^{-1} \cdot  -\hat r_{_{\overline{S}}} =  -\hat r_{_{\overline{S}-1}} $.
 \end{Lem}

\begin{proof} Since the initial digit of  $\overline{b}(-k,v)$ is $(1,2)$, using Lemma~\ref{l:conjByRofAtoPc}  we evaluate 
\[
\begin{aligned}
 R AC^2 R^{-1}\cdot -\hat r_0 &= -(CAC)^{-1}\cdot (C \cdot r_0(\eta_{-k,v})\,) = -(AC)^{-1}\cdot r_0(\eta_{-k,v})\\
   &= -  C^{-1}\cdot \ell_0(\eta_{-k,v})= -\mathfrak b_{\eta_{-k,v}} = -r_{_{\overline{S}}}(\eta_{-k,v}) = -\hat r_{_{\overline{S}}}. 
 \end{aligned}
 \]
 
Similarly, we treat the four possible cases when $1 \le j <\overline{S}$.   If $M_j = AC$ and $r_{j-1}(\eta_{-k,v}) \in \Delta_{\eta_{-k,v}}(1,1)$, then 
\[R AC R^{-1}\cdot -\hat r_j = - (AC)^{-1}\cdot r_j(\eta_{-k,v}) = -r_{j-2}(\eta_{-k,v}) =   - \hat r_{j-1}.\]
Whereas if instead $r_{j-1}(\eta_{-k,v}) \in \Delta_{\eta_{-k,v}}(1,2)$, then
\[R AC R^{-1}\cdot -\hat r_j = - (AC)^{-1}\cdot (AC^2 \cdot r_{j-1}(\eta_{-k,v})\,)  = - C\cdot r_{j-1}(\eta_{-k,v}) =   - \hat r_{j-1}.\]
If  $M_j = AC^2$ and $r_{j-1}(\eta_{-k,v}) \in \Delta_{\eta_{-k,v}}(1,1)$, then 
\[R AC^2 R^{-1}\cdot -\hat r_j = - (CAC)^{-1}\cdot C \cdot r_j(\eta_{-k,v}) = -r_{j-1}(\eta_{-k,v}) =   - \hat r_{j-1}.\]
And,  if instead $r_{j-1}(\eta_{-k,v}) \in \Delta_{\eta_{-k,v}}(1,2)$, then
\[R AC^2 R^{-1}\cdot -\hat r_j = - (CAC)^{-1}\cdot C \cdot r_j(\eta_{-k,v}) = - C\cdot r_{j-1}(\eta_{-k,v}) =   - \hat r_{j-1}.\]

Finally, 
 \[
 R AC R^{-1}\cdot -\hat r_{_{\overline{S}}} = -(AC)^{-1}\cdot r_{_{\overline{S}}}(\eta_{-k,v}) =-\hat r_{_{\overline{S}-1}},
 \]
by considering the two possibilities for $l_{\overline{S}-1}$.
\end{proof}

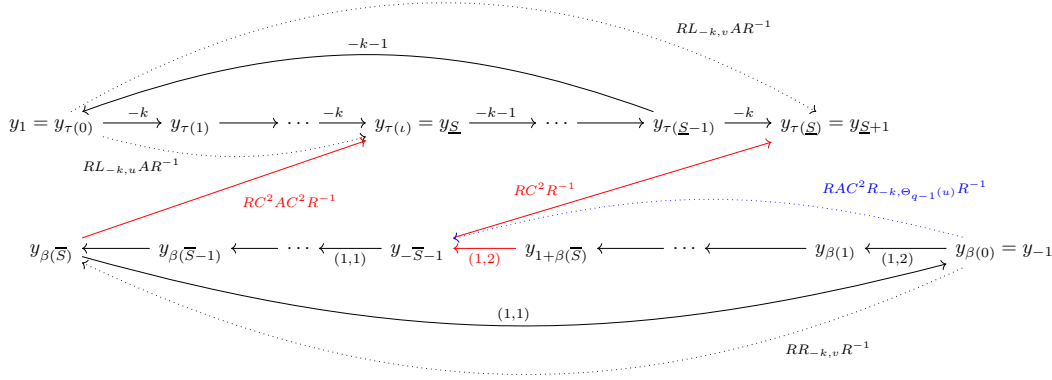
\begin{figure}
\begin{tikzpicture}[baseline= (a).base]
\node[scale=.75] (a) at (0,0){
\begin{tikzcd}[column sep=2pc,row sep=2pc]
y_1= y_{\tau(0)}  \ar[pos=0.6]{r}{-k} \ar[bend left = 30, dotted, pos=0.8]{rrrrrr}{\;R L_{-k,v}AR^{-1}}\ar[bend right = 15, dotted,swap,pos=0.3]{rrr}{R L_{-k,u}AR^{-1}}
 &y_{\tau(1)}  \ar{r}& \cdots  \ar[pos=0.3]{r}{-k} &y_{\tau(\iota)}=y_{\underline{S}}  \ar[pos=0.4]{r}{-k-1}& \cdots  \ar{r}&y_{\tau(\underline{S}-1)}  \ar[pos=0.4]{r}{-k} \ar[bend right = 20, swap, pos=0.5]{lllll}{-k-1} &y_{\tau(\underline{S})}= y_{\underline{S}+1} \\
\\
y_{\beta(\overline{S})}\ar[bend right = 15,  pos=0.5]{rrrrrrr}{(1,1)}\ar[red, swap, pos=0.5]{uurrr}{\;\;\;RC^2AC^2R^{-1}}&y_{\beta(\overline{S}-1)} \ar{l}& \cdots  \ar{l} &
y_{-\overline{S}-1}\ar[pos=0.5]{l}{(1,1)} \ar[red, pos=0.4]{uurrr}{RC^2R^{-1}}&y_{1+\beta(\overline{S})}    \ar[red, pos=0.5]{l}{(1,2)}& \cdots  \ar{l} 
&y_{\beta(1)}  \ar{l}&y_{\beta(0)} = y_{-1}\ar[pos=0.6]{l}{(1,2)}\ar[bend left = 25, dotted, pos=0.2]{lllllll}{R R_{-k,v}R^{-1}}\ar[blue, dotted, swap, bend right=15, pos=0.3]{llll}{\;\;\;R AC^2R_{-k,\Theta_{q-1}(u)}R^{-1}}
\end{tikzcd}
};
\end{tikzpicture}
\caption{{\bf Relations on the heights of rectangles for general $k, v$ and $\alpha \in (\eta_{-k,v}, \delta_{-k,v})$.}  The equality $\tau(\iota) = \underline{S}$  is  Lemma~\ref{l:largestLvalue}.   The red paths are used to prove that lamination occurs; the red arrows are due to Proposition~\ref{p:RelationBottomAndTop} and Corollary~\ref{c:someLamRels}.    The blue dotted arrow is due to Lemma~\ref{l:smallestRvalue}.   The highest solid arrow is due to Lemma~\ref{l:relForSharedTopOfTwoLeftmostCylinders}, the lowest solid arrow is due to Corollary~\ref{c:someLamRels}.   The lowest dotted arrow is due to Corollary~\ref{c:theWholeWord}.  
The top horizontal arrows are as for the small $\alpha$ setting,  the lower horizontal arrows follow mainly from the definition of the $y_b$ combined with Lemma~\ref{l:whyHatr}.    The top two dotted  arrows are implied by definitions, in direct analogy with the case of small $\alpha$.       
}
\label{f:orbitRelations}
\end{figure}

\smallskip

\subsection{Examples with `intermediate' values}\label{ss:ExamplesKis1FirstHalfOfSyncInt}

Recall from  (\cite{CaltaKraaikampSchmidt}, Example~7.4)
 that $R_{-1, n-2} = CA^{-1}C\; (A^{-1}C)^{n-2}A^{-1} = \text{Id}$ holds for all $n\ge 3$; thus $\overline{S} (-1, n-2)= 0$.  Recall that $\eta_{-1,n-2} = \gamma_n$, from whose definition it easily follows that the initial digit of $r_0(\gamma_n)$ is $(-1,2)$.   Thus, on $J_{-1, n-2}$ we have a single bottom height of $y_{-1} = - r_0(\gamma_n)$.   

We now fix $n=3$.   Here $\eta_{-1,1} = \gamma_3 = g^2/2$ where  $g = G-1$ and $G = (1+\sqrt{5})/2$.   Thus $r_0(\eta_{-1,1}) = g^2$; one easily verifies that the $T_{g^2/2}$-orbit of $\ell_0(g^2/2)$ is periodic of length one with preperiod $-G\mapsto -g^2$.  Therefore,  for all $\alpha \in (\gamma_3, \delta_{-1,1})$  we have $y_1 = g^2,  y_2 = G$ and $y_{-1}  =- g^2$.   That is, we find a ``backwards L"-shape:
\[\Omega_{3,\alpha} = (\, [\ell_0(\alpha), \ell_1(\alpha)]\times [-g^2, g^2]\,) \cup (\,  [\ell_1(\alpha), r_0(\alpha)]\times [-g^2, G]\,)\;\; \forall \alpha \in (\gamma_3, \delta_{-1,1}).\]

\smallskip
\subsection{Relations on heights of rectangles, for large $\alpha$ in left portion of synchronization interval}\label{s:RelHtsLargeAlps}  

 The main aim of this subsection is to establish the relations needed to show that $\mathcal T_{n,\alpha}$ is bijective up to measure zero on $\Omega_{n,\alpha}$.   Other than the result that the vertices have the correct images (see Lemma~\ref{l:vertexProgression}, below),  these relations are all indicated in Figure~\ref{f:orbitRelations}.   We also show that $\Omega_{n, \alpha}$ has the expected shape --- the vertices increase in $y$-value as their $x$-coordinates increase.  See Lemma~\ref{l:htsIncrease} for the bottom vertices, the similar behavior of top vertices follows directly from arguments of the small $\alpha$ setting.
 
 Note that since Subsection~\ref{ss:ExamplesKis1FirstHalfOfSyncInt} directly establishes that the lower boundary of $\Omega_{\alpha}$ is of constant height when $\alpha \in J_{-1, n-2}$, in what follows   we will occasionally tacitly assume that $\alpha>\zeta_{-1, n-2}$.

\bigskip

 \begin{Lem}\label{l:oneLamRel}   For $\alpha \in [\eta_{-k,v}, \delta_{-k,v})$, we have\\
  
 \begin{enumerate} 
 \item[i.)]   $r_{\overline{S}}(\alpha) \ge \mathfrak b_{\alpha}$ with equality only for $\alpha = \eta_{-k,v}$, \\

 \item[ii.)]    $r_{\overline{S}}(\alpha)$ lies to the left of $\Delta_{\alpha}(1,2)$,\\
 
\item[iii.)] $\beta(\overline{S}) = -1 - e(-k,v)$,\\
 
\item[iv.)] $\beta(-1+ j_{-\overline{S}-1}) = - e(-k,v)$.
 \end{enumerate}
\end{Lem}

\begin{proof} For all $k \in \mathbb N$ and $v \in \mathcal V$,  the definition of  $\eta_{-k,v}$ results in $r_{\overline{S}}(\eta_{-k,v})= \mathfrak b_{\eta_{-k,v}}$.   The commonality of the initial digits along the synchronization interval then leads to the inequality. 

For $\alpha \in (\eta_{-k,v}, \delta_{-k,v})$ we have $\ell_{\underline{S}}(\alpha) < \mathfrak b_{\alpha}$, and hence  $C^{-1}A C^{-1}\cdot r_{\overline{S}}(\alpha) < C^{-1}\cdot \ell_0(\alpha)$,  whence $A C^{-1}\cdot r_{\overline{S}}(\alpha)< \ell_0(\alpha)$.  Since $C^{-1}$ acts as $C^2$, each $r_{\overline{S}}(\alpha)$ lies to the left of $\Delta_{\alpha}(1,2)$.

We thus have that there are $e(-k,v)$ orbit entries larger than $r_{\overline{S}}(\alpha)$, and hence $\beta(\overline{S}) = -1 - e(-k,v)$.

 We certainly have that $r_{j_{-\overline{S}-1}}$ is the  leftmost $r_j$ ($j< \overline{S}$) in  the cylinder $\Delta_{\alpha}(1,1)$.     Since all $r_j$  with $j< \overline{S}$ are in $\Delta_{\alpha}(1,1) \cup \Delta_{\alpha}(1,2)$ while  $r_{\overline{S}}$  lies between these cylinders,  it follows that $r_{-1+j_{-\overline{S}-1}}$ is the image of the leftmost $r_j$ in 
$\Delta_{\alpha}(1,2)$. That is,  $\beta(-1+ j_{-\overline{S}-1}) = -e(-k,v)$. 
\end{proof}

\newpage%

 Compare the following  with Figures~\ref{f:orbitRelations} and \ref{f:topBottomVertices}. 
 \begin{Cor}\label{c:someLamRels}     The following hold:
 
 \begin{enumerate}
 \item[i.)] 
 $RC^2R^{-1} \cdot y_{-\overline{S}-1} = y_{\underline{S}+1}$, \\
 
 \item[ii.)] $y_{-\overline{S}-1}= R AC^2 R^{-1}\cdot y_{1 + \beta(\overline{S})}\,$,\\
 
 \item[iii.)]  $y_{\underline{S}+1}= R C^{-1}AC^2 R^{-1}\cdot y_{1 + \beta(\overline{S})}\,$,\\
 
  \item[iv.)]  $y_{-1}= R AC R^{-1}\cdot y_{\beta(\overline{S})}\,$.
 \end{enumerate} 
\end{Cor}

\begin{proof}
By definition,  $y_{-\overline{S}-1} =  - \hat r_{j_{- e(-k,v)-1}}$.   The lemma then implies that $y_{-\overline{S}-1} =  -  r_{\overline{S}}(\eta_{-k,v})$.  By definition, this equals $-\mathfrak b_{\eta_{-k,v}}$.    Lemma~\ref{l:conjByRofAtoPc}  gives $RC^2R^{-1} \cdot  (-\mathfrak b_{\eta_{-k,v}}) = - C\cdot \mathfrak b_{\eta_{-k,v}} = -  \ell_0(\eta_{-k,v})$.   This final value equals  $y_{\underline{S}+1}$.

From the lemma, $\beta(\overline{S})+1 =  - e(-k,v)$.  Hence $y_{1 + \beta(\overline{S})} = -\hat r_{j_{-1}}$.   Since $j_{-1} = 0$ we have that $y_{1 + \beta(\overline{S})} = -\hat r_0$ and Lemma~\ref{l:whyHatr}  gives $R AC^2 R^{-1}\cdot (-\hat r_0) = - \hat r_{\overline{S}}$. From the above, this value is indeed  $y_{\underline{S}+1}$.   

The third equality follows directly from the first two. 

Finally, from Definition~\ref{d:bottomYvaluesLargeAlps}, Lemma~\ref{l:whyHatr}, Lemma~\ref{l:oneLamRel} and again Definition~Lemma~\ref{d:bottomYvaluesLargeAlps}, we have
\[ R AC R^{-1}\cdot y_{\beta(\overline{S})} = R AC R^{-1}\cdot  (-\hat r_{j_{-\overline{S}-1}}) = -\hat r_{-1 +j_{-\overline{S}-1}} = -\hat r_{j_{-e} }= y_{-1}.\]
\end{proof} 
 
\bigskip

Recall from \cite{CaltaKraaikampSchmidt} that $U = AC (AC^2)^{n-2}$, $\tilde{\mathcal E} = (AC^2)^{n-3}U^{k-2}AC$, and $\tilde{\mathcal F} = (AC^2)^{n-3}U^{k-1}AC$.  And, 
\begin{equation} \label{e:expressRsubminKv}
R_{-k,v} =  \tilde{\mathcal E}^{c_s}  \tilde{\mathcal F}^{d_{s-1} }\;\tilde{\mathcal E}^{c_{s-1}} \tilde{\mathcal F}^{d_{s-2}}\cdots \tilde{\mathcal E}^{c_{2}}\tilde{\mathcal F}^{d_{1}} \tilde{\mathcal E}^{c_1}(AC^2)^{n-2}.
\end{equation}
 Recall also that $r_{j_{-\overline{S}-1}}$ is the least element in the initial $T_{\alpha}$-orbit of $r_0 = r_0(\alpha)$; from  \S~\ref{ss:ExamplesKis1FirstHalfOfSyncInt},  when $(k,v) = (1, n-2)$ that initial orbit consists only of $r_0$ itself.
 In the following, we temporarily use $s$   to also denote a  word in $\mathcal V$.

 \begin{Lem}\label{l:smallestRvalue}   Suppose that $(k,v) \neq (1, n-2)$.   Then
 \[ r_{j_{-\overline{S}-1}} = \begin{cases}    (AC^2)^{n-3-c}\cdot r_0&\text{if} \;\; v = c,  k=1\,;\\
 \\
  (AC^2)^{n-2} U^{k-2}\cdot r_0&\text{if} \;\; v = c,  k>1\,;\\
 \\
                                                                  AC^2 R_{-k,\Theta_{q-1}(s)}\cdot r_0&\text{if}\;\; v = \Theta_{0}^{h}\circ\Theta_q(s),\, \text{with}\; q\ge 1, h\ge 0.\\
                                        \end{cases}
\]   
\end{Lem}
\begin{proof}  We first treat the  case of $k\ge 2$. We can uniquely express $\overline{b}(-k,v)$ as a word in the letters that we temporarily use:  $a:= (1,1),  b:= (1,2)^{n-3}, c:= (1,2)^{n-2}$,  where we view these letters as having increasing value (of course, when $n=3$ we suppress $b$).   We  introduce a fourth symbol,  $\star$ to represent the expansion of $r_{\overline{S}}$,   so that $\overline{b}_{[1, \infty)}^{\alpha}$ can be viewed as a  word of length  $\overline{S}(-k,v)+1$ in these four letters, with the final symbol appearing only as a suffix.  Of course,  $r_{j_{-\overline{S}-1}}$ has expansion given by the least suffix of this word. 

  In terms of these letters,   $u = c a, \mathcal E = a(ca)^{k-2}b, \mathcal F = a(ca)^{k-1}b$.    By Lemma~\ref{l:oneLamRel},  $a \prec \star \prec b \prec  c$,  and thus   $\mathcal E \prec \mathcal F \prec u$.  
  
 Since \eqref{e:expressRsubminKv} gives the expression for cancelling the initial digits of $\overline{b}_{[1,\infty)}^{\alpha}$,
\begin{equation}\label{e:rightExpansionSimplifiedLetters}
\overline{b}_{[1,\infty)}^{\alpha} = c [a(ca)^{k-2}b]^{c_1}  [a(ca)^{k-1}b]^{d_1} \cdots [a(ca)^{k-2}b]^{c_{s-1}} \,\mathcal [a(ca)^{k-1}b]^{d_{s-1}} \,[a(ca)^{k-2}b]^{c_s} \star.
\end{equation}
 If $v = 1$, then $\overline{b}_{[1,\infty)}^{\alpha} =  c(ac)^{k-2}  a b \,  \star$.  The least suffix is visibly $ab\star$ and we find
 \[r_{j_{-\overline{S}-1}} =  [(AC^2)^{n-2} AC]^{k-2}(AC^2)^{n-2} \cdot r_0=   (AC^2)^{n-2} U^{k-2}\cdot r_0\,.
 \]
Similarly, if $v=c_1, c_1 > 1$ or $|v|>1$ then the least three letter word (when $n>3$) is $aba$ (when $n=3$ then this is the least two letter word).  Thus, when $v=c_1$   we find that the least suffix   begins with the letters $ab$ found at the end of the first copy of $\mathcal E$. (Indeed, all other suffixes beginning with an $ab$  reach $\star$ at a length at which this suffix reaches a copy of $a$.) That is, here also $r_{j_{-\overline{S}-1}} =   (AC^2)^{n-2} U^{k-2}\cdot r_0$.

When $|v|>1$, recall that  $c_s= c_1 = \max\{c_j\}$  and $d_1 = \min\{d_j\}$.  
There are no internal occurrences of $c_j = c_1$ in $v$   exactly when $v$ is of the form $v = \Theta_q(c_1)$.  In this case, 
the minimal suffix is $ab \,[a(ca)^{k-2}b]^{c_s}\,\star$.   Note that $AC^2 \tilde{\mathcal E}\cdot \mathcal F = (1,1) (1,2)^{n-3}$.   
Hence, depending on  whether $c_1 =1$ or not, we find $M = AC^2 \tilde{\mathcal E} \tilde{\mathcal F}^{q-1} \tilde{\mathcal E} (AC^2)^{n-2}$ or $M=  AC^2 \tilde{\mathcal E}\,[\tilde{\mathcal E}^{c_1-1}  \tilde{\mathcal F}]^{q} \tilde{\mathcal E}^{c_1}(A C)^{n-2}$ sends $r_0$ to $r_{j_{-\overline{S}-1}}$.
Therefore, 
\[r_{j_{-\overline{S}-1}} =   AC^2\, R_{-k, \Theta_{q-1}(c_1)} \cdot r_0\, .
 \]
 
If there are internal occurrences of $c_j = c_1$, then arguing as in in the proof to Lemma~\ref{l:lastEllValueIsLargest} 
, we find that the minimal word, up to its suffix of $\star$ and prefix $ab$,  corresponds to the  parent of $v$.   Thus, when $v = \Theta_q(s)$,   we find  $r_{j_{-\overline{S}-1}} =  (1,1),(1,2)^{n-3}, \sigma^{n-2}(\,\overline{b}(-k, s)\,), \star$. 
Following again the proof of Lemma~\ref{l:lastEllValueIsLargest}, we
find $AC^2 R_{-k,\Theta_{q-1}(s)}\cdot r_0 = AC^2\tilde{\mathcal E}\cdot\mathcal F\,\sigma^{n-2}(\,\overline{b}(-k, s)\,), \star$.   Thus, this case holds as well.\\
 
  We turn to the case $k=1$.  Recall that the only $v \in \check{\mathcal V}$ containing the letter $n-2$ is the one letter word $v = n-2$ itself, thus the case $n=3$ here is excluded by hypothesis.
For $n>3$,  the brief discussion at the end of \S~\ref{ss:terseForBigAlps} can be shown to imply that in fact
\[
\begin{aligned} 
R_{-1,v} &=  (AC^2)^{n-3-c_s}  AC \, \tilde{\mathcal F}^{-1+d_{s-1} }\;(AC^2)^{n-3-c_{s-1}} AC \,  \tilde{\mathcal F}^{-1 +d_{s-2}} \cdots \\
\\
& \phantom{moveover}\cdots (AC^2)^{n-3-c_{2}} AC \tilde{\mathcal F}^{-1+d_{1}} (AC^2)^{n-2-c_1}.  
\end{aligned}
\]
Equivalently, using $\star$ as above, 
\[ 
\begin{aligned} \overline{b}_{[1,\infty)}^{\alpha} &= (1,2)^{n-2-c_1}(1,1) [(1,2)^{n-3} (1,1)]^{-1+d_1}  (1,2)^{n-3-c_2}(1,1) [(1,2)^{n-3} (1,1)]^{-1+d_2} \dots \\
& \phantom{moveover} \dots (1,2)^{n-3-c_{s-1}}(1,1) [(1,2)^{n-3} (1,1)]^{-1+d_{s-1}} (1,2)^{n-3-c_s}\star.
\end{aligned}
\]
  If $v = c$ with $1\le c <n-2$, then $ \overline{b}_{[1,\infty)}^{\alpha} = (1,2)^{n-2-c}\star$ and we find that our  extremal suffix is given by $(AC^2)^{n-3-c}\cdot r_0$.    For all other words, arguing as above we find that the extremal suffix begins with $(1,1)$ and continues with the digits corresponding to the longest suffix/prefix.  That is, the argument for $k>1$  succeeds in this case as well. 
\end{proof} 
 
\bigskip 

\begin{figure}[h]
\scalebox{1.1}{
\begin{tikzpicture}[x=6cm,y=6cm] 
\draw  (-0.5, 0)--(1, 0); 
\node at (-.6, 0) {$\cdots$}; 
\node at (-0.4, -0.15)[pin={[pin edge=->, pin distance=10pt]90:{}}] {$-\overline{S}-1$};
\node at (-0.18, 0) [red] {$\cdots$};
\node at (-0.20, -0.1) {$\cdots$};
  \node at (-0.01, -0.15)[pin={[pin edge=->, pin distance=10pt]90:{}}] {$-e-2$};
%
\draw  (0.09, -0.01)--(0.11, 0.01); 
\draw   (0.10,  -0.01)--(0.12, 0.01);  
\node at (0.21, 0.1) [red] {$\beta(\overline{S})$}; 
\node at (0.22, -0.1) [red] {$-e-1$};
%
\draw  (0.32, -0.01)--(0.34, 0.01); 
\draw   (0.33,  -0.01)--(0.35, 0.01);   
\node at (0.46, -0.1) {$-e$};
\node at (0.56, 0) [blue] {$\cdots$}; 
\node at (0.69, 0) [blue] {$\cdots$}; 
 \node at (0.77, -0.15)[pin={[pin edge=->, pin distance=10pt]90:{}}] {$e-\overline{S}-2$};
\node at (0.75, 0.1) {$\beta(1)$};
\node at (0.85, 0) [red] {$\cdots$}; 
\node at (0.92, 0) [red] {$\cdots$}; 
\node at (1, -0.1) {$-1$};
\node at (1.05, 0.1) {$\beta(0)$};
\draw [decorate, ultra thick,
    decoration = {calligraphic brace,amplitude=10pt}] (-0.4,0.1) --  (-0.01,0.1);
\draw[->, red] (-0.2, 0.2) -- (-0.2, 0.32) -- (0.87, 0.32) --(0.87, 0.1);
\node at (0.4, 0.4) [red] {$b \mapsto b+e$};
%
\draw [decorate, ultra thick,
    decoration = {calligraphic brace,mirror,amplitude=5pt}] (0.48,-0.15) --  (0.61,-0.15);
\draw [decorate, ultra thick,
    decoration = {calligraphic brace,mirror,amplitude=10pt}] (-0.4,-0.2) --  (-0.01,-0.2);
\draw[->, blue] (0.55,-0.19) -- (0.55, -0.4) -- (-0.2,-0.4) -- (-0.2,-0.3);
\node at (0.8, -0.5) [blue] {$b \mapsto b+e - \overline{S}-1$};
 \foreach \x/\y in {-0.4/0,   -0.31/0, -0.24/0, -0.12/0, -0.01/0}%
{ \node at (\x,\y)[orange] {$\bullet$}; } 
\node at (0.19, 0) [red] {$\star$};
 \foreach \x/\y in {0.48/0,  0.61/0, 0.77/0%
} { \node at (\x,\y)[blue] {$\bullet$}; } 
\node at (1, 0) [blue] {$\star$};
\end{tikzpicture}  
}
\caption{{\bf Dynamics of the initial portion of the $T_{\alpha}$-orbit of $r_0(\alpha)$.}   Recall that $\beta(j) = j_b$ means that $r_j$ is the $j_b$-th real value of this orbit, from the right.  Indicated below the $x$-axis are the values of $b$, above it a few  corresponding $\beta(j)$.   See Corollaries~\ref{c:anotherBottomHeight} and ~\ref{c:changeOrder}.}
\label{f:betaChanges}%
\end{figure}
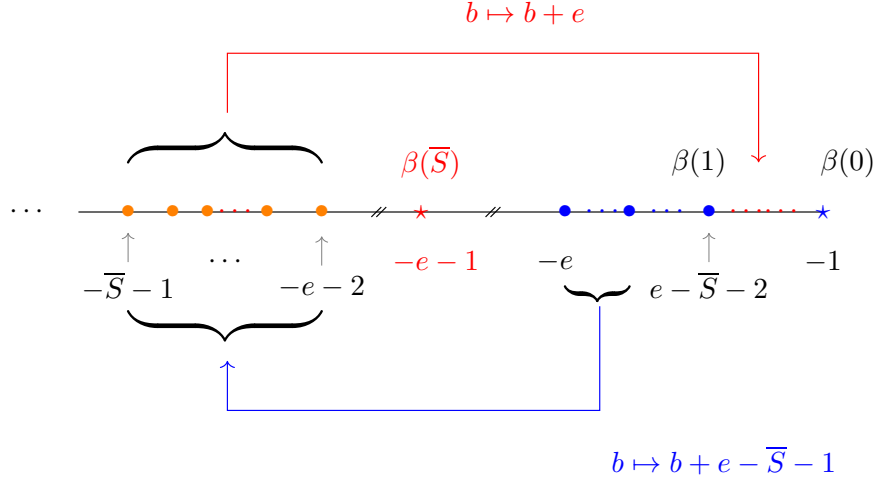

 Compare the following  with Figures ~\ref{f:topBottomVertices} , ~\ref{f:orbitRelations}, ~\ref{f:meaningOfrHat} and  ~\ref{f:betaChanges}. 
 \begin{Cor}\label{c:anotherBottomHeight}    We have 
$1 + \beta(1) = \beta(1+ j_{-\overline{S}-1})\,$.
\end{Cor}
\begin{proof}[Sketch]  We treat only the main case of $v = \Theta_q(s)$.  The proof of the lemma implies that $r_{1+j_{-\overline{S}-1}} =   (1,2)^{n-3}, \sigma^{n-2}(\,\overline{b}(-k, s)\,), \star$.   
We easily find that $r_1 = (1,2)^{n-3}, \sigma^{n-2}(\,\overline{b}(-k, v)\,), \star$.  Now,  $(1,2)^{n-3}, \sigma^{n-2}(\,\overline{b}(-k, s)\,)$ is a prefix of $(1,2)^{n-3}, \sigma^{n-2}(\,\overline{b}(-k, s)\,)(1,2)^{n-3}, \sigma^{n-2}(\,\overline{b}(-k, v)\,)$ which is followed by a digit of $(1,1)$.  Since (mixing notation and temporary notation),  $(1,1)\prec \star$ we have  $r_{1+j_{-\overline{S}-1}}>r_1$.  As  $r_1$ is the rightmost of the $AC^2\cdot r_j$ while $r_{1+j_{-\overline{S}-1}}$ is the leftmost of the $AC\cdot r_j$,  the statement holds. 
\end{proof}

A second direct result of the previous lemma gives the basic dynamics of  the $T_{\alpha}$-orbit of $r_0(\alpha)$ in terms of 
the indices for the real ordering.   See Figure~\ref{f:betaChanges}.

 \begin{Cor}\label{c:changeOrder}   The position of the successor of $r_{j_b}$  is given by
 \[ \beta(1 + j_b) = \begin{cases}  b+e  \phantom{moveiton}\text{if}\;\; b \le -e-2;\\
 \\
                                                    b+e-\overline{S}-1\;\;\;\text{if}\;\;  b \ge -e.
                                                       \end{cases}       
\] 
\end{Cor}

\begin{proof}    The $r_j, j <\overline{S}$ in the cylinder $\Delta_{\alpha}(\, (1,2)^{n-2}\,)$ are $r_0 = r_{j_{-1}}$  and the images of the larger 
$r_j, j <\overline{S}$ in $\Delta_{\alpha}(1,1)$.   In particular, the largest of these latter values, $r_{j_{-e-2}}$,  is sent to $r_{j_{-2}}$.    On the other hand, we have that the smallest of the    $r_j, j <\overline{S}$ in $\Delta_{\alpha}(1,1)$ has image greater than $r_1$.    Hence, the image of  $\Delta_{\alpha}(\, (1,2)^{n-2}\,)$ lies to the left of the image of $\Delta_{\alpha}(1,1)$.  It follows that $\beta(1 + j_b) =  b+e$ for all $r_{j_b} \in  \Delta_{\alpha}(1,1)$.   This is the set of $r_{j_b}$ with $b\le -e-2$.    

Since there are $e$ values  $r_{j_b}$ in $\Delta_{\alpha}(1,2)$ and the one value, $r_{\overline{S}}$ outside of  $\Delta_{\alpha}(1,1) \cup \Delta_{\alpha}(1,2)$,  the set  
$\{b\mid r_{j_b} \in \Delta_{\alpha}(1,1)\}$ has cardinality  $\overline{S}-e$.   Hence,  
\begin{equation}\label{e:BetaOfOne}
 \beta(1) = - \overline{S}+e -2.
 \end{equation}
     Therefore for all $r_b \in \Delta_{\alpha}(1,2)$, 
we have $\beta(1 + j_b) = b+e-\overline{S}-1$.
\end{proof} 

 \begin{Lem}\label{l:vertexProgression}    For $0\le j < \overline{S}$, let $M_j$ be as in Lemma~\ref{l:whyHatr}.   For each  $b$ we have 
 \[ \mathcal T_{M_{j_b}}(r_{j_b}, y_b) =  (r_{1+j_b}, y_{\beta(1+j_b)}).  \]
 \end{Lem}
\begin{proof} We treat the setting where $n>3$ and hence $(1,1)$ appears only as an isolated letter in  $\overline{b}(-k, v)\,)$.  
(When $n=3$, it is $(1,2)$ that appears as an isolated letter, allowing analogous arguments.)

  {\bf Case 1.  $-1 \ge b \ge -e\,$:} \quad   From Corollary~\ref{c:changeOrder},   we have  $\beta(1 + j_b) = b+e-\overline{S}-1$.   Thus,   Definition~\ref{d:bottomYvaluesLargeAlps} gives
\[ y_{\beta(1 + j_b)} =  \begin{cases} -\hat{r}_{j_{-(b+2e +1)}}\;\; \text{if} \;\;\; -e> \beta(1 + j_b)\\
\\
                                                                        -\hat{r}_{j_{-(b+2e -\overline{S})}}\;\; \text{otherwise}.  
                                                  \end{cases}  
\]                         
 
  Definition~\ref{d:bottomYvaluesLargeAlps} also gives   $y_b = - \hat r_{j_{-(b+  e +1)}}$.  Note that  $ M_{j_b} =  M_{j_{-(b+  e +1)}} = AC^2$ here.       If $b=-e$, then $y_b = - \hat r_0$ and Lemma~\ref{l:whyHatr} gives in this case $R M_{j_b} R^{-1} \cdot y_b = -\hat{r}_{\overline{S}}$.   Here $y_{\beta(1 + j_b)} = -\hat r_{j_{-e-1}} = -\hat{r}_{\overline{S}}$; the result holds for $b=-e$.

 For $-1 \ge b >-e$,  Lemma~\ref{l:whyHatr} combines with Corollary~\ref{c:changeOrder} to give   
\[ R M_{j_b} R^{-1} \cdot y_b = \begin{cases} -\hat{r}_{j_{-(b+2e +1)}}\;\; \text{if} \;\;\; \beta(-1+j_{-(b+e+1)}) \le -e-2,\\
\\
                                                                        -\hat{r}_{j_{-(b+2e -\overline{S})}}\;\; \text{if}  \;\;\;  -1\ge \beta(-1+j_{-(b+e+1)}) \ge -e.
                                                  \end{cases}  
\]                         

We see   that the result holds whenever  $\beta(-1+j_{-(b+e+1)}), \beta(1 + j_b)$ are both at least, or are both less than, $-e$. 

 In this first case, it remains to show that the two remaining possible combinations of inequalities on the values of $\beta(-1+j_{-(b+e+1)})$ and $\beta(1 + j_b)$ can never be fulfilled.  
For this, note that the cardinality of  the set $\{r_j \mid  r_j \in  \Delta_{\alpha}(\, (1,2)(1,1)\,), j\le \overline{S} \}$ equals   $\overline{S}-e$, as the set has image  consisting of all of the $r_j$ in $\Delta_{\alpha}(1,1)$.   Therefore,  the corresponding indices $j$ are such that the real ordering gives $-e \le  \beta(j) \le \overline{S}-2 e-1$.  

Thus,  if $\beta(1 + j_b)<-e-1$, then    $-1 \ge  -(b+e+1)\ge -\overline{S}+e$. From \eqref{e:BetaOfOne},  $-(b+e+1) > \beta(1)$.   Therefore, $\beta(-1+j_{-(b+e+1)})< -e-1$.      Now,  if  $ -1\ge \beta(-1+j_{-(b+e+1)}) \ge -e$, then (as also $-(b+e+1) \ge -e$),  we have $-(b+e+1)\le \beta(1)$; but then  $\beta(1 + j_b)\ge -e-1$.     The result thus holds in this first case.

 \bigskip 
 \noindent
  {\bf Case 2.  $-e > b\,$:} \quad  In this case, Definition~\ref{d:bottomYvaluesLargeAlps} gives   $y_b = - \hat r_{j_{-(b+  e  +  \overline{S} +2)}}$.   Of course $- \overline{S}-1 \le b$ and thus we find  $-(b+  e  +  \overline{S} +2)<-e$.   Therefore,  $ M_{j_b} =  M_{j_{-(b+  e  +  \overline{S} +2)}} = AC$ here; as well   $-1> \beta(-1+j_{-(b+  e  +  \overline{S} +2)}) \ge -e$.  
   
     Hence, Lemma~\ref{l:whyHatr} combines with Corollary~\ref{c:changeOrder} to give   
\[ R M_{j_b} R^{-1} \cdot y_b = -\hat{r}_{j_{-(b+ 2 e +1)}}.\]
                                         
 From Corollary~\ref{c:changeOrder},   we have  $\beta(1 + j_b) =  b+e$, which must be at least $-e$ (assuming $n>3$)   Thus, 
\[ y_{\beta(1 + j_b)} =    -\hat{r}_{j_{-(b+2e +1)}}.
\]                         
\end{proof}

 \begin{Cor}\label{c:theWholeWord}    We have $y_{\beta(\overline{S})} = R R_{-k,v} R^{-1}\cdot y_{\beta(0)}$.
\end{Cor}

\begin{proof}  Since $R_{-k,v} =  M_{\overline{S}} \cdots M_0$, while for each $j\le \overline{S}$ we have   $M_j \cdots M_0 \cdot r_0 = r_j$,  this does follow from the lemma.
\end{proof}

\bigskip

 \begin{Lem}\label{l:htsIncrease}   The values of the
 $y_b, b<0$ increase with $b$.     
 \end{Lem}

\begin{proof}  By Definition~\ref{d:bottomYvaluesLargeAlps}, $y_b = -C\cdot r_{j_{-e-1 -b}}(\eta_{-k,v})$  when $b \in \{-e, \cdots, -1\}$.  
 By definition,   $r_{j_{-e}}(\eta_{-k,v}) \le \cdots \le r_{j_{-1}}(\eta_{-k,v})$.  The map $b \mapsto -e-1 -b$ is order reversing, the action of $C$ is order preserving, and finally multiplication by $-1$ is order reversing.   Thus,  $b \mapsto y_b$ increases with $b$ for these values of $b$. 
 
   Since $r_{1+j_{-\overline{S}-1}}>r_1$,  we deduce that the $e$ number of values $C\cdot r_{j_{-e-1 -b}}(\eta_{-k,v})$ all lie to the left of the values of any of the  $r_j(\eta_{-k,v}), 0\le j \le \overline{S}$.     In particular,  those $y_b$ with $-1 \ge b \ge -e$ are larger than the remaining values. 

For these remaining values,  $y_b =  - \hat r_{j_{-(b+  e  +  \overline{S} +2)}} = - r_{j_{-(e  +  \overline{S} +2)-b}}(\eta_{-k,v})$ increases with $b$ because the multiplication by $-1$ reverses the decreasing nature of the   $r_{j_{-(e  +  \overline{S} +2)-b}}(\eta_{-k,v})$.   The result thus holds, we have:   $y_{-\overline{S}-1} \le \cdots \le y_{-1}$. 
\end{proof} 

\bigskip
We give a result that could have been stated directly after the definition of the set of the various $y_{\tau(i)}$.
 \begin{Lem}\label{l:relForSharedTopOfTwoLeftmostCylinders}    We have 
\[ RA^{-k-1}CR^{-1} \cdot y_{\tau(\underline{S}-1)} = y_1.\]   
 \end{Lem}
 
\begin{proof} By definition,  $y_{\tau(\underline{S}-1)}  = - \ell_1(\eta_{-k,v})$.  By Lemma~\ref{l:conjByRofAtoPc}, $RA^{-k-1}CR^{-1}$ sends this to $-x$ where $x$ satisfies $\ell_1 = A^{-k-1}C\cdot x$.  Now, by (\cite{CaltaKraaikampSchmidt},  Lemma~6.3),  $x = \ell_{\underline{S}}(\eta_{-k,v})$.   Since $y_1 = y_{\tau(0)}$, by definition this has the value $-\ell_{\underline{S}}(\eta_{-k,v})$, and the result holds. 
\end{proof}

\bigskip

Since $\underline{d}(-k, v)$ differs only slightly from  $\overline{d}(k, v)$, we use notation as in Section~\ref{s:relationsOnHts}.   We also revert to using $u$ to denote a word in $\mathcal V$.  (We trust that the reader can easily distinguish this from our other standard usage of $u$.)
 \begin{Def}\label{d:iotaLargeAlps}   Just as for $\alpha<\gamma_{n}$, we let                                                                
\[ \iota = \begin{cases} 0 & \text{if}\;  v = 1,\\
                                     \underline{S}(-k,u)& \text{if}\;  v = \Theta_q(u).
                                     
           \end{cases}                          
\]
\end{Def}

 Recall from (\cite{CaltaKraaikampSchmidt}, Lemma~6.7) 
 that   $v \prec w$ if and only if $\underline{d}(-k,v) \succ \underline{d}(-k,w)$.
 Lemma~\ref{l:2ndLargestR} thus implies the following.
 \begin{Lem}\label{l:largestLvalue}   We have 
    \[ \ell_\iota = \max_{0\le i < \underline{S}}\,\ell_i\,.\]
    Equivalently,   $\tau(\iota) =  \underline{S}$.                                                                   
\end{Lem}
 \bigskip

 Compare the following with Figure~\ref{f:orbitRelations}.
 \begin{Prop}\label{p:RelationBottomAndTop}   We have 
    \[y_{\underline{S}}  = R C^2AC^2R^{-1}\cdot y_{\beta(\overline{S})}  .\]                                                                  
\end{Prop}
\begin{proof}[Sketch]  We treat the main case of $v = \Theta_q(u)$,  $q\ge 1$. 
From Corollary~\ref{c:someLamRels}, combining Lemma~\ref{l:vertexProgression} with Lemma~\ref{l:smallestRvalue}, and Corollary~\ref{c:theWholeWord}, 
\[
y_{\beta(\overline{S})} = R\, R_{-k,v}\,[AC^2 R_{-k, \Theta_{q-1}(u)}]^{-1} C R^{-1}\cdot y_{\underline{S}+1}.\\
\]

Using $R'_{-k, w}$ to denote  $R_{-k,w}[(AC^2)^{n-2}]^{-1}$ (for general words $w$) as well as the notation of the proof of (\cite{CaltaKraaikampSchmidt}, Proposition~4.13), 
 we have 
\[
\begin{aligned}
R_{-k,v}\,[AC^2 R_{-k, \Theta_{q-1}(u)}]^{-1} C &= R_{-k, u'u''} R_{u''}^{-1} (AC^2)^{-1} C\\
                                                                          &=  R'_{-k, \mathfrak z' \mathfrak a \mathfrak Z\mathfrak a \mathfrak z} {R'}_{\mathfrak a \mathfrak z}^{-1} (AC^2)^{-1} C\\
                                                                          &=  R'_{-k, \mathfrak a \overleftarrow{\mathfrak z'}\mathfrak Z\mathfrak a \mathfrak z} {R'}_{\mathfrak a \mathfrak z}^{-1} (AC^2)^{-1} C\\
                                                                          & R'_{-k, \overleftarrow{\mathfrak z'}u} {R'}_{\mathfrak z}^{-1} (AC^2)^{-1} C\\
                                                                          &=R'_{-k, u} \tilde{\mathcal F} {\tilde{\mathcal E}}^{-1} (AC^2)^{-1} C\\
                                                                          &= R'_{-k, u} \tilde{\mathcal F} \tilde{\mathcal E}^{-1} (AC^2)^{-1} C\\
                                                                          &= R'_{-k, u} (AC^2)^{n-3}U  [(AC^2)^{n-3}]^{-1} (AC^2)^{-1} C\\
                                                                          &= R'_{-k, u} (AC^2)^{n-2}\\
                                                                          &= R_{-k, u}\,.
\end{aligned}  
\]
Therefore,   $y_{\beta(\overline{S})} = R\, R_{-k,u}R^{-1}\cdot y_{\underline{S}+1}$.  Hence,   $R C^2AC^2R^{-1}\cdot y_{\beta(\overline{S})} =  R L_{k,u}R^{-1}\cdot  y_{\underline{S}+1}$.

On the other hand, by Lemma~\ref{l:largestLvalue}, $y_{\underline{S}} = y_{\tau(\iota)}$.  Therefore, from the definition of the $y_a$, we have that  $y_{\underline{S}} = R L_{k,u}AR^{-1}\cdot  y_{\tau(0)}$.    Arguing as for Lemma~\ref{l:topValuesFirstRelations}, one has  $R A^{-k-1}CR^{-1}\cdot y_{\tau(\underline{S}-1)} = y_{\tau(0)}$ and $R A^{-k}CR^{-1}\cdot y_{\tau(\underline{S}-1)} = y_{\tau(\underline{S})}$.   Since $ y_{\tau(\underline{S})} =  y_{\underline{S}}$, one finds that   $y_{\underline{S}}$ also is equal to $R L_{k,u}R^{-1}\cdot  y_{\underline{S}+1}$.
\end{proof}

\section{Bijectivity of $\mathcal T_{\alpha}$ on  $\Omega_{\alpha}$ for large $\alpha$ in left portion of synchronization interval}\label{leftBlocks}    

\subsection{Partitioning $\Omega$ by blocks $\mathcal B_i$}\label{ss:theBlocksLargeAlps}

For each $i \in \{-k, -k+1, \dots\}\cup \mathbb N$ and $j\in \{1,2\}$,  let the {\em block} $\mathcal B_{i,j}$ be the closure of the set $\{(x,y)\in \Omega\,|\, x \in \Delta_{\alpha}(i, j)\,\}$.   Thus the  blocks partition $\Omega$ up to $\mu$-measure zero.  

Since $\mathcal T$ is invertible, it is clear that Theorem ~\ref{t:Omega} when $\alpha \in (\eta_{-k,v},  \delta_{-k,v})$ follows from the following two results, the proofs of which are given by combining the results of the ensuing three subsections.  

 \begin{Prop}\label{p:topFromNegI}     The union of $\mathcal T(\mathcal B_{i,j})$ taken over all   $i \in \{-k, -k+1, \dots\}$ and $j\in \{1,2\}$ equals $\Omega^{+}$ up to $\mu$-measure zero.  
 \end{Prop}

 \begin{Prop}\label{p:bottomFromPosI}     The union of $\mathcal T(\mathcal B_{i,j})$ taken over all  $i\in\mathbb N$ and $j\in \{1,2\}$ equals $\Omega^{-}$ up to  $\mu$-measure zero. 
 \end{Prop}

\subsection{Blocks laminate one above the other}\label{ss:laminateLargeAlps}     

The arguments for the case of  $\alpha<\gamma_{n}$ give also the following.
\begin{Lem}\label{l:topsOfBlocksLargeAlps}    Let $(c,d)$ denote the first digit of $\ell_{\underline{S}}$.  Then 
for all $(i,j) \notin \{(-k,1), (-k-1,1)\}$,   
 the top boundary of the block $\mathcal B_{i,j}$ is given by 
\[ \begin{cases}    
    y = y_{\underline{S}+1}& \text{if} \;\;(c,d) \prec (i,j)\,; \\
     (y = y_{\underline{S}+1})\cup (y = y_{\underline{S}})& \text{if} \;\;(c,d) = (i,j)\,; \\
     y = y_{\underline{S}}& \text{if} \;\;(c,d) \succ (i,j)\,.
     \end{cases}
\]     
\end{Lem}

\bigskip 

 \begin{Lem}\label{l:bottomsOfBlocksLargeAlps}     Let $(a,2)$ denote the first digit of $r_{\overline{S}}$.  
Then 
for all $(i,j) \notin \{(1,1), (1,2)\}$,   
 the lower boundary of the block $\mathcal B_{i,j}$ is given by
\[ \begin{cases}    
    y = y_{-\overline{S}-1}& \text{if} \;\;j=1\,; \\
    y = y_{\beta(\overline{S})}& \text{if} \;\;j=2\, \text{and}\; (a,2) \succ (i,j)\,;\\
     (y = y_{\beta(\overline{S})} )\cup (y = y_{1+\beta(\overline{S})})& \text{if} \;\;(a,2) = (i,j)\,; \\
     y = y_{1+\beta(\overline{S})}& \text{if} \;\;(a,2) \prec (i,j)\,.
     \end{cases}
\]     
 \end{Lem}

\begin{proof}   Since all of the $r_j$, with $j\le \overline{S}$, lie in or to the right of $\Delta_{\alpha}(1,1)$,   the result for $j=1$ holds.  All of the $r_j$ with $j< \overline{S}$  lie in  $\Delta_{\alpha}(1,1)\cup \Delta_{\alpha}(1,2)$, and thus the bottom heights of $\mathcal B_{i,2}$ for $i\neq 1$ are one or both
$y_{\beta(\overline{S})}, y_{1+\beta(\overline{S})}$ as determined by the location of $r_{\overline{S}}$ as per the remaining statements.
\end{proof}

 \begin{Lem}\label{l:laminationInEasyCaseLargeAlps}  Suppose that $(i,j) \notin\{(1,1), (1,2), (-k,2)\}$.  Then $\mathcal T_{\alpha}(\mathcal B_{i,j})$  laminates above  $\mathcal T_{\alpha}(\mathcal B_{i', j'})$, where 
\[ (i',j') =    \begin{cases} (i, 2)& \text{if}\; j=1,\\
\\
                 (i-1, 1)& \text{otherwise}.
\end{cases}                        
\] 
\end{Lem}
 
\begin{proof} We first prove the result in the case of $j=1$.   By Lemma~\ref{l:bottomsOfBlocksLargeAlps},  $\mathcal B_{i,1}$ has lower boundary height $y_{-\overline{S}-1}$.  Similarly, since $\ell_{\underline{S}}(\alpha) < \mathfrak b_{\alpha}$,  Lemma~\ref{l:topsOfBlocksLargeAlps} gives that for any $i$, 
$\mathcal B_{i,2}$ has upper boundary height $y_{\underline{S}+1}$.  From Lemma~\ref{c:someLamRels}, $y_{-\overline{S}-1} = R CR^{-1}\cdot  y_{\underline{S}+1}$.     All cylinders are full with the possible exception of those of index $(-k,1), (-k,2),  (1,2)$, and the $T_{\alpha}$-images of the cylinders of index $(-k,1), (-k,2)$ agree.  Therefore, the $T_{\alpha}$-images of the respective $\Delta_{\alpha}(i,j)$ and $\Delta_{\alpha}(i',j')$ agree.
Lemma~\ref{l:lamEqsSameAexpon}  thus gives the result in this case.   

We now treat the case of $j=2$.    The proof of (\cite{CaltaKraaikampSchmidt}, Lemma~6.2)
 shows that $r_{\overline{S}} \in \Delta_{\alpha}(a,2)$ (with  $a\ge 2$ because $\alpha \in (\eta_{-k,v}, \delta_{-k,v})\,$) if and only if 
$\ell_{\underline{S}} \in \Delta_{\alpha}(a-1,1)$.   Thus, if $j=2$ and $(a,2) \succ (i,j)$ then also $(c,d) \succ (i-1,1)$ (with the notation of Lemma~\ref{l:topsOfBlocksLargeAlps}); in this case,  the top of $\mathcal B_{i-1,1}$ is $y_{\underline{S}}$ and the bottom of  $\mathcal B_{i,2}$ is $y_{\beta(\overline{S})}$.   Both $\Delta_{\alpha}(i-1,1)$ and $\Delta_{\alpha}(i, 2)$ are full cylinders, thus  as Proposition ~\ref{p:RelationBottomAndTop} shows that the hypothesis for Lemma~\ref{l:lamEqsDifferingAandCexpons}  is fulfilled, the result holds in this subcase. 

When  $j=2$ and $(a,2) \prec (i,j)$ we have also $(c,d) \prec (i-1,1)$;  Lemmas~\ref{l:topsOfBlocksLargeAlps}, \ref{l:bottomsOfBlocksLargeAlps} and Corollary~\ref{c:someLamRels}  show that  the hypothesis for Lemma~\ref{l:lamEqsDifferingAandCexpons}  is fulfilled, and the result also holds in this subcase. 

 Finally, in the case where $\mathcal B_{i,2}$ has two heights at its bottom, we find that $\mathcal B_{i-1,1}$ has two heights on its top, with the pairs of heights fulfilling the hypothesis of Lemma~\ref{l:lamEqsDifferingAandCexpons}.   Because of synchronization,  the images of the pairs of boundary pieces indeed match perfectly. 
\end{proof}

\begin{Lem}\label{l:laminLeftmostBlock}    
The  $\mathcal T_{\alpha}$-image of the block $\mathcal B_{-k,2}$ laminates above a portion of $\mathcal T_{\alpha}(\mathcal B_{-k-1,1})$.
\end{Lem}

\begin{proof}    The cylinder $\Delta_{\alpha}(-k,2)$ is right full, with image having left endpoint $\ell_1(\alpha)$.  The cylinder $\Delta_{\alpha}(-k-1,1)$ is full.   Otherwise, the lamination is as above. 
\end{proof}

\begin{Lem}\label{l:laminRightmostBlock}    
The  $\mathcal T_{\alpha}$-image of the block $\mathcal B_{1,1}$ laminates above a portion of $\mathcal T_{\alpha}(\mathcal B_{1,2})$.
\end{Lem}

\begin{proof}    The cylinder $\Delta_{\alpha}(1,2)$ is left full, with image having right endpoint $r_1(\alpha)$.  The cylinder $\Delta_{\alpha}(1,1)$ is full.   Otherwise, the lamination is as above.
\end{proof}

\subsection{Upper boundary  is in image}\label{ss:upperBoundaryInImageLargeAlps}  

The statements and proofs in the case of $\alpha<\gamma_{n}$, see Subsection~\ref{ss:upperBoundaryInImage},  go through here with only minor adjustments.

\subsection{Lower boundary  is in image} \label{ss:lowerBoundaryInImageLargeAlps} 

Using the definition of the $y_b$ and  Corollary~\ref{c:someLamRels}, the arguments for the case of $\alpha<\gamma_{n}$, see Subsection~\ref{ss:lowerBoundaryInImage},  succeed here.   
 
\subsection{Bijectivity and ergodicity at left endpoints}   Whereas the left endpoint of a synchronization interval for small $\alpha$ is determined by the orbit of $r_{\overline{S}}$  being at the left end of $\mathbb I_{\alpha}$ (and thus there   is one less rectangle in $\Omega^{-}$ than for the values in the interior of its synchronization interval),  that of $J_{-k,v}$ is announced by  $r_{\overline{S}} = \mathfrak b$.    Thus, see \eqref{e:frakBis}, we have  $C\cdot r_{\overline{S}}(\eta_{-k,v}) = \ell_0(\eta_{-k,v})$.  Hence, $r_{\overline{S}+1}(\eta_{-k,v}) = A^{-k} C^2 \cdot r_{\overline{S}}(\eta_{-k,v}) = \ell_1(\eta_{-k,v})$.   That is, here the synchronization is of the same form as for the rest of $(\eta_{-k,v}, \delta_{k,v})$.  Thus,  
there is no change necessary to the definitions of $\Omega^{\pm}$.   

\begin{Prop}\label{p:OmegaForEtaKneg}   Fix $n\ge 3,  k \in \mathbb N, v \in \mathcal V$  (with $v \in \check{\mathcal V}$ if $k=1$) and $\alpha = \eta_{-k,v}$.    Let $\Omega^{+} $ be as in Definition~\ref{d:topYvaluesLargeAlp} and $\Omega^{-}$ be as in Definition~\ref{d:bottomYvaluesLargeAlps}.    

Then 
$\mathcal T_{n,\alpha}$ is bijective on $\Omega_{n, \alpha}  := \Omega^{+} \cup \Omega^{-}$, up to 
$\mu$-measure zero.   
\end{Prop}
 
\begin{Eg}\label{e:kIsMin1LeftEndpt} Recall the setting of Subsection~\ref{ss:ExamplesKis1FirstHalfOfSyncInt}:   $n=3$, $k=1$, $v=1$.  Since $g^2/2 = \gamma_{3} =  \eta_{-1,1}$, we find 
\[\Omega_{3,g^2/2} = (\, [-G,  -g^2]\times  [-g^2, g^2]\,) \cup (\, [-g^2, g^2]\times [-g^2, G]\,).\]
\end{Eg}

\begin{Prop}\label{p:ergodicLargeEta}   With $\alpha = \eta_{-k,v}$ as in the previous proposition, 
The system $(\mathcal T_{\alpha}, \Omega_{\alpha}, \mathscr B'_{\alpha}, \mu_{\alpha})$ is ergodic.   

Furthermore,    this two dimensional system is the natural extension of  $(T_{\alpha},\mathbb I_{\alpha},  \mathscr B_{\alpha}, \nu_{\alpha})$, where $\nu_{\alpha}$ is the marginal measure of $\mu_{\alpha}$ and $\mathscr B_{\alpha}$  the Borel sigma algebra on $\mathbb I_{\alpha}$.   In particular, the one dimensional system is ergodic. 
\end{Prop}

\begin{proof} [Sketch]    From (\cite{CaltaKraaikampSchmidt},  Lemma~6.3),  
 $\ell_0(\alpha)$ has a periodic $\alpha$-expansion.   Thus, $r_0(\alpha)$ and $\mathfrak b_{\alpha}$ are also all of finitely many $\alpha$-digits. In view of Remark~\ref{rmk:bddNonfullCylinders}, it follows that 
 the non-full cylinders of $T_{\alpha}$ are bounded in range.   Due to this and the finiteness of $\mu(\Omega_{\alpha})$ and of its vertical fibers, the hypotheses of (\cite{CKStoolsOfTheTrade}, Theorem~2.3) are met, thus giving the result here.
\end{proof} 

\medskip  
\section{Bijectivity for  large $\alpha$ in right portion of synchronization interval}\label{s:largeAlpsRightSide}
This section  treats the right portion of synchronization intervals for large $\alpha$ in a completely analogous manner to that of \S~\ref{s:bigLefties} and  \S~\ref{leftBlocks} for the left portions. 

\subsection{Bijectivity on the interior} 
Suppose now that $\alpha \in (\delta_{-k,v}, \zeta_{-k,v})$.
We define $\Omega_{\alpha}^{-}$ with ``one more" rectangle than for the left portion, and
sketch the proof that bijectivity of $\mathcal T_{\alpha}$ on $\Omega_{\alpha}$ then holds here as well.

\begin{figure}[h]
\scalebox{0.9}{
\begin{tikzpicture}[x=6cm,y=6cm] 
\draw  (-.26, -0.83)--(-0.18, -0.83)-- (-0.18, -0.74) -- (0.65, -0.74)-- (0.65, -0.44)-- (1.46, -0.44)--(1.46, -0.39)--(1.74, -0.39) --(1.74, 0.35)--(0.85, 0.35)--(0.85, 0.2)--(-.26, 0.2)--cycle; 
 \draw  (0.0, -0.74)--(0.0, .2); 
 \draw  (1, -0.44)--(1, .35); 
 \fill [opacity= 0.15, gray]   (-.26, -0.83)--(-0.18, -0.83)-- (-0.18, -0.74)--(-.26, -0.74)-- cycle; 
\draw[thin,dashed] (-0.21, -0.83)--(-0.21, 0.2);    
\draw[thin,dashed] (-0.15, -0.74)--(-0.15, 0.2);  
\draw[thin,dashed] (0.19, -0.74)--(0.19, 0.2); 
\draw[thin,dashed] (0.31, -0.74)--(0.31, 0.2);      
\draw[thin,dashed] ( 0.78, -0.44)--( 0.78, 0.2);  
\draw[thin,dashed] (0.83, -0.44)--(0.83, 0.2);  
\draw[thin,dashed] (0.87, -0.44)--(0.87, 0.35);  
\draw[thin,dashed] (1.2, -0.44)--(1.2, 0.35);  
\draw[thin,dashed] (1.4, -0.44)--(1.4, 0.35);  
\draw[thin,dashed] (1.08, -0.44)--(1.08, 0.35);  
\node at (-.4, -0.2)[pin={[pin edge=->, pin distance=12pt]0:{}}] {$\mathcal B_{-2,1}$};
\node at (-.4, -0.4)[pin={[pin edge=->, pin distance=24pt]0:{}}] {$\mathcal B_{-3,1}$};     
\node at (0.55, -0.3) {\tiny{$(1,1)$}};  
\node at (0.25, -0.3) {\tiny{$(2,1)$}}; 
\node at (0.65, 0.05)[pin={[pin edge=->, pin distance=12pt]0:{}}] {$\mathcal B_{-2,2}$};    
\node at (0.65, -0.1)[pin={[pin edge=->, pin distance=20pt]0:{}}] {$\mathcal B_{-3,2}$};   
\node at (0.10, -0.3) {\tiny{$\cdots$}}; 
\node at (-0.06, -0.3) {\tiny{$\cdots$}};   
\node at (0.96, 0) {\tiny{$\cdots$}}; 
\node at (1.05, 0) {\tiny{$\cdots$}}; 
\node at (1.14, 0) {\tiny{$(3,2)$}};  
\node at (1.3, 0) {\tiny{$(2,2)$}};   
\node at (1.55, 0) {\tiny{$(1,2)$}};   
\node at (-.42, -0.87) {$(\ell_0, y_{-4})$}; 
\node at (-.42, 0.2) {$(\ell_0, y_1)$}; 
\node at (0.7,  0.35)  {$(\ell_1, y_2)$};  
\node at (0.7, -0.85)  {$(r_1, y_{-3})$}; 
\node at (1.5, -0.55)  {$(r_2, y_{-2})$}; 
\node at (-0.05, -0.87)  {$(r_3, y_{-4})$}; 
\node at (.78, -0.5)  {$\mathfrak b$};
\node at (1.9, -0.39)  {$(r_0, y_{-1})$};  
\node at (1.75,0.45)  {$(r_0, y_{2})$};   
\node at (0,0)  {$0$};  
\node at (1, -0.5)  {$1$};  
 \foreach \x/\y in {-.26/-0.83, -0.18/ -0.83, 0.65/-0.74, 1.46/-0.44,
1.74/-0.39,  1.74/0.35, 0.85/0.35, -.26/0.2%
} { \node at (\x,\y) {$\bullet$}; } 
\end{tikzpicture}  
}
\caption{The domain $\Omega_{3, 0.87}$. Blocks $\mathcal B_{i,j}$ also denoted by $(i,j)$. Here $L_{-k,v} = A^{-2}CA^{-1}$ and $R_{-k,v} = AC \,AC^{2}$, and $\alpha$ is an interior point of $J_{-2,1}$ lying to the right of $\delta_{-2,1}$. Compare with $\Omega_{3, 0.86}$, of Figure~\ref{f:omegaLargeAlpLessThanDelta_{-k,v}}.  Since both $0.86, 0.87$ give $\alpha$ values in $J_{-2,1}$,  the heights $y_1, y_2, y_{-1}, y_{-2}, y_{-3}$ are the same.  Highlighted in gray: rectangle of lower height not seen on left portion of $J_{-k,v}$; here, $y_{-4} = y_{-\overline{S}-2} = RAC^2R^{-1}\cdot y_{\beta(\overline{S})}$, as per Definition~\ref{d:bottomYvaluesPastDelta}.}
\label{f:omegaLargeAlpBiggerThanDelta_{-k,v}}%
\end{figure}
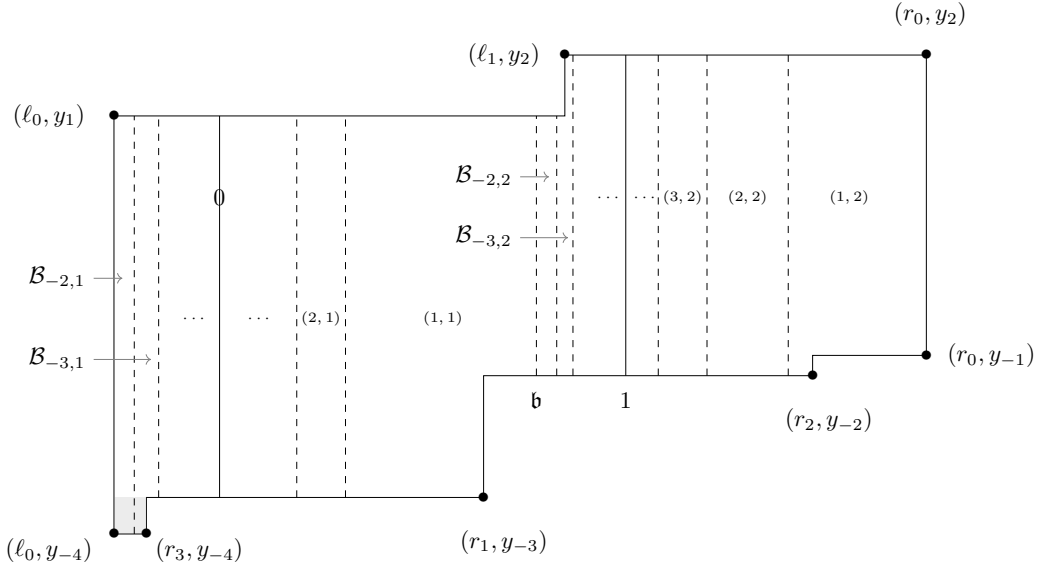

\bigskip
As a first step, we show that $r_{1+\overline{S}}$ lies to the left of the initial portion of the orbit of $r_0$.
\begin{Lem}\label{l:locatingFinalRpastDelta}   We have 
$r_{1+\overline{S}}(\alpha)< r_{j_{-\overline{S}-1}}(\alpha)$.
\end{Lem}
\begin{proof}  Here,   $\ell_{\underline{S}}(\alpha)  \ge \mathfrak b_{\alpha}$.  The proof of Lemma~\ref{l:oneLamRel} shows that here also  $r_{\overline{S}}(\alpha)\in \Delta_{\alpha}(1,2)$, and that now synchronization occurs with $\ell_{1+\underline{S}}(\alpha) = r_{2+\overline{S}}(\alpha)$.  Furthermore, (\cite{CaltaKraaikampSchmidt}, Lemma~6.2)
 shows that in this setting  we have $\ell_{\underline{S}}(\alpha) \in \Delta_{\alpha}(s,2)$ if and only if    $r_{2+\overline{S}}(\alpha) = A^sCAC^2\cdot r_{\overline{S}}(\alpha)$.  In particular,  $r_{\overline{S}}(\alpha)$ lies to the right of $\Delta_{\alpha}(1,1)$ and to the left of all values in 
$\Delta_{\alpha}(1,2)$ whose images are also in $\Delta_{\alpha}(1,2)$; 
thus, the proof of Lemma~\ref{l:smallestRvalue} succeeds also in this setting.     When $r_{2+\overline{S}}(\alpha) = A^sCAC^2\cdot r_{\overline{S}}(\alpha)$  with $(s,1) \preceq  (2,1)$,  we certainly have that $r_{1+\overline{S}}(\alpha) < r_{j_{-\overline{S}-1}}(\alpha)$.  We must consider the extreme case when $s=1$.   Since $r_{\overline{S}}(\alpha) = C A^{-1}C\cdot \ell_{\underline{S}}(\alpha)$ and $\ell_{\underline{S}}(\alpha)<r_0(\alpha)$, we have that 
$r_{1+\overline{S}}(\alpha) < AC^2 \cdot( C A^{-1}C\cdot r_0(\alpha)\,)$.   That is, $r_{1+\overline{S}}(\alpha) < C\cdot r_0(\alpha)$ and therefore, 
$AC\cdot r_{1+\overline{S}}(\alpha) < r_1(\alpha)$.  In terms of expansions, using the notation of the proof of Lemma~\ref{l:smallestRvalue},  $r_{1+\overline{S}}(\alpha) < (1,1),(1,2)^{n-3}, \sigma^{n-2}(\,\overline{b}(-k, v)\,), \star$ which one easily finds is less than   $(1,1),(1,2)^{n-3}, \sigma^{n-2}(\,\overline{b}(-k, u)\,), \star$.   That is,   $r_{1+\overline{S}}(\alpha)< r_{j_{-\overline{S}-1}}(\alpha)$.
\end{proof}

We must ``update" Definition ~\ref{d:bottomYvaluesLargeAlps}, confer Figure~\ref{f:omegaLargeAlpBiggerThanDelta_{-k,v}}.   
\begin{Def}\label{d:bottomYvaluesPastDelta} 
 Refine $L_{-\overline{S}-1} = [r_{1+\overline{S}}, r_{\overline{S}}]$, and let 
$L_{-\overline{S}-2}  = [\ell_0, r_{1+\overline{S}})$.  
Set 
 \[ y_{-\overline{S}-2} = RAC^2R^{-1}\cdot y_{\beta(\overline{S})}.\]
Let 
\[ \Omega^{-} = \bigcup_{b=-1}^{ -\overline{S}-2}\,   L_{b} \times [y_b, 0].\]
\end{Def}

The following is necessary to show that the leftmost blocks $\mathcal B_{i,1}$ continue to laminate above $\mathcal B_{i,2}$.   
\begin{Lem}\label{l:locatingFinalRpastDeltaOtherAlp}   We have 
$y_{-\overline{S}-2} = RCR^{-1}\cdot y_{\underline{S}}$.
\end{Lem}

\begin{proof}   This follows directly from the definition of $y_{-\overline{S}-2}$ and Proposition~\ref{p:RelationBottomAndTop}.
\end{proof} 
\bigskip  

We must be sure that lamination, thus Lemma~\ref{l:laminationInEasyCaseLargeAlps},   still holds.   The only  change is that now there are blocks $\mathcal B_{i,1}$ whose bottom height is $y_{-\overline{S}-2}$\,.    For each such $\mathcal B_{i,1}$,  due to the relation between $\ell_{\underline{S}}$ and $r_{1+\overline{S}}\,$,  the block $\mathcal B_{i, 2}$ has top height $y_{\underline{S}}$ (or this and $y_{\underline{S}+1}$).   Therefore,  Lemma~\ref{l:locatingFinalRpastDeltaOtherAlp} shows that Lemma~\ref{l:lamEqsSameAexpon} continues to apply (in the two bottom heights case, synchronization yields that the $x$-coordinates are correct for lamination).   The remainder of the proof goes through directly.

\begin{Eg}\label{e:kIs1SecondHalfOfSyncInt}  As in Subsection~\ref{ss:ExamplesKis1FirstHalfOfSyncInt} and Example~\ref{e:kIsMin1LeftEndpt}, let  $n=3$, $k=1$, $v=1$; thus,   $y_1 = g^2,  y_2 = G$, $y_{-1} = - g^2$.   We now consider $\alpha \in  (\delta_{-1,1},  \zeta_{-1,1}) =  (\delta_{-1,1},  \epsilon_3)$.  
From its definition, we find  $y_{-2} =  RAC^2R^{-1}\cdot (- g^2) = -(5 +\sqrt{5})/10$.     From the proof of Lemma~\ref{l:locatingFinalRpastDeltaOtherAlp}, and the fact that 
$\Delta_{\alpha}(s,2)$ lies to the right of $\Delta_{\alpha}(s,1)$, we certainly have that $r_{\overline{S}+1}< \ell_{\underline{S}}$.   Thus here we find that for all $\alpha \in (\delta_{-1,1},\epsilon_3 )$, one has 
\[\Omega_{3,\alpha} = (\, [\ell_0,  r_1]\times  [-(5 +\sqrt{5})/10, g^2]\,) \cup ([r_1, \ell_1]\times [-g^2, g^2]\,) \cup (\,[\ell_1, r_0]\times [-g^2, G]\,).\]
\end{Eg}
 
\smallskip
\subsection{Bijectivity and ergodicity at right endpoints} 
Although for   $k \in \mathbb N, v \in \mathcal V$, the $ \zeta_{-k,v}$  are values of $\alpha$ for which the  $T_{\alpha}$-obits of $\ell_0(\alpha), r_0(\alpha)$ do not synchronize,  we still can precisely describe domains on which their associated two-dimensional maps are bijective.     Indeed, one has $\ell_{\underline{S}} = \ell_0$   and thus the various preceding sections now easily imply the following results.  

\smallskip  
\begin{Prop}\label{p:OmegaForZetaKneg}   Fix $n\ge 3,  k \in \mathbb N, v \in \mathcal V$  (with $v \in \check{\mathcal V}$ if $k=1$) and $\alpha = \zeta_{-k,v}$.     Let $\Omega^{-}$ be as in Definition~\ref{d:bottomYvaluesPastDelta}.  Let $\Omega^{+}$ be as in Definition~\ref{d:topYvaluesLargeAlp} except that we redefine $K_{\underline{S}}$ to be  $[\ell_{i_{\underline{S}}}, r_0]$, and $\Omega^{+} = \bigcup_{a=1}^{ \underline{S}}\,   K_a \times [0, y_a]$.    

Then 
$\mathcal T_{n,\alpha}$ is bijective on $\Omega_{n, \alpha}  := \Omega^{+} \cup \Omega^{-}$, up to 
$\mu$-measure zero.   
\end{Prop}

\smallskip 

\begin{Eg}\label{e:kIsMin1RightEndpt}  We continue the examples with $n=3$, $k=1$, $v=1$; since $G/2 = \epsilon_3 =  \zeta_{-1,1}$,  
we  find 
 \[\Omega_{3,G/2} = (\,  [-g^2, g^2]\times  [-(5 +\sqrt{5})/10, g^2]\,)  \; \cup\;  (\,[g^2, G]\times [-g^2, g^2]\,) .\]
\end{Eg}

\smallskip 
\begin{Prop}\label{p:ergodicLargeZeta}   With $\alpha = \zeta_{-k,v}$ as in the previous proposition, 
the system $(\mathcal T_{\alpha}, \Omega_{\alpha}, \mathscr B'_{\alpha}, \mu_{\alpha})$ is ergodic.   

Furthermore,    this two dimensional system is the natural extension of  $(T_{\alpha},\mathbb I_{\alpha},  \mathscr B_{\alpha}, \nu_{\alpha})$, where $\nu_{\alpha}$ is the marginal measure of $\mu_{\alpha}$ and $\mathscr B_{\alpha}$  the Borel sigma algebra on $\mathbb I_{\alpha}$.   In particular, the one dimensional system is ergodic. 
\end{Prop}

\begin{proof}     From (\cite{CaltaKraaikampSchmidt},  Lemma~6.3),  
 both $\ell_0(\alpha)$ and  $r_0(\alpha)$ have  periodic $\alpha$-expansions.   From the first of these,  $\mathfrak b_{\alpha}$  also has of finitely many $\alpha$-digits.  In view of Remark~\ref{rmk:bddNonfullCylinders},  it follows that 
 the non-full cylinders of $T_{\alpha}$ are bounded in range.   Due to this and the finiteness of $\mu(\Omega_{\alpha})$ and of its vertical fibers, the hypotheses of (\cite{CKStoolsOfTheTrade}, Theorem~2.3) are met, thus giving the result here.
\end{proof}

\subsection{Bijectivity and ergodicity at division point between portions: $\alpha = \delta_{-k,v}$}\label{ss:smackDab}  At $\delta = \delta_{-k,v}$, we will find that the ``extra" rectangle in $\Omega_{\alpha}$ for $\alpha >\delta$ (as opposed for $\alpha<\delta$ in the same synchronization interval) has width zero and thus can be ignored.    From the proof of Lemma~\ref{l:locatingFinalRpastDelta}, one has that for the value of $s$ such that $\ell_1(\delta) = A^sC \cdot \ell_0(\delta)$ we also have $A^sC^2\cdot \ell_{\overline{S}}(\delta) = \ell_1(\delta)$ and $A^s C \cdot r_{\overline{S}+1}(\delta) = \ell_1(\delta)$.   (Of course, here $s = -k$.)  That is, $r_{\overline{S}+1}(\delta) = \ell_0(\delta)$.    Thus,    Definition ~\ref{d:bottomYvaluesPastDelta} reverts back to 
Definition~\ref{d:bottomYvaluesLargeAlps} and the bijectivity of $\mathcal T_{\delta}$ on the resulting $\Omega_{\delta}$ holds as above. 

\begin{Eg}\label{e:kIsMin1Delta}  We continue the examples with $n=3$, $k=1$, $v=1$, as in Subsection~\ref{ss:ExamplesKis1FirstHalfOfSyncInt} and Examples~\ref{e:kIsMin1LeftEndpt}, ~\ref{e:kIs1SecondHalfOfSyncInt}.  We have  $\delta_{-1,1} = (5-\sqrt{5})/4$ and thus $\ell_0(\delta_{-1,1}) =  -g$.       We  find 
 \[\Omega_{3,\delta_{-1,1}} = (\,  [-g, g]\times  [-g^2, g^2]\,)  \; \cup\;  (\,[g, G]\times [-g^2, G]\,) .\]
\end{Eg}

\medskip
\begin{Prop}\label{p:ergodicDelta}   With $\alpha = \delta_{-k,v}$ as in the previous proposition, let $\mathscr B'_{\alpha}$ denote the Borel sigma algebra on $\Omega_{\alpha}$.   
The system $(\mathcal T_{\alpha}, \Omega_{\alpha}, \mathscr B'_{\alpha}, \mu_{\alpha})$ is ergodic.   
Furthermore,    this two dimensional system is the natural extension of  $(T_{\alpha},\mathbb I_{\alpha},  \mathscr B_{\alpha}, \nu_{\alpha})$, where $\nu_{\alpha}$ is the marginal measure of $\mu_{\alpha}$ and $\mathscr B_{\alpha}$  the Borel sigma algebra on $\mathbb I_{\alpha}$.   In particular, the one dimensional system is ergodic. 
\end{Prop}

\begin{proof}   Since  $\ell_0(\alpha)$ has a periodic $\alpha$-expansion and synchronization holds, in view of Remark~\ref{rmk:bddNonfullCylinders},  it follows  that 
 the non-full cylinders of $T_{\alpha}$ are bounded in range.   Due to this and the finiteness of $\mu(\Omega_{\alpha})$ and of its vertical fibers, the hypotheses of (\cite{CKStoolsOfTheTrade}, Theorem~2.3) are met, thus giving the result here.
\end{proof}  

\smallskip   
 
\section{Limit values for larger $\alpha$}\label{s:limLarge}   
 The main result of this section are Propositions~\ref{p:fillUpFromW} and ~\ref{p:biDomIsNatExtErgLargeNonSyn} giving a bijectivity domain for $\mathcal T_{\alpha}$ and then showing this to be an ergodic natural extension of $T_{\alpha}$,  for each large non-synchronizing $\alpha$.  As in \S~\ref{s:ContSmall}, due to the piecewise `twisted'  action on $y$-values, the proof of surjectivity requires that certain reversed words are admissible.  Here also, we prove this admissibility by showing the admissibility of such words for the systems of endpoints of synchronization intervals   and then arguing by taking limits.    
 
\bigskip
\noindent
{\bf  Convention}   Just as in Subsection~\ref{s:ContSmall}  where we used $u$ as a general word in the setting of small $\alpha$ values,  here we   use $w$ as a general word.  

\subsection{Domain for non-synchronizing values}
We introduce some more notation and define basic sets for our construction.    For legibility,  we suppress various indications of dependence on $\alpha$ and the like, trusting that this causes no confusion for the reader.  
\begin{Def}\label{d:notationForEndpointsLargeAlps}   Fix $n\ge 3,  k \in \mathbb N$ and $\alpha > \gamma_{n}$. Let $\mathcal L= \mathcal L_{\alpha}$ denote the language of admissible   digits.   For finite length $u \in \mathcal L$,  as usual let $\Delta_{\alpha}(u)$ be the corresponding $T_{\alpha}$-cylinder.    Let $\lambda_u,  \rho_u$ denote the left- and right endpoints, respectively of $\Delta_{\alpha}(u)$; note that $\rho_{1,1} = \lambda_{-k,2} = \mathfrak b$.   For any $d \in \mathbb N$, let  $\mu_{-d,1} = (A^{-d}C)^{-1}\cdot \mathfrak b$.    

We then define vertical fibers, see  Figure~\ref{f:theW},  
\[ \Phi_1 = [ -\mathfrak b, -\mu_{-k-1, 1}], \; \Phi_2 = [ -\mathfrak b, -\mu_{-k,1}], \]
\[\Phi_3 = [ -\mathfrak b, - \ell_0], \;\Phi_4 = [-C\cdot r_0, - \ell_0],  \; \text{and}\; \Phi_5 = [-\rho_{2,1}, -\ell_0]\,,\]  

and regions
\[ \mathcal J= \mathbb I_{\alpha} \times [-\rho_{2,1},  -\mu_{-k-1, 1}]\;\cup\; [\ell_1, r_0]\times [-\rho_{-k,1},  -\mu_{-k, 1}] \]

and 
\[
\begin{aligned} 
\mathcal W =  [\ell_0, \ell_1]  \times \Phi_1\;&\cup\;  [\ell_1, \lambda_{-k-2,1})\times \Phi_2 \;\cup \;[\lambda_{-k-2,1}, \rho_{2,1}]\times \Phi _3 \\
                                                                     & \;\cup\; [ \rho_{2,1}, \rho_{2,2}) \times \Phi _4 \;\cup \; [\rho_{2,2}, r_0) \times \Phi_5\,.
\end{aligned}
 \]
   Note that $\mathcal W \supset \mathcal J$.
\end{Def}

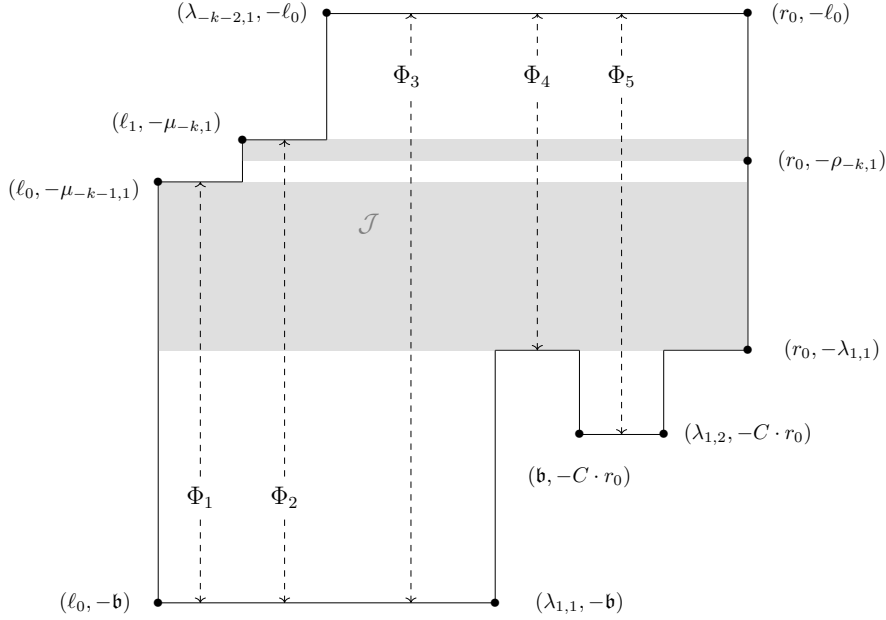
\begin{figure}
\scalebox{.8}{
\begin{tikzpicture}[x=7cm,y=7cm] 
\draw  (0, 0)--(0, 1.0); 
\draw  (0, 1)--(0.2, 1);
\draw  (0.2,1)--(0.2, 1.1);
\draw  (0.2,1.1)--(0.4, 1.1);
\draw  (0.4,1.1)--(0.4, 1.4);
\draw  (0.4, 1.4)--(1.4, 1.4);
\draw (1.4, 1.4)--(1.4, 0.6);
\draw (1.4, 0.6)--(1.2, 0.6);
\draw (1.2, 0.6)--(1.2, 0.4);
\draw (1.2, 0.4)--(1, 0.4);
\draw (1, 0.4)--(1, 0.6);
\draw (1, 0.6)--(.8, 0.6);
\draw (.8, 0.6)--(.8, 0.0);
\draw (.8, 0.0)--(0,0);
\fill [opacity= 0.25, gray] (0.2, 1.1) -- (1.4,1.1) --(1.4, 1.05)--(0.2, 1.05)--cycle;
\fill [opacity= 0.25, gray] (0, 1) -- (1.4,1) --(1.4, 0.6)--(0, 0.6)--cycle;
\node [gray] at (0.5,  .9) {\Large{$\mathcal J$}}; 
\draw[<-, dashed] (0.1, 0)--(0.1, .2);
\node at (0.1,  .25) {\Large{$\Phi_1$}}; 
\draw[->, dashed] (0.1, 0.3)--(0.1, 1.0);
\draw[<-, dashed] (0.3, 0)--(0.3, .2);
\node at (0.3,  .25) {\Large{$\Phi_2$}}; 
\draw[->, dashed] (0.3, 0.3)--(0.3, 1.1);
\draw[<-, dashed] (0.6, 0)--(0.6, 1.2);
\node at (0.59,  1.25) {\Large{$\Phi_3$}}; 
\draw[->, dashed] (0.6, 1.3)--(0.6, 1.4);
\draw[<-, dashed] (.9, 0.6)--(.9, 1.2);
\node at (.9,  1.25) {\Large{$\Phi_4$}}; 
\draw[->, dashed] (.9, 1.3)--(.9, 1.4);
\draw[<-, dashed] (1.1, 0.4)--(1.1, 1.2);
\node at (1.1,  1.25) {\Large{$\Phi_5$}}; 
\draw[->, dashed] (1.1, 1.3)--(1.1, 1.4);
 \foreach \x/\y in {0/0,   0/1, 0.2/1.1, 0.4/1.4, 1.4/1.4, 1.4/0.6,1.4/1.05,1.2/0.4,1/0.4,.8/0.0%
 } { \node at (\x,\y) {$\bullet$}; } 
\node at ( -0.15, 0) {$(\ell_0, -\mathfrak b)$};  
\node at ( -0.2, 0.98) {$(\ell_0, -\mu_{-k-1,1})$};  
\node at ( 0.02, 1.14) {$(\ell_1, -\mu_{-k,1})$};  
\node at ( 0.2, 1.4) {$(\lambda_{-k-2,1}, -\ell_0)$};  
\node at ( 1,0.3) {$(\mathfrak b, -C \cdot r_0)$};
\node at ( 1.4,0.4) {$(\lambda_{1,2}, -C \cdot r_0)$};
\node at ( 1.6,1.05) {$(r_0, -\rho_{-k,1})$};
\node at (1.55, 1.4) {$(r_0, -\ell_0)$};  
\node at (1.6, 0.6) {$(r_0, -\lambda_{1,1})$};  
\node at (1, 0) {$(\lambda_{1,1}, -\mathfrak b)$};  
\end{tikzpicture}
}
\caption{Schematic representation of $\mathcal W_{\alpha}$, see Definition~\ref{d:notationForEndpointsLargeAlps}.   For large $\alpha = \eta_{-k,v}$ or for non-synchronizing  $\alpha$,   Lemma~\ref{l:jIsCorrect} shows that the image of the regions in $\mathcal W_{\alpha}$  fibering over all cylinders other than those of the digits $(-k,1), (-k-1,1), (1,1), (1,2)$ gives the  (gray) region $\mathcal J = \mathcal J_\alpha$  (of two connected components).    Proposition~\ref{p:fillUpFromW}  shows that the  complementary pieces are sent by powers of  $\mathcal T_{\alpha}$ to fill in a region that strictly includes the complement of $\mathcal J_\alpha$ in $\mathcal W_{\alpha}$.}
\label{f:theW}%
\end{figure}

 \begin{Lem}\label{l:jIsCorrect}   Fix a non-synchronizing $\alpha$ with $\alpha >\gamma_{n}$.   Then 
 \[J_\alpha =    \cup_{\substack{(i,1)\succ (-k-1,1)\\ i\neq 1}}\;   \mathcal T_{\alpha}(\, \Delta_{\alpha}(i,1)\times \Phi_3\,)  \bigcup\,  \cup_{(i,2)\prec (1,2)}\,  \mathcal T_{\alpha}(\,\Delta_{\alpha}(i,2)\times \Phi_5\,).\]
 \end{Lem}
 \begin{proof}  From Lemma~\ref{l:conjByRofAtoPc},  for each $i \in \mathbb Z$, 
 \[  RA^iCR^{-1}\cdot -\ell_0 = -\lambda_{i,1}, \;RA^iCR^{-1}\cdot -\mathfrak b = -\mu_{i,1} = RA^iC^2R^{-1}\cdot -\ell_0, \;RA^iC^2R^{-1}\cdot (-C\cdot r_0) = \rho_{i,1}.\]
Since adjacent cylinders share a common endpoint, and $(-k,2)$ is the only non-full cylinders considered in the right hand side of the expression,   the result follows. 
 \end{proof}
 
\smallskip
\noindent
{\bf Convention.}  Throughout this section, we will use $w$ to denote a general word.  (Compare this to our use of $u$ in Section \ref{s:ContSmall}.)

 \begin{Prop}\label{p:fillUpFromW}     Fix a non-synchronizing $\alpha$ with $\alpha >\zeta_{-1, n-2}$. 
 With notation as above, let $\mathcal T = \mathcal T_{\alpha}$ and define 
 \[\Omega = \overline{\cup_{j=1}^{\infty}\, \mathcal T^j(\, \mathcal W)\,}\,.\]
 Then    $\Omega \supset \mathcal W$ and $\mathcal T$ is bijective  up to $\mu$-measure zero on $\Omega$.   Furthermore,  letting 
 \[  \mathscr A =  \{(-k,1), (-k-1,1)\}^{^*}\cap \mathcal L_{\alpha},\;\;\;   \mathscr B =  \{(1,1), (1,2)\}^{^*}\cap \mathcal L_{\alpha} \] 
 and 
\[ \Phi'_2 =  \Phi_2\setminus (-\mu_{-k-1,1},-\rho_{-k,1})\]
we have that $\Omega$ is the closure of 
 \begin{equation} \label{e:jtUnionBig} 
 \begin{aligned} \mathcal J \;  
 &\bigsqcup \, \sqcup_{\substack{w\in \mathscr A  \\ \rho_w < \ell_1}}\;  \mathcal T^{|w|}(\, \Delta_{\alpha}(w) \times \Phi_1 \,)
 \, \bigsqcup \, \sqcup_{\substack{w\in \mathscr A  \\ \lambda_w > \ell_1}}\;  \mathcal T^{|w|}(\, \Delta_{\alpha}(w) \times \Phi'_2 \,)  \\
 \\
 &\,\bigsqcup \, \sqcup_{s = 2}^{\infty}\, \mathcal T^{|w|}(\, \Delta_{\alpha}(\underline{b}_{[2, s]}^{\alpha}) \times \Phi_1 \,) \, \sqcup \, [\ell_s, r_0)\times  RM_{\underline{b}_{[2, s]}^{\alpha}}R^{-1}(\, \Phi'_2\setminus \Phi_1\,)
\\
\\
                                      &\, \bigsqcup \,\sqcup_{w\in   \mathscr B \, \cap \mathcal L_{\alpha}}\; \mathcal T^{|w|}(\, \Delta_{\alpha}(w) \times \Phi_4 \,).
                         \end{aligned}
\end{equation} 
Furthermore,  the system $(\mathcal T_{\alpha}, \Omega_{\alpha}, \mathscr B'_{\alpha}, \mu_{\alpha})$ is the natural extension of  $(T_{\alpha},\mathbb I_{\alpha},  \mathscr B_{\alpha}, \nu_{\alpha})$, where $\nu_{\alpha}$ is the marginal measure of $\mu_{\alpha}$ and $\mathscr B_{\alpha}$  the Borel sigma algebra on $\mathbb I_{\alpha}$.  Finally,   both systems are ergodic.
\end{Prop}

 The surjectivity of $\mathcal T$ on $\Omega$ follows from the definition of $\Omega$.    Key to proving the remaining statements of the proposition is to show that each $x\in \mathbb I_{\alpha}$ of  first digit other than $\{(-k,1), (-k-1,1), (1,1), (1,2)\}$    has the same $y$-fiber in $\Omega$ as in $\mathcal W$.

\subsection{Agreement of fibers on lower portions of $\mathcal W, \Omega$}   

\begin{Lem}\label{l:fiberIntervals}  For 
 \[ w = (1,1)^{a_1}(1,2)^{b_1}\cdots (1,1)^{a_s}(1,2)^{b_s} \in \mathcal L_{\alpha}\,,\]
 $\mathcal T^{|w|}(\, \Delta_{\alpha}(w) \times \Phi_4 \,)$ has $y$-fiber equal to 
\[\begin{cases} [ - ( M_{\overleftarrow{w}} AC)^{-1} \cdot \ell_0,\, - ( (AC)^{-1} M_{\overleftarrow{w}} AC)^{-1} \cdot \ell_0]& \text{if}\;\; a_1>0\\
\\
                        [-(M_{\overleftarrow{w}} AC)^{-1} \cdot \ell_0,\, - (AC (AC^2)^{-1} M_{\overleftarrow{w}} AC)^{-1} \cdot r_0]& \text{if}\;\; a_1=0
  \end{cases}
 \]                       
\end{Lem}

 \begin{proof}  We have  $M_w = (AC^2)^{b_s}(AC)^{a_s}\cdots (AC^2)^{b_1}(AC)^{a_1}$,  and  for any real $x$  Lemma~\ref{l:conjByRofAtoPc} gives  
 \[ R M_w R^{-1}\cdot -x = - (\,(AC)^{a_1} (CAC)^{b_1}\cdots (AC)^{a_s}(CAC)^{b_s} \,)^{-1}\cdot x.\]  
 Independent of the vanishing of $a_1$,  one finds that  $R M_w R^{-1}\cdot -x$ equals $- ( (AC)^{-1} M_{\overleftarrow{w}} AC)^{-1} \cdot x$.  When $x = \lambda_{1,1} = (AC)^{-1}\cdot \ell_0$, this value equals $- ( M_{\overleftarrow{w}} AC)^{-1} \cdot \ell_0$.  When $a_1 = 0$,  the expression  $(AC)^{-1} M_{\overleftarrow{w}}$ factors as a $C$ followed by a product of powers of $AC^2$ and $AC$;  replacing $\ell_0$ by $r_0 = A\cdot \ell_0$ results in an expression which is a word in our matrices, while giving the formula indicated for the largest value in the image of the fiber.  
 \end{proof}

The previous result shows that corresponding to $w$ are fibers that, upon using a negative exponent to indicate cancelling a letter in a word, one  might express as  $[-\lambda_{(1,1)\overleftarrow{w}}, -\lambda_{(1,1)\overleftarrow{w}(1,1)^{-1}}\,]$  and  $[-\lambda_{(1,1)\overleftarrow{w}}, -\rho_{(1,1)\overleftarrow{w}(1,2)^{-1}(1,1)}\,]$.  For these expressions to be valid,  we must then also prove that each of the subscripts is given by an admissible word.      The next two lemmas give us this result.

\begin{Lem}\label{l:fillDelOneOne}    Fix $n\ge 3$,  and  either a non-synchronizing $\alpha > \gamma_{n}$ or $\alpha = \eta_{-k,v}$ with $k\in \mathbb N$ and $v \in \mathcal V$.   Then 
 \[ \Delta_{\alpha}(1,1)  = \overline{\sqcup_{w\in (1,1)\{(1,1), (1,2)\}^{^*}\cap \mathcal L_{\alpha}}\, [\lambda_w,   \lambda_{w,(1,1)}) \cup [\rho_{w,(1,1)},   \,  \lambda_{w,(1,2)})}.\]
\end{Lem}

 \begin{proof}  By definition in the union, we include $w = (1,1)$.  Note that any $w\in \{(1,1), (1,2)\}^{^*}$ is in $\mathcal L_{\alpha}$ if and only if   $w \preceq \overline{d}_{[1, \infty)}^{\alpha}$; in particular, whenever $w \in \mathcal L_{\alpha}$ so is also $w,(1,1)$. 
 
  For any $w\in (1,1)\{(1,1), (1,2)\}^{^*}\cap \mathcal L_{\alpha}$,   we have 
 \[  \Delta_{\alpha}(w) = [\lambda_w,   \lambda_{w,(1,1)}) \cup \Delta_{\alpha}(w, (1,1)\,)\, \cup [\rho_{w,(1,1)},   \,  \lambda_{w,(1,2)}) \cup  \Delta_{\alpha}(w, (1,2)\,),\]
 where by definition $ \Delta_{\alpha}(w, (1,2)\,) = \emptyset$ if $w, (1,2) \notin \mathcal L_{\alpha}$.   From this, by beginning with $w=(1,1)$,   it follows that $\Delta_{\alpha}(1,1)$ does equal the closure of the union over all $w\in (1,1)\{(1,1), (1,2)\}^{^*}\cap \mathcal L_{\alpha}$ of  $ [\lambda_w,   \lambda_{w,(1,1)}) \cup [\rho_{w,(1,1)},   \,  \lambda_{w,(1,2)})$.   
 
 We have tacitly used the fact that each admissible infinite word in $\{(1,1), (1,2)\}^{^*}$  represents a unique point,  this holds due to the dynamics of the $T_{n,1}$; see (\cite{CaltaKraaikampSchmidt}, \S ~3) and (\cite{CKStoolsOfTheTrade}, \S~2.6).
\end{proof}
 
\begin{Lem}\label{l:admissible}    Fix $n\ge 3$,  and  either a non-synchronizing $\alpha > \gamma_{n}$ or $\alpha = \eta_{-k,v}$ with $k\in \mathbb N$ and $v \in \mathcal V$.    Let $w\in (1,1)\{(1,1), (1,2)\}^{^*}\cap \mathcal L_{\alpha}$.  Then $(1,1)\overleftarrow{w} \in\mathcal L_{\alpha}$ and so are $(1,1)\overleftarrow{w}(1,1)^{-1}$ if $w$ has prefix $(1,1)$ and $(1,1)\overleftarrow{w}(1,2)^{-1}(1,1)$ otherwise.  
\end{Lem}

 \begin{proof}  For $w$ of digits in $\{(1,1), (1,2)\}$, we have $w$ is admissible if and only if it and all its  suffixes are smaller or equal to the expansion of $r_0(\alpha)$, thus than $\overline{b}_{[1, \infty)}^{\alpha}$.  In particular,  $w$ is admissible if and only if $w, (1,1)$ is.  Furthermore, 
 $\overleftarrow{w} \in\mathcal L_{\alpha}$ if and only if the various $(1,1)\overleftarrow{w}, (1,1)\overleftarrow{w}(1,1)^{-1}, (1,1)\overleftarrow{w}(1,2)^{-1}(1,1)$ are (with cases depending again on whether $w$ has prefix $(1,1)$ or not).    
 
 We thus aim to show that $w \in \mathcal L_{\alpha}$ implies  $\overleftarrow{w} \in\mathcal L_{\alpha}$.  For this, we argue by contradiction and can assume that  $w$ is of minimal length such that $\overleftarrow{w} \notin  \mathcal L_{\alpha}$.    We can now argue as in the proof of  Lemma~\ref{l:admissibility}  and Corollary~\ref{c:symmetryInLimitAlph}.
\end{proof} 
 
\begin{Lem}\label{l:correctFibersBottom}   With $\alpha$ as above, each $x\in \mathbb I_{\alpha}$ of  first digit other than $\{(1,1), (1,2)\}$    has the same $y$-fiber in $\Omega \cap \{y\le 0\}$ as in $\mathcal W \cap \{y\le 0\}$.  
\end{Lem}
 
\begin{proof} Lemmas~\ref{l:fiberIntervals}, \ref{l:fillDelOneOne}, and \ref{l:admissible} combine to show that each $x$ less than the infimum of the elements of the $T_{\alpha}$-orbit has $y$ fiber in $\Omega_{\alpha}$ meeting $\{y\le 0\}$ in exactly  $[\mathfrak b, - \lambda_{1,1}]$.    Thus equality of fibers holds for all such $x$. 
 
 It remains to show the equality when $x \in [\mathfrak b, \lambda_{1,2})$.   For this, it suffices to show that $[-\mathfrak b, -C\cdot r_0]$ equals the closure of the union of the images of $\Phi_4$ corresponding to the cylinders  $\Delta_{\alpha}(w)$ such that $T^{|w|}(\rho_w) = r_j$ for some $j$ and such that $r_j \in  \Delta_{\alpha}(1,1)$.   Let us say that such a word $w$ is {\em of $(1,1)$ right value}.
 
 We first note that for any $\eta_{-k,v}$ this equality holds due to Proposition~\ref{p:OmegaForEtaKneg}.    Let $\alpha$ be a non-synchronizing value, then there  there is a fixed $k$ and a branch of $\mathcal V$ such that taking $v$ along the branch gives values of the form $\eta_{-k,v}$ converging to $\alpha$ from the right, and such that the $\overline{b}_{[1, \infty)}^{\eta_{-k,v}}$ agree to ever greater length with $\overline{b}_{[1, \infty)}^{\alpha}$.      Given $w\in (1,1)\{(1,1), (1,2)\}^{^*}\cap \mathcal L_{\alpha}$, since in particular $w \preceq \overline{b}_{[1, \infty)}^{\alpha}$ the finiteness of $w$ implies that for all $v$ sufficiently far along the branch of $\mathcal V$ that also $w \in \mathcal L_{\eta_{-k,v}}$.    Furthermore,  if $w$ is of $(1,1)$ right value with respect to $\alpha$ then it is also for all for all $v$ sufficiently far along the branch.   
 
 The convergence of the various $\eta_{-k,v}$  to $\alpha$ gives that  
 $RM_wR^{-1} \cdot -\ell_0(\eta_{-k,v}) \to RM_wR^{-1}\cdot -\ell_0(\alpha)$ and also   $-C\cdot r_0(\eta_{-k,v})\to -C\cdot r_0(\alpha)$.   Since for each $w$ of $(1,1)$ right value  $RM_wR^{-1} \cdot -\ell_0(\eta_{-k,v})$ is below $-C\cdot r_0(\eta_{-k,v})$,  we conclude that also $RM_wR^{-1} \cdot -\ell_0(\alpha) \le -C\cdot r_0(\alpha)$.    For any other $w\in (1,1)\{(1,1), (1,2)\}^{^*}\cap \mathcal L_{\alpha}$, we argue similarly to find that $RM_wR^{-1}\cdot -\ell_0(\alpha) \ge -C\cdot r_0(\alpha)$.   We conclude that $[-\mathfrak b, -C\cdot r_0]$ equals the closure of the union of the images of $\Phi_4$ corresponding to the cylinders indexed by words of $(1,1)$ right value.  The result thus holds. 
\end{proof} 

\subsection{Agreement of fibers on upper portions of $\mathcal W, \Omega$}   

\subsubsection{Statement and beginning results}
We announce the main result of this subsection.
 \begin{Prop}\label{p:correctFibersTop}   With $\alpha$ as above, each $x\in \mathbb I_{\alpha}$ of  first digit other than $\{(-k,1), (-k-1,1)\}$    has the same $y$-fiber in $\Omega \cap \{y\ge 0\}$ as in $\mathcal W \cap \{y\ge 0\}$.  
 \end{Prop}

\noindent
{\bf Convention.} 
The arguments in this subsection mainly involve the digits $(-k,1)$ and $(-k-1,1)$,  thus  we will use simplified digits.  (In particular,  we consider  $\underline{b}_{[1, \infty)}^{\alpha}$ to be equal to $\underline{d}_{[1, \infty)}^{\alpha}$.)\\

For further typographic ease,   for $w \in \mathcal L_{\alpha}$,  let  $\nu_w$ be the point of simplified digit expansion  $w, (-k-1)^{\infty}$. \\

The following is proved just as in the small $\alpha$ setting, see Lemma~\ref{l:admissibility} and its corollary.    
\begin{Lem}\label{l:symmetry}   With $\alpha$ as above,   if $w \in \{-k, -k-1\}^{*} \cap \mathcal L_{\alpha}$   then   also $\overleftarrow{w} \in \mathcal L_{\alpha}$.
\end{Lem}

\subsubsection{Images of fibers}
 We begin by collecting images of $\Phi_1$.
\begin{Lem}\label{l:extendToFixedPoint}   With $\alpha$ as above,   if $w \in \{-k, -k-1\}^{*} \cap \mathcal L_{\alpha}$   then 
 \[\cup_{j=0}^{\infty}   \, R M_{w, (-k-1)^j} R^{-1} \cdot \Phi_1 =   -(\nu_{\overleftarrow{w}}, \mu_{\overleftarrow{w}}].\]
\end{Lem}

 \begin{proof} Given $w \in \{-k, -k-1\}^{*} \cap \mathcal L_{\alpha}$ from the previous result we have that $\overleftarrow{w} \in  \mathcal L_{\alpha}$.  Furthermore, for words of digits $-k, -k-1$ admissibility is determined by being not less than $\underline{b}_{[1, \infty)}^{\alpha}$ and thus  (since $-k \prec -k-1$) appending (or prepending) any power of $-k-1$ to such an admissible word results in an admissible word.  Thus, the right hand side of the displayed equation is sensical. 
 
 When $w$ is a word in simplified digits, Lemma~\ref{l:conjByRofAtoPc} gives that $R M_w R^{-1}\cdot -x = (M_{\overleftarrow{w}})^{-1}\cdot x$.  Thus, for $w \in \{-k, -k-1\}^{*} \cap \mathcal L_{\alpha}$ we find that $R M_{w} R^{-1} \cdot \Phi_1 =- [\mu_{\overleftarrow{w}, -k-1}, \mu_{\overleftarrow{w}}]$.  The result easily follows. 
\end{proof} 
 
\begin{Def}\label{d:wTilde}  For $w \in \{-k, -k-1\}^{*} \cap \mathcal L_{\alpha}$ of length at least two and with suffix $-k$,  let 
\[ \widetilde{w} = \overleftarrow{\phantom{b}w_{[1, |w|-1}]}.\]
\end{Def}

 \begin{Lem}\label{l:goodRelationWwTilde}   With $\alpha$ as above,   the function $w \mapsto \widetilde{w}$ gives  a one-to-one correspondence between the set of $w \in \{-k, -k-1\}^{*} \cap \mathcal L_{\alpha}$ of length at least two and with suffix $-k$ and the set of $w \in \{-k, -k-1\}^{*} \cap \mathcal L_{\alpha}$ with $\widetilde{w} \succ \underline{b}_{[2, \infty)}^{\alpha}$.    
 \end{Lem}

 \begin{proof}  We first note that the function $w \mapsto \widetilde{w}$ is clearly injective into the set $\{-k, -k-1\}^{*}$.
  
If  $w \in \{-k, -k-1\}^{*} \cap \mathcal L_{\alpha}$ of length at least two and with suffix $-k$, then we know that also $\overleftarrow{w} \in \mathcal L_{\alpha}$ and in particular  $\overleftarrow{w} \succeq \underline{b}_{[1, \infty)}^{\alpha}$.    Hence, $\widetilde{w} \succeq \underline{b}_{[2, \infty)}^{\alpha}$, since $w$ is of finite length we have that   $\widetilde{w} \succ \underline{b}_{[2, \infty)}^{\alpha}$. 
 
 On the other hand, if $w \in \{-k, -k-1\}^{*} \cap \mathcal L_{\alpha}$ is such that $w \succ \underline{b}_{[2, \infty)}^{\alpha}$, then  $-k, w \succ \underline{b}_{[1, \infty)}^{\alpha}$ and thus  $-k, w \in \mathcal L_{\alpha}$.  From this,   $\overleftarrow{-k,w} \in \mathcal L_{\alpha}$ is of length at least two and has $-k$ as a suffix.   
\end{proof} 
  
\begin{Lem}\label{l:givesFiberPieceForW}   With $\alpha$ as above,   
 
 \[ - [\mu_w, \rho_w) = R M_{\widetilde{w}}R^{-1} (\, \Phi'_2\setminus \Phi_1\,).\] 
\end{Lem}

 \begin{proof}  Since $\Phi'_2\setminus \Phi_1 = -[\mu_{-k}, \rho_{-k})$,  using Lemma~\ref{l:conjByRofAtoPc}   we find that $R M_{\widetilde{w}}R^{-1} (\, \Phi'_2\setminus \Phi_1\,) = - [\mu_w, \rho_w)$.
 \end{proof}

 \subsubsection{Partitioning an $x$-interval related to the upper portion of fibers $\Phi_3, \Phi_4, \Phi_5$}

\begin{Def}\label{d:children}  Suppose that $w \in \{-k, -k-1\}^{*} \cap \mathcal L_{\alpha}$.   Since every $w, (-k-1)^d \in \mathcal L_{\alpha}$, there is a minimal $d\ge 0$ such that $w, (-k-1)^d, -k$ is admissible. The {\em children} of $w$ are the words $w, (-k-1)^i, -k$ with $i \ge d$.      
\end{Def}

 \begin{Lem}\label{l:childrenAreAdmissible}   With $\alpha$ as above and $w \in \{-k, -k-1\}^{*} \cap \mathcal L_{\alpha}$, all of the children of $w$ are admissible.  Furthermore, for each child  $w'$ of $w$,  we have $\rho_{w'} \le \nu_{w}$.     Finally, if $x \in [\lambda_w, \nu_w)$ then there is a largest  child $w'$ of $w$ such that $x \in \Delta_{\alpha}(w')$. 
\end{Lem}

\begin{proof}  Recall that for any  $w \in \{-k, -k-1\}^{*} \cap \mathcal L_{\alpha}$ we have that every  $w, (-k-1)^i$ with $i\in \mathbb N$ is admissible.   Admissibility for words in  $\{-k, -k-1\}^{*}$ is a matter of being at least as large as $\underline{b}_{[1, \infty)}^{\alpha}$, thus the requirement that $j$ be at least as large as the minimal $d$ with $w, (-k-1)^d, -k$ admissible does imply that all of the children are admissible.
 
Let $w'$ be a child of $w$.   For any $x \in \Delta_{\alpha}(w')$, there some $i$ such that the simplified digits $x$ has prefix $w'$  since $\nu_w$ is the point of simplified digits $w, (-k-1)^{\infty}$ we find that $x < \nu_w$.  By continuity we find that $\rho_{w'} \le \nu_w$ and by induction on $i$ one  finds that strict inequality holds. 
  
  For any $j\ge 0$,  we have  $w, (-k-1)^j, -k \prec w, (-k-1)^{j+1}, -k$.  Given $x \in [\lambda_w, \nu_w)$,  there is a largest $i\ge 0$ such that the simplified digits of $x$ begin $w, (-k-1)^{i+1}, -k$.  Hence with $w' = w, (-k-1)^{i+1}, -k$ we have both that $x \in \Delta_{\alpha}(w')$ and that $w'$ is the largest child of $w$ with this property.    
\end{proof} 

 Recall that the final letter of a word $w$ is denoted $w_{[-1]}$.
\begin{Lem}\label{l:coversInterval}   With $\alpha$ as above,   
  \[[\ell_0, \mu_{-k}] \cup [\lambda_{-k-1}, \mu_{-k-1}] =  \overline{ (\nu_{-k}, \mu_{-k}] \cup (\nu_{-k-1}, \mu_{-k-1}] \cup\, \bigcup_{\substack{w \in \{-k, -k-1\}^{*} \cap \mathcal L_{\alpha}\\ |w|\ge 2\\w_{[-1]}=-k}} (\nu_{w}, \rho_{w}) }\,.\] 
\end{Lem}

 \begin{proof}  Since $\ell_0 = \lambda_{-k}$, we have that each $x \in [\ell_0, \mu_{-k}] \cup [\lambda_{-k-1}, \mu_{-k-1}]$ is in $[\lambda_w, \nu_w)\cup [\nu_w, \mu_w]$ for one $w \in \{-k, -k-1\}$.   If $x \notin [\nu_w, \mu_w]$ then there is a largest child $w'$ such that $x\in [\lambda_{w'}, \nu_{w'})\cup [\nu_{w'}, \mu_{w'}]$.   Since every child $w'$ is in the set $\{-k, -k-1\}^{*} \cap \mathcal L_{\alpha}$ with  $|w|\ge 2$ and $w_{[-1]}=-k$, the result follows by induction.   
\end{proof}

 \subsubsection{Proof that fibers agree:  reaching top}    Recycling notation, let  $\mathscr L = \sup_j\, \{\ell_j \}$.  Note that $\mathscr L < \lambda_{-k-2,1}$.   
\begin{Lem}\label{l:topFilledIn}       The intersection of $[\mathscr L, r_0)\times [-\mu_{-k-1}, \infty)$ with the closure of 
\[\begin{aligned} 
&\bigsqcup \, \sqcup_{\substack{w\in \mathscr A  \\ \rho_w < \ell_1}}\;  \mathcal T^{|w|}(\, \Delta_{\alpha}(w) \times \Phi_1 \,)
 \, \bigsqcup \, \sqcup_{\substack{w\in \mathscr A  \\ \lambda_w > \ell_1}}\;  \mathcal T^{|w|}(\, \Delta_{\alpha}(w) \times \Phi'_2 \,)  \\
 \\
 &\,\bigsqcup \, \sqcup_{s = 2}^{\infty}\, \mathcal T^{|w|}(\, \Delta_{\alpha}(\underline{b}_{[2, s]}^{\alpha}) \times \Phi_1 \,) \, \sqcup \, [\ell_s, r_0)\times  RM_{\underline{b}_{[2, s]}^{\alpha}}R^{-1}(\, \Phi'_2\setminus \Phi_1\,)
 \end{aligned}
 \]
equals  
$[\mathscr L, r_0)  \times -(  [\ell_0, \mu_{-k}] \cup [\lambda_{-k-1}, \mu_{-k-1}])$. 
\end{Lem}

\begin{proof}   Since $\ell_1 \in  [\lambda_w, \rho_w)$ exactly for $w$ of the form $\underline{b}_{[2, s]}^{\alpha}$,  and $\Phi'_2 \supset \Phi_1$, we can rewrite the displayed quantity as 
\[\begin{aligned} 
&\bigsqcup \, \sqcup_{w\in \mathscr A}\;  \mathcal T^{|w|}(\, \Delta_{\alpha}(w) \times \Phi_1 \,)\\ \\
& \; \bigsqcup \, \sqcup_{\substack{w\in \mathscr A  \\ \lambda_w > \ell_1}}\;  \mathcal T^{|w|}(\, \Delta_{\alpha}(w) \times (\Phi'_2\setminus \Phi_1) \,) \,\bigsqcup \, \sqcup_{s = 2}^{\infty}\,  \, [\ell_s, r_0)\times  RM_{\underline{b}_{[2, s]}^{\alpha}}R^{-1}(\, \Phi'_2\setminus \Phi_1\,).
 \end{aligned}
 \]
By Lemma~\ref{l:goodRelationWwTilde}, this equals 
\[\begin{aligned} 
&\bigsqcup \, \sqcup_{w\in \mathscr A}\;  \mathcal T^{|w|}(\, \Delta_{\alpha}(w) \times \Phi_1 \,)
 \, \bigsqcup \, \sqcup_{\substack{w\in \mathscr A\\w_{[-1]} = -k\\ \text{or}\, w=-k-1}}\;  \mathcal T^{|\widetilde{w}|}(\, \Delta_{\alpha}(\widetilde{w}) \times (\Phi'_2\setminus \Phi_1) \,),
 \end{aligned}
 \] 
 which we can of course rewrite as 
\[\begin{aligned} 
&\bigsqcup \, \sqcup_{w\in \mathscr A}\;   T^{|w|}(\, \Delta_{\alpha}(w)\,)  \times R M_w R^{-1} (\Phi_1) \,)
 \, \bigsqcup \, \sqcup_{\substack{w\in \mathscr A\\w_{[-1]} = -k\\ \text{or}\, w=-k-1}}\;  T^{|\widetilde{w}|}(\, \Delta_{\alpha}(\widetilde{w})\,)  \times RM_{\widetilde{w}}R^{-1}(\,\Phi'_2\setminus \Phi_1 \,),
 \end{aligned}
 \] 
\bigskip 
 
  Since the image of the left endpoint of every cylinder lies to the left of $\mathscr L$,  we find that the $y$-fiber in $\Omega$ above any $x \in [\mathscr L, r_0)$ includes   
\[\begin{aligned} 
\mathcal F_{x}^{+}&= \sqcup_{\substack{w\in \mathscr A\\w_{[-1]} = -k\\ \text{or}\, w=-k-1}}\;  RM_{\widetilde{w}}R^{-1}(\,\Phi'_2\setminus \Phi_1 \,),
 \end{aligned}
 \]   
 
For all $w$ and all $i$, we have  $\nu_{w, (-k-1)^i} = \nu_w$, and thus  applying Lemmas \ref{l:extendToFixedPoint}  and \ref{l:givesFiberPieceForW}  and then Lemma~\ref{l:coversInterval} shows that 
$\mathcal F_{x}^{+} = -(  [\ell_0, \mu_{-k}] \cup [\lambda_{-k-1}, \mu_{-k-1}])$.
\end{proof}

 \subsubsection{Proof that fibers agree: filling in gap}  It remains to show that for $x \in [\ell_1, \mathscr L)$ that $\Phi_2\setminus\Phi'_2$ is contained in the $y$-fiber of $\Omega$ over $x$.   Since $R M_{-k-1} R^{-1} (\Phi_3) =  [-\mu_{-k-1}, -\lambda_{-k-1}]$ and of course $\lambda_{-k-1} = \rho_{-k}$ we must simply show that $(A^{-k-1}C)^{-1} \cdot \ell_1 \ge \mathscr L$.   
 
 Just as Lemma~\ref{l:largestLvalue} follows from Lemma~\ref{l:2ndLargestR}, so also for any $\alpha = \eta_{-k,v}$ we have that $\ell_{\underline{S}}$ is maximal by arguing analogously to Lemma~\ref{l:lastRValueIsLeast}.   For such $\alpha$, we also have  $A^{-k-1}C\cdot \ell_{\underline{S}} = \ell_1$, and thus our results holds for these values of  $\alpha$.   By taking limits,  we find that it holds for all of the values of $\alpha$ which we are considering.  \qed

 \subsection{Ergodic natural extension for non-synchronizing large $\alpha$} 
 
\begin{Prop}\label{p:finMeasLargeAlp}     Fix $n$ and a non-synchronizing $\alpha$ with $\alpha >\gamma_{n}$.  
 With $\Omega$ as in Proposition~\ref{p:fillUpFromW},  $\mu(\Omega)< \infty$.
\end{Prop}
\begin{proof} As in the proof of Proposition~\ref{p:finMeas}, we can show the finiteness of the measure of the upper portion of $\Omega$ by comparing with $\Omega_{n,0}$.   Similarly, here one can show the finiteness of the measure of the lower portion of $\Omega$ by comparing with $\Omega_{n,1}$, see (\cite{CKStoolsOfTheTrade},  Proposition~2.4).  
See also (\cite{CKStoolsOfTheTrade},  Figure~2).   
\end{proof} 

 The following is proven as is Proposition~\ref{p:bijectivityDomIsNatExtErgSmallNonSyn},  {\it mutatis mutandis}.
\begin{Prop}\label{p:biDomIsNatExtErgLargeNonSyn}       With $\alpha$ as in the notation of  Proposition~\ref{p:fillUpFromW}, 
the system $(\mathcal T_{\alpha}, \Omega_{\alpha}, \mathscr B'_{\alpha}, \mu_{\alpha})$ is the natural extension of  $(T_{\alpha},\mathbb I_{\alpha},  \mathscr B_{\alpha}, \nu_{\alpha})$, where $\nu_{\alpha}$ is the marginal measure of $\mu_{\alpha}$ and $\mathscr B_{\alpha}$  the Borel sigma algebra on $\mathbb I_{\alpha}$.  Finally,   both systems are ergodic. 
\end{Prop}

\section{Continuity of mass for larger $\alpha$; Completing proof of Theorem~\ref{t:continuityOfMass}}\label{s:conLarge}  
 In view of Theorem~\ref{t:continuity}, the results of this section complete the proof of Theorem~\ref{t:continuityOfMass}.

\begin{Thm}\label{t:continuityBigAlps}     The function  $\alpha \mapsto  \mu(\Omega_{\alpha})$ is continuous on $(\gamma_{n},1)$.
\end{Thm}

\begin{proof}  Since the both top and bottom heights are constant along any synchronization interval $J_{-k,v}$, with fibers constant between the various $\ell_i(\alpha), r_j(\alpha)$, continuity and indeed smoothness,  of $\alpha \mapsto  \mu(\Omega_{\alpha})$ in each of $(\eta_{-k,v}, \delta_{-k,v})$ and $(\delta_{-k,v}, \eta_{-k,v})$ is clear.   Propositions~\ref{p:OmegaForEtaKneg}  and ~\ref{p:OmegaForZetaKneg} along with Subsection~\ref{ss:smackDab}  then imply that  continuity holds on the closed interval.  

Continuity throughout    $I_{-k} := \cup_{v\in \mathcal V}\, I_{k,v}$ for fixed $k$ is shown as in the proof of Theorem~\ref{t:continuity}, using that we can express every $\Omega_{\eta_{-k,v}}$ and $\Omega_{\zeta_{-k,v}}$ as the union of regions of the type of $\mathcal W$ similar to the result of Proposition~\ref{p:fillUpFromW}.   

Continuity at the boundaries of the various $I_{-k}$   follows from arguing analogously to the proof of Theorem~\ref{t:continuity}.    Indeed, the right endpoint of $I_{-k}$ is $\zeta_{-k,1}$, and one argues that $\Omega_{\zeta_{-k-1,h}} \to \Omega_{\zeta_{-k,1}}$ as $h \to \infty$.  For this,  (\cite{CaltaKraaikampSchmidt}, Lemma~6.6) gives $\underline{d}{}^{\zeta_{-k,1}}_{[1,\infty)}  = \overline{-k-1}$ while  $\underline{d}{}^{\zeta_{-k-1,h}}_{[1,\infty)}  = (-k)^{h-1}, \overline{-k-1}$ (in simplified digits) from which also the $\overline{d}{}^{\zeta_{-k,1}}_{[1,\infty)}$ and   $\overline{d}{}^{\zeta_{-k-1,h}}_{[1,\infty)}$  sufficiently agree to deduce both the convergence of the  $\Omega_{\zeta_{-k-1,h}}^{+}$ to $\Omega_{\zeta_{-k,1}}^{+}$ and similarly for the lower portions. 
\end{proof}

 There is only one remaining $\alpha$ value at which we must prove continuity. 
\begin{Lem}\label{l:atGammaToo}   Fix $n$.  The function $\alpha \mapsto \mu(\, \Omega_{\alpha}\,)$ is continuous at $\alpha = \gamma_n$.
\end{Lem}

\begin{proof} Recall that $\gamma_n = \eta_{-1, n-2}$, thus we have already shown continuity from the right.  

Since $\eta_{1,h}$ converges from the left to $\gamma_n$, by Theorem~\ref{t:continuity}  it suffices to show that $\mu(\Omega_{\eta_{1,h}}) \to \mu(\Omega_{\eta_{-1,n-2}})$ as $h \to \infty$.   By Lemma~\ref{l:fillUpWithZimagesZetaEta},  for each $h$   we have that $\Omega_{\eta_{1,h}}$ is swept out by the $\mathcal T_{\eta_{1,h}}$-orbit of $\mathcal Z_{\eta_{1,h}}$.  Furthermore, a version of \eqref{e:jtUnion}   holds for these values as well.   That is,   $\Omega_{\eta_{1,h}}$ is the union of the image of $\mathcal Z_{\eta_{1,h}}$ with pieces of the form $\mathcal T^{|u|}(\, \Delta_{\alpha}(u) \times \Phi^{\pm}\,)$, with signs depending upon whether $u \in \{-1, -2\}^*$ or $u \in \{2,1\}^*$ (in simplified digits).       

Now, the only admissible words in $\mathcal L_{\eta_{-1,n-2}}$ of digit not of the form  $(a, 1)$ are prefixes of  $\overline{b}{}^{\eta_{-1,n-2}}_{[1,\infty)}$.  Thus we can temporarily delete these from the language and use simplified digits.   One can then extend   the definition of $\mathcal Z_{\alpha}$  in Definition~\ref{d:notationForEndpoints}    to include $\alpha = \eta_{-1,n-2}$, and easily find that (upon taking closures) here also a partition  of the type described in \eqref{e:jtUnion}   holds for $\alpha = \eta_{-1,n-2}$.   Therefore, we can once again argue as in Theorem~\ref{t:continuity} and conclude that continuity does hold here.
\end{proof}
     
\section{Ergodicity and natural extensions;  Proof of Theorem~\ref{t:naturallyErgodicParTout}}\label{s:dynamProp}   
In \cite{CKStoolsOfTheTrade}, we gave theoretical tools which can be applied to show main dynamical properties of systems of the type studied here.   (Indeed, we placed several theoretical tools in that paper so as to retain the brevity of this section.)   There, the  ``bounded non-full range" condition is a key hypothesis to prove ergodicity of a planar system associated to interval maps such as our own, as well as to show that the planar system is the natural extension.   See Remark~\ref{rmk:bddNonfullCylinders} for a reminder of this condition.  As recalled in the proof of Theorem~\ref{t:naturallyErgodicParTout}, see page~\pageref{pf:TheoremOne} below, we have already proven the theorem in various cases.  This was done invoking (\cite{CKStoolsOfTheTrade}, Theorem~2.3).  In \S~\ref{ss:bddNonFullRangeValDense}, we show the  denseness  of the set of parameter values indexing maps which satisfy the bounded non-full range condition.  This dense set hence parametrizes maps whose planar natural extensions are ergodic.

A second result in \cite{CKStoolsOfTheTrade} is that the properties of ergodicity  and being a natural extension of a planar system are induced onto sufficiently nearby systems by way of our quilting.
In  \S~\ref{ss:closeNeighbors}   establish that every $\alpha$ for which the bounded non-full range condition fails is contained in an open neighborhood of parameter values  all of which are sufficiently nearby so that ergodicity  and being a natural extension of a planar system could  be 
shared with its system by way of quilting.    The aforementioned denseness result than allows the section to end with a completion of the proof of Theorem~\ref{t:naturallyErgodicParTout}.

We begin the section by proving that all of our maps are eventually expansive.   

\subsection{Eventual expansivity; Proof of Proposition~\ref{p:expansive}}     
 We show eventual expansiveness for all of our maps follows from results of (\cite{CKStoolsOfTheTrade}.    
  

\begin{proof}[{\bf Proof of Proposition~\ref{p:expansive}}]  Fix $n, \alpha$.  We must show that the hypotheses (a)--(c) of (\cite{CKStoolsOfTheTrade}, Proposition~2.3), listed in the statement of (\cite{CKStoolsOfTheTrade}, Theorem~2.3), are fulfilled.    
We have already established that for any $n, \alpha$ the domain $\Omega_{n, \alpha}$ has all  vertical fibers being intervals of finite length.  In the case of $\alpha \notin \{0,1\}$ we have also established that these fibers are a bounded distance from the locus of $y =- 1/x$.   Finally,   we have also showed that every block of  $\Omega_{n, \alpha}$  is such that its vertical fibers are intervals such that the ratio  of the Lebesgue measure  of the image  by  $\mathcal T_{n, \alpha}$-image to that of the receiving vertical fiber is bounded away from zero and one.    That is, the proposition holds for all of our maps. 
\end{proof}

\subsection{Close neighbors: shared properties and related entropy}\label{ss:closeNeighbors} 
  
 We aim to invoke a result from (\cite{CKStoolsOfTheTrade},  which depends on the use of ``matching exponents" (there denoted $m,n$ --- here we will use $e_\ell, e_r$).    In our setting, if $\alpha < \gamma_n$ is in some synchronization interval $J_{k,v}$ then we let $e_\ell = \underline{S}+1$ and $e_r = \overline{S}+1$;  if  $\alpha > \gamma_n$ is in the left hand portion of some synchronization interval $J_{-k,v}$  then let $e_\ell = \underline{S}+1$ and $e_r = \overline{S}+1$; finally, if $\alpha > \gamma_n$ is in the left hand portion of some synchronization interval $J_{-k,v}$  then let $e_\ell = \underline{S}+1$ and $e_r = \overline{S}+2$.   Then equations \eqref {e:synchronizationExplicit}, \eqref{e:synchroExplicitLargeAlps} show that for all such $\alpha$, 
 \[ T_{\alpha}^{e_\ell}(\, \ell_0(\alpha) \,) = T_{\alpha}^{e_r}(\, r_0(\alpha) \,).\]  
Furthermore,  in \cite{CaltaKraaikampSchmidt}, we show that for all but finitely many $\alpha$ in any synchronization  interval that these exponents are appropriately minimal for such an equality to hold.  
 
Now consider $n \ge 3$  fixed and suppose that $\alpha' \neq\alpha$.  
Denote the orbits of the respective endpoints of the intervals of definition by $\ell'_i, r'_i$ for each $\ell_i(\alpha'), r_i(\alpha')$, and similarly simply $\ell_i, r_i$ for each $\ell_i(\alpha), r_i(\alpha)$, respectively.     For the ease of the reader we 
 include the following direct translation of (\cite{CKStoolsOfTheTrade},  Definition~3.8).  

 \bigskip
 
\begin{Def} 
  \label{d: tightSyn} Suppose that $J$ is a matching interval with corresponding matching exponents $e_\ell, e_r$.  We say that $\alpha, \alpha' \in J$  are {\em close neighbors} if
  $\ell'_i, \ell_i, r'_j, r_j \in  \mathbb I_{\alpha'} \cap \mathbb I_{\alpha}$ for all  $1\le i   \le  e_\ell$ and $1 \le j   \le e_r$.  
\end{Def}

Recall from \S~\ref{ss:notation} that $b_{\alpha}(x)$ gives the $\alpha$-digit of $x$.  We let 
 \[  \mathcal C   = \{ \, (x,y) \in \Omega_{\alpha} \mid  x \in \mathbb I_{\alpha} \cap \mathbb I_{\alpha'} \,,\, b_{\alpha}(x) \neq \, b_{\alpha'}(x)\, \}\,.   
  \]
  
   Compare the following with Figure~\ref{f:smallAlpQuilt}.  
\begin{Prop}\label{p:quiltCloseNeighsSmall}  Fix $n$.   Suppose that $\alpha, \alpha'$, with $\alpha' < \alpha< \gamma_n$, are close neighbors. 
  Then both of  $\Omega_{\alpha'}, \Omega_{\alpha}$  can be finitely quilted from each other.  In particular, 
\begin{equation}\label{e:peterPaulWorks}
\Omega_{\alpha'} = \bigg(\, \Omega_{\alpha} \setminus \coprod_{i=1}^{\overline{S}+1}\, \mathcal T_{\alpha}^i(\,\mathcal C\,)\,\bigg) \amalg\; \coprod_{j=1}^{\underline{S}+1}\,\mathcal T_{\alpha'}^j(\,\mathcal C\,).
\end{equation}
  Furthermore,  if $\mathcal T_{\alpha}$ is the natural extension of $T_{\alpha}$, then the entropy of $T_{\alpha'}$ is 
 \[ h(T_{\alpha'}) = \bigg(1 + (\underline{S} - \overline{S})  \;\nu_{\alpha}(\,[r'_0, r_0]\,) \bigg)^{-1} h(T_{\alpha}), \]
where $\nu_{\alpha}$ is $T_{\alpha}$-invariant probability measure induced from $\mu$ on $\Omega_{\alpha}$. 
Finally, since we always have $\underline{S} > \overline{S}$,  the entropy is a monotonically increasing continuous function in $\alpha'$ (for fixed $\alpha$). 
\end{Prop}
 
\begin{figure}[h]
\scalebox{.75}{
\begin{tikzpicture}[x=5cm,y=5cm] 
\draw  (-1.72, -0.4)--(-0.55, -0.4); 
\draw  (-0.55, -0.4)--(-0.55, -0.2); 
\draw  (-0.55, -0.2)--(0.28, -0.2); 
\draw  (0.28,  1.79)--(-0.25,  1.79); 
\draw  (-0.25,  1.79)--(-0.25,  1.64); 
\draw  (-0.25,  1.64)--(-0.44,  1.64); 
\draw  (-0.44,  1.64)--(-0.44,  0.74); 
\draw  (-0.44,  0.74)--(-0.74,  0.74);
\draw  (-0.74,  0.74)--(-0.74,  0.44);
\draw  (-0.74,  0.44)--(-1.64,  0.44);
\draw  (-1.64,  0.44)--(-1.64,  0.39);   
\draw  (-1.64,  0.39)--(-1.72,  0.39);  
\draw  (-1.72,  0.39)--(-1.72,  -0.4); 
\draw  (0.0, -0.2)--(0.0, 1.79);         
\draw[thin,dashed] ( -0.84, -0.4)--(-0.84, 0.44);        
\draw[thin,dashed] (-0.32, 1.64)--(-0.32, -0.2); 
\draw[thin,dashed] (-0.19, 1.79)--(-0.19, -0.2); 
\node at (-1, 0) {$-1$};  
\node at (-0.55, 0) {$-2$};      
\node at (-0.26, 0) {$-3$};      
\node at (-0.1, 0) {$\cdots$};  
\node at (0.13, 0) {$\cdots$}; 
 \node at (-1.7, -0.5) {$(\ell_0, y_{-2})$}; 
 \node at (-1.95,  0.40) {$(\ell'_0, y_1)$}; 
 \node at ( -1.5, 0.55) {$(\ell_3, y_2)$}; 
 \node at (-0.82, 0.82) {$(\ell'_2, y_3)$};
 \node at (-0.6, 1.65)  {$(\ell_1, y_4)$}; 
 \node at (-0.3, 1.9)  {$(\ell'_4, y_5)$}; 
 \node at (0.5, -0.2) {$(r_0, y_{-1})$}; 
 \node at (0.35, 1.9) {$(r'_0, y_5)$}; 
 \node at (-0.5, -0.5)  {$(r_1, y_{-2})$}; 
 \node at (0, -0.2)  {$0$}; 
 \foreach \x/\y in {-1.72/-0.4,-0.55/-0.4, 
 0.28/-0.2,  0.24 /1.79, -0.28/1.79, -0.44/1.64, -0.77/0.74,
-1.64/0.44, -1.76/0.39%
} { \node at (\x,\y) {$\bullet$}; } 
\draw[pattern=north east lines, pattern color=red]    (0.24, -0.2)--(0.24, 1.79) -- (0.27, 1.79) --(0.27, -0.2) -- cycle; 
\draw[pattern=north east lines, pattern color=red]     (-0.55, -0.4)--(-0.55, -0.2) --  (-0.59, -0.2)--(-0.59, -0.4) -- cycle; 
\draw[pattern=north west lines, pattern color=blue]    (-1.72, -0.4)--(-1.76, -0.4) -- (-1.76, 0.39) --(-1.72, 0.39) -- cycle; 
\draw[pattern=north west lines, pattern color=blue]   (-1.64,  0.44)--(-1.64,  0.39) -- (-1.67,  0.39)--(-1.67,  0.44) -- cycle; 
\draw[pattern=north west lines, pattern color=blue]    (-0.74,  0.74)--(-0.74,  0.44) -- (-0.77,  0.44)--(-0.77,  0.74) -- cycle; 
\draw[pattern=north west lines, pattern color=blue]     (-0.44,  1.64)--(-0.44,  0.74) --  (-0.47,  0.74)--(-0.47,  1.64) -- cycle;  
\draw[pattern=north west lines, pattern color=blue]    (-0.25,  1.79)--(-0.25,  1.64) --  (-0.28,  1.64)--(-0.28,  1.79) -- cycle;  
\draw[pattern=north west lines, pattern color=blue]    (-1.0, 0.35 )--(-1.0, 0.3 ) --  (-1.1, 0.3 )--(-1.1, 0.35 ) -- cycle; 
\draw[pattern=north east lines, pattern color=red]    (-1.0, 0.35 )--(-1.0, 0.3 ) --  (-1.1, 0.3 )--(-1.1, 0.35 ) -- cycle; 
\node at (0.27,  0.8)[pin={[pin edge=<-, pin distance=12pt]0:{$\mathcal D := \mathcal T_{\alpha}(\mathcal C)$}}] {};
\node at (-0.55,  -0.3)[pin={[pin edge=<-, pin distance=12pt]0:{$\mathcal T_{\alpha}^{\,2}(\mathcal C)$}}] {};
\node at (-1.76,  0)[pin={[pin edge=<-, pin distance=12pt]180:{$ \mathcal T_{A^{-1}}(\mathcal D) = \mathcal T_{\alpha'}(\mathcal C)$}}] {};
\node at (-1.67,  0.42)[pin={[pin edge=<-, pin distance=60pt]135:{$\mathcal T_{\alpha'}^{\,2}(\mathcal C)$}}] {};
\node at (-0.77,  0.58)[pin={[pin edge=<-, pin distance=80pt]160:{$\mathcal T^{\,3}_{\alpha'}(\mathcal C)$}}] {};
\node at (-0.47,  1.2)[pin={[pin edge=<-, pin distance=60pt]180:{$\mathcal T^{\,4}_{\alpha'}(\mathcal C)$}}] {};
\node at (-0.3,  1.74)[pin={[pin edge=<-, pin distance=80pt]180:{$\mathcal T^{\,5}_{\alpha'}(\mathcal C)$}}] {};
\node at (-1.1,  0.3)[pin={[pin edge=<-, pin distance=5pt]250:{$\mathcal T^{\,6}_{\alpha'}(\mathcal C) = \mathcal T^{\,3}_{\alpha}(\mathcal C)$}}] {};
\end{tikzpicture} 
}
\caption{{\bf Quilting for close neighbors.}  Quilting from  the  $\Omega_{3, 0.14}$ to $\Omega_{3, 0.135}$, with blocks $\mathcal B_i, i = -1, -2, -3$ indicated (not fully to scale); compare with Figure~\ref{f:omegaSmallAlp}.  The forward $\mathcal T_{\alpha}$-orbit of $\mathcal C$ is deleted, while the forward  $\mathcal T_{\alpha'}$-orbit of $\mathcal C$ is added, until the ``hole" created by excising  $\mathcal T_{\alpha}^{\,3}(\mathcal C)$ is ``patched" in by  $\mathcal T_{\alpha'}^{\,6}(\mathcal C)$.   See Proposition~\ref{p:quiltCloseNeighsSmall}. }%
\label{f:smallAlpQuilt}%
\end{figure}
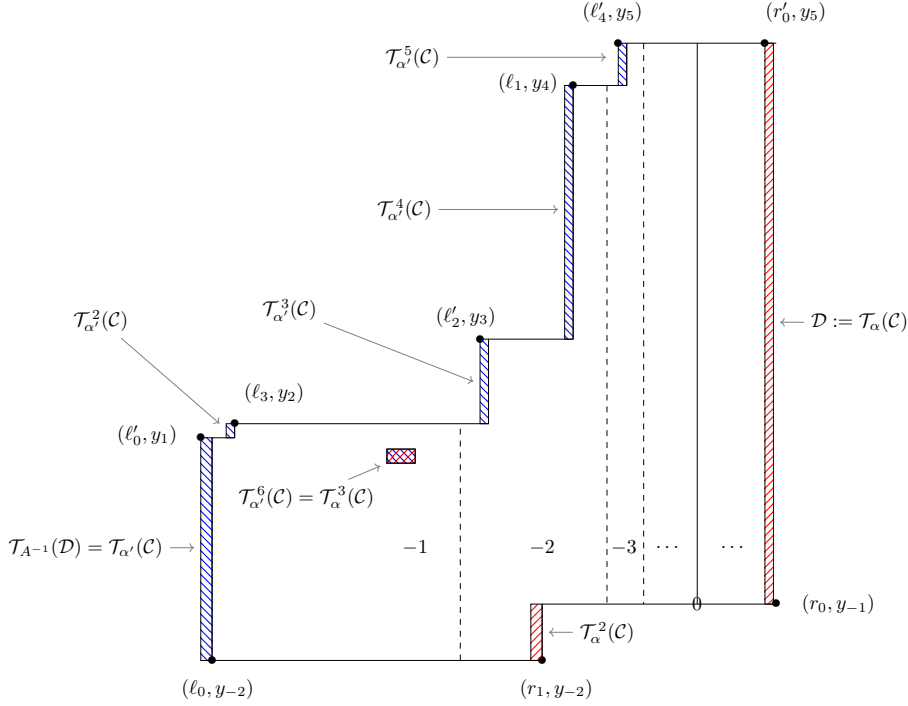

\begin{proof}    This follows from combining (\cite{CKStoolsOfTheTrade}, Proposition~3.10)
with (\cite{CKStoolsOfTheTrade}, Theorem~3.3 
and its associated propositions). 
\end{proof}

\begin{figure}[h]
\scalebox{.7}{
\begin{tikzpicture}[x=6cm,y=6cm] 
\draw  (-.28, -0.72)--(0.6, -0.72); 
\draw  (0.6, -0.72)--(0.6,  -0.44); 
\draw  (0.6, -0.44)--(1.1,  -0.44);   
\draw  (1.1, -0.44)--(1.1,  -0.39);   
\draw  (1.1, -0.39)--(1.72,  -0.39);   
\draw  (1.72, -0.39)--(1.72,  0.35); 
\draw  (1.72, 0.35)--(0.25,  0.35); 
\draw  (0.25, 0.35)--(0.25,  0.2);  
\draw  (0.25, 0.2)--(-.28,  0.2);  
\draw  (-.28, 0.2)--(-.28,  -0.72);
 \draw  (0.0, -0.72)--(0.0, .2); 
 \draw  (1, -0.44)--(1, .35); 
\draw[thin,dashed] (-0.17, -0.72)--(-0.17, 0.2);    
\draw[thin,dashed] (0.19, -0.72)--(0.19, 0.2); 
\draw[thin,dashed] (0.33, -0.72)--(0.33, 0.35);      
\draw[thin,dashed] ( 0.78, -0.44)--( 0.78, 0.35);  
\draw[thin,dashed] (0.92, -0.44)--(0.92, 0.35);  
\draw[thin,dashed] (1.2, -0.39)--(1.2, 0.35);  
\draw[thin,dashed] (1.4, -0.39)--(1.4, 0.35);  
\draw[thin,dashed] (1.08, -0.44)--(1.08, 0.35);  
\node at (-.23, 0)[pin={[pin edge=<-, pin distance=20pt]300:{$\mathcal B_{-2,1}$}}] {};
\node at (0.55, -0.2) {\tiny{$(1,1)$}};  
\node at (0.25, 0) {\tiny{$(2,1)$}}; 
\node at (0.85, 0.12) {\tiny{$(-2,2)$}};      
\node at (0.10, 0) {\tiny{$\cdots$}}; 
\node at (-0.1, 0) {\tiny{$\cdots$}};   
\node at (0.96, -0.01) {\tiny{$\cdots$}}; 
\node at (1.05, -0.01) {\tiny{$\cdots$}}; 
\node at (1.14, 0) {\tiny{$(3,2)$}};  
\node at (1.3, 0) {\tiny{$(2,2)$}};   
\node at (1.55,  -0.2) {\tiny{$(1,2)$}};   
\node at (-.3, -0.8) {$(\ell_0, y_{-3})$}; 
\node at (-.45, 0.2) {$(\ell'_0, y_1)$}; 
\node at (0.3,  0.45)  {$(\ell_1, y_2)$};  
\node at (0.62, -0.8)  {$(r_1, y_{-3})$}; 
\node at (1.2, -0.5)  {$(r_2, y_{-2})$}; 
\node at (1.7, -0.5)  {$(r'_0, y_{-1})$};  
\node at (1.75,0.45)  {$(r_0, y_{2})$};   
\node at (0,0)  {$0$};  
\node at (1, -0.5)  {$1$};  
 \foreach \x/\y in {-.28/-0.72, 0.6/-0.72, 1.1/-0.44, 
 1.68/-0.39,  1.72/0.35, 0.25/0.35, -.3/0.2%
} { \node at (\x,\y) {$\bullet$}; } 
\draw[pattern=north west lines, pattern color=orange]    (1.68, -0.3)--(1.68, 0.35) -- (1.72, 0.35)--(1.72,-0.3)-- cycle; 
\draw[pattern=north east lines, pattern color=red]    (1.68, -0.39)--(1.68, -0.3) -- (1.72, -0.3)--(1.72,-0.39)-- cycle; 
\draw[pattern=north west lines, pattern color=orange]     (0.6, -0.65)--(0.6,  -0.44) --  (0.55, -0.44)--(0.55,  -0.65)-- cycle; 
\draw[pattern=north east lines, pattern color=red]          (0.6, -0.72)--(0.6,  -0.65) --  (0.55, -0.65)--(0.55,  -0.72)-- cycle; 
\draw [decorate, ultra thick,
    decoration = {calligraphic brace,mirror,amplitude=6pt}] (0.63,-0.73) --  (0.63,-0.45);
\node at (0.8, -0.6){$\mathcal T_{\alpha}^{\,2}(\mathcal C\,)$};    
\node at (1.9, 0.1){$\mathcal D_1 = $};
\node at (1.73, 0)[pin={[pin edge=<-, pin distance=12pt]0:{$\mathcal T_{\alpha}(\mathcal C_1\,)$}}] {};
\node at (1.73,  -0.35)[pin={[pin edge=<-, pin distance=12pt]0:{$\mathcal T_{\alpha}(\mathcal C_2\,)$}}] {};
\draw[pattern=north west lines, pattern color=blue]    (-0.28, -0.72)--(-0.28, 0.2) -- (-0.3,0.2)--(-0.3,-0.72)-- cycle; 
\draw[pattern=north east lines, pattern color=green]    (0.25,0.2)--(0.22, 0.2) -- (0.22, 0.28)--(0.25,0.28)-- cycle; 
\draw[pattern=north west lines, pattern color=blue]    (0.25,0.35)--(0.22, 0.35) -- (0.22, 0.28)--(0.25,0.28)-- cycle; 
\node at (-.3,  -0.3)[pin={[pin edge=<-, pin distance=12pt]180:{$\mathcal T_{\alpha'}(\mathcal C_1)$}}] {};
\node at (-.55, -0.2){$\mathcal T_{A^{-1}}(\mathcal D_1) = $};
\node at (0.22,  0.22)[pin={[pin edge=<-, pin distance=40pt]165:{$\mathcal T_{\alpha'}(\mathcal C_2\,)$}}] {};
\node at (0.22,  0.32)[pin={[pin edge=<-, pin distance=40pt]130:{$\mathcal T_{\alpha'}^{\,2}(\mathcal C_1\,)$}}] {};
\fill [opacity= 0.15, gray]  (0.72,0.35)--(0.78, 0.35) -- (0.78, -0.44)--(0.72,-0.44)-- cycle; 
\node at (.75, 0.38)[pin={[pin edge=<-, pin distance=12pt]90:{$\mathcal C_2$}}] {};
\end{tikzpicture}  
}
\caption{ {\bf Quilting for close neighbors, large $\alpha$.}  Quilting from  $\Omega_{3, 0.86}$ to $\Omega_{3, 0.855}$ (not fully to scale); compare with Figure~\ref{f:omegaLargeAlpBiggerThanDelta_{-k,v}}. Blocks $\mathcal B_{i,j}$, also denoted by $(i,j)$.   The forward $\mathcal T_{\alpha}$-orbit of $\mathcal C  = \mathcal C_1 \cup \mathcal C_2$ is deleted, while the forward  $\mathcal T_{\alpha'}$-orbit of $\mathcal C$ is added, until synchronization causes a ``hole" excised due to the first of these, but to be ``patched" due to the second (not shown here, but compare with Figure~\ref{f:smallAlpQuilt}\,).   See Proposition~\ref{p:quiltCloseNeighsLargeLeft}. }%
\label{f:quiltLargeAlpLessThanDelta}%
\end{figure}

   Given large $\alpha, \alpha'$,  their respective digits of an $x \in \mathbb I_{\alpha'}\cap \mathbb I_{\alpha}$ can differ in two basic ways.  
\begin{Def}\label{d:changeDigBigAlp}   Fix $n$ and suppose $\alpha>\alpha'>\gamma_n$ are as above.    
We define  the two sets 
 \[
 \begin{aligned} 
  \mathcal C_1 &= \{ \, (x,y) \in \Omega_{\alpha} \mid  x \in \mathbb I_{\alpha} \cap \mathbb I_{\alpha'} \,,\, b_{\alpha}(x) = (i,j), \, b_{\alpha'}(x) = (i',j)\, \text{with}\,   i\neq i' \, \}\,, \\
  \mathcal C_2 &= \{ \, (x,y) \in \Omega_{\alpha} \mid  x \in \mathbb I_{\alpha} \cap \mathbb I_{\alpha'} \,,\,  b_{\alpha}(x) = (i,j), \, b_{\alpha'}(x) = (i',j')\, \text{with}\,   j\neq j'   \, \}.
  \end{aligned}
  \]
\end{Def} 
 
 Compare the following with Figure~\ref{f:quiltLargeAlpLessThanDelta}.
\begin{Prop}\label{p:quiltCloseNeighsLargeLeft}  Suppose that  $\alpha, \alpha'$ are close neighbors in the  synchronization interval $J_{-k,v}$ for some $k,v$, with $\alpha'<\alpha< \delta_{-k,v}$.     Then both of  $\Omega_{\alpha'}, \Omega_{\alpha}$  can be finitely quilted from each other.  In particular,
\begin{equation}\label{e:peterPaulWorksLarge}
\Omega_{\alpha'} = \bigg(\, \Omega_{\alpha} \setminus \coprod_{i=1}^{\overline{S}+1}\, \mathcal T_{\alpha}^i(\,\mathcal C\,)\,\bigg) \amalg\; \coprod_{j=1}^{\underline{S}+1}\,\mathcal T_{\alpha'}^j(\, \mathcal C_1\,)\, \amalg\; \coprod_{j=1}^{\underline{S}}\,\mathcal T_{\alpha'}^j(\,\mathcal C_2\,).
\end{equation}

 Furthermore,  if $\mathcal T_{\alpha}$ is the natural extension of $T_{\alpha}$, then entropy of $T_{\alpha'}$ is 
  \[ h(T_{\alpha'}) = \bigg(1 + (\underline{S} - \overline{S})  \;\nu_{\alpha}(\,[r'_0, r_0]\,) - \nu_{\alpha}(\,[\mathfrak b_{\alpha'}, \mathfrak b_{\alpha}]\,) \bigg)^{-1} h(T_{\alpha}), \]
where $\nu_{\alpha}$ is $T_{\alpha}$-invariant probability measure induced from $\mu$ on $\Omega_{\alpha}$. Finally,   the entropy is a monotonically decreasing  continuous function in $\alpha'$ (for fixed $\alpha$). 
\end{Prop}

\begin{proof}  This also follows from combining (\cite{CKStoolsOfTheTrade}, Proposition~3.10) 
with (\cite{CKStoolsOfTheTrade}, Theorem~3.3 
and its associated propositions). 
\end{proof} 

  We now treat the case of $\alpha, \alpha'$ close neighbors in the right portion of some $J_{-k,v}$.   We have $\mathcal C_1, \mathcal C_2$ as above. 

 \begin{Prop}\label{p:quiltCloseNeighsLargeRight}  Suppose that  $\alpha, \alpha'$ are close neighbors in the  synchronization interval $J_{-k,v}$ for some $k,v$, with $\delta_{-k,v}< \alpha'<\alpha$.     Then both of  $\Omega_{\alpha'}, \Omega_{\alpha}$  can be finitely quilted from each other.  In particular,
\[
\Omega_{\alpha'} = \bigg(\, \Omega_{\alpha} \setminus \coprod_{i=1}^{\overline{S}+2}\, \mathcal T_{\alpha}^i(\,\mathcal C\,)\,\bigg) \amalg\; \coprod_{j=1}^{\underline{S}+1}\,\mathcal T_{\alpha'}^j(\, \mathcal C_1\,)\, \amalg\; \coprod_{j=1}^{\underline{S}}\,\mathcal T_{\alpha'}^j(\,\mathcal C_2\,).
\]

 Furthermore,  if $\mathcal T_{\alpha}$ is the natural extension of $T_{\alpha}$, then entropy of $T_{\alpha'}$ is 
  \[ h(T_{\alpha'}) = \bigg(1 + (1+\underline{S} - \overline{S})  \;\nu_{\alpha}(\,[r'_0, r_0]\,) - \nu_{\alpha}(\,[\mathfrak b_{\alpha'}, \mathfrak b_{\alpha}]\,) \bigg)^{-1} h(T_{\alpha}), \]
where $\nu_{\alpha}$ is $T_{\alpha}$-invariant probability measure induced from $\mu$ on $\Omega_{\alpha}$. Finally,   the entropy is a monotonically decreasing  continuous function in $\alpha'$ (for fixed $\alpha$). 
 \end{Prop}

\begin{proof} 
This again follows from combining (\cite{CKStoolsOfTheTrade}, Proposition~3.10) 
with (\cite{CKStoolsOfTheTrade}, Theorem~3.3 
and its associated propositions). 
\end{proof}
 
\subsection{Bounded non-full range values are dense}\label{ss:bddNonFullRangeValDense}   

\begin{figure}
\scalebox{.65}{
\begin{tikzpicture}[x=7cm,y=7cm] 
 \draw[domain= -0.25:1.2, smooth,  variable=\x, thick, dashed ] plot ({\x}, {(3) *(-1/(\x+3)+0.625)});
 \node at (1.25, 1.25) {$y= \rho_{a, b}(\alpha)$};
 \draw[domain= -0.2:1.2, smooth,  variable=\x, thick] plot ({\x}, {-1/(\x+1)+1.3});
  \node at (-0.3, 0.2) {$y= \lambda_{a, b}(\alpha)$};
 \draw[domain= 0:1, smooth, red, variable=\x, thick] plot ({\x}, {(2*\x -1)/(\x+1)+1});
\draw[gray] (0, -0.1)--(0, 1.55);
\draw[gray] (1, -0.1)--(1, 1.55);
\draw[<-] (0, 1.55)--(0.45, 1.55);
\node at (0.5,  1.55) {$J$};
\draw[->] (0.55, 1.55)--(1, 1.55);
\draw[blue, thin, <-] (-0.1, 0.75)--(1, 0.75);
\draw[cyan, thin] (0.333, 0.55)--(0.333, 0.97);
\draw[cyan, thin] (0.333, 0.55)--(0.224, 0.55);
\node at (0.224, 0.55)[cyan] {$\bullet$};
\draw[cyan, dashed] (0.224, -0.1)--(0.224, 0.55);
\node at (0.23, -0.2) {$\phi_1$};
\draw[cyan, thin] (0.333,  0.97)--(0.47,  0.97);
\node at (0.475, 0.97)[cyan] {$\bullet$};
\draw[cyan, dashed] (0.475, -0.1)--(0.475, 0.97);
\node at (0.475, -0.2) {$\phi_2$};
\draw[ dotted] (0.176, -0.1)--(0.176, 0.45);
\node at (0.176, 0.45)[red] {$\bullet$};
\node at (0.16, -0.2) {$\beta_1$};
\draw[ dotted] (0.512, -0.1)--(0.512, 1.02);
\node at (0.512, 1.02)[red] {$\bullet$};
\node at (0.512, -0.1)[pin={[pin edge=<-, pin distance=12pt]300:{$\beta_2$}}] {};
\node at (0.818, 0.75)[blue] {$\bullet$};
\draw[ dotted] (0.818, -0.1)--(0.818, 0.75);
\node at (0.818, -0.2) {$\beta_4$};
\node at (0.333, 0.75) {$\bullet$};
\node at (0.333, 0.75)[pin={[pin edge=<-, pin distance=12pt]150:{$(\alpha_0,p)$}}] {};
\draw[ dotted] (0.333, -0.1)--(0.333, 0.75);
\node at (0.333, -0.2) {$\alpha_0$};
\end{tikzpicture}
}
\caption{Determining the open set $N(p)$ of $\alpha$ with agreement of $\alpha$- and $\alpha_0$-digits for both $p$ and corresponding $T_{\alpha}$-image of endpoint, where $p$ is in the $T_{\alpha_0}$-orbit of an endpoint of $\mathbb I_{\alpha_0}$ (one possible configuration), see the proof of Proposition~\ref{p:goodNeighborsAreThere}.}
\label{f:goodNeighborsExist}%
\end{figure}
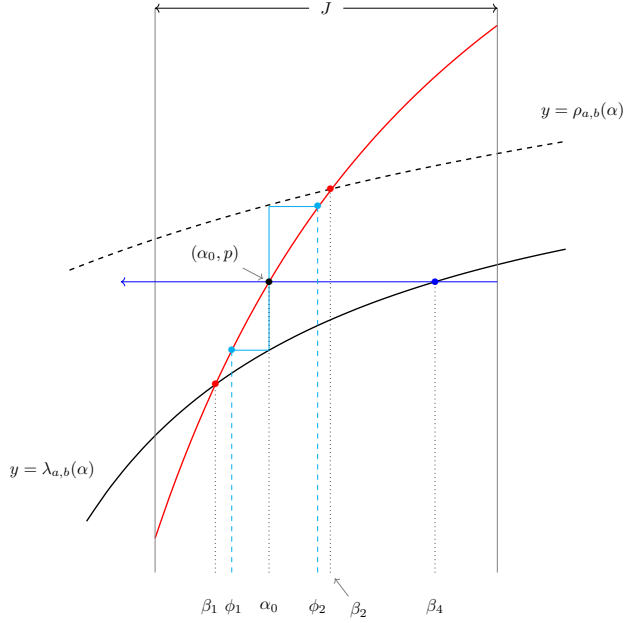

\begin{Prop}\label{p:goodNeighborsAreThere}   If $T_{\alpha}$ is not of bounded non-full range, then there is an open neighborhood of $\alpha$ within which every $\alpha'$ is a close neighbor of $\alpha$. 
\end{Prop}

\begin{proof}   All endpoints of synchronization intervals, all $\delta_{-k,v}$, and all non-synchronizing $\alpha$ are such that $T_{\alpha}$ is of bounded non-full range.   We thus fix  $\alpha_0$ in the interior of some  interval $J$ of the form $J_{k,v}$, $(\eta_{-k,v}, \delta_{-k,v})$ or $(\delta_{-k,v}, \zeta_{-k,v}, )$.   

On $J$,  each $\alpha \mapsto  \ell_i(\alpha)$ with $0\le i \le \underline{S}$ and $\alpha \mapsto   r_j(\alpha)$ with $0\le j \le \overline{S}$ is  continuous.  
Similarly,  letting  $\Delta(a,b) = \cup_{\alpha\in [0,1]}\, \Delta_{\alpha}(a,b)$,  we have that the left and right boundaries $\lambda_{a,b}(\alpha)$  and $\rho_{a,b}(\alpha)$ also give continuous functions  (at least for those $(a,b)$ which are admissible for all of $J$), compare this with (\cite{CaltaKraaikampSchmidt}, Figure~1.4), where such boundaries are plotted over all of $[0,1]$ in the setting of $n=3$.   

Fix $p \in \{\ell_i(\alpha_0) \mid 1\le i \le \underline{S}\} \cup \{r_j(\alpha_0) \mid 1\le j \le \overline{S}\}$, and thus also fix the value of $i$ or $j$ corresponding to $p$.  We now solve for those $\alpha$ for which  the $T_{\alpha}$-digit of $p$ agrees with its $T_{\alpha_0}$-digit, confer Figure~\ref{f:goodNeighborsExist}.    Denote the  $T_{\alpha_0}$-digit of $p$ by $(a,b)$.  Let $\beta_1, \beta_2$ be the endpoints of the subinterval of $J$ where all $\ell_i(\alpha)$ (respectively,  $r_j(\alpha)$\,)  have $\alpha$-digit equal to $(a,b)$.  Thus,  $\beta_1 < \alpha_0 < \beta_2$.    Let $\beta_3, \beta_4$ be defined by  $\rho_{(a,b)}(\beta_3) = p$ and  $\lambda_{(a,b)}(\beta_4) = p$.   Then $\beta_3 < \alpha_0 < \beta_4$.  Now let $M(p)$ be the open interval whose left endpoint is the greatest of  $\beta_1, \beta_3$ and whose right endpoint is the least of  $\beta_2, \beta_4$.   Then for all $\alpha \in  M(p)$,   the $\alpha$-digit of $p$ is also $(a,b)$.    Thus,  for all $\alpha \in M(p)$,  both $p \in \mathbb I_{\alpha_0}\cap \mathbb I_{\alpha}$ and   $T_{\alpha_0}(p) = A^aC^b\cdot p = T_{\alpha}(p)$.

Now let $\phi_1$ be defined by $\lambda_{a,b}(\alpha_0) = \ell_i(\phi_1)$ and $\phi_2$  by $\rho_{a,b}(\alpha_0)= \ell_i(\phi_2)$ (respectively, $r_j(\phi_1), r_j(\phi_2)\,$).
Let  $N(p)$ be the open interval whose left endpoint is the greatest of  $\beta_1, \beta_3, \phi_1$ and whose right endpoint is the least of   $\beta_2, \beta_4, \phi_2$.   Then for all $\alpha \in N(p)$,    $T_{\alpha_0}( \ell_i(\alpha)) = A^aC^b\cdot \ell_i(\alpha)) = T_{\alpha}( \ell_i(\alpha))$ (respectively, with $r_j$ appropriately replacing $\ell_i$).   

We now let $\mathcal N(\alpha_0) = \cap_p \, N(p)$.   Then every $\alpha$ in the open set $\mathcal N(\alpha_0)$ is a close neighbor of $\alpha_0$.    
\end{proof}

For our purposes, we only need that in every open set $\mathcal N(\alpha)$ of close neighbors there is some $\alpha'$ with $T_{\alpha'}$ of bounded non-full range.  We show a stronger result.   
\begin{Prop}\label{p:bddNonFullAreDense}  The set of $\alpha$ with $T_{\alpha}$ of bounded non-full range is dense in $[0,1]$. 
\end{Prop}
\begin{proof}   The bounded non-full range condition is fulfilled by the maps indexed by: the endpoints of synchronization intervals, see the proofs of Propositions~\ref{p:NaturallyErgodicEndPts},  ~\ref{p:ergodicLargeEta}, ~\ref{p:ergodicLargeZeta}; and the division points of synchronization intervals for large $\alpha$, see the proof of  Proposition~\ref{p:ergodicDelta}; and, in light of Remark~\ref{rmk:bddNonfullCylinders}, also  at all non-synchronization points:  for each such $\alpha$,  both endpoints of the interval of definition have only finitely many distinct digits in their $\alpha$-expansions.

  Fix any $\alpha$ with $T_{\alpha}$ not of bounded non-full range.  From the above,   $\alpha$ lies in the interior of some synchronization interval;   the $T_{\alpha}$-orbits of $\ell_0(\alpha)$ and $r_0(\alpha)$ meet.  
Since $T_{\alpha}$ is  not of bounded non-full range, the $T_{\alpha}$-orbit of $\ell_0(\alpha)$ must visit every full cylinder at least once.  Choose an increasing subsequence of $i_j, j \ge 1$ with $i_1 >\underline{S}$   such that $\ell_{i_j-1}(\alpha)$ is in a full cylinder.   Zero is clearly included in the $T_{\alpha}^{i_j}$-image of the corresponding rank $i_j$ cylinder, for each $j$ .  Set  $z_j = \underline{b}_{[1, i_j]}^{\alpha}$    and let $a_j = M_{z_j}^{-1}\cdot 0$.  
For each sufficiently large  $j$, there is an $\alpha_j$ such that $\ell_0(\alpha_j) = a_j$ and such that $\alpha_j$ is in the same synchronization interval as $\alpha_0$.   If $M_{z_j}$ is $\alpha_j$-admissible, then the  $T_{\alpha_j}$-orbit of $\ell_0(\alpha_j)$,  and hence that of $r_0(\alpha_j)$,  reaches zero in finitely many steps.  In particular,  these $T_{\alpha_j}$ are all of bounded non-full range.   

  Since the $\alpha_j$ certainly converge to $\alpha$, we are done unless in fact we cannot find an infinite subsequence along which each $M_{z_j}$ is $\alpha_j$-admissible.   If this obstruction exists,  then it must be that for all  $i_j$ sufficiently large there is some $i<i_j$ such that $M_i\cdot a_j \notin [a_j, t+ a_j)$, where $M_i$ corresponds to $\underline{b}_{[1, i]}^{\alpha}$.   For each such pair $(i,j)$, either $M_i\cdot a_j <a_j$ or  $M_i\cdot a_j \ge t+ a_j$.   In the first case, $M_i\cdot a_j <a_j < \ell_0(\alpha) < M_i\cdot \ell_0(\alpha)$ and hence $M_i$ has a fixed point in $(a_j, \ell_0(\alpha)\,)$.   Recall that $A\cdot x = x + t$ for any $x$;   in the second case,  $A^{-1}M_i\cdot a_j <a_j < \ell_0(\alpha) < A^{-1}M_i\cdot \ell_0(\alpha)$ and  it is $A^{-1}M_i$ that has a fixed point in $(a_j, \ell_0(\alpha)\,)$.    
   
  We now define a new subsequence of   indices, $(i_k)_{k> 0}$ such that the distance of a fixed point of $M_{i_k}$ or of $A^{-1}M_{i_k}$ to $\ell_0(\alpha)$ is less than all such distances for all $0<i<i_k$.    Let $b_k$ be the fixed point defining $i_k$, and let $\alpha_k$ be such that $\ell_0(\alpha_k) = b_k$.    By construction,  each $M_i, i <i_k$ is $\alpha_k$-admissible,  and $T_{\alpha_k}^{i_k+1}$ fixes $\ell_0(\alpha_k)$.   In particular,  the $T_{\alpha_k}$-orbit of $\ell_0(\alpha_k)$ meets only finitely many cylinders.     We can and so assume that  $i_k>\underline{S}$ holds for all $k$,  and hence each $T_{\alpha_k}$ is of bounded non-full range.    The result holds.
\end{proof}

\begin{proof}[\bf {Proof of Theorem~\ref{t:naturallyErgodicParTout}}.] \label{pf:TheoremOne}   We have already shown that the result holds for: any endpoint of a synchronization interval, by Propositions~\ref{p:NaturallyErgodicEndPts} ,~\ref{p:ergodicLargeEta}, and ~\ref{p:ergodicLargeZeta};    all of the $\delta_{-k,v}$   by Proposition~\ref{p:ergodicDelta}; and,  also  all non-synchronizing values of $\alpha$ by  Propositions~\ref{p:bijectivityDomIsNatExtErgSmallNonSyn}   and ~\ref{p:biDomIsNatExtErgLargeNonSyn}.    It thus remains to consider the case of $\alpha$ in an interval $J$ of the form $J_{k,v}$,  or $(\eta_{-k, v}, \delta_{-k,v})$,  or $(\delta_{-k,v}, \zeta_{-k,v})$.   By Proposition~\ref{p:goodNeighborsAreThere} , either $\alpha$ already satisfies the hypotheses of   (\cite{CKStoolsOfTheTrade}, Theorem~2)  and the result holds, or else     the set of close neighbors of $\alpha$ is an open neighborhood.  Thus, the denseness  of the set of $\alpha'$ with $T_{\alpha'}$ of bounded non-full range implies that  we can apply the results of  Subsection~\ref{ss:closeNeighbors}  to conclude that our statements hold for this value of $\alpha$ as well. 
\end{proof}

\section{Entropy changes continuously}\label{s:contEntropy}   As in  (\cite{KraaikampSchmidtSteiner}, proof of Theorem~2), we employ one of the key tools for proving statements about change of entropy:  Abramov's formula relates entropy values of systems which have a common subsystem to which the respective  first return maps agree.    We show such common first returns to subsystems of the $\alpha =0$ system  for a leftmost parameter subinterval, and  similarly use $\alpha=1$ for a rightmost parameter subinterval.   With this in hand, continuity of entropy on each of these subintervals does hold; fortunately enough, the subintervals overlap and the continuity holds on the entire parameter interval.

\subsection{Common first returns} 

For each $\alpha<1$, we seek a subset of $\Omega_{\alpha}$ contained in $\Omega_0$ for which $\mathcal T_{\alpha}$ and $\mathcal T_0$ have common first returns.   Due to the shape of $\Omega_0$, such a subset must surely have  negative $x$-values and positive $y$-values; this leads to the following set $N$, confer Figures~\ref{f:omegaSmallAlp}  and \ref{f:omegaLargeAlpLessThanDelta_{-k,v}}.   
 \begin{Lem}\label{l:conversionAlpToZero}  Suppose that $1/t_n \ge \alpha>0$ and set  $N = \bigcup_{i \in \mathbb N}\, \Delta_{\alpha}(-i,1) \cup \Delta_{\alpha}(-i,2)$.   Suppose $x \in \mathbb I_{\alpha}$ is negative and that there exists $m\ge 0$ such that both $T_{\alpha}^{m}(x) \in N$ and $T_{\alpha}^{m+1}(x)<0$.   Then there is some $k\in \mathbb N$ such that $T_{0}^{k}(x) =  T_{\alpha}^{m+1}(x)$.    Furthermore,  for any $y$ one has $\mathcal T_{0}^{k}(x,y) =  \mathcal T_{\alpha}^{m+1}(x,y)$. 
\end{Lem}
\begin{proof}  Note that some of the cylinders in the expression defining $N$ may be empty for any given $\alpha$.  

We can and do assume that $m$ is minimal with respect to the hypotheses.    From the definition of our maps, we can write
\[ T_{\alpha}^{m+1}(x) = A^{a_{m+1}}C^{c_{m+1}}A^{a_{m}}C^{c_{m}} \cdots A^{a_1}C^{c_1}\cdot x,\]
with $c_1 = 1$ and both $a_{m+1}, a_1$ negative.   If $m=0$, then $A^{a_1}C^{c_1}\cdot x = T_0(x)$ and $k=1$.   We now assume that $m>0$ and  
 perform word processing on $A^{a_{m+1}}C^{c_{m+1}}A^{a_{m}}C^{c_{m}} \cdots A^{a_1}C^{c_1}$ with the goal to achieve an expression for this element of the group $G_n$ that is in an admissible form for $T_0$.  
 
 We process from right to left,  using the following substitution rules.  Here ``;" is placed so that the processed portion lies to its right.   

\noindent
{\bf Substitution rules}:
 \[
 \begin{aligned} 
(i)&\;\;\;\;\;\;\;\;\;  a<0,\;   A^a C;                       &\mapsto &\;\;\;\;\; A; A^{a-1}C&\\
(ii)&\;\;\;\;\;\;\;\;\;   a>0,\;   A^a C A; \;               &\mapsto &\;\;\;\;\; C A; W^a&\\  
(iii)&\;\;\;\;\;\;\;\;\;  A^a C^2 A; \;                        &\mapsto &\;\;\;\;\; A^{a-1}C; (A^{-1}C)^{n-2}&\\  
(iv)&\;\;\;\;\;\;\;\;\;  a>0, b<0, \; A^a C;  A^b C\; &\mapsto &\;\;\;\;\; C A; W^a A^{b-1}C&\\    
(v)&\;\;\;\;\;\;\;\;\;   b<0, \;  A^a C^2;  A^b C\;     &\mapsto &\;\;\;\;\; A^{a-1}C;  (A^{-1}C)^{n-2}A^{b-1}C&\\     
  \end{aligned}
  \]
  Recall from (\cite{CaltaKraaikampSchmidt}, Equation~2.1) that 
 \begin{equation}\label{e:ItooAmW} W = A^{-2}C (A^{-1}C)^{n-3} A^{-2}C (A^{-1}C)^{n-2}.
 \end{equation}
After a rule is applied, we either treat the $A^{a-1}C;$ which has been defined, or we multiply on the left by the next unprocessed $A^{a_i}C^{c_i}$,  and continue until only the empty word is to the left of the demarcating ``;".

 We claim that this process ends with a successful conversion of the element of $G_n$ factored into a $T_0$-admissible form.   Note first that each rule, up to ignoring the separation marker ``;", is the result of applying a group identity.  The first of these is obvious; the second and fourth result from induction and the simple identity $ACA = CAW$ see   (\cite{CaltaKraaikampSchmidt}, Equation~2.2)
and compare with  (\cite{CaltaKraaikampSchmidt}, Lemma~5.1);
finally, the third and fifth result from $C^2 A = C^{-1}A = (A^{-1}C)^{n-1}$.       We now argue that there is exactly one way to apply the rules.  First note that no two rules are possible for any situation.   It remains to show that there is always is a rule to apply at any (non-final) step.\\ 

First assume that $\alpha \le \gamma_n$.   Hence, there is no occurrence of $c_i = 2$.   Since $x<0$, we have $a_1<0$.    Since $m>0$,   there are no consecutive $a_i, a_{i+1}$ which are both negative.  Thus, we apply the first rule immediately, and  the second rule is always applied only after the first rule is.  Both the second and the fourth result in a step where the third rule must be applied.   The third and fifth result in a step where the first or fourth rule must be applied, unless in either case $a=1$ when both result in the fifth rule  being applied.  It follows that there is always an applicable rule at each step.     Since $a_{m+1}<0$ we have $a_m >0$ and thus the final step occurs with an application of the third or fifth rule;   at this point, we can disregard the semi-colon, and have rewritten the original  element of the group.    The main requirements for $T_0$-admissibility of a word in letters of the form $A^a C$ with each $a<0$ come from the fact that  $-t = \ell_0(0)$ is sent to itself by the admissible word $W$.   These requirements are that $A^{-1}C$ cannot appear more than $n-2$ times in a row, and appearances of $A^{-2}C \,(A^{-1}C)^{n-2}$ are isolated.   It is easily checked that our processed word meets these requirements.    The process fully succeeds in this case.\\

Now suppose $1/t_n \ge \alpha >\gamma_n$.  Note that the complement of $N$ in $\mathbb I_{\alpha}$ is $[0,\mathfrak b_{\alpha})$.     We of course can again assume $m>0$, and also observe that at most one rule can be applied at any step of the word processing.  Since $x<0$ and we cannot have consecutive occurrences of $(a_i, 1)$ with $a_i$ negative, the first rule is again applied immediately. It then is followed by an application of the second or third rule. Both the second and the fourth result in a step where the third rule must be applied, unless the next unprocessed portion is of the form $A^{a_i} C^2$.  In this case,  a simplification gives $A^{a_i} C^2 C A; W^a = A^{a_i-1}  C; (A^{-1}C)^{n-3} A^{-2}C (A^{-1}C)^{n-2} W^{a-1}$ and (assuming processing is not yet complete) either rule $i$ or rule $iv$ is applied, unless $a_i = 1$ in which case similar considerations are applied with $A^{a_{i+1}} C^{c_{i+1}} C; \cdots$.  As well, there are further, easily determined, possibilities after either the third of fifth rule is applied and $a_i = 1$.  That is,  here also the process fully succeeds.\\ 

Finally, with $M = A^{a_{m+1}}C^{c_{m+1}}A^{a_{m}}C^{c_{m}} \cdots A^{a_1}C^{c_1}$, both $\mathcal T_{0}^{k}(x,y) =  \mathcal T_M(x,y)$ and $\mathcal T_{\alpha}^{m+1}(x,y)=  \mathcal T_M(x,y)$.
 \end{proof}

  We now wish to find a region for which the {\em first} return under $\mathcal T_0$ agrees with that of $\mathcal T_{\alpha}$.  To this end, we restrict $N$ and the region to which we return, so that $\{-1,-2\}$ digits introduced in the word processing could never index  an `early return' of a $\mathcal T_0$-orbit. 
 
\begin{Lem}\label{l:firstReturnToDoubleNegVal}  Suppose that $1/t_n \ge \alpha>0$.  Let $d = \min\{-4,  -2 + d_{\alpha}(\ell_0(\alpha))\,\}$,  $N_d = \bigcup_{i \le d}\, \Delta_{\alpha}(i,1) \cup \Delta_{\alpha}(i,2)$  and  $\Omega^{+}_{d}=  \Omega_{\alpha} \cap \{(x,y) \mid y>0, x<0,  x \in N_d\; \text{and}\; x' \in N_d, {\text where}\; \mathcal T_{\alpha}^{-1}(x,y) = (x',y')\,\}$.   
  The first return maps of  $\mathcal T_0$ and of $\mathcal T_{\alpha}$ to $\Omega^{+}_{d}$ agree. 
\end{Lem}

\begin{proof}  Since $\Omega_{\alpha}^{+}$ is of finite $\mu$-mass, it follows that $\Omega^{+}_{d} \subset \Omega_0$ and hence both return maps are defined.   Suppose that  $(x,y) \in \Omega^{+}_{d}$ and  $\mathcal T_{\alpha}^{m+1}(x,y)$ is its first $\mathcal T_{\alpha}$-return. 
From the previous lemma, we can deduce that this first return map   is also  {\em some} return under $\mathcal T_0$.  

 If $m=0$ then certainly $T_{\alpha}(x) = T_0(x)$, and the first return of the two-dimensional maps  also agree.  
 Now assume that $m>0$.  
Suppose  
 \[ T_{\alpha}^{m+1}(x) = A^{a_{m+1}}C^{c_{m+1}}A^{a_{m}}C^{c_{m}} \cdots A^{a_1}C^{c_1}\cdot x.\]   
Of course,   if for all $i$ both $c_i = 1$ and $a_i$ are negative, then the orbits are the same and the result clearly holds.  In general, there may be intermediate consecutive appearances of this type.  We thus add a special case to the first rule:
 \[
(i')\;\;\;\;\;  A^{a'} C  A^a C;\;     \mapsto     A^{a'} C;  A^a C \;\;\text{if}\;   a<0, a'<0,                    
 \]
 and process so as to achieve an admissible $T_0$ expression.   
 Whenever any of our substitution rules inserts $T_0$-steps,  these insertions are always of corresponding simplified $0$-digits  in $\{-1, -2\}$.   Of course,  any $x$-value  with either of these as its  simplified $0$-digit cannot have simplified $\alpha$-digit less than or equal to $d$.
 
 On the other hand, our rules also make changes in the exponents of appearances of $A$.       However,  any such exponent is changed at most once throughout the process, and if it is changed then it is decreased exactly by one.   Since each substitution realizes a group identity,  it follows that  any element in $N_d$ that is in the $T_0$-orbit segment of $x$  but not already present in the $T_{\alpha}$-orbit segment is either (1) isolated (with respect to this property),  or (2) is sent by $T_0$ to  $T_{\alpha}^{m}(x)$ while its $T_0$-orbit predecessor is not in $N_d$.    But, by hypothesis, $\mathcal T_{\alpha}^{m}(x, y) \notin \Omega^{+}_{d}$.   Therefore the result holds.  \end{proof} 
 
We now treat the case of large $\alpha$.    Note that $\alpha=1$ gives a system defined on positive $x$, but importantly each digit is then of the form $(i,1)$ or $(i,2)$ with $i$ {\em positive}.   
\begin{Lem}\label{l:conversionAlpToOne}  Let $\gamma_n< \alpha<1$ and set $P = \bigcup_{i \in \mathbb N}\, \Delta_{\alpha}(i,1) \cup \Delta_{\alpha}(i,2)$.     Suppose $x \in P$ is such that there exists $m\ge 0$ such that both $T_{\alpha}^{m}(x)$ and $T_{\alpha}^{m+1}(x)$ are in $P$.   Then there is some $k\in \mathbb N$ such that $T_{1}^{k}(x) =  T_{\alpha}^{m+1}(x)$.    Furthermore,  for any $y$ one has $\mathcal T_{1}^{k}(x,y) =  \mathcal T_{\alpha}^{m+1}(x,y)$. 
\end{Lem}

\begin{proof}  We argue as in the case of small $\alpha$.   Here also the case of $m=0$ is trivial, and we now suppose that $m>0$ and minimal.  

Recall that $U = A C (AC^2)^{n-2}$.   The relation given in (\cite{CaltaKraaikampSchmidt}, Lemma~8.1) 
directly implies that for any $i \in \mathbb N$ we have 
\[ A^{-i}C A^{-1} = C^2 (AC^2)^{n-2}U^{i-1}.\]
Suppose  
 \[ T_{\alpha}^{m+1}(x) = A^{a_{m+1}}C^{c_{m+1}}A^{a_{m}}C^{c_{m}} \cdots A^{a_1}C^{c_1}\cdot x,\]  
with the right hand side being the $\alpha$-admissible factorization.   

We  process $M = A^{a_{m+1}}C^{c_{m+1}}A^{a_{m}}C^{c_{m}} \cdots A^{a_1}C^{c_1}$ by  the following rules.   Note that if    $c_2=2$   then $a_2<0$ and $T_{\alpha}(x) \in [\mathfrak b_{\alpha}, 1)$.  It is easily verified that  $ [\mathfrak b_{\alpha}, 1) \subset \Delta_1(1,1)$ and hence the $1$-word retains $A^{a_1}C^{c_1}$ and continues by the result of a processing of   $\cdots A^{a_2}C^2 (AC)^{-1} AC$.  Similar considerations apply anytime some intermediate $c_i = 2$. 

 \[
 \begin{aligned} 
A^{a_2}C^{c_2} A^{a_1}C^{c_1}; &\mapsto  \begin{cases}  C^2(AC^2)^{n-2}; U^{-a_2-1}A^{1 + a_1}C^{c_1}&\text{if}\;\; c_2 = 1\\
                                                                                              C^2(AC^2)^{n-2}; U^{-a_2-1}AC A^{a_1}C^{c_1}&\text{otherwise.}
                                                                      \end{cases}\\
A^{a_i}C^{c_i} \,C^2(AC^2)^{n-2}; &\mapsto   \begin{cases} C^2  (AC^2)^{n-2};U^{-a_i-1} AC  (AC^2)^{n-3}&\text{if}\;\; c_i = 1\\
                                                                                              C^2(AC^2)^{n-2}; U^{-a_i-1}A^2C^2 (AC^2)^{n-3} &\text{otherwise.}
                                                                      \end{cases}\\ 
A^{a_{m+1}}C^{c_{m+1}}\, C^2(AC^2)^{n-2}; &\mapsto   \begin{cases}  ;A^{a_{m+1}}C (AC^2)^{n-2} &\text{if}\;\; c_{m+1} = 2\\
                                                                                                           ;A^{1+ a_{m+1}}C^2 (AC^2)^{n-3} &\text{otherwise.}
                                                                                   \end{cases}
  \end{aligned}
  \]
Once again, upon suppressing the semi-colon of demarcation, each rule is an identity.   
The processed word is $1$-admissible, and the result holds. 
 \end{proof} 

\begin{Lem}\label{l:firstReturnToDoubleLarge}  Suppose that $\gamma_n < \alpha <1$, set $Q   = \bigcup_{i \ge 4}\, \Delta_{\alpha}(i,1) \cup \Delta_{\alpha}(i,2)$, and let 
$\Omega^{Q}=  \Omega_{\alpha} \cap \{(x,y) \mid y<0;  x, x' \in Q,\; {\text where}\; \mathcal T_{\alpha}^{-1}(x,y) = (x',y')\,\}$.
 The first return maps of  $\mathcal T_1$ and of $\mathcal T_{\alpha}$ to $\Omega^{Q}$ agree. 
\end{Lem}

\begin{proof}  We have that $\Omega^{Q} \subset \Omega_1$ and hence the return maps are defined.     Suppose that  $(x,y) \in \Omega^{Q}$ and  $\mathcal T_{\alpha}^{m+1}(x,y)$ is its first $\mathcal T_{\alpha}$-return. 
From the previous lemma, we can deduce that this first return map   is also  {\em some} return under $\mathcal T_1$.  

 If $m=0$ then certainly $T_{\alpha}(x) = T_1(x)$, and the first return of the two-dimensional maps  also agree.  
 Now assume that $m>0$.  Suppose  
 \[ T_{\alpha}^{m+1}(x) = A^{a_{m+1}}C^{c_{m+1}}A^{a_{m}}C^{c_{m}} \cdots A^{a_1}C^{c_1}\cdot x.\]   
Of course,   if this $x$-orbit stays within $Q$, then it certainly stays within the positive reals and we find that this gives also the $T_1$-orbit  and the result  holds.  In general,  we process as in the previous lemma each subword corresponding to a return to a consecutive pair in $Q$ (including the possibility of the use of the ``$m=0$ case"). Note that  our rules insert only steps corresponding to one  of $AC, A^2C, AC^2, A^2 C^2$ and cause  one-time changes of the form $(a_i, c_i)$ replaced by  $(1+ a_i, 1), (1+ a_i, 2)$ for certain positive $a_i$.   These one time changes can only occur at the end of a return to a consecutive pair in $Q$, or at the beginning of an excursion from such a pair.  In the latter case, the value of the point at the beginning of the excursion is unchanged.  When such an increase to $1+a_i$ is made at the end of an excursion, the $T_1$-predecessor of the return lies in $\Delta_1(1+a_i,2)$;  in all cases its own $T_1$-predecessor lies in $\Delta_1(1,1) \cup \Delta_1(1,2)$
(with only the case of $n=3$ and  $a_{i-1} = -1$ at all complicated).   

Since the left endpoint of $\Delta_{1}(2,2)$ is sent by $A^2C^2$ to $0$, we find that $\Delta_{1}(2,2) \cup \Delta_{1}(1,2) \subset \Delta_{\alpha}(2,2) \cup \Delta_{\alpha}(1,2) \cup [r_0(\alpha), t)$.  One similarly finds that $\Delta_{1}(2,1) \cup \Delta_{1}(1,1) \subset \Delta_{\alpha}(2,1) \cup \Delta_{\alpha}(1,1) \cup [\mathfrak b_{\alpha}, 1)$.   Thus,   $\bigcup_{i=1}^{2}\, \Delta_{1}(i,1) \cup \Delta_{1}(i,2)$ does not intersect  $Q$.     Thus, our $1$-admissible word is such that it gives the   {\em first}  return under $\mathcal T_1$. 
\end{proof}

\subsection{Entropy times mass is constant;  Proofs of Theorems~\ref{t:prodIsConst},~\ref{t:continuityOfEntropy} }\label{ss:EM=C}
We now argue by way of Abramov's formula, as in \cite{CKStoolsOfTheTrade},   that the following holds. 

\begin{Prop}\label{p:continuityOfEntropyOnIntervals}     The function   $\alpha \mapsto  h(T_{\alpha}) \mu(\Omega_{\alpha})$ is constant on each of the intervals $(0, 1/t_n]$ and $(\gamma_n, 1)$.  
\end{Prop}

\begin{proof}  Let $\alpha, \alpha'$ both belong to one of the parameter intervals being considered.    In the respective cases, the subset $\Omega^{+}_{d}$ or $\Omega^{Q}$ of $\Omega_{\alpha}$ is of positive $\mu$-mass, and similarly for when $\alpha'$ replaces $\alpha$.  Let $\Omega_{\alpha, \alpha'}$ be the intersection of these respective subsets.    From Lemmas~\ref{l:firstReturnToDoubleNegVal} and ~\ref{l:firstReturnToDoubleLarge} we deduce that the first return map of  $\mathcal T_{\alpha}$ and $\mathcal T_{\alpha'}$  agree.   These hence define the same dynamical system, and Abramov's formula yields that their entropy can be expressed as 
\[ \dfrac{h(\mathcal T_{\alpha}) \, \mu( \Omega_{\alpha})}{\mu(\Omega_{\alpha, \alpha'})} =   \dfrac{h(\mathcal T_{\alpha'}) \, \mu( \Omega_{\alpha'})}{\mu(\Omega_{\alpha, \alpha'})}.\]
Since the entropy of a dynamical system equals the entropy of the natural extension system, the result holds.   
\end{proof} 

  Since the two parameter subintervals of the proposition overlap, that is $1/t_n>\gamma_n$, Theorem~\ref {t:prodIsConst} is proven.  Recall that from this, and Theorem~\ref{t:continuityOfMass}, Theorem~\ref{t:continuityOfEntropy}  follows. \qed 

\section{Conjecture: Entropy times $\mu$-mass equals volume of unit tangent bundle}\label{ss:theEnd}
 
The volume of the unit tangent bundle of the hyperbolic orbifold uniformized by the triangle group $G_n$ is  
\begin{equation}\label{e:volare}
\text{vol}_n = \dfrac{2 (2 n - 3 )\pi^2}{3 n}.
\end{equation}
     This is easily checked, as the hyperbolic area of the base orbifold is twice the hyperbolic area of any hyperbolic triangle of angles  $\pi/3,  \pi/n, 0$,  and  $\text{vol}_n$ is this times $\pi$,   see  say \cite{ArnouxCodage} for this calculation when $G_n$ is replaced by the standard modular group.    From this, Conjecture~\ref{con:Vol} can be restated as the following.
\begin{Conj}   For all $n\ge 3$ and for all $\alpha \in (0,1)$ we conjecture that   
\[     h(T_{\alpha})\, \mu(\Omega_{\alpha})  = \emph{vol}_n\,.  \] 
\end{Conj}

In \S~\ref{ss:RockTheVolume} we relate Rohlin's formula for entropy to volumes of unit tangent bundles.   This reduces the above conjecture to  the computation of an integral.   \S~\ref{ss:conjWithNis3} shows that the conjecture holds when $n=3$.  \S~\ref{ss:ImpFirst}  gives in particular the proof of Theorem~\ref{t:expansiveUisFactorOfFlow}   (which implies the equivalence of Conjectures~\ref{con:returnReturn} and ~\ref{con:Vol}). Finally, in \S~\ref{ss:conjComputation4to12} we report on computational confirmation for $4\le n \le 12$.   

\bigskip 
\subsection{Rohlin's formula and the volume of unit tangent bundles}\label{ss:RockTheVolume}  We show that for each $n$, the conjecture can be reduced to an integral computation, involving either $\Omega_0$ or $\Omega_1$.
\begin{Def}\label{d:tau}  
  Fix $n \ge 3$.   For $ \alpha \in [0,1]$ and for each $x \in \mathbb I_{\alpha}$,  let $\tau_{\alpha}(x) = - 2 \log \vert c x + d\vert$ where $T_{\alpha}(x) = (a x + b)/(c x + d)$.   Of course, this is simply  $\tau_{\alpha}(x) =  \log \vert T'_{\alpha}(x)\vert$. 
\end{Def}

  We now show that the integral of $\tau_{\alpha}(x)$ over $\Omega_{\alpha}$ with respect to $d \mu$ gives the product of  $h(T_{\alpha})$ with  $\mu(\Omega_\alpha)$.   Furthermore,  these integrals are constant with respect to $\alpha$, including at the endpoints, this although both $\mu(\Omega_0)$ and $\mu(\Omega_1)$ are infinite.
 
\begin{Lem}\label{l:integralsMatch}  Fix $n \ge 3$.   For $ \alpha \in [0,1]$ and for each $x \in \mathbb I_{\alpha}$,  let $\tau_{\alpha}(x) = - 2 \log \vert c x + d\vert$ where $T_{\alpha}(x) = (a x + b)/(c x + d)$.   Then for $0 <\alpha < 1$,  
 \[  h(T_{\alpha})\, \mu(\Omega_{\alpha})  = \int_{\Omega_{\alpha}} \tau_{\alpha}(x)\,  d \mu = \int_{\Omega_{0}} \tau_{0}(x)\,  d \mu =\int_{\Omega_{1}} \tau_{1}(x)\,  d \mu. 
\]
\end{Lem}

\begin{proof}     For $0 < \alpha < 1$,   Rohlin's formula (see \cite{DenkerKellerUrbanski}) gives 
\[  h(T_{\alpha}) = \int_{\mathbb I_{\alpha}}\, \log \vert T'_{\alpha}(x) \vert\, d \nu_{\alpha}=  \int_{\mathbb I_{\alpha}}\, \tau_{\alpha}(x) d \nu_{\alpha} =  \dfrac{\int_{\Omega_{\alpha}}\, \tau_{\alpha}(x) d \mu}{\mu(\Omega_{\alpha})},
\] 
where the last equality holds because $\nu_{\alpha}$ is the marginal measure for the probability measure on $\Omega_{\alpha}$ induced by $\mu$.   Since  $\mu(\Omega_{\alpha})< \infty$, our first equality holds.  We now show the remaining two equalities hold by showing that each holds for $\alpha$ in a certain subinterval; these subintervals overlap, from this the result itself holds.    
 
 Assume first that $\alpha \le 1/t$.  We employ a tower construction, with base $\Omega^{+}_{d}$.  For this, note that Poincar\'e recurrence for each of $\mathcal T_{\alpha}^{\pm 1}$ to this positive measure subspace of $\Omega_{\alpha}$ implies bijectivity of the first return of $\mathcal T_\alpha$ to $\Omega^{+}_{d}$.    For each $k \in \mathbb N$ we define $R_k \subset \Omega^{+}_{d}$ as the set of points whose first return to $\Omega^{+}_{d}$  is given by applying $\mathcal T_{\alpha}^k$, thus the $R_k$ partition $\Omega^{+}_{d}$.   We have that  $\mathcal T_{\alpha}$ bijectively maps    the tower formed  by the disjoint  union of the sets of the form $\mathcal T_{\alpha}^i(R_k)$ for all $k$ and $0\le i <k$ to itself.  By the ergodicity of $\mathcal T_{\alpha}$, we deduce that this union gives $\Omega_{\alpha}$ up to measure zero.

Let  $\rho$ denote projection onto the $x$-coordinate of the first return map on $\Omega^{+}_{d}$.   
 Due to  the logarithm turning the multiplication appearing in the Chain Rule to addition  ---  in the sense that:  For any reasonable map $T$, we have $\log |(T^k)'(x)| = \sum_{i=0}^{k-1}\, \log |T'(\, T^i(x)\,)|$ ---  and the additivity of integration, one deduces that 
\[ \int_{\Omega_{\alpha}} \tau_{\alpha}(x)\,  d \mu = \int_{\Omega^{+}_{d}}  \log \vert \rho'(x)\vert\,  d \mu.\]
Due to the infinite measure of $\Omega_0$, we cannot simply repeat this argument to directly achieve the equality with $\Omega_0$ replacing $\Omega_{\alpha}$.  However, the proof of Lemma~\ref{l:firstReturnToDoubleNegVal} can easily be adapted to show that the first return map  itself also  agrees with the first return map for the accelerated system that \cite{CaltaSchmidt} associates to $\mathcal T_0$. 
This system is again a finite $\mu$-measure ergodic dynamical system and thus one finds that $\int_{\Omega_{\alpha}} \tau_{\alpha}(x)\,  d \mu$ equals the integral of the logarithm of the absolute value of the first derivative of the accelerated one-dimensional map.    Next, one can show that by taking the union of that accelerated two-dimensional domain with  the forward $\mathcal T_0$ orbits of its points fibering over the subinterval where acceleration occurs (denoted $[-\tau, \epsilon_0)$ in \cite{CaltaSchmidt}, note that their $\tau$ is our $t$)  gives  $\Omega_0$  up to zero measure.    Putting this together, we do find that $\int_{\Omega_{\alpha}} \tau_{\alpha}(x)\,  d \mu = \int_{\Omega_{0}} \tau_{0}(x)\,  d \mu$.     

 We can now argue completely analogously,  beginning with $\alpha >\gamma_n$,  using Lemma~\ref{l:conversionAlpToOne}, and the acceleration for the $\alpha=1$ setting given in (\cite{CKStoolsOfTheTrade}, \S~2.6),  to find that $\int_{\Omega_{\alpha}} \tau_{\alpha}(x)\,  d \mu = \int_{\Omega_{1}} \tau_{1}(x)\,  d \mu$.  
\end{proof}

\subsection{Proof of Conjecture~\ref{con:Vol} when $n=3$}\label{ss:conjWithNis3}
 
 In light of the previous subsection, the following result proves  Conjecture~\ref{con:Vol} when $n=3$.
\begin{Lem}\label{l:alpOneGivesIt}  Let  $n = 3$.   We have 
\[ \int_{\Omega_{1}} \tau_{1}(x)\,  d \mu = \emph{vol}_3\,.
\]
\end{Lem}

\begin{proof}    Restricting  (\cite{CKStoolsOfTheTrade}, Proposition~2.4) to our case of $n=3$, one has 
\[\Omega_1 =  (\, [0,1]\times [-1,0] \,) \; \cup \;  (\,[1, 2]\times [-1/2,0] \,).  \]
For $x<1, T_1(x) = A^{-k}C\cdot x$ for some $k \in \mathbb Z$ and in particular,   $\tau_1(x) = -2 \log  x$.   Similarly, for $2 \ge x>1$ we have $\tau_1(x) = -2 \log \vert x-1 \vert $.   Recall that for any $M \in \text{SL}_2(\mathbb R)$ that $\mathcal T_M$ locally preserves $d \mu$.   One easily verifies that   $M = \begin{pmatrix} 1&1\\0&1\end{pmatrix}$ is such that $\mathcal T_M$ sends  $[0,1]\times [-1,0]$ to  $[1, 2]\times [-1/2,0]$, and hence  a change-of-variables calculation gives
\[ \int_{[0,1]\times [-1,0]}  -2 \log  x     \; d\mu = \int_{ [1, 2]\times [-1/2,0]}  -2 \log \vert x-1 \vert   \, d\mu\, . \] 
That is, $ \int_{\Omega_{1}} \tau_{1}(x)\,  d \mu = 2 \int_{[0,1]\times [-1,0]}  -2 \log  x     \; d\mu$.   The following lemma thus implies our result.
\end{proof}

Note that \eqref{e:volare} specializes to $\text{vol}_3 = 2 \pi^2/3$.   
\begin{Lem}\label{l:useMtwoNthreeAlpZero}     We have 
\[  \int_{[0,1]\times [-1,0]}  -2 \log  x     \; d\mu = \pi^2/3. 
\]
\end{Lem}

\begin{proof}  It is an exercise to show that the by-excess map   on $[-1,0)$ to itself defined by $x \mapsto -1/x - \lfloor  1 -1/x \rfloor$  has as planar natural extension domain exactly   $[-1,0 ]\times [0,1]$.    (Recall that this is the $\alpha = 0$ Nakada-continued fraction, see say \cite{KraaikampSchmidtSteiner}.)     Arguing as in \cite{ArnouxSchmidtCross} shows that $ \int_{[-1,0 ]\times [0,1]}  -2 \log  \vert x  \vert   \; d\mu =  \pi^2/3$.    (Again, see  say \cite{ArnouxCodage}   for related discussion.)  Since $d \mu = (1 + xy)^{-2}\, dx\, dy$,   the measure is preserved under the map $(x,y) \mapsto  (-x, -y)$ and hence a change-of-variables calculation gives our result.
\end{proof}  

\subsection{First  expansive  power maps, Proofs of Theorems~\ref{t:expansiveUisFactorOfFlow}, ~\ref{t:conjConfirmedNis3}}\label{ss:ImpFirst}  We show that under Conjecture~\ref{con:Vol} every one of our maps can be appropriately accelerated so as to define an interval map which is naturally given by a geometric system. 

\subsubsection{Cross sections for geodesic flow: $n=3, \alpha>1/2$}
We now show that the results of the previous subsection imply that each of infinitely many of our maps has its natural extension given by the geodesic flow returning to a cross section in the  unit tangent bundle of the hyperbolic orbifold uniformized by $G_3$.   See, say,  \cite{Manning} for related background.
 \begin{Cor}\label{c:oneFactorOfFlow}    For $n=3$ and each $\alpha > 1/2$ there is a cross section in the unit tangent bundle of $G_3\backslash \mathbb H$ which when equipped with the first return map under geodesic flow gives the natural extension of the  interval map $T_{3, \alpha}$.
\end{Cor}    

\begin{proof} [Sketch]  For any $n\ge 3$  and $\alpha\in [0,1]$, Arnoux's transversal as in \cite{ArnouxSchmidtCross,ArnouxSchmidtCommeUneFraction}  allows for the map $v: \Omega_{\alpha} \to T^1(\,G_n\backslash \mathbb H\,)$ given by $(x,y) \mapsto [ M\cdot \hat i]_{G_n}$ where  $M = \begin{pmatrix} x&-1/(1+ xy )\\1& y/(1+ xy )\end{pmatrix}$,  $\hat i$ is the upwards directed unit tangent vector based at $i \in \mathbb H$, $M\cdot \hat i$ is given by the usual action, and  $ [ M\cdot \hat i]_{G_n}$ denotes the projection to the quotient orbifold of the unit tangent vector $M\cdot \hat i$.   Thus, the cross section here is $v(\Omega_{\alpha})$.   From the cited works, the map $v$ is injective up to measure zero.   Furthermore, the map $v\circ \mathcal T_{n, \alpha}\circ v^{-1}$  has $\mu$ as invariant measure. (Note that the cited works have Lebesgue measure as invariant measure, but  $(x,y) \mapsto (x, y/(1+x y)\,)$ pushes forward $\mu$ to Lebesgue measure.  See (\cite{CKStoolsOfTheTrade}, Section~2.3).   Since $\mathcal T_{n, \alpha}$ on $\Omega_{\alpha}$ gives the natural extension for the interval map, so does the conjugate map on the cross section. 

For any $n\ge 3$  and $\alpha\in [0,1]$, the derivative of $T_{n,\alpha}(x)$ equals $x^{-2}$ if $x < \mathfrak b_{\alpha}$ and  $(x-1)^{-2}$ otherwise.   Thus,   $T_{n,\alpha}(x)$ is expansive exactly when both ($i$) $\ell_0(\alpha) > -1$ and  ($ii$) $r_0(\alpha) <2$.   Since $r_0(\alpha) = \ell_0(\alpha)+ t = \alpha t$, these conditions are ($i$) $\alpha >1-1/t$ and ($ii$) $\alpha<2/t$.

We now fix $n=3$, thus $t=2$, and restrict to $1>\alpha >1/2$.  
Here each interval map $T_{\alpha} = T_{3, \alpha}$ is {\em expansive}. It follows, again from the cited works,  that at each $v(x,y)$ the map $v\circ \mathcal T_{\alpha}\circ v^{-1}$ is given by {\em some} return of the positively oriented geodesic flow to the cross section.  Key to this is that there is a geodesic arc connecting  the basepoint of $v(x,y)$ to that of $v\circ \mathcal T_{n, \alpha}(x,y)$, whose length is in fact  $ \tau_{\alpha}(x)$; 
see    (\cite{ArnouxSchmidtCommeUneFraction}, Proposition~4 and its proof).     By definition,  the geodesic segments determined by each $v(x,y)$ and the first return under the flow to the cross section form a flow invariant set. The ergodicity of the flow implies that this set is all of the unit tangent bundle;   the integral over the cross section of the lengths of the geodesic segments defined by the first return is hence equal to the volume of this bundle.   (See,  \cite{Manning, ArnouxSchmidtCross} for details of the measures involved in this.) Since $  \int_{\Omega_{\alpha}}   \tau_{\alpha}(x)\,  d \mu = \text{vol}_3$ holds here, one finds that $v\circ \mathcal T_{\alpha}\circ v^{-1}$ agrees almost everywhere with the {\em first} return of the geodesic flow.  
\end{proof}  

\subsubsection{First expansive power maps and geodesic flow, under the conjecture}
\begin{Def}\label{d:firstExpansiveReturn}  
Let $T$ be a function on  an interval  $I$, with  an invariant probability measure $\nu$ which is equivalent to Lebesque measure. Suppose that $T$ is eventually expansive, thus there is some $r \in \mathbb N$ such that $T^r$ is expansive.    By the Well Ordering Principle of the integers, for each $x\in I$ there is  a least $\ell(x) \in \mathbb N$ such that $T^{\ell(x)}(x)$ is expansive in the sense that the derivative here is greater than 1 in absolute value.    We define the {\em first pointwise expansive power of $T$}  as the map  $U: I\to I$  sending each $x$  to $T^{\ell(x)}(x)$.
\end{Def}
 
Of course, $\ell(x) \le r$ for all $x \in I$ and also $U(x) = T(x)$ for all $x \in E$, where $E$ is the maximal subset of $I$ on which $T$ itself is expansive.     For each $k \le r$, let $E_k \subset I$ be the set on which $\ell(x) = k$.    Of course,  the $E_k$ give a finite partition of $I$.   We have that $|(T^k)'(x)|>1$ for all $x \in E_k$;  
by  the Chain Rule, $|(T^k)'(x)| = \prod_{i=0}^{k-1}\, |T'(\, T^i(x)\,)|$.   The {\em minimality} of $\ell(x)$ hence implies both  $T(E_k)\subseteq \cup_{i=1}^{k-1} \, E_i$ for $k\ge 2$ and $T^{k-1}(E_k) \subset E_1$ for all $k$.

Note that in the setting where $I$ is an interval and $T$ is given piecewise by M\"obius transformations,  then also  $U$ is so given.

\medskip

We now identify the natural extension of the first pointwise expansive power of $T$ as an induced system within the natural extension of $T$.
 \begin{Prop}\label{p:expansiveMapNatIsInducedSys}     Suppose $T$ as above is eventually expansive.  Let now $\mathcal T, \Omega, \mu$ be the usual aspects of the natural extension of the system of $T$, and for each $k$   let $\mathcal E_k \subset \Omega$ be the portion which projects to  $E_k\subset I$.  Set
 \[ \mathcal F = \Omega \setminus \bigcup_{k=2}^{r}\, \bigcup_{j=1}^{k-1}\, \mathcal T^j(\mathcal E_k).\]

 Suppose that there is a full $T$-cylinder contained in $E_1$\,. Then the system on $\mathcal F \subset \Omega$ induced by $\mathcal T$, in other words the first return system of $\mathcal T$ on $\mathcal F$, gives a natural extension for $U$. 
\end{Prop}

 \begin{proof}  We can and do assume that $r>1$.    
  
 Since there is a full cylinder contained in $E_1$, there is positive measure subset of $\Omega$ consisting of points in $\mathcal E_1$ whose $\mathcal T^{-1}$-orbit segments of length $r$ remain in $\mathcal E_1$.  Since this set is contained in $\mathcal F$, we find that $\mathcal F$ has positive measure and furthermore  $\mathcal F \subset \Omega$ projects surjectively to $I$.   Let $\mathcal R: \mathcal F \to \mathcal F$ be the first return map of $\mathcal T$.   Since the $\mathcal E_k$ give a partition of $\Omega$ and $\mathcal T$ is bijective on $\Omega$, the fact that one deletes at least $\mathcal T^{k-1}(\mathcal E_k)$ for $k\ge 2$ shows that $\mathcal F \subset \mathcal T(\mathcal E_1)$.
  
We now show that $\mathcal R$ on each $\mathcal F \cap \mathcal E_k, k \ge 1$ is given by $\mathcal T^k$.    Given the definition of the set $\mathcal F$, certainly within this subset the first return map $\mathcal R$ is given by $\mathcal T^j$ with $j\ge k$.  It suffices to show that $\mathcal T^k(\mathcal F \cap \mathcal E_k) \subset \mathcal F$.    

Consider first $ k \ge 2$. We thus show that for each $2\le i \le r$ and $1\le j <i$ we have 
$\mathcal T^k(\mathcal F \cap \mathcal E_k) \cap \mathcal T^j(\mathcal E_i) = \emptyset$.  We discern three cases.  If $k<j$ then $ \mathcal F  \cap \mathcal T^{j-k}(\mathcal E_i) = \emptyset$ implies the result.  Suppose  $j<k$; if  $x \in T^{k-j}(E_k)$ then  we must have that $|(T^j)'(x)|>1$ but since $j<i$ this implies that $x \notin E_i$.   Should  $j=k$  then  since  the $\mathcal T^j(\mathcal E_i), 1\le i \le r$ partition $\Omega$; we find that $i=k$ but $j<i=k$ allows the previous argument to apply. 
Therefore,   $\mathcal R$ on $\mathcal F \cap \mathcal E_k, k \ge 2$ is given by $\mathcal T^k$.  

Finally, if $k=1$ there are two cases.  When $j=1$, since $\mathcal E_1 \cap \mathcal E_i = \emptyset$ the result holds.    Otherwise, $1<j<i$ and $\mathcal F  \cap \mathcal T^{j-1}(\mathcal E_i) = \emptyset$ implies the result.   Thus, we have shown that $\mathcal R$ on each $\mathcal F \cap \mathcal E_k, k \ge 1$ is given by $\mathcal T^k$.

It is now clear that the system of $\mathcal R$ on $\mathcal F$ is an extension of the system of $U$ on $I$.  (Note that $\mu$ and $\nu$ remain the respective invariant measures, up to normalization.)  Since $\mathcal T$ on $\Omega$ gives the natural extension of $T$ on $I$,  the induced system on $\mathcal F$ inherits in particular the property of being a minimal extension.
\end{proof}

 We can now show that under certain conditions, an interval map can be appropriately accelerated so as to define an interval map which is naturally given by a geometric system. 
\begin{proof} [{\bf Proof of Theorem~\ref{t:expansiveUisFactorOfFlow}}] Arguing as in the proof of Corollary~\ref{c:oneFactorOfFlow},  we need only show that $h(U)\, \mu(\mathcal F) = h(T)\, \mu(\Omega)$. But, this is immediate from Abramov's formula.
\end{proof}

Combining the above results and arguments yields Theorem~\ref{t:conjConfirmedNis3}.
 
\subsection{Computational confirmation for $4\le n \le 12$}\label{ss:conjComputation4to12}   In view of Subsection~\ref{ss:RockTheVolume},  we discuss $\int_{\Omega_{0}} \tau_{0}(x)\,  d \mu$ for fixed $n$.
\smallskip

Let  $s: (x,y) \mapsto (-y, -x)$  
  and note that  for any reasonable region $\mathscr R$ in the plane   one has
\[  \int_{\mathscr R}\, -2 \log|xy| d\mu =  \int_{\mathscr R}\, -2 \log|x| d\mu  +  \int_{s(\mathscr R)}\, -2 \log|x| d\mu.\]
Now, for each $n$, the domain  $\Omega_0$ is symmetric under $s$, see \cite{CaltaSchmidt}.    We could thus integrate $ -2 \log|xy|$ with respect to  $\mu$ over half of the domain; indeed, it seems easier for Mathematica to evaluate integrals of this type.    A further improvement is reached by noting that   each $\Omega_0$ contains the square $\mathcal S = [-1,0]\times [0,1]$; using Lemma~\ref{l:useMtwoNthreeAlpZero}, the integral of  $ -2 \log|x| d\mu$ over $\mathcal S$ equals $\pi^2/3$.  
We thus let $\mathcal I(n)$ denote the integration of 
$ -2 \log|xy| d\mu$  over the portion of $\Omega_0$ projecting to  $x$-values less than $-1$.

Using Mathematica, we numerically approximated the values 
$Z(n) = \text{vol}_n - (\,\mathcal I(n)+ \pi^2/3)$ for $4\le n \le 12$.  Clearly, each $Z(n)$ is a real number, and our conjecture states that each $Z(n)$ equals zero.  The approximated values of $Z(n), 4 \le n \le 6$ were real numbers of absolute value less than $10^{-8}$.   The approximated values of $Z(n), 7 \le n \le 11$ were complex (!)  numbers whose real parts were of absolute value less than   $10^{-8}$  and imaginary parts less than $10^{-12}$.  The same code run for $n=12$ gave a complex number of large norm; however, upon using virtually the same code  to separately  recalculate one of constituent integrals, the reported value of $Z(12)$ was a complex number of  real part less than $10^{-4}$ and imaginary part less than $10^{-6}$.   

We have the impression that as $n$ increases the computational difficulty of evaluating $\mathcal I(n)$ increases due to both the increasing number of constituent integrals and also the increasing algebraic degree of the numbers defining the limits of integration.   

\subsection{Remarks on the conjecture}  
 In the setting of {\em expansive} piecewise M\"obius interval maps appropriately given by elements of some Fuchsian group that uniformizes a surface (or orbifold) of finite hyperbolic area,  one can reasonably expect a result as conjectured above.   Indeed,     there is a long history of these matters,  going back at least to E.~ Artin's \cite{Art}   coding by way of regular continued fractions of the geodesic flow on the unit tangent bundle of the hyperbolic orbifold uniformized by the modular group; some 60 years later   Series \cite{Series} determined a cross section to this flow for which the first return map  has as a factor 
 the interval map defining these continued fractions.
For more recent work see  \cite{AF, KatokUgarcoviciStructure}    and   in particular the overview \cite{KatokUgarcoviciBullAMS}.   We ourselves learned of these matters from  Arnoux,  see \cite{ArnouxCodage} and also the discussion of `Arnoux's transversal' in \cite{ArnouxSchmidtCross}.      

Some caution is necessary.  First, an example of an expansive piecewise M\"obius interval map, given by elements of the modular group $\text{PSL}(2, \mathbb Z)$, with a nice planar natural extension but which is not of ``first return type" is given in \cite{ArnouxSchmidtCross}.   To be of first return type, see \cite{ArnouxSchmidtCommeUneFraction}, a map must be such that the induced function on Arnoux's cross section in the unit tangent bundle of the surface/orbifold agrees with the first return map of the geodesic flow.   Under this property, again see \cite{ArnouxSchmidtCommeUneFraction}, one has an equality of the form of our conjectures, but in the setting of expansive maps.

Various approaches to proving the conjectures naturally present themselves.    First, by direct integration.   The integrals which appear are naturally related to dilogarithmic functions and hence are notoriously difficult to evaluate exactly.    See  \cite{FriedSymDynTri} for  discussion of the appearance of the dilogarithm in a related setting.     Secondly, one could attempt to prove that each of our continued fractions gives a coding for the geodesic flow on the corresponding orbifold and then argue as in \cite{KatokUgarcoviciStructure}.    Closely related to this, one could follow \cite{MayerStroemberg08} to determine a natural extension for the interval maps as a cross section to the geodesic flow consisting of  unit tangent vectors (mainly) along the boundary of a particular fundamental domain for the Fuchsian group in consideration.   With this in hand,   one could hope to   use Bonahon's \cite{Bonahon} geodesic current formula as in \cite{AbramsKURigidandFlexible} to obtain the integral value.     Thirdly,  given any of our explicit planar natural extensions, Arnoux's transversal provides  an easily determined cross section. When the cross section is lifted to the upper half-plane,  each $(x,y)$ in the planar natural extension is associated to an inward pointing unit tangent vector on a horocycle tangent to $x$.   One could hope to determine if the interval map is of first return type by consideration of the flow in this model.      We would very much like to see a method to pass from this cross section to one consisting of unit tangent vectors along the boundary of some fundamental domain for the Fuchsian group at hand.


\begin{thebibliography}{Thurst88}   
 

\bibitem[AKU]{AbramsKURigidandFlexible} A.~Abrams, Adam, S.~Katok and I.~Ugarcovici,  {\em Rigidity and flexibility of entropies of boundary maps associated to Fuchsian groups}, Math. Res. Rep. 3 (2022), 1--19. 

\bibitem[AF]{AF}  R.~Adler and L.~Flatto, {\em  Geodesic flows, interval maps, and symbolic dynamics},  Bull. Amer. Math. Soc. (N.S.)  25  (1991),  no. 2, 229--334.

 
\bibitem[Ar]{ArnouxCodage} P. Arnoux, {\em Le codage du flot g\'eod\'esique sur la surface modulaire}, 
Enseign. Math. (2) 40 (1994), no. 1-2, 29--48.


\bibitem[AS]{ArnouxSchmidtCross} P. Arnoux and T. A. Schmidt, {\em Cross sections for geodesic flows and $\alpha$-continued fractions},  Nonlinearity 26 (2013), 711--726.

\bibitem[AS2]{ArnouxSchmidtCommeUneFraction}  \bysame, {\em
Commensurable continued fractions},  
Discrete Contin. Dyn. Syst. 34 (2014), no. 11, 4389--4418. 

\bibitem[Art]{Art}  E. Artin, {\em Ein mechanisches System mit quasi-ergodischen Bahnen}, Abh. Math. Sem. Hamburg 3
(1924)
170--175 (and {\em Collected Papers}, Springer-Verlag, New
York,1982,  499--505).   
 
\bibitem[Bon]{Bonahon} F.~Bonahon, {\em The geometry of Teichmüller space via geodesic currents}, Invent. Math. 92 (1988), no. 1, 139--162.
  
\bibitem[CKS]{CaltaKraaikampSchmidt}
 K. Calta, C. Kraaikamp and T.~A.~Schmidt, {\em Synchronization is full measure for all  $\alpha$-deformations of an infinite class of continued fraction transformations},  Ann. Sc. Norm. Super. Pisa Cl. Sci.(5)
Vol. XX (2020), 951--1008.
 
\bibitem[CKS2]{CKStoolsOfTheTrade}  
\bysame,  {\em Proofs of ergodicity using planar extensions}, \href{https://arxiv.org/abs/2208.03807}{preprint 2022, 25 pp.: Arxiv/2208.03807}.

\bibitem[CS]{CaltaSchmidt} K. Calta and T. A. Schmidt, {\em  Continued fractions for a class of triangle groups},  J. Austral. Math. Soc., 93 (2012) 21--42.
 
\bibitem[CIT]{CIT} C. Carminati, S. Isola and G. Tiozzo,
{\em Continued fractions with $\text{SL}(2, \mathbb Z)$-branches: combinatorics and entropy},
 Trans. Amer. Math. Soc. 370 (2018), no. 7, 4927--4973. 
 
\bibitem[CT]{CT} C. Carminati  and G. Tiozzo,
{\em  A canonical thickening of $\mathbb Q$ and the entropy of $\alpha$-continued fraction transformations}, Ergodic Theory Dynam. Systems 32 (2012), no. 4, 1249--1269.
 
\bibitem[CT2]{CT2} \bysame, {\em Tuning and plateaux for the entropy of $\alpha$-continued fractions},  Nonlinearity 26 (2013), no. 4, 1049--1070.

\bibitem[DKS]{DKS} K. Dajani, C. Kraaikamp, W.~Steiner, {\em Metrical theory for $\alpha$-Rosen fractions},  J. Eur. Math. Soc. (JEMS) 11 (2009), no. 6, 1259--1283.
\bibitem[DKU]{DenkerKellerUrbanski} M. Denker, G. Keller and M. Urba\'nski, {On the uniqueness of equilibrium states for piecewise monotone mappings}, Studia Math. 97 (1990), no. 1, 27--36.

\bibitem[F]{FriedSymDynTri}   D.~Fried, {\em Symbolic dynamics for triangle groups},  Invent. Math. 125 (1996), no. 3, 487--521.  

\bibitem[KU]{KatokUgarcoviciBullAMS} 
S.~Katok, I.~Ugarcovici, 
{\em Symbolic dynamics for the modular surface and beyond},
Bull. Amer. Math. Soc. (N.S.) 44 (2007), no. 1, 87--132. 

\bibitem[KU2]{KatokUgarcoviciStructure} 
\bysame,
{\em Structure of attractors for (a,b)-continued fraction transformations},
J. Mod. Dyn. 4 (2010), no. 4, 637--691. 

\bibitem[KU3]{KatokUgarcovici12} \bysame
{\em Applications of (a,b)-continued fraction transformations}
Ergodic Theory Dynam. Systems 32 (2012), no. 2, 755--777. 
  
\bibitem[K]{KraaikampNewClass} C.~Kraaikamp, \emph{A new class of
continued fraction expansions}, Acta Arith.\ 57 (1991), no.\ 1,
1--39.

\bibitem[KSSm]{KraaikampSchmidtSmeets} C. Kraaikamp, T. A. Schmidt and  I. Smeets,
{\em Natural extensions for $\lambda$-Rosen continued fractions}, 
J. Math. Soc. Japan  62, No. 2 (2010), 649--671
 
\bibitem[L]{L}  P.~L\'evy, {\em Sur le d\'eveloppement en fraction continue d'un nombre choisi au
hasard},  Compositio Math. 3 (1936), 286--303.

\bibitem[KSS]{KraaikampSchmidtSteiner} C. Kraaikamp, T. A. Schmidt and  W. Steiner, {\em Natural extensions and entropy of $\alpha$-continued fractions}, Nonlinearity 25 (2012), no. 8, 2207--2243.
 
\bibitem[LM]{LuzziMarmi08}  L.~Luzzi and S.~Marmi,  \emph{On the entropy of Japanese
continued fractions}, Discrete Contin.\ Dyn.\ Syst.\ \textbf{20}
(2008), no.~3, 673--711.

\bibitem[M]{Manning} 
\newblock A. Manning, 
{\em Dynamics of geodesic and horocycle flows on surfaces of constant negative curvature}, 
\newblock  in {Ergodic theory, symbolic dynamics, and hyperbolic spaces (Trieste, 1989)},  eds.  T.~Bedford, M. Keane and C.~Series,  Oxford Sci. Publ., Oxford Univ. Press, New York, (1991),  71--91.

\bibitem[MS]{MayerStroemberg08} 
D.~Mayer and F.~Str\"omberg, {\em Symbolic dynamics for the geodesic flow on Hecke surfaces}, J. Mod. Dyn. 2 (2008), no. 4, 581--627. 

\bibitem[N]{N} H.~Nakada, \emph{Metrical theory for a class of continued fraction
transformations and their natural extensions}, Tokyo J.\ Math.\ 4
(1981), 399-426.

\bibitem[N2]{NakadaLenstra} \bysame, \emph{On the Lenstra constant associated to the Rosen
continued fractions},  J. Eur. Math. Soc. (JEMS) 12 (2010), no.\ 1, 55--70.

\bibitem[NIT]{NakadaItoTanaka} H.~Nakada, S.~Ito, and S.~Tanaka, \emph{On the invariant
measure for the transformations associated with
some real continued-fractions}, Keio Engrg.\ Rep.\ 30 (1977), no.\
13, 159--175.

\bibitem[R]{Rohlin61} V.~A.~Rohlin, {\em Exact endomorphisms of a Lebesgue space},  Izv. Akad. Nauk SSSR Ser. Mat. 25 (1961), 499--530;
Am. Math. Soc. Transl. Ser. 2, 39 (1964)1--36 (Engl. transl.).
 
\bibitem[S]{Series} 
\newblock C. Series,
\newblock {\em The modular surface and continued fractions},
\newblock  {J. London Math. Soc. (2)}  \textbf{31}  (1985), 69--80.
 
\end{thebibliography}
\end{document}